\documentclass[aos, preprint]{imsart}

\RequirePackage[OT1]{fontenc}
\RequirePackage{graphicx,color,amssymb,amsmath,amsthm,mathrsfs}
\RequirePackage[numbers]{natbib}
\RequirePackage[colorlinks,citecolor=blue,urlcolor=blue, hypertexnames=false]{hyperref}
\RequirePackage{tikz}
\RequirePackage{bbm}
\RequirePackage{algpseudocode}
\RequirePackage{algorithm2e}
\usepackage{caption}

\allowdisplaybreaks

\startlocaldefs
\newtheorem{definition}{Definition}[section]
\newtheorem{thm}[definition]{Theorem}
\newtheorem{lemma}[definition]{Lemma}
\newtheorem{proposition}[definition]{Proposition}

\newtheorem{remark}[definition]{Remark}

\def\R{\mathbb{R}}
\def\C{\mathbb{C}}
\def\N{\mathbb{N}}
\def\Q{\mathbb{Q}}

\def\1{\mathbbm{1}}

\def\E{\mathbb{E}}
\def\Pr{\mathbb{P}}
\endlocaldefs

\makeatletter
\renewcommand\tableofcontents{%
    \@starttoc{toc}%
}
\makeatother

\begin{document}

\begin{frontmatter}

\title{Sharp adaptive similarity testing with pathwise stability for ergodic diffusions}
\runtitle{Similarity testing for ergodic diffusions}


\author{\fnms{Johannes} \snm{Brutsche}\ead[label=e1]{johannes.brutsche@stochastik.uni-freiburg.de}}
\author{\fnms{Angelika} \snm{Rohde}\ead[label=e2]{angelika.rohde@stochastik.uni-freiburg.de}}


\runauthor{J. Brutsche and A. Rohde}

\affiliation{Albert-Ludwigs-Universit\"at Freiburg}

\begin{abstract}
Within the nonparametric diffusion model, we develop a multiple test to infer about \textit{similarity} of an unknown drift $b$ 
to some reference drift $b_0$: At prescribed significance, we simultaneously identify those regions where violation from similiarity occurs, without a priori knowledge of their number, size and location. This test is shown to be minimax-optimal and adaptive. At the same time, the procedure is robust under small deviation from Brownian motion as the driving noise process. A detailed investigation for fractional driving noise, which is neither a semimartingale nor a Markov process, is provided for Hurst indices close to the Brownian motion case.
\end{abstract}



\end{frontmatter}


\numberwithin{equation}{section}
\addtocontents{toc}{\protect\setcounter{tocdepth}{0}}

\section{Introduction}

\paragraph{Motivation}
In many areas such as physics or financial mathematics, numerous time-continuous dynamics are modeled by a diffusion process. Moreover,
diffusions arise frequently as scaling limits of Markov chains and jump processes. For instance, the stochastic SIS model in epidemiology reveals an Ornstein--Uhlenbeck process as a scaling limit.Therefore, stochastic differential equations have become an important subject of investigation in nonparametric statistics.

In the nonparametric scalar diffusion model $dX_t = b(X_t) dt + \sigma dW_t$ with driving noise $W$ being a Brownian motion and $\sigma>0$, the aim of our paper is to infer about \textit{similarity} of an unknown drift $b$ to some reference drift $b_0$ based on a continuous record $(X_t)_{t\in [0,T]}$ of observations.
Here, a drift $b$ is said to be similar to $b_0$ at tolerance $\eta\geq 0$ within some interval $I$ if
\[ b_0(x)-\eta \leq b(x)\leq b_0(x) +\eta\quad \textrm{ for all } x\in I.\]
If $b$ and $b_0$ are not similar at tolerance $\eta$, there are numerous possible regions of deviation. Thus, it is of major interest to the statistician to identify with a certain confidence \textit{where} violation from similarity occurs as illustrated in Figure~\ref{Fig_1} in the supplementary material in Section~\ref{App_SIS}.

To this aim, we develop an efficient multiple test to simultaneously identify those regions where violation from similarity occurs, without a priori knowledge of their number, size and location.
Once regions of deviation are identified at level of significance $\alpha$, the composite null hypothesis of the \textit{similarity testing problem}
\begin{align}\label{eq: Testing_problem}
H_0:\ \sup_{x\in I}|b(x)-b_0(x)| \leq \eta \quad \textrm{ versus }\quad H_1:\ \sup_{x\in I} |b(x)-b_0(x)|>\eta
\end{align} 
is rejected. The relevance of tolerant testing, such as hypotheses of the form \eqref{eq: Testing_problem} for $\eta>0$, has meanwhile been widely acknowledged in many different fields of statistical inference such as financial, medical, pharmaceutical or environmental statistics, see \cite{Altman}, \cite{Buecher}, \cite{Dette}, \cite{Fogarty}, \cite{McBride}, \cite{Romano} and \cite{Wellek} including references cited therein. Note that testing $b=b_0$ may miss the point in many applications because 'sufficiently close' is equally convenient. In particular, testing for similarity avoids the consistency problem mentioned in~\cite{Berkson}, i.e. that any consistent test will detect arbitrary small deviations in the drift if the amount of data is sufficiently large.
We emphasize that tolerant testing is not model specific but in constrast to the Gaussian white noise or regression model, moving from the simple to the composite hypothesis in the diffusion model requires to invent conceptionally new ideas. 
The reason is that constructing a stochastically dominating random variable of the test statistic on the composite null cannot be reduced to the boundary cases of the hypothesis any longer (even not asymptotically) which makes the calibration of the test to the significance level $\alpha$ a highly challenging task. From this point of view, the diffusion model serves as one of the simplest (central) nonparametric statistical models where this obstacle occurs when testing for similarity of the model parameter. 

Including the tolerance $\eta$ in the null hypothesis \eqref{eq: Testing_problem} urges likewise to ask for stability properties of the statistical methodology with respect to \textit{any} small deviation from the idealized model assumption. Such stability of algorithms and sophisticated statistical inference procedures is of increasing importance: it justifies to employ them even for data that is known to be not generated by the idealized model as long as the true model is sufficiently close to the idealized model in a suitable sense. This is of particular relevance if the true model is extremely complex and no efficient statistical methodology is available at present. For the statistical analysis of stochastic process data given by a stochastic differential equation this includes deviation from the driving noise assumption which may be extremely subtle to describe. The crucial obstacle in the diffusion model is that the likelihood ratio involves stochastic integration against the diffusion but the stochastic It\^{o} integral is exclusively given for semimartingales. Thus, on the one hand, proving stability of our inference procedure in the above sense first raises the question about existence of a continuation which is given for arbitrary continuous paths as pioneered in~\cite{Diehl/Friz} for parametric maximum likelihood estimation. On the other hand, deviating from the semimartingale context drastically complicates the solution of the similiarity testing problem as many tools from stochastic analysis are missing. In case of stability, if the true model is sufficiently close to our idealized diffusion model, one might then use our procedure as this is so far the only way to address the similarity testing problem.

We conclude this paragraph with exemplarily illustrating the interplay of tolerant testing in a diffusion approximation and the stability property. As mentioned above, the stochastic SIS model, widely used in epidemiology, possesses an Ornstein--Uhlenbeck process as a scaling limit. The SIS model, however, is by far too simplified in order to capture the full dynamics and can adequately describe the data at most within a certain tolerance.
Imposing that the (suitably rescaled) data generating jump process also possesses a diffusion approximation suggests to develop the tolerant testing methodology in the technically much more convenient diffusion model with the Ornstein--Uhlenbeck limit as reference in the null hypothesis. Given stability of the similarity test with respect to the diffusion approximation then justifies to apply it to the original data. Note that in this example, such stability has to cover approximation schemes with laws singular to the one of the diffusion limit (more details and a graphical illustration are provided in Section~\ref{App_SIS}).

\paragraph{Main contributions}
For $\eta>0$, the null in \eqref{eq: Testing_problem} is a composite hypothesis, and our goal is to construct for any significance level $\alpha\in (0,1)$ a multiple testing procedure $\phi_\eta$ to infer about local deviations from similarity under the constraint $\sup_{b\in H_0} \E_b\left[ \phi_\eta\right] \leq \alpha.$
To this aim, we employ the multiscale approach that has been proven to be successful in a large variety of scenarios (cf. \cite{Datta}, \cite{Walther}, \cite{Duembgen/Spokoiny}, \cite{Koenig}, \cite{Proksch}, \cite{Rohde_2} and  \cite{Rohde}), though neither including stochastic differential equations nor stability considerations. 
While composite hypotheses in the context of multiscale testing have been 
studied in situations where the boundary of the hypothesis is least favourable in the sense of stochastic ordering (cf. \cite{Walther}, \cite{Duembgen/Spokoiny}), the situation for the composite similarity hypothesis for the drift of ergodic diffusions is substantially more intricate. 
Although our multiscale test statistic is motivated by the idea of simultaneously testing $b\leq b_0+\eta$ and $b\geq b_0-\eta$ pointwisely, 
there is no evidence that the boundary cases are least favourable for the null hypothesis of similarity. Indeed, the stochastic order relation required for this purpose may be missing even for the corresponding local likelihood ratio statistics. The reason is that their distribution does not only depend on local values of the drift $b$, but on the entire drift function via the invariant density.\\

Our main contributions are the following: 

\begin{itemize}
\item[(i)] Based on a multiscale statistic in the spirit as described above and for any significance level $\alpha\in (0,1)$, we construct a threshold level such that the resulting test $\phi_T^\eta$ for the testing problem \eqref{eq: Testing_problem} satisfies
\begin{align}\label{eq: level_asymp}
\limsup_{T\to\infty} \sup_{b\in H_0} \E_b\left[ \phi_T^\eta\right] \leq \alpha,
\end{align} 
where $T$ denotes the time horizon of the diffusion's observation. Note that \eqref{eq: level_asymp} is a substantially stronger statement than the pointwise relation $\limsup_{T\to\infty} \E_b[\phi_T^\eta]\leq\alpha$ for all $b\in H_0$. For the derivation of \eqref{eq: level_asymp}, we construct a random variable $Y_\eta$
\begin{itemize}
\item that provably dominates the test statistic \textit{uniformly} on the similarity hypothesis in stochastic order asymptotically and 
\item whose distribution depends continuously on the level $\eta$ of similarity, and $Y_0$ equals the limiting distribution of the test statistic under the simple null hypothesis, i.e. $\eta=0$ in \eqref{eq: Testing_problem}.
\end{itemize}

The cornerstone for the construction of $Y_\eta$ is the identification of the weak limit of the multiscale test statistic \textit{uniformly} in $b\in H_0$. Whereas weak limit results for supremum statistics like ours have been derived in various settings (cf. \cite{Walther}, \cite{Proksch}, \cite{Rohde_2} and \cite{Rohde}), the additional uniformity in the drift parameter accounting for the composite null in \eqref{eq: Testing_problem} is new and considerably more involved on a mathematical level.

\item[(ii)] We prove optimality and adaptivity for the similarity test in the minimax sense, as introduced in \cite{Ingster1} and \cite{Ingster2}.
We exemplarily consider the case of alternatives belonging to some Hölder class $\mathcal{H}(\beta, L)$ where deviations are measured in weighted supremum norm which is the equivalent to weighted risk definitions in sharp adaptive drift estimation like \cite{Dalalyan} or \cite{Strauch_2}. Our similarity test is shown to be rate-optimal in the minimax sense, adaptive in both the unknown parameters $\beta$ and $L$, optimal in the constant for the regime $\beta\leq 1$ and here, even sharp adaptive in $L$. 
The hypotheses construction in the proof of the lower bound involves a fixed point problem as the drift itself appears in the invariant density which pops up in the deviation measure between null and alternative.

\item[(iii)] We prove stability properties of our test with respect to deviation from the model assumption. As our test statistic for $\phi_T^\eta$ involves a stochastic integral which is not even defined for data that is not given by a semimartingale a priori, we introduce in Subsection~\ref{SubSec_cont} a pathwise continuation of the statistic as a function of the data that is shown to be continuous with respect to the topology of uniform convergence. In Subsection~\ref{SubSec_fractional_model}, we then address the problem of stability for the particular example of fractional diffusion models where the driving Brownian motion is replaced by a fractional Brownian motion with Hurst index $H\in (0,1)$. The reason for this choice is that on the one hand, fractional diffusions are neither semimartingales nor Markov processes for $H\neq 1/2$, while on the other hand a minimax optimal similarity test for the fractional model is at present out of reach. 
Although most of the present literature focuses on the case where $H$ clearly deviates from $1/2$, investigating the fractional diffusion model for $H\rightarrow 1/2$ has been iniciated in \cite{Diehl/Friz} for parametric maximum likelihood estimation. We prove that the test statistic built from observations in the fractional diffusion model has strong performance properties as the fractional driving noise approaches Brownian motion in the following sense:
\begin{itemize}
\item The test is \textit{uniformly} over the hypothesis of similarity of approximate level $\alpha$, i.e. (slightly simplified)
\begin{align}\label{eq: level_alpha_H}
\limsup_{T\to\infty} \limsup_{H\to\frac12}\sup_{b\in H_0} \E_b^H\left[\phi_T^\eta \right] \leq \alpha,  
\end{align} 
where $\E_b^H$ denotes the expectation when applied to fractional diffusion with Hurst index $H$ and drift $b$.

\item We prove that minimax optimality is preserved in a certain sense as the fractional driving process approaches Brownian motion. This relies on $L^1(\Pr)$-convergence of likelihood ratios of the fractional diffusion model to those of the standard model and is based on (deterministic) fractional calculus (cf. \cite{Samko}).
\end{itemize}

\end{itemize}

The article is organized as follows. Model description and notation are given in Section~\ref{Sec_Notation}. In Section~\ref{Sec_simple}, we tackle the technical difficulties of constructing a powerful multiscale statistic for the simple null $b=b_0$ which are due to the context of ergodic diffusions. 
Section~\ref{Sec_similarity} contains the results described in (i), including the development of a multiscale test for the composite similarity hypothesis. 
Power properties as summarized in (ii) are given in Section~\ref{Sec_Power}.
The stability results (iii) are content of Section~\ref{Sec_Stability}. 
In Section~\ref{Sec_weak}, a route of proof of the main result in (i) is presented. Here, the crucial limit theorem for the supremum statistic with weak convergence \textit{uniformly} over the null hypothesis is stated.
An outline of the proof of the lower bound with the fixed point argument of (ii) is presented in Section~\ref{Sec_supp} and an outlook to the multidimensional case is given in Section~\ref{Sec_outlook}.
All proofs as well as an extended simulation study are deferred to the supplement.

\section{Model assumptions and notation}\label{Sec_Notation}

For the problem of similarity testing, we assume throughout that a continuous record of observations $(X_t)_{t\in [0,T]}$ is available, where $X$ denotes an It\^{o} diffusion satisfying the one-dimensional homogeneous stochastic differential equation (SDE) of the form
\begin{align}\label{eq: SDE}
dX_t = b(X_t) dt + \sigma dW_t,\quad X_0 = \xi,
\end{align} 
with drift $b:\R\rightarrow\R$, $\sigma>0$, $W=(W_t)_{t\geq 0}$ a standard one-dimensional Brownian motion and initial condition $\xi$ independent of $W$. In this setup, the diffusion coefficient is identifiable using the semimartingale quadratic variation of the diffusion and the problem of its estimation does not arise. This remains also true if $\sigma$ was replaced by $\sigma(X_t)$ in \eqref{eq: SDE}. For conciseness and clarity in the representation, we however restrict attention to constant diffusion coefficient. The extension of our results to a non-constant diffusion coefficient is straightforward, except for Section~\ref{Sec_Stability} because a suitable notion of stochastic integrals is then needed to even define the fractional SDE.
For arbitrary but fixed constants $A,\gamma, \sigma >0$ and $C\geq 1$, the  drift $b$ belongs to 
\begin{align*}
\Sigma(C,A,\gamma, \sigma) &:=\left\{ b\in \textrm{Lip}_\textrm{loc}(\R):\ |b(x)|\leq C(1+|x|)\ \forall x\in\R \textcolor{white}{\frac12} \right.\\
&\hspace{3cm}\left.\textrm{ and } \frac{b(x)}{\sigma^2}\textrm{sign}(x)\leq -\gamma\ \forall |x|\geq A\right\}.
\end{align*}
Here, $\textrm{Lip}_\textrm{loc}(\R)$ denotes the local Lipschitz functions on $\R$, see Appendix~\ref{App_prelim}. The first two constraints ensure that the SDE \eqref{eq: SDE} has a unique strong solution and the last one is a typical assumption to guarantee ergodicity and the existence of an invariant measure. For each $b\in\Sigma(C,A,\gamma,\sigma)$ we denote this invariant measure by $\mu_b$ and it is a classical result (cf. \cite{Kutoyants}, Theorem~$1.16$) that it admits the invariant probability density
\[ q_b(x) := \frac{1}{C_{b,\sigma}} \exp\left( \int_0^x \frac{2b(y)}{\sigma^2} dy\right) \quad\textrm{ for all } x\in\R,\]
with normalizing constant $C_{b,\sigma}$.
For $x<0$ the integral should be read as $\int_0^x f(y)dy = -\int_{x}^0 f(y) dy$.
For ease of representation, we assume that $\xi\sim\mu_b$ such that $X$ is stationary and ergodic. Extensions are possible, see Remark~\ref{remark_fixed_point}.
Subsequently, we denote by $\Pr_b$ the law of $X$ satisfying \eqref{eq: SDE} with drift $b$ and by $\E_b$ the corresponding expectation.

For any set $I\subset \R$ and bounded function $f:I\rightarrow \R$ we denote
\[ \|f\|_I := \sup_{z\in I} |f(z)|.\]
For any compact set $K\subset\R$ we denote by $\mathcal{C}(K)$ the set of continuous functions $f: K\rightarrow\R$. Unless stated otherwise, we denote by $\|\cdot\|_{L^2}$ the $L^2$-norm with respect to the Lebesgue measure on $\R$.

\section{The case of the simple null $b=b_0$}\label{Sec_simple}

Although our main contribution is the development of a test for similarity, we start with the hypothesis $b=b_0$ as a preliminary step before the similarity test is presented in the next section. 
Besides being of independent interest, this presentation comprises solutions to technical difficulties that arise merely from the setting of ergodic diffusions in the context of multiscale testing - and not those attributed to the composite null.\\ 
The precise testing problem we address in this section is given for some $b_0\in\Sigma(C/2, A, \gamma,\sigma)$ by
\begin{align}\label{eq: H_0-simple}
H_0:\ \|b-b_0\|_{[-A,A]} =0 
\end{align} 
versus one of the following alternatives:
\begin{align*}
&H_{\neq}:\ \left\{\|b-b_0\|_{[-A,A]}>0 \right\} \cap\Sigma(C,A,\gamma,\sigma), \\
&H_{>}:\ \left\{\exists x\in [-A,A]:\ b(x)-b_0(x)>0 \right\} \cap\Sigma(C,A,\gamma,\sigma), \\
&H_{<}:\ \left\{\exists x\in [-A,A]:\ b(x)-b_0(x)<0 \right\} \cap\Sigma(C,A,\gamma,\sigma).
\end{align*}  
The intersection with $\Sigma(C,A,\gamma,\sigma)$ accounts for the fact that we only test against ergodic diffusions. In a first step, we consider $H_0$ against the two-sided alternative $H_{\neq}$. Afterwards in Subsection \ref{SubSec_one-sided} we consider the alternatives $H_<$ and $H_>$.

\subsection{Testing the two-sided alternative}
Suppose that we want to test a simple drift hypothesis $b_0$ against a simple alternative $b_1$. Then by the Neyman-Pearson-Lemma, an optimal test is given by the likelihood ratio statistic $d\Pr_{b_1}/d\Pr_{b_0}(X)$, in our particular case given by means of Girsanov's theorem as 
\[ \frac{d\Pr_{\xi_1}}{d\Pr_{\xi_0}}(X_0)\exp\left( \int_0^T \frac{b_1(X_s) - b_0(X_s)}{\sigma^2} dX_s - \frac12\int_0^T \frac{b_1(X_s)^2 - b_0(X_s)^2}{\sigma^2} ds\right).\]
When moving on to the composite alternative $\{\|b-b_0\|_{[-A,A]}>0\}$ which can be represented as
\[ \bigcup_{\delta>0}\bigcup_{y \in [-A,A]} \{ |b(y) -b_0(y)|\geq \delta\}, \]
alternatives of the form $b=b_0+K_{y}$ for some localized deviation $K_y$ with $K_y(y)=\delta$ seem to be hardest to detect for each set of the union. If some regularity of the alternative $b$ is imposed, a deviation $|b(y)-b_0(y)|=\delta$ implies that $|b(x)-b_0(x)|\neq 0$ for all $x$ within some neighborhood of $y$ as well. The size of this neighborhood depends on the regularity of the alternative which is typically unknown.
The idea is now to develop a multiple test in the spirit of \cite{Duembgen/Spokoiny} that simultaneously tests all locations with likelihood ratio statistics of localized deviations 
\[ K_{y,h}(x) := K\left(\frac{x-y}{h}\right), \quad x\in\R, \]
with different scaling parameters $h>0$. As our approach combines standardized local $\log$-likelihood ratio statistics, the particular value $\delta$ will cancel out.


\subsubsection*{Suitable standardization of the local likelihood statistics}
Omitting the initial values, the $\log$-likelihood ratio of local deviation $b=b_0+K_{y,h}$ and $b_0$ is given by
\[  \frac{1}{\sigma^2}\int_0^T K_{y,h}(X_s) dX_s - \frac{1}{2\sigma^2} \int_0^T \left( 2b_0(X_s) K_{y,h}(X_s) + K_{y,h}(X_s)^2 \right) ds. \]
For our construction of the multiple test, standardization under $\Pr_{b_0}$ is required. Whereas an additive correction for centering under $\Pr_{b_0}$ is obvious, we do not divide by the standard deviation of the stochastic integral $\sigma^{-1}\int_0^T K_{y,h}(X_s) dW_s$ for normalizing the variance, but choose its random analogue, the square root of its quadratic variation, which is purely data dependent. Thus, the standardized local $\log$-likelihood statistic is given by
\begin{align}\label{eq: Psi}
\Psi_{T,y,h}^{b_0}(X) := \frac{\int_0^TK_{y,h}(X_s) dX_s - \int_0^T K_{y,h}(X_s) b_0(X_s) ds}{\sigma\sqrt{ \int_0^T K_{y,h}(X_s)^2 ds}},
\end{align} 
where $\Psi_{T,y,h}^{b_0}(X):=0$ if the denominator equals zero. The deeper reason behind normalizing with the quadratic variation is that it
provides a suitable standardization of the martingale part $\int_0^T K_{y,h}(X_s)dW_s$ for \textit{any} drift $b$. This enables us to attain efficiency when moving on to the construction of a multiple test for the \textit{composite} null hypothesis of similarity in Section~\ref{Sec_similarity}.
Note that the numerator in \eqref{eq: Psi} is a martingale  under $\Pr_{b_0}$.

\subsection*{Developing the multiple test}
The following result shows that the local statistics $\Psi_{T,y,h}^{b_0}(X)$ in \eqref{eq: Psi} can be combined for all $(y,h)$ within
\begin{align}\label{eq: cT}
\mathcal{T} := \left\{ (y,h) \mid h\in (0,A] \textrm{ and } y\in [-A+h, A-h]\right\}\cap\ (\Q\times\Q)
\end{align} 
in a specific way that enables to construct the desired multiple test.

\begin{thm}\label{Multiscale_Lemma}
Let $\Psi_{T,y,h}^{b_0}(X)$ be given as above for a continuous kernel $K$ of bounded variation with support $[-1,1]$ and $\|K\|_{[-1,1]} \leq 1$. Define
\[ \hat\sigma_T(y,h)^2 := \frac1T \int_0^T K_{y,h}(X_s)^2 ds \quad \textrm{ and }\quad \hat\sigma_{T,\max}^2 := \frac1T \int_0^T \1_{[-A,A]}(X_s)^2 ds.\]
Then under $\Pr_{b_0}$ the family (indexed in $T$)
\[\sup_{(y,h)\in\mathcal{T}}\left( \left|\Psi_{T,y,h}^{b_0}(X)\right| - \Upsilon\big(\hat\sigma_T(y,h)^2/\hat\sigma_{T,\max}^2\big)\right)\]
is asymptotically tight, where $0/0:= 0$ in the argument of $\Upsilon(\cdot)$ and 
\begin{align*}
\Upsilon(r) :=(2\log(1/r))^\frac12 \1_{\{r>0\}}.
\end{align*}
\end{thm}

The proof of this result is deferred to Appendix~\ref{App_Multiscale} and relies on a delicate interplay of stochastic analysis and empirical processes. Various variants of identifying the above correction $\Upsilon(\cdot)$ have been established in the theory of multiscale testing, see for example \cite{Walther}, \cite{Duembgen/Spokoiny}  or \cite{Rohde}. We derive a further extension of such results, where in particular the sub-gaussian tail bounds of the local test statistics allow for an additional $\log$-factor, see Theorem~\ref{Multiscale_preparation}. 
With these preliminaries we now define the global test statistic
\begin{align}\label{eq: T_simple}
T_T^{b_0}(X) := \sup_{(y,h)\in\mathcal{T}} \left( |\Psi_{T,y,h}^{b_0}(X)| - \Upsilon\big(\hat\sigma_T(y,h)^2/\hat\sigma_{T,\max}^2\big)\right).
\end{align} 
Theorem \ref{Multiscale_Lemma} ensures that the corresponding quantiles
\begin{align}\label{eq: quantile_simple}
\kappa_{T,\alpha}^{\neq b_0} := \min\left\{ r\in\R\mid \Pr_{b_0}\left( T_T^{b_0}(X)\leq r\right) \geq 1-\alpha\right\} 
\end{align} 
are well-defined and $\limsup_{T\to\infty}|\kappa_{T,\alpha}^{\neq b_0}|<\infty$. An asymptotic power investigation of the resulting test
\begin{align}\label{eq: phi_0}
\phi_T^{b_0}(X) = \1_{\left\{ T_T^{b_0}(X)> \kappa_{T,\alpha}^{\neq b_0}\right\}} 
\end{align}
is given in Section \ref{Sec_Power}.



\subsection{Testing for one-sided alternatives}\label{SubSec_one-sided}

Following the same approach that was taken to construct the test against $H_{\neq}$ we construct a test of $H_0$ versus $H_>$. The same reasoning yields the test statistic
\begin{align*}
T_T^{>b_0}(X) := \sup_{(y,h)\in\mathcal{T}} \left( \Psi_{T,y,h}^{b_0}(X) - \Upsilon\big(\hat\sigma_T(y,h)^2/\hat\sigma_{T,\max}^2\big)\right)
\end{align*} 
which equals that in \eqref{eq: T_simple} except for the missing absolute value signs around $\Psi_{T,y,h}^{b_0}$. They drop out as we only test against $b$ being larger than $b_0$. By Theorem \ref{Multiscale_Lemma} the corresponding quantile 
\[ \kappa_{T,\alpha}^{>b_0} := \min\left\{ r\in\R\mid \Pr_{b_0}\left( T_T^{<b_0}(X)\leq r\right) \geq 1-\alpha\right\} \]
is well-defined.
For testing against $H_<$ the same approach with local alternatives of the form $b = b_0- K_{y,h}$ leads to the test statistic
\begin{align*}
T_T^{<b_0}(X) := \sup_{(y,h)\in\mathcal{T}} \left( -\Psi_{T,y,h}^{b_0}(X) - \Upsilon\big(\hat\sigma_T(y,h)^2/\hat\sigma_{T,\max}^2\big)\right)
\end{align*}
with the quantile $\kappa_{T,\alpha}^{<b_0}$ under $\Pr_{b_0}$ defined correspodingly.
It is important to note that we restrict attention to the \textit{simple} null hypothesis \eqref{eq: H_0-simple}, where the quantiles $\kappa_{T,\alpha}^{\neq b_0}, \kappa_{T,\alpha}^{<b_0}$ and $\kappa_{T,\alpha}^{>b_0}$ of our test statistics have to be determined under $\Pr_{b_0}$. As exemplarily the one-sided alternative $H_>$ is also a reasonable alternative for the composite null $\{ b\leq b_0\}$, the question arises whether
\[ \sup_{b\in\{b\leq b_0\}} \Pr_b\left( T_T^{>b_0}(X) > \kappa_{T,\alpha}^{>b_0}\right) \leq \Pr_{b_0}\left(T_T^{>b_0}(X) > \kappa_{T,\alpha}^{>b_0}\right),  \]
in order to guarantee validity on the composite null $\{b\leq b_0\}$. However, it is totally unclear if this inequality is true, see Section \ref{Sec_similarity}. This missing stochastic order relationship crucially complicates the construction of the similarity test.

\section{The multiscale test for similarity}\label{Sec_similarity}

In this section we construct a test statistic for the similarity testing problem presented in the introduction. In particular, it comprises the derivation of a quantile that ensures validity on the composite null which is a highly non-trivial contribution as indicated in Section \ref{SubSec_one-sided}.\\
Before we start, we fix some additional notation to formulate the testing problem in a mathematical rigorous way. Therefore, letting $\eta\geq 0$ and choosing a reference drift $b_0$ we formulate the composite null hypothesis as
\[ H_0(b_0,\eta) := \left\{ b\mid \| b-b_0\|_{[-A,A]} \leq \eta\right\} \cap \Sigma(C,A,\gamma,\sigma).\]
This composite hypothesis will be tested against its complement within $\Sigma(C,A,\gamma,\sigma)$, i.e.
\begin{align}\label{eq: H1_eta}
H_1(b_0,\eta) := \left\{ b\mid \| b-b_0\|_{[-A,A]} > \eta\right\} \cap \Sigma(C,A,\gamma,\sigma).
\end{align}
In this notation $\eta$ describes the extent of similarity. The smaller $\eta$, the more similar are $b_0$ and $H_0(b_0,\eta)$. In particular, we also cover the case $\eta=0$ where the null is the simple hypothesis $\{b_0\}$ from the preceeding section. 

\subsection*{Construction of the similarity test}

The first insight for constructing the similarity test is that $\{ \|b-b_0\|_{[-A,A]} \leq \eta\} $ may be written as
\begin{align*}
\bigcap_{y\in [-A,A]} \{ b(y)\leq b_0(y)+\eta\} \cap \{ b(y)\geq b_0(y)-\eta \},
\end{align*}
proposing that testing for similarity is the same as testing for two one-sided hypotheses. The construction from Section~\ref{SubSec_one-sided} then suggests the statistic
\begin{align}\label{eq: Motivation_test_similarity}
\begin{split}
&\max\left\{ \Psi_{T,y,h}^{b_0+\eta}(X), -\Psi_{T,y,h}^{b_0-\eta}(X)\right\} \\
&\hspace{2cm} = \max\left\{ \Psi_{T,y,h}^{b_0} - \Lambda_{T,y,h}^\eta(X), -\Psi_{T,y,h}^{b_0}(X) - \Lambda_{T,y,h}^\eta(X)\right\}\\
&\hspace{2cm} = \left|\Psi_{T,y,h}^{b_0}(X)\right| - \Lambda_{T,y,h}^\eta(X)
\end{split}
\end{align}
for testing for deviation at a certain location $y$ with 
\[ \Lambda_{T,y,h}^\eta(X) := \frac{\eta \int_0^T K_{y,h}(X_s) ds}{\sigma\sqrt{ \int_0^T K_{y,h}(X_s)^2 ds}},\]
where $0/0$ is read as zero. Combining those localized statistics in the same way as in Section~\ref{Sec_simple} for $(y,h)\in\mathcal{T}$ given in \eqref{eq: cT} then yields the ansatz
\begin{align}\label{eq: T_eta}
T_T^\eta(X) := \sup_{(y,h)\in\mathcal{T}} \left(\left| \Psi_{T,y,h}^{b_0}(X)\right| - \Lambda_{T,y,h}^\eta(X)- \Upsilon\big(\hat\sigma_T(y,h)^2/\hat\sigma_{T,\max}^2\big) \right)
\end{align} 
as a test statistic for $H_0(b_0,\eta)$ against $H_1(b_0,\eta)$ given in \eqref{eq: H1_eta}.
As our test has to be valid on $H_0(b_0,\eta)$ we have to identify a value $\kappa$ such that
\begin{align}\label{eq: kappa_T}
\sup_{b\in H_0(b_0,\eta)} \Pr_b\left( T_T^\eta(X) >\kappa \right) \leq \alpha.
\end{align} 
In order to guarantee high power on the alternative, $\kappa$ should be chosen as small as possible to fullfill \eqref{eq: kappa_T}. However, the standard approach which is identifying a least favourable case $\tilde{b}\in H_0(b_0,\eta)$, namely
\[ \sup_{b\in H_0(b_0,\eta)} \Pr_b\left( T_T^{\eta}(X) > \kappa\right) \leq \Pr_{\tilde b}\left(T_T^{\eta}(X) > \kappa\right),  \] 
and choosing $\kappa$ as the quantile of $T_T^\eta$ under $\Pr_{\tilde{b}}$ fails: One might expect that least favourable cases are given by the boundary cases $b_0\pm\eta$ in terms of stochastic order, 
but this is totally unclear which can be seen as follows. 
When decomposing $|\Psi_{T,y,h}^{b_0}(X)|$ into
\begin{align}\label{eq: 8.1}
\left| \frac{\int_0^T K_{y,h}(X_s) dW_s}{\sqrt{  \int_0^T K_{y,h}(X_s)^2 ds}} +\frac{\int_0^T K_{y,h}(X_s)(b(X_s) - b_0(X_s)) ds}{\sigma\sqrt{  \int_0^T K_{y,h}(X_s)^2 ds}}  \right|, 
\end{align} 
the construction reveals that for any  $b\in H_0(b_0,\eta)$, the first summand within the absolute value is tight, whereas for $K\geq 0$ the absolute value of the second one is bounded by $\Lambda_{T,y,h}^\eta(X)$ and equal to $\Lambda_{T,y,h}^\eta(X)$ for the boundary cases $b=b_0\pm\eta$. Although this suggests these boundary cases to be least favourable candidates in the sense of stochastic ordering, the distribution of the first summand in \eqref{eq: 8.1} still depends on $b$ -- as a process in $(y,h)$ even asymptotically in terms of finite dimensional distributions -- and no stochastic order relationship as in \cite{Duembgen/Spokoiny} is available. Moreover, the distribution of the argument of $\Upsilon(\cdot)$ in \eqref{eq: T_eta} depends on $b$ as well. Nevertheless, $b_0\pm \eta$ seems to be \textit{close} to the least favourable case. Indeed, we find this closeness to be true \textit{uniformly} in the limit $T\to\infty$.

\begin{thm}\label{worst_case_delta}
Let $T_T^\eta$ be given as in \eqref{eq: T_eta} with a non-negative continuous kernel function $K$ of bounded variation supported in $[-1,1]$ with $\|K\|_{[-1,1]}\leq 1$. 
Furthermore, assume that $b_0\in\Sigma(C/2-\eta,A,\gamma+\eta/\sigma^2,\sigma)$.
Then we have for any $r\in\R$,  
\[ \limsup_{T\to\infty} \sup_{b\in H_0(b_0,\eta)}\Pr_b\left( T_T^\eta(X) \geq r\right) \leq \Pr\left( U_1\vee U_2 + 4\sqrt{A\eta/\sigma^2} \geq r\right),\]
where $U_1\vee U_2$ denotes the (pointwise) maximum of $U_1$ and $U_2$ given by
\begin{align*}
U_1 &:= \sup_{(y,h)\in\mathcal{T}} \left( \left|\frac{\int_{-A}^A K_{y,h}(z)\sqrt{q_{b_0+\eta}(z)} dW_z}{\|K_{y,h}\sqrt{q_{b_0+\eta}}\|_{L^2}}\right| - \Upsilon\left(\frac{\|K_{y,h}\sqrt{q_{b_0+\eta}}\|_{L^2}^2}{\| \1_{[-A,A]}\sqrt{q_{b_0+\eta}}\|_{L^2}^2}\right)\right),\\
U_2 &:= \sup_{(y,h)\in\mathcal{T}} \left( \left|\frac{\int_{-A}^A K_{y,h}(z)\sqrt{q_{b_0-\eta}(z)} dW_z}{\|K_{y,h}\sqrt{q_{b_0-\eta}}\|_{L^2}}\right| - \Upsilon\left(\frac{\|K_{y,h}\sqrt{q_{b_0-\eta}}\|_{L^2}^2}{\| \1_{[-A,A]}\sqrt{q_{b_0-\eta}}\|_{L^2}^2}\right)\right).
\end{align*} 
\end{thm}

The proof of this result is the most elaborate one in this article and Section~\ref{Sec_weak} contains a route of it together with a uniform weak convergence result of the supremum statistic developed in Section~\ref{Sec_simple}. This uniform weak convergence is both the most important ingredient for Theorem~\ref{worst_case_delta} and interesting on its own. In addition, some general results about uniform weak convergence are derived that may be of independent interest and can be found in Section~\ref{Sec_weak} as well. The complete proof of Theorem~\ref{worst_case_delta}, which can be found in Appendix~\ref{App_least}, is then a combination of this weak convergence result and stochastic analysis tools together with the explicit representation of the invariant density that occurs in the limiting statistic.

\begin{remark}\label{remark_Piterbarg}
It is shown in \cite{Piterbarg} that a supremum statistic like $U_1$, however without the normalizing factor $\sqrt{q_b}$ in nominator and denominator, has a distribution without point mass. We prove that this is also true for $U_1\vee U_2$ in Appendix~\ref{SubSecD2}, which is important to conclude that our test based on the test statistic $T_T^\eta$ is uniformly asymptotically of level $\alpha$, see \eqref{eq: asym_level_alpha} below.
\end{remark}


The limiting statistics $U_1$ and $U_2$ in Theorem \ref{worst_case_delta} are almost surely finite which can be seen analogously to Theorem \ref{Multiscale_Lemma} and hence the quantiles
\begin{align}\label{eq: kappa_eta_alpha}
\kappa_{\eta, \alpha} := \min \left\{ r\in\R:\ \Pr( U_1\vee U_2 + 4\sqrt{A\eta/\sigma^2} \leq r)\geq 1-\alpha \right\}
\end{align} 
are well-defined. 
For the testing problem $H_0(b_0,\eta)$ versus the alternative~\eqref{eq: H1_eta}, the test 
\begin{align}\label{eq: phi_eta}
\phi_T^\eta(X) = \1_{\left\{ T_T^\eta(X)> \kappa_{\eta,\alpha}\right\}} 
\end{align} 
is by Theorem \ref{worst_case_delta} uniformly (over $H_0(b_0,\eta)$) asymptotically of level $\alpha$, i.e.
\begin{align}\label{eq: asym_level_alpha}
\limsup_{T\to\infty}\sup_{b\in H_0(b_0,\eta)} \Pr_b\left(\phi_T^\eta(X) =1 \right) \leq \alpha. 
\end{align}

\begin{remark}[Simultaneous detection of regions of deviation]\label{remark_D_alpha}
The test statistic $T_T^\eta(X)$ exceeds the $(1-\alpha)$-significance level if, and only if, the random family
\[ \mathcal{D}_\alpha^\eta := \left\{ (y,h)\in\mathcal{T}: |\Psi_{T,y,h}^{b_0}(X)| - \Lambda_{T,y,h}^\eta(X) > \Upsilon\left( \hat\sigma_T(y,h)^2/\hat\sigma_{T,\max}^2\right) + \kappa_{\eta,\alpha} \right\} \]
is non-empty. Therefore, one may conclude that with confidence $1-\alpha$ there is a deviation from $H_0(b_0,\eta)$ on \textbf{every} interval $[y-h, y+h]$ with $(y,h)\in\mathcal{D}_\alpha^\eta$. An illustration is given in the simulation study in Section~\ref{SubSec_Imp2} of the supplementary material.
\end{remark}

\section{Minimax optimality and sharp adaptivity}\label{Sec_Power}

In this section we will show that the similarity test $\phi_T^\eta$ possesses minimax optimality and adaptivity properties. To this aim we restrict the alternative $H_1(b_0,\eta)$ given in \eqref{eq: H1_eta} to Hölder-regular deviations $b-b_0$. In most situations, precise knowledge of the regularity of this deviation is unrealistic and we will establish that a suitable chocie of the kernel $K$ in the definition of the test statistic $T_T^\eta$ in \eqref{eq: T_eta} allows for (sharp) adaptive results.

\subsection{Measuring distances from $H_0(b_0,\eta)$}

When establishing minimax rates and optimal constants it is crucial to specify a distance between a given function $b$ and the null $H_0(b_0,\eta)$. We define this distance as
\begin{align}\label{eq: Delta_J}
\Delta_J(b) := \inf_{\tilde{b}\in H_0(b_0,\eta)} \left\| |b - \tilde{b}| \left(\frac{q_b}{\sigma^2}\right)^\frac{\beta}{2\beta+1}\right\|_J
\end{align} 
with compact $J\subset (-A,A)$ to avoid boundary effects. For $b\notin H_0(b_0,\eta)$ it is given by
\[ \Delta_J(b) = \sup_{x\in J} \Big(|b(x)-b_0(x)| - \eta\Big) \left(\frac{q_b(x)}{\sigma^2}\right)^\frac{\beta}{2\beta+1}\]
which corresponds to the boundary cases $\tilde{b} = b_0\pm\eta$. Note that the occurence of $q_b$ and $\sigma^2$ meet our intuition: deviations at a point $y$ can be expected to be easier to detect if the process spends more time around $y$, i.e. when $q_b(y)$ is large. On the other hand, detecting is more challenging the more noise we have, i.e. the larger $\sigma^2$ is. The reason we do not shift the factor $(q_b/\sigma^2)^{\beta/(2\beta+1)}$ into rate or constant is that $q_b$ depends on the location. For testing in supremum norm this is the equivalent to weighted risk definitions used in sharp adaptive drift estimation, see for example \cite{Dalalyan} and \cite{Strauch_2}.

\subsection{Optimal power properties}
For $\beta, L>0$ the Hölder class $\mathcal{H}(\beta,L)$ is given by the set of functions $f:\R\rightarrow\R$ such that for each $k=0,\dots,\lfloor \beta\rfloor$ the Hölder-condition
\[ \left| f^{(\lfloor\beta\rfloor)}(x) - f^{(\lfloor\beta\rfloor)}(y)\right| \leq L|x-y|^{\beta-\lfloor\beta\rfloor}\]
is valid, where $f^{(n)}$ denotes the $n$-th derivative of $f$ and $\lfloor \beta\rfloor$ the maximal integer strictly smaller than $\beta$. Our power results specify those functions in
\[ \bigcup_{\beta, L>0} H_1(b_0,\eta) \cap\{ b-b_0\in \mathcal{H}(\beta, L)\} ,\]
close enough to the null $H_0(b_0,\eta)$ in the distance $\Delta_J$ that can be detected by $\phi_T^\eta$ given in \eqref{eq: phi_eta} with probability tending to one. For this aim, we define the rate
\[ \delta_T =\delta_T(\beta) := \left(\frac{\log T}{T}\right)^{\frac{\beta}{2\beta +1}} \]
and constant
\begin{align}\label{eq: c_optimal}
c_* = c_*(\beta,L) := \left( \frac{2L^{\frac{1}{\beta}}}{(2\beta +1)\|K_\beta\|_{L^2}^2}\right)^{\frac{\beta}{2\beta+1}}.
\end{align} 
Here $K_\beta$ is the unique solution of the following optimization problem:
\begin{align}\label{eq: optimal_recovery}
\textrm{Minimize } \|K\|_{L^2} \textrm{ over all } K\in\mathcal{H}(\beta,1) \textrm{ with } K(0)\geq 1.
\end{align} 
We call $K_\beta$ the optimal recovery kernel. In the case $0<\beta\leq 1$ it is not difficult to see that
\[ K_\beta (x) = \1_{\{|x|\leq 1\}} \left(1-|x|^\beta\right),\]
as we even have $K_\beta (x) \leq f(x)$ for all $x\in [-1,1]$ and $f\in\mathcal{H}(\beta,1)$ in this case. For $\beta =2$ an explicit solution is known (see~\cite{Leonov}). For details on how this function can be constructed numerically, see \cite{Donoho} and \cite{Leonov}. Furthermore, for all $\beta>0$, $K_\beta$ is compactly supported, an even function and satisfies $K_\beta (0)=1>|K_\beta(x)|$ for $x\neq 0$. \\ 


For the power consideration we start with the lower bound. In the next theorem we will show that for every test of level $\alpha$ of the hypothesis $H_0(b_0,\eta)$ there exist drift functions in the alternative that deviate $(1-\epsilon_T)c_*\delta_T$ from $H_0(b_0,\eta)$ in the distance $\Delta_J$ which will not be detected with probability $1-\alpha-o(1)$ or larger. In particular, this is even true in the knowledge of both smoothness parameters $\beta$ and $L$.

\begin{thm}\label{lower_bound}
Let $\eta\geq 0$ and $\psi_T$ be a test that is uniformly over $H_0(b_0,\eta)$ of level $\alpha$, i.e. $\sup_{b\in H_0(b_0,\eta)}\E_{b}[\psi_T]\leq\alpha$, for some drift function $b_0\in\Sigma(C/2-\eta,A,\gamma+\eta/\sigma^2,\sigma)$. Then for arbitrary numbers $\epsilon_T>0$ with $\lim_{T\to\infty}\epsilon_T =0$ and $\lim_{T\to\infty}\epsilon_T \sqrt{\log T} =\infty$,
\[ \limsup_{T\to\infty}\ \inf_{\substack{b\in H_1(b_0,\eta) \cap\{ b-b_0\in \mathcal{H}(\beta, L)\}:\\ \Delta_J(b)\geq (1-\epsilon_T)c_*\delta_T}}\ \E_{b}\left[\psi_T\right]\ \leq\ \alpha\]
for any fixed compact interval $J\subset (-A,A)$.
\end{thm}

Although the proof of this result follows common ideas that have to be applied in the context of stochastic analysis, there appear two unusual obstacles. The first relates to the definition of local alternatives. Those are commonly defined via some disturbance function $g$.
However, in our case we cannot just add some hat $g$ with absolute height $(1-\epsilon_T)c_*\delta_T$ on boundary cases of the null hypothesis, because the distance $\Delta_J$ involves a scaling by the invariant density of the local alternative itself, which in turn depends on the choice of $g$. This leads to a fixed point problem which in Appendix~\ref{App_lower} is proven to be solvable, enabling us to construct alternatives $b$ with $\Delta_J(b)=(1-\epsilon_T)c_*\delta_T$. The second obstruction is that the likelihoods of the above constructed hypotheses are not independent as for example in~\cite{Walther} or~\cite{Duembgen/Spokoiny}, but only  uncorrelated asymptotically, see Proposition~\ref{lemma_E_absvalue} and Remark~\ref{remark_E_absvalue}. More details are provided in the route of the proof in Section~\ref{Sec_supp}.\\

Next, we establish the corresponding upper bounds of $\phi_T^\eta$ in \eqref{eq: phi_eta}. Note that the validity of Theorem~\ref{worst_case_delta} cannot be guaranteed for kernels of higher order than one because such kernels necessarily take negative values.


\begin{thm}\label{Upper_bound}
Let $\beta, L>0$, $\eta>0$, $b_0\in\Sigma(C/2-\eta,A,\gamma+\eta/\sigma^2,\sigma)$ and let $K$ be a non-negative kernel of bounded variation supported in $[-1,1]$ with $\|K\|_{[-1,1]}=1$.
Then for arbitrary numbers $\epsilon_T>0$ with $\lim_{T\to\infty}\epsilon_T =0$ and $\lim_{T\to\infty} \epsilon_T \sqrt{\log T}=\infty$ there exists a constant $c=c(\beta, L, K)$ such that for the test $\phi_T^\eta$ given in~\eqref{eq: phi_eta},
\[ \lim_{T\to\infty}\ \inf_{\substack{b\in H_1(b_0,\eta) \cap\{ b-b_0\in \mathcal{H}(\beta, L)\}:\\ \Delta_J(b)\geq (1+\epsilon_T)c\delta_T}}\ \Pr_b\left( \phi_T^\eta(X) =1\right)\ =\ 1\]
for any fixed compact interval $J\subset (-A,A)$. In the case $\beta\in (0,1]$ we can choose $K=K_\beta$ and the result is true for $c=c_*$.
\end{thm}

This result may be read as follows: If the underlying drift function of our diffusion deviates from the composite null $H_0(b_0,\eta)$ in distance $\Delta_J$ by at least $(1+\epsilon_T)c\delta_T$, then the test detects the deviation and rejects the null hypothesis of similarity with probability close to one. Hence, for any non-negative kernel of bounded variation supported in $[-1,1]$, the test $\phi_T^\eta$ is minimax rate-optimal as it attains over the whole range of $\beta, L>0$ the corresponding rate of the lower bound in Theorem \ref{lower_bound}. In case $\beta\leq 1$, it is even optimal in the constant. The next theorem even states that rate-adaptivity is attained uniformly over parameter ranges of the form $[\beta_1,\beta_2]\times [L_1,L_2]$ and sharp adaptivity for fixed $\beta$ over $L\in [L_1,L_2]$.

\begin{thm}[Adaptivity]\label{thm_adaptivity}
Let $\beta, L>0$ and $b_0,K$ and $\epsilon_T$ be specified as in Theorem~\ref{Upper_bound}.
Then for a compact subset $[\beta_1,\beta_2]\times[L_1, L_2]\subset (0,\infty)^2$ the test $\phi_T^\eta$ is rate-adaptive in both parameters $\beta$ and $L$ in the sense
\[  \lim_{T\to\infty}\ \inf_{(b,L)\in [\beta_1,\beta_2]\times [L_1,L_2]}\ \inf_{\substack{b\in H_1(b_0,\eta) \cap\{ b-b_0\in \mathcal{H}(\beta, L)\}:\\ \Delta_J(b)\geq (1+\epsilon_T)c(\beta, L,K)\delta_T}}\ \Pr_b\left( \phi_T^\eta(X) =1\right)\ =\ 1.\]
For $\beta\leq 1$ we have sharp adaptivity in the parameter $L\in [L_1,L_2]\subset (0,\infty)$ in the sense
\[  \lim_{T\to\infty}\ \inf_{L\in [L_1,L_2]}\ \inf_{\substack{b\in H_1(b_0,\eta) \cap\{ b-b_0\in \mathcal{H}(\beta, L)\}:\\ \Delta_J(b)\geq (1+\epsilon_T)c_*\delta_T}}\ \Pr_b\left( \phi_T^\eta(X) =1\right)\ =\ 1.\]
\end{thm}

For $\eta=0$ the hypothesis $H_0(b_0,\eta)$ reduces to the simple hypothesis $b=~b_0$ (on $[-A,A]$). 
Here, we find minimax optimality including the efficiency constant even on the full range $\beta,L>0$.

\begin{thm}\label{thm_simple_hypothesis}
Let $\beta,L>0$ and $b_0$ and $\epsilon_T$ be specified as in Theorem~\ref{Upper_bound}.
Then the test $\phi_T^{b_0}$ given in \eqref{eq: phi_0} with the kernel $K=K_\beta$ satisfies for any compact interval $J\subset (-A,A)$
\[ \lim_{T\to\infty}\ \inf_{\substack{b\in H_{\neq} \cap\{ b-b_0\in \mathcal{H}(\beta, L)\}:\\ \Delta_J(b)\geq (1-\epsilon_T)c_*\delta_T}}\ \Pr_b\left( \phi_T^{b_0}=1\right)\ =\ 1.\]

\end{thm}

\begin{remark}[Diffusions started at a fixed point]\label{remark_fixed_point}
For technical convenience, all of our results were derived under the assumption that the diffusion $X$ is stationary, i.e. $X_0\sim\mu_b$. However, it is not necessary. In Remark~\ref{remark_fixed_start} and \ref{remark_fixed_start_upper} we give the details how Theorem~\ref{worst_case_delta}, \ref{lower_bound}, \ref{Upper_bound}, \ref{thm_adaptivity} and \ref{thm_simple_hypothesis} can be established for $X$ started at a fixed point $x_0\in [-A,A]$.
\end{remark}

\section{Pathwise stability of the similarity test}\label{Sec_Stability}

The aim of the section is to show that it is reasonable to employ our inference procedure in case of deviation from the idealized model assumptions as long as the difference is moderate in a suitable sense. As the definition of $T_T^\eta(X)$ involves a stochastic (It\^{o}-) integral, it is not even clear how it could be defined for data $X$ that is not given by a semimartingale. Moreover, as all results in Sections \ref{Sec_simple}, \ref{Sec_similarity} and \ref{Sec_Power} use very specific properties of the It\^{o} diffusion, any deviation from this model assumption might cause a failure of those results. Of course, there are many imaginable deviations from the diffusion model. They range from processes whose distribution is absolutely continuous to the law of $(X_t)_{t\in [0,T]}$ in \eqref{eq: SDE} such as inhomogeneous It\^{o} diffusions (with the same driving noise) to those with singular path properties such as fractional diffusions. Whereas the first case mentioned is the object of statistical investigation in questions about model misspecification, the second scenario is rather uncommon and statistically questionable at first sight. However, studying deviation from the idealized model assumption in this second scenario is meaningful if the design of an efficient statistical procedure in the true model is too difficult. This happens easily in the context of stochastic processes that are not given by semimartingales as for those many tools from stochastic analysis are not available. Correspondingly, based on a continuous record of observations, \cite{Diehl/Friz}~analyzes the parametric MLE for diffusions with regard to its pathwise stability properties as well as robustness to the very nature of the noise.

In Subsection \ref{SubSec_cont} we will propose a natural extension of our test statistic $T_T^\eta$ that is defined pathwise and is in fact well defined for any continuous path $X$. Furthermore, we establish continuity 
of this extended test statistic with respect to the topology of uniform convergence. This opens the door for thoroughly studying the performance of our procedure beyond the semimartingale context. In Subsection \ref{SubSec_fractional_model} we study the particular example of fractional misspecification of the driving noise. In this way, we supplement the pioneering work \cite{Diehl/Friz} where the study subject to the asymptotics $H\to 1/2$ in parametric statistical inference has been initiated in their Section~$6$.
We establish that even uniformly in the drift, the extended test statistic is stable as the fractional noise approaches Brownian motion. This uniformity, which is substantially harder to derive than the corresponding pointwise result for any fixed drift, is crucial in order to deduce that the test is actually uniformly (over the drift) asymptotically of level $\alpha$ in this limiting scenario, see \eqref{eq: level_H}.
Moreover, we prove that the minimax optimality is preserved in a certain sense as the fractional driving process approaches Brownian motion.

\subsection{Continuation of the multiple test statistic}\label{SubSec_cont}

As mentioned in the introduction of this section, the first obstacle appearing is that the test statistic $T_T^\eta(X)$ given in \eqref{eq: T_eta} can only be evaluated for semimartingles as it involves the It\^{o} integral $  \int_0^T K_{y,h}(X_s) dX_s, $
where ther kernel $K$ satisfies the requirements of Theorem \ref{worst_case_delta}. To give a pathwise definition, we have to generalize the test statistic. To this aim, we assume in addition the kernel $K$ to be continuously differentiable and apply It\^{o}'s formula to $f(x) = \int_0^x K_{y,h}(z) dz$ which gives for the diffusion $X$ solving~\eqref{eq: SDE}
\[ \int_0^T K_{y,h}(X_s) dX_s= \int_{X_0}^{X_T} K_{y,h}(z) dz - \frac{\sigma^2}{2}\int_0^T (K_{y,h})'(X_s) ds.\]
On the right-hand side, it is perfectly possible to insert any continuous function $f\in\mathcal{C}([0,T])$. Therefore, we define
\begin{align*}
\tilde{I}_T:\ \mathcal{C}([0,T]) &\longrightarrow\ \R, \\
 f\ \ \  & \longmapsto\ \int_{f(0)}^{f(T)} K_{y,h}(z) dz - \frac{\sigma^2}{2}\int_0^T (K_{y,h})'(f(s)) ds.
\end{align*} 
$\tilde I_T$ is measurable as it is continuous with respect to $\|\cdot\|_{[0,T]}$. 
Let 
\[ D_{A,T}:= \{ f\in\mathcal{C}([0,T]) \mid -A,A\in f([0,T]) \} \]
be the set of real-valued continuous functions $f$ on $[0,T]$ whose image set $f([0,T])$ contains the interval $[-A,A]$.
For any $f\in D_{A,T}$ we denote
\begin{align}\label{eq: tilde_psi}
\tilde\Psi_{T,y,h}^{b_0}(f) &:=  \frac{\tilde I_T(f)}{\sigma\sqrt{ \int_0^T K_{y,h}(f(s))^2 ds}} -\frac{\int_0^T K_{y,h}(f(s)) b_0(f(s)) ds}{\sigma\sqrt{ \int_0^T K_{y,h}(f(s))^2 ds}}.
\end{align} 
Based on those $\tilde\Psi_{T,y,h}^{b_0}(f)$, 
the final extended test statistic is given pathwise by
\[ \tilde T_T^{\eta}(f) := \sup_{(y,h)\in\mathcal{T}_T} \left(\left| \tilde\Psi_{T,y,h}^{b_0}(f)\right|- \Lambda_{T,y,h}^\eta(X) - \Upsilon\big(\hat\sigma_T(y,h)^2/\hat\sigma_{T,\max}^2\big) \right) .\]
Here, $\hat\sigma_T(y,h), \hat\sigma_{T,\max}, \Lambda_{T,y,h}^\eta(X)$ and $\Upsilon(\cdot)$ are defined as in Section~\ref{Sec_simple} and~\ref{Sec_similarity} where no problem occurs as all involved integrals are  classical integrals. Note that in the respective definitions of $\tilde\Psi_{T,y,h}^{b_0}(f)$ and $\tilde T_T^\eta(f)$, no division by zero occurs due to the restriction of $f$ to $D_{A,T}$.
Additionally, as a consequence of Remark~\ref{remark_bandwidth} in Appendix~\ref{Sec_Implementation} on implementation, it is sufficient to restrict attention to $\mathcal{T}_T := \left\{ (y,h)\in\mathcal{T}\mid h\geq h_{\min}(T)\right\}$.

\begin{thm}\label{thm_continuity}
The mapping $\tilde{T}_T^\eta: D_{A,T}\longrightarrow\R$ is continuous with respect to the topology of uniform convergence.
\end{thm}

\subsection{Fractional misspecification of the noise}\label{SubSec_fractional_model}

As an example beyond a common semimartingale or Markovian setup with non-trivial dependence structure in the driving noise, we consider dynamics of the form
\begin{align}\label{eq: SDE_H}
dX_t^H = b(X_t^H) dt + \sigma dW_t^H, \quad X_0^H =x_0,
\end{align} 
where $W^H$, $H\in (0,1)$, is a fractional Brownian motion, i.e. a Gaussian process with covariance structure
\begin{align}\label{eq: R_H}
R_H(t,s) = \E\left[ W_t^H W_s^H\right] = \frac12\left( t^{2H} + s^{2H} - |t-s|^{2H}\right). 
\end{align} 
Although developing a similarity test for this fractional model and arbitrary $H$ is obstructed by missing developments in stochastic analysis that enable to perform sophisticated empirical process theory, for $H\approx 1/2$ the model in \eqref{eq: SDE} may be a good description of the true dynamics in \eqref{eq: SDE_H}. We propose to still use $\phi_T^\eta$ based on the wrong model in this context. For justification, we subsequently prove that for $H\approx 1/2$, the test statistic $\tilde T_T^\eta(X^{H,b})$ is close to $\tilde T_T^\eta(X^{b})$, uniformly in $b$. 


Throughout this section we fix a probability space $(\Omega, \mathcal{A}, \Pr)$ that supports a Brownian motion $W=(W_t)_{t\geq 0}$. As we are interested in convergence for varying Hurst parameter $H$ it is important that all $W^H$ are defined on the same probability space. Therefore, we define $W^H=(W_t^H)_{t\in [0,T]}$ by
\[ W_t^H := \int_0^t K_H(t,s) dW_s, \]
where $K_H$ is some kernel function specified in \eqref{eq: K_H}. One also has that $W^H|_{H=1/2} = W$ is a standard Brownian motion. \\
Existence of a unique strong solution to \eqref{eq: SDE_H} was established in \cite{Nualart} under the condition 
\begin{enumerate}
\item[(i)] $|b(x)|\leq C(1+|x|)$ if $H\leq \frac12$ and 
\item[(ii)] $|b(x) - b(y)|\leq C|x-y|^\alpha$ for some $1>\alpha>1-\frac{1}{2H}$ in the case $H>\frac12$.
\end{enumerate}
In fact, in the second case one can also use Lipschitz continuous $b$ as in~\cite{Tudor}. As we are interested to compare our results for the case $H=1/2$ with the fractional model, we restrict ourselves to Lipschitz continuous $b\in\Sigma(C,A,\gamma,\sigma)$. For such $b$ we denote the solution process of \eqref{eq: SDE_H} by $X^{H,b}$ and by $X^b$ the solution of \eqref{eq: SDE} for drift $b$ and initial condition $x_0$, respectively. \\
As a first result, we have $X^{H,b}\rightarrow X^b$ in probability uniformly on $[0,T]$ as $H\to 1/2$. For any fixed Lipschitz continuous drift function $b$, this convergence seems to be well-known in the literature. Here, the convergence is required \textit{uniformly} over Lipschitz balls for the crucial conclusion \eqref{eq: level_H}, and hence a short proof is presented in Appendix~\ref{App_stability}.

\begin{proposition}\label{lemma_X^H-X}
For every $\epsilon,L>0$, we have
\[ \sup_{b\in\Sigma(C,A,\gamma,\sigma)\cap\mathcal{H}(1,L)}\Pr\left( \|  X^{H,b} - X^b \|_{[0,T]} >\epsilon\right)\ \stackrel{H\to\frac12}{\longrightarrow}\ 0.\]
\end{proposition}

The preceding Proposition \ref{lemma_X^H-X} shows that for small deviations of $H$ from $1/2$ the diffusion model \eqref{eq: SDE} is still a good description of the dynamics of $X^{H,b}$.
As a consequence of this Proposition and Theorem \ref{thm_continuity}, the next theorem states that uniformly in $b$, $\tilde T_T^\eta$ is stable as $H$ approaches $1/2$ conditional on the event that the statistics are defined.

\begin{thm}\label{prop_stability}
Conditional on the event $\{X^{H,b},X^{b}\in D_{A,T}\}$, the test statistic $\tilde T_T^\eta(X^{H,b})$ converges to $\tilde T_T^\eta(X^{b})$ in probability uniformly over drift functions $b\in\Sigma(C,A,\gamma,\sigma)\cap\mathcal{H}(1,L)$ for any $L>0$, i.e. for every $\epsilon>0$,
\[ \sup_{b\in\Sigma(C,A,\gamma,\sigma)\cap\mathcal{H}(1,L)}\Pr\left( \left. \left| \tilde T_T^\eta(X^{H,b}) - \tilde T_T^\eta(X^{b}) \right| >\epsilon \ \right| X^{H,b},X^{b}\in D_{A,T} \right) \stackrel{H\to\frac12}{\longrightarrow} 0\]
for $T$ sufficiently large. Moreover, we have for the conditioning event
\begin{align*}
&\liminf_{T\to\infty} \liminf_{H\to\frac12} \inf_{b\in\Sigma(C,A,\gamma,\sigma)\cap\mathcal{H}(1,L)}  \Pr\left( X^{H,b}, X^{b}\in D_{A,T}\right)\\
&\hspace{1cm}=\liminf_{H\to\frac12} \liminf_{T\to\infty} \inf_{b\in\Sigma(C,A,\gamma,\sigma)\cap\mathcal{H}(1,L)}\Pr\left( X^{H,b}, X^{b}\in D_{A,T}\right)  =1.
\end{align*} 
\end{thm}

The idea of the proof is to use the continuity of $\tilde T_T^\eta$ from Theorem \ref{thm_continuity}. However, continuity solely is not sufficient to derive the uniformity in $b$ as stated in Theorem \ref{prop_stability}. By showing tightness of the family of measures $\{\Pr^{X^{H,b}} : b\in\Sigma(C,A,\gamma,\sigma)\cap\mathcal{H}(1,L), H\in \{H_n:n\in\N\}\}$ for any sequence $(H_n)_{n\in\N}$ with $H_n\to\frac12$, we are able to restrict attention to a compact subset of $D_{A,T}$ in $\mathcal{C}([0,T])$. On this, $\tilde T_T^\eta$ is uniformly continuous and we can apply the uniform continuous mapping theorem for convergence in probability, see Lemma~\ref{lemma_CMT_uniform}. The proof of tightness does not follow the route of verifying asymptotic stochastic equicontinuity, but is instead based on the Gronwall lemma, Proposition \ref{lemma_X^H-X}, the Arzelà--Ascoli theorem and Prohorov's theorem, see Lemma~\ref{lemma_X^H_tight}.

\begin{remark}\label{remark_stability}
With the uniform convergence in conditional probability in Proposition \ref{prop_stability}, it is reasonable to define a test $\tilde\phi_T^\eta$ based on $\tilde T_T^\eta$ as it was done in \eqref{eq: phi_eta} for the Brownian diffusion.  For some $\kappa> \kappa_{\eta,\alpha}$, where $\kappa_{\eta,\alpha}$ is given in \eqref{eq: kappa_eta_alpha}, we set 
\[ \tilde\phi_T^\eta\left(X^H\right) := \1_{\left\{ \tilde T_T^\eta(X^H)>\kappa\right\}}.\]
Then one obtains as a corollary of Proposition \ref{prop_stability} that this test $\tilde\phi_T^\eta$ is in the limit $H\to\frac12$ uniformly (over $H_0(b_0,\eta)\cap\mathcal{H}(1,L)$) asymptotically of level $\alpha$ in the sense that
\begin{align}\label{eq: level_H}
\limsup_{T\to\infty} \limsup_{H\to\frac12}\sup_{b\in H_0(b_0,\eta)\cap\mathcal{H}(1,L)} \Pr_b\left(\tilde\phi_T^\eta(X^H) =1 \right) \leq \alpha. 
\end{align} 
\end{remark}

%



Finally, we show that the test $\tilde\phi_T^\eta$ is indeed powerful for $H\approx 1/2$ by establishing a lower bound in the spirit of Theorem \ref{lower_bound} for $H$ being close to the Brownian case $H=1/2$. A corresponding upper bound follows in the same way as \eqref{eq: level_H} from Theorem~\ref{prop_stability}.


\begin{thm}\label{Continuity_lower_H}
Let $x_0\in [-A,A]$ and $b_0\in\Sigma(C/2-\eta,A,\gamma+\eta/\sigma^2,\sigma)$. Then for every $\epsilon>0$ there exists a sequence $\delta(T):=\delta(T,\epsilon)>0$ such that
\[ \limsup_{T\to\infty} \sup_{\substack{H:\ |H-\frac12|<\delta(T)}} \sup_{\psi_T^H}\ \inf_{\substack{b\in H_1(b_0,\eta)\cap \{b-b_0\in\mathcal{H}(\beta,L)\}:\\ \Delta_J(b)\geq (1-\epsilon_T)c_*\delta_T}} \E\left[\psi_T^H(X^{H,b})\right]\leq \alpha +\epsilon,\]
where $\sup_{\psi_T^H}$ is taken over tests $\psi_T^H$ with $\sup_{b\in H_0(b_0,\eta)}\E\left[ \psi_T^H(X^{H,b})\right] \leq \alpha$.
\end{thm}

The idea of the proof is to bound $\sup_{b\in H_0(b_0,\eta)}\E\left[ \psi_T^H(X^{H,b})\right]-\alpha$ from above by an expression involving an average of likelihoods of certain unfavourable alternatives as it was done in the proof of Theorem \ref{lower_bound}. Then we prove 
\begin{align}\label{eq: L1_Z_H}
\E\left[\left| Z_T^H(b_k(X_\cdot^{H,b})) - Z_T^{1/2}(b_k( X_\cdot^b))\right|\right] \stackrel{H\to\frac12}{\longrightarrow} 0
\end{align} 
where $Z_T^{H}$ denotes the likelihood from Girsanov's theorem, to apply Theorem \ref{lower_bound} for the case of $X^b$.
The $L^1(\Pr)$-convergence in \eqref{eq: L1_Z_H} is established by rewriting the explicit form of $Z_T^H$ in terms of fractional integrals and derivatives and then derive continuity results of those. Note that the likelihood ratios $Z_T^H$ for $H\neq 1/2$ are far more complicated than in case of $H=1/2$, see Appendix~\ref{SubSub_fBM}.

\section{Sketch of the proof of Theorem \ref{worst_case_delta}}\label{Sec_weak}

In Subsection \ref{SubSec_sketch} a route of proof for Theorem \ref{worst_case_delta} is given. It is built on a uniform weak convergence result for the supremum statistic $T_T^{b_0}(X)$ given in \eqref{eq: T_simple}. This crucial uniform weak convergence result is formulated in Theorem \ref{weak_conv}. A route of its proof is explained in Subsection \ref{SubSec_weak}, which strongly relies on the abstract Proposition~\ref{prop_unif_weak}.

\subsection{Sketch of proof of Theorem \ref{worst_case_delta}}\label{SubSec_sketch}

The idea of the proof is to upper bound 
\[ \limsup_{T\to\infty}\sup_{b\in H_0(b_0,\eta)}\Pr_b(T_T^\eta(X) \geq r)\] 
by decomposing the set $\mathcal{T}$ (over which the supremum $\sup_{(y,h)\in\mathcal{T}}$ in the definition \eqref{eq: T_eta} of $T_T^\eta(X)$ is taken) into three subsets $\mathcal{T}_1,\mathcal{T}_2(b)$ and $\mathcal{T}_3(b)$: 
\begin{itemize}
\item $\mathcal{T}_1$ consists of all $(y,h)\in\mathcal{T}$ with \textit{very small} $h$. Using the same techniques from empirical process theory as for the proof of Theorem \ref{Multiscale_Lemma}, we show that uniformly in $b$ the supremum over $\mathcal{T}_1$ is negligible as $T\rightarrow\infty$.
\item $\mathcal{T}_2(b)$ consists of those $(y,h)\in\mathcal{T}$ such that $b\in H_0(b_0,\eta)$ is \textit{bounded away from the boundary cases $b_0\pm\eta$} on $[y-h,y+h]$ while $h$ is \textit{not too small}. It is shown that the supremum in $T_T^\eta$ given in \eqref{eq: T_eta}, restricted to those $(y,h)$, converges to $-\infty$ uniformly in $b$ as $T\rightarrow\infty$. Hence, the supremum in \eqref{eq: T_eta} is not attained asymptotically at $(y,h)\in\mathcal{T}_2(b)$.
\item $\mathcal{T}_3(b)$ contains those $(y,h)\in\mathcal{T}$ for which $b\in H_0(b_0,\eta)$ is \textit{close to some boundary case $b_0\pm\eta$} on $[y-h,y+h]$ while $h$ is \textit{not too small}. Uniformly over $b$, the probability that the supremum over $\mathcal{T}_3(b)$ exceeds a given $r$ is proven to be upper bounded by $\Pr( U_1\vee U_2+ 4\sqrt{A\eta/\sigma^2}\geq r)$ in the limit $T\rightarrow\infty$. 
\end{itemize}
The last item is the most involved part of the proof and is explained in what follows. By \eqref{eq: 8.1} and the triangle inequality,
\begin{align*}
&\left| \Psi_{T,y,h}^{b_0}(X)\right| - \Lambda_{T,y,h}^\eta(X)- \Upsilon\big(\hat\sigma_T(y,h)^2/\hat\sigma_{T,\max}^2\big)\\
&\hspace{1cm} \leq \frac{1}{\sqrt{T}\hat\sigma_T(y,h)}\left|\int_0^T K_{y,h}(X_s) dW_s\right|  -\Upsilon\left(\hat\sigma_T(y,h)^2/\hat\sigma_{T,\max}^2\right).
\end{align*} 
Therefore, it is sufficient to determine the limes superior as $T\to\infty$ of
\begin{align}\label{eqp: 8.2}
\sup_{b\in H_0(b_0,\eta)} \Pr_b\left( \sup_{(y,h)\in\mathcal{T}_3(b)} \left(  \left| \frac{\frac{1}{\sqrt{T}}\int_0^T K_{y,h}(X_s) dW_s}{\hat\sigma_T(y,h)}\right| - \Upsilon\left(\frac{\hat\sigma_T(y,h)^2}{\hat\sigma_{T,\max}^2}\right)\right) \geq r\right). 
\end{align}  
This is severely challenging for two reasons: 
\begin{itemize}
\item First, the random variable is defined as a supremum over normalized stochastic integrals, whose distribution depends via the process $X$ on the parameter $b$.
\item Secondly, we need to establish the result uniformly over $b$ belonging to $H_0(b_0,\eta)$, which is mathematically substantially more involved as compared to other weak limit results of supremum statistics.
\end{itemize}
Note that the expression within the probability now is of the same form 
as $T_T^b(X)$ given in \eqref{eq: T_simple} under $\Pr_b$ (up to restricting the supremum to $\mathcal{T}_3(b)$ instead of $\mathcal{T}$).
Thus, we investigate the weak limit of $T_T^b(X)$ \textit{uniformly} over $H_0(b_0,\eta)$ which substantially increases the technical effort. This uniformity strongly relies on $b\in\Sigma(C,A,\gamma,\sigma)$, but does not use $\|b-b_0\|_{[-A,A]}\leq\eta$ and hence it is derived on the whole class $\Sigma(C,A,\gamma,\sigma)\supset H_0(b_0,\eta)$. This yields the following Theorem~\ref{weak_conv}. Its proof is deferred to Appendix~\ref{App_weak}, while a route is presented in Subsection~\ref{SubSec_weak} below.

\begin{thm}\label{weak_conv}
Let $K$ be given as in Theorem~\ref{worst_case_delta} and define the random variable
$ Z_b(y,h) := \int_{-A}^A K_{y,h}(z) \sqrt{q_b(z)} dW_z $
with a two-sided Brownian motion $W$ on $[-A,A]$, $\sigma_b(y,h):=\| K_{y,h} \sqrt{q_b}\|_{L^2}$, 
$\sigma_{b,\max} :=\|\1_{[-A,A]}\sqrt{q_b}\|_{L^2}$,
and based on those
\[ S_b :=  \sup_{(y,h)\in\mathcal{T}} \left( \frac{\left|Z_b(y,h)\right|}{\sigma_b(y,h)}  - \Upsilon \left(\frac{\sigma_b(y,h)^2}{\sigma_{b,\max}^2}\right)\right). \]
Then the following uniform weak convergence holds true:
\[ \sup_{b\in\Sigma(C,A,\gamma,\sigma)} d_{BL}^b\left( T_T^b(X),S_b \right)\stackrel{T\to \infty}{\longrightarrow} 0.\]
\end{thm}
Here, $d_{BL}$ denotes the dual bounded Lipschitz metric which metrizes weak convergence and is given in Appendix~\ref{SubSec_C1}. The superscript $b$ in $d_{BL}^b$ indicates the dependence of the distribution of $X$ on $b$.\\

To continue with \eqref{eqp: 8.2}, note that it is almost immediate from our proof that Theorem \ref{weak_conv} remains true for $\mathcal{T}$ replaced by $\mathcal{T}_3(b)$. This allows to conclude that for any $\delta>0$, the limes superior as $T\to\infty$ of \eqref{eqp: 8.2} is bounded from above by
\begin{align}\label{eqp: 8.3}
\sup_{b\in H_0(b_0,\eta)} \Pr\left( \sup_{(y,h)\in\mathcal{T}_3(b)} \left(  \frac{|Z_b(y,h)|}{\sigma_b(y,h)}- \Upsilon\bigg(\frac{\sigma_b(y,h)^2}{\sigma_{b,\max}^2}\bigg)\right) \geq r-\delta\right),
\end{align} 
see Step (3) in Appendix~\ref{App_least}. The remaining problem is the dependence of the random variable within the probability in \eqref{eqp: 8.3} on $b$. It is finally shown that $U_1\vee U_2 + 4\sqrt{A\eta/\sigma^2}$ stochastically dominates the random variables appearing in \eqref{eqp: 8.3} as a suppressed (in this sketch of proof) hyperparameter that specifies \textit{closeness} of $b$ to $b_0\pm\eta$ in the definition of $\mathcal{T}_3(b)$ is tending to zero. At this point, the restriction to $\mathcal{T}_3(b)$ is crucial.

\begin{remark}
In order to find bounds for an expression of the form \eqref{eqp: 8.2}, typical techniques are Gaussian approximations for suprema as in \cite{CCK} or the Hungarian construction that was first described in \cite{KKT}. However, for an application of the results in \cite{CCK}, we miss the structure of a classical empirical process, and to the best of our knowledge, there is no Gaussian coupling result available for $(\int_0^T K_{y,h}(X_s) dW_s)_{(y,h)\in\mathcal{T}}$.
\end{remark}

\subsection{Route of the proof of Theorem \ref{weak_conv}}\label{SubSec_weak}

The proof of Theorem~\ref{weak_conv} is a combination of sophisticated empirical process theory, stochastic analysis und the theory of (uniform) weak convergence. Subsequently, we use the notation $\mathcal{T}_\delta :=\{(y,h)\in\mathcal{T}\mid h\geq \delta\}$ and
\[Z_T(y,h) := \frac{1}{\sqrt{T}} \int_0^T K_{y,h}(X_s) dW_s. \]
Slightly simplified, the proof is split into the following three steps which will be explained right after:
\begin{enumerate}
\item[(i)] First, we prove that for any $\delta>0$, the processes $(Z_T(y,h))_{(y,h)\in\mathcal{T}_\delta}$ converge in distribution to $(Z_b(y,h))_{(y,h)\in\mathcal{T}_\delta}$, uniformly in $b$, i.e.
\[ \sup_{b\in\Sigma(C,A,\gamma,\sigma)} d_{BL}^b \left( (Z_T(y,h))_{(y,h)\in\mathcal{T}_\delta}, (Z_b(y,h))_{(y,h)\in\mathcal{T}_\delta}\right) \stackrel{T\to\infty}{\longrightarrow} 0.\]
\item[(ii)] Next, we establish for any $\delta>0$
\[ \sup_{b\in\Sigma(C,A,\gamma,\sigma)} d_{BL}^b\left( T_T(\delta,1), S_b(\delta,1)\right) \stackrel{T\to\infty}{\longrightarrow} 0 \]
for  
\begin{align*}
T_T(\delta, 1) := \sup_{(y,h)\in\mathcal{T}_\delta} \left(\frac{\left|\frac{1}{\sqrt{T}} \int_0^T K_{y,h}(X_s) dW_s\right|}{\hat\sigma_T(y,h)} - \Upsilon\left( \frac{\hat\sigma_T(y,h)^2}{\hat\sigma_{T,\max}^2}\right) \right), 
\end{align*} 
where the definition of $\hat\sigma_T, \hat\sigma_{T,\max}$ can be found in Theorem \ref{Multiscale_Lemma}, and
\[ S_b(\delta, 1) := \sup_{(y,h)\in\mathcal{T}_\delta} \left(\frac{\left|\int_\R K_{y,h}(z) \sqrt{\rho_b(z)} dW_z\right|}{\sigma_b(y,h) } - \Upsilon\left( \frac{\sigma_b(y,h)^2}{\sigma_{b,\max}^2}\right) \right).\]
\item[(iii)] Finally, we conclude with the extension to 
\[\sup_{b\in\Sigma(C,A,\gamma,\sigma)} d_{BL}^b\left( T_T(0,1), S_b(0,1)\right)\stackrel{T\to\infty}{\longrightarrow} 0.\]
Note that $S_b(0,1)=S_b$ and $T_T(0,1) = T_T^b(X)$ under $\Pr_b$.
\end{enumerate}

On a conceptual level, the proof of the step (i) is based on the following result that is of independent interest. Our proof of this result is given in Appendix~\ref{SubSec_C1}. For any metric space $E$, $\mathcal{F}_{BL}(E)$ denotes the closed unit ball of bounded Lipschitz functions on $E$, see Appendix~\ref{SubSec_C1}. For any pseudometric space $(\mathcal{S},\rho)$ and any $u>0$, we write $N(u,\mathcal{S},\rho)$ for the covering number with closed balls at radius $u$, see its definition~\eqref{eq: covering_number}. $l_\infty(\mathcal{S})$ denotes all functions $f:\mathcal{S}\rightarrow\R$ with $\|f\|_\mathcal{S} <\infty$.

\begin{proposition}\label{prop_unif_weak}
Let $\left(X_n(s)\right)_{s\in\mathcal{S}}$ and $\left(Y_n(s)\right)_{s\in\mathcal{S}}$ be two stochastic processes with values in $l_\infty(\mathcal{S})$, where $\mathcal{S}$ is countable and equipped with a metric $\rho_\theta$ that depends on some parameter $\theta\in\Theta$. Suppose the following three conditions hold:
\begin{enumerate}
\item[(a)] For any integer $k>0$ we have
\[ \sup_{\substack{B\subset\mathcal{S}: \# B\leq k}}\ \sup_{\theta\in\Theta} \sup_{f\in\mathcal{F}_{BL}(\R^{\# B})} \hspace{-0.05cm}\big| \E_\theta\big[ f\big((X_{n}(s))_{s\in B}\big)\big] -\E_\theta\big[ f\big((Y_{n}(s))_{s\in B}\big)\big]\big| \hspace{-0.05cm}\to 0 \]
in the limit $n\to\infty$. Here, $\#B$ denotes the cardinality of $B$.
\item[(b)] For each $\epsilon>0$ we have for $Z_n \in \{X_n, Y_n\}$,
\[ \lim_{\delta\searrow 0} \limsup_{n\to\infty} \sup_{\theta\in\Theta} \Pr_\theta\left( \sup_{\rho_\theta(s,s')\leq\delta} |Z_n(s) - Z_n(s')| >\epsilon \right) =0.\]
\item[(c)] For all $u>0$ we have $\sup_{\theta\in\Theta}N(u, \mathcal{S},\rho_\theta) <\infty$.
\end{enumerate}
Then we have
\[ \sup_{\theta\in\Theta} \sup_{f\in\mathcal{F}_{BL}(l_\infty(\mathcal{S}))} \left| \E_\theta\left[ f(X_n)\right] - \E_\theta\left[f(Y_n)\right]\right| \stackrel{n\to\infty}{\longrightarrow}\ 0.\]
\end{proposition}

For step (ii), we first show that the random denominators $\hat\sigma_T(y,h)$ for $(y,h)\in\mathcal{T}_\delta$ in $\tilde T_T(\delta,1)$ can be replaced by their deterministic limiting counterparts $\sigma_b(y,h)$, see Lemma~\ref{lemma_weak_2}.
In order to continue with continuous mapping type arguments, we are facing the problem that 
\begin{align}\label{eq: map_l_infty}
(x(y,h))_{(y,h)\in\mathcal{T}} \mapsto \sup_{(y,h)\in\mathcal{T}} \left|\frac{x(y,h)}{\sigma_b(y,h)} - \Upsilon\big(\sigma_b(y,h)^2/\sigma_{b,\max}^2\big)\right| 
\end{align} 
is not continuous on $l_\infty(\mathcal{T})$. This is the point where the restriction to $\mathcal{T}_\delta$ is necessary. The uniformity over $b$ from step (i) is transferred by the continuous mapping theorem because the mapping \eqref{eq: map_l_infty} restricted to $l_\infty(\mathcal{T}_\delta)$ is Lipschitz continuous, see Lemma~\ref{uniform_continuous_mapping}.

Step (iii) relies on results about empirical process theory that were developed in derivation of Theorem~\ref{Multiscale_Lemma}.

\section{Sketch of the proof of Theorem \ref{lower_bound}}\label{Sec_supp}

The proof of the lower bound in Theorem \ref{lower_bound} relies on the construction of several drift functions that belong to the alternative $H_1(b_0,\eta)$, but are close to the null in the distance $\Delta_J$ the theorem is formulated for. The innovation is that constructing these alternatives close enough to the null leads to a fixed point problem.

The start of the proof is to deduce the classical inequality
\begin{align}\label{eq: lower_av_likelihood}
\inf_{\substack{b\in H_1(b_0,\eta) \cap\{ b-b_0\in \mathcal{H}(\beta, L)\}:\\ \Delta_J(b)\geq (1-\epsilon_T)c_*\delta_T}} \E_b[\psi] - \alpha 
\leq  \E_{b_0+\eta}\left[ \left| \frac1N \sum_{k=1}^N \frac{d\Pr_{b_k}}{d\Pr_{b_{0}+\eta}}(X) - 1\right|\right]
\end{align}
that holds true for each test $\psi$ with $\sup_{b\in H_0(b_0,\eta)}\E_{b}[\psi]\leq\alpha$ and suitable alternatives $b_k$.
Those will be constructed in such a way that the last expression tends to zero for $T\to\infty$, in particular each $b_k$ and their number $N$ will depend on $T$. It will turn out from our construction that the likelihood ratios are not independent and proceeding with Cauchy-Schwarz' inequality does not yield tight enough bounds. For this purpose, we prove the following result.

\begin{proposition}\label{lemma_E_absvalue}
Let $Z_1,\dots, Z_m$ be strictly positive random variables with $\E[Z_i]=1$ and $\E[Z_i Z_j]\leq C_0$ for all $1\leq i,j\leq m$ and some constant $C_0\geq 1$. Then for all $\epsilon>0$ and $0<\nu\leq 1$ we have
\begin{align*}
\E\left[\left| \frac1m \sum_{i=1}^m Z_i - 1\right|\right] &\leq \sqrt{\epsilon} + \left( 2\epsilon^{-\nu} m^{-(1+\nu)} \sum_{i=1}^m \E[ Z_i^{1+\nu}] \right)^\frac12 + \sqrt{C_0-1}  \\
&\hspace{1cm} +2\epsilon^{-\nu}m^{-(1+\nu)}\sum_{i=1}^m \E[ Z_i^{1+\nu}].
\end{align*} 
\end{proposition}

The construction of the hypotheses $b_k$ that allow to achieve the optimal constant in Theorem \ref{lower_bound} is more sophisticated as compared to the standard approach.
Typically, $b_k\in H_1(b_0,\eta) \cap\{ b-b_0\in \mathcal{H}(\beta, L)\}$ with $\Delta_J(b_k)\geq (1-\epsilon_T)c_*\delta_T$ is constructed by adding a localized kernel to some boundary case of the null at location $y_k$, for different $k$, i.e. $b_k$ is of the form
\[ b_k = b_0 + \eta + \textrm{localized kernel at } y_k.\]
Furthermore, with $\Delta_J$ as defined in \eqref{eq: Delta_J}, $b_k$ should be smallest possible in the sense
\begin{align}\label{eq: lower_delta}
\begin{split}
\Delta_J(b_k) &=\sup_{x\in J} \Big(|b_k(x)-b_0(x)| - \eta\Big) \left(\frac{q_{b_k}(x)}{\sigma^2}\right)^\frac{\beta}{2\beta+1}=  (1-\epsilon_T)c_*\delta_T,
\end{split}
\end{align}  
where $q_{b_k}$ denotes the invariant density corresponding to the drift $b_k$. The problem is to find $y_k$ and the corresponding localized kernels such that \eqref{eq: lower_delta} ist satisfied. This is rather involved due to the dependence of $q_{b_k}$ on $b_k$ itself.
Our ansatz is as follows: we define
\[ b^w(x) := b_{0}(x) +\eta + L(1-\epsilon_T)(h_T^w)^\beta K_T^\beta\left(\frac{x-y^w}{h_T^w}\right)\]
with 
\[ h_T^w := \left(\frac{c_*}{L}\right)^\frac{1}{\beta} \left( \frac{\sigma^2 \log T}{T w}\right)^\frac{1}{2\beta+1},\]
where $y^w$ is a location that depends continuously on $w$ and $K_T^\beta$ is a local Lipschitz approximation to the solution $K_\beta$ of \eqref{eq: optimal_recovery} in case $\beta<1$. The particular dependence on $w$ in $h_T^w$ is inspired by the optimal bandwidth for detection where the bandwidth depends in the same way on the invariant density.
For such a drift function $b^w$, \eqref{eq: lower_delta} is equivalent to the following \textit{fixed point problem} 

\[ w = q_{b^w}(y^w) K_T(0)^\frac{2\beta+1}{\beta}.\]

The following result states that there exists a solution. Its proof can be found in Appendix~\ref{App_lower}.

\begin{lemma}\label{lemma_fixed_point}
Let $R>0$. Choose $y\in [-A', A'] $ and set $y^w = y+Rh_T^w$. Then for $T\geq T_0$ large enough and $\frac12<c_T<1$, there exists $\tilde{w}\in [c_TL_*, c_T L^*]$, such that \vspace{-0.7ex}
\[ \tilde{w}= c_Tq_{b^{\tilde{w}}}(y^{\tilde{w}}). \]
\end{lemma}

With this lemma, we can define the drifts $b_k$ satifying \eqref{eq: lower_delta} interatively. As the height and support of the additive localized kernel depends on $h_T^w$, it will be different with varying location. 

By the aforementioned Proposition~\ref{lemma_E_absvalue}, the proof is now reduced to bounding the $(1+\gamma)$-moments of likelihood ratios of diffusions driven by different drifts. Those are available using Girsanov's theorem and the further calculation heavily depends on the occupation times formula and the concentration result~\ref{moments_local-invariant} for the empirical density. 

\begin{remark}
Using the local asymptotic equivalence of the diffusion model to a Gaussian white noise model as given in Section~$2$ of \cite{Dalalyan/Reiss} provides another way to prove the lower bound in case of the simple null $\eta=0$.
\end{remark}

\section{Outlook to the multidimensional case}\label{Sec_outlook}

The theory of ergodic diffusions solving a stochastic differential equation of the form \eqref{eq: SDE} is not limited to the scalar case and neither is the statistical analysis.\\
While the construction of a multiscale test statistic is rather straightforward in higher dimension, a generalization of Theorem~\ref{worst_case_delta}, which was already highly non-trivial in dimension $d=1$, is not available a priori. However, the identification of a quantile is immediate in case of the simple null hypothesis $b=b_0$, corresponding to $\eta=0$. Here, also rate-optimality in the minimax sense transfers to higher dimension, where the main obstacle is that no (point) local time and occupation times formula are available in dimension $d>1$. This is merely of technical nature and alternative tools are outlined in \cite{Strauch_2} (cf. Lemma~$2$), and~\cite{Strauch_3}.

The stability results of Section \ref{Sec_Stability} do need a completely different approach. For the stochastic integrals in our test statistic, we can no longer have a continuous dependence on the data in supremum metric for $d>1$, see Section~$7$ in \cite{Diehl/Friz}. To overcome this problem, we may employ rough path theory to construct a rough path extension of our test. Such an extension, similar in spirit to \cite{Diehl/Friz}, gives a pathwise definition of the test statistic. A continuity result in analogy to Theorem~\ref{thm_continuity} can then be derived with respect to the $\alpha$-Hölder rough path metric instead of the supremum metric. 
Whether higher dimensional analogues of the strong results of Section~\ref{SubSec_fractional_model} with uniformity over the drift can be deduced in this way remains, however, totally unclear.

\smallskip
\textbf{Acknowledgments.} This work has been supported in part by the \textit{Research Unit $5381$} of the German Research Foundation (DFG).

\newpage
\setcounter{section}{0}
\renewcommand*\thesection{\Alph{section}}

\begin{frontmatter}

\title{\vspace{-0.3cm} Supplement to "Sharp adaptive and pathwise stable similarity testing for scalar ergodic diffusions"}
\runtitle{Supplement to "Similarity testing for ergodic diffusions"}

\author{\fnms{Johannes} \snm{Brutsche}\ead[label=e1]{johannes.brutsche@stochastik.uni-freiburg.de}}
\and
\author{\fnms{Angelika} \snm{Rohde}\ead[label=e2]{???}}

\begin{aug}
\end{aug}

\runauthor{J. Brutsche and A. Rohde}

\affiliation{Albert-Ludwigs-Universit\"at Freiburg}

\end{frontmatter}

\vspace{-0.1cm}
This supplementary material is organized as follows:\\
\vspace{-0.2cm}

{\small
\tableofcontents
\addtocontents{toc}{\protect\setcounter{tocdepth}{3}} 
}

\section{Notation}

For any set $E$ and function $f:E\rightarrow \R$ we denote
\[ \|f\|_E := \sup_{x\in E} |f(x)|. \]
When $E=\R$ we denote $\|f\|_\infty := \sup_{x\in\R} |f(x)|$. For a compact set $K\subset \R^n$ we denote by $\mathcal{C}(K)$ the set of continuous functions $f:K\rightarrow\R$ and consider it as a normed space with the norm $\|\cdot\|_K$. Moreover, we set 
\[ l_\infty(E) :=\{ f:E\rightarrow \R: \|f\|_E <\infty\}. \]
We denote by $\|\cdot\|_{L^p(\Pr)}$ the $L^p$-norm with respect to the probability measure $\Pr$. On the other hand, for a subset $I\subset\R$ we write $\|\cdot\|_{L^p(I)}$ for the $L^p$-norm with respect to the Lebesgue measure $\lambda(\cdot)$ on $\R$ and denote the corresponding $L^p$-space by $L^p(I)$. If $I=\R$, we simply write $\|\cdot\|_{L^p}$. By $\|\cdot\|_2$ we denote the Euclidean norm on $\R^d$.\\
For real numbers $a,b\in\R$ we use the notation $a\wedge b:=\min\{a,b\}$ and $a\vee b :=\max\{a,b\}$, as well as $a_+ := \max\{a,0\}$ and $\textrm{sign}(a) = a/|a|$. For functions, in particular random variables, maximum and minimum are understood pointwise and also abbreviated with $\wedge$ and $\vee$. The cardinality of a finite set $B$ is denoted by $\#B$.\\
We use the Landau symbols $\mathcal{O}(a_T)$ and $o(a_T)$, where 
\[ b_T=\mathcal{O}(a_T)\Leftrightarrow \limsup_{T\to\infty} \left|\frac{b_T}{a_T}\right|<\infty\quad \textrm{ and }\quad b_T=o(a_T)\Leftrightarrow \lim_{T\to\infty} \left|\frac{b_T}{a_T}\right| =0.\] 
In case $a_T=1$, we write $\mathcal{O}_T(1)$ and $o_T(1)$ to indicate the running index.\\
By $\rightarrow_{as}$, $\rightarrow_{\Pr}$ and $\rightarrow_{\mathcal{D}}$ we denote almost sure convergence, convergence in probability and convergence in distribution, respectively. By $\Rightarrow$ we denote weak convergence of measures. $X_n=o_\Pr(1)$ means that $(X_n)_{n\in\N}$ converges to zero in probability and $X_n=\mathcal{O}_\Pr(1)$ that $(X_n)_{n\in\N}$ is stochastically bounded. We frequently use the notation
\[ \Pr( A,B ) := \Pr(A\cap B).\]
Furthermore, throughout the whole paper, we fix constans $A,\gamma,\sigma>0$ and $C\geq 1$ that are used in the definition of the drift function class $\Sigma(C,A,\gamma,\sigma)$ given in Section~\ref{Sec_Notation}. Additionally, $L_*$ and $L^*$ are fixed as the constants for the uniform lower and upper bounds on the invariant density $q_b$ over the class $\Sigma(C,A,\gamma,\sigma)$ that are provided by Lemma~\ref{bound_invariant_density}. 


\section{Computational aspects and simulation study}\label{Sec_Implementation}

In this section we first discuss how the test statistic $T_T^\eta(X)$ given in \eqref{eq: T_eta} can be computed for a given observation $X=(X_t)_{t\in [0,T]}$, see Subsection~\ref{SubSec_Imp1}. In Subsection~\ref{SubSec_Imp2}, an extended simulation study is given that includes the identification of regions of derivation from the null hypothesis which are illustrated in the following Figure~\ref{Fig_1}.

\tiny
\begin{center}
\begin{figure}[h]
\begin{tikzpicture}
\draw [thick] (0,0) -- (10,0);


\draw[blue, domain=0.5:9.5] plot (\x,{2+0.01*(\x-5)^3}) node(1){};
\draw[blue, domain=0.5:9.5, dotted] plot (\x,{1.5+0.01*(\x-5)^3}) node(2){};
\draw[blue, domain=0.5:9.5, dotted] plot (\x,{2.5+0.01*(\x-5)^3}) node(3){};

\draw[green, domain=0.5:4] plot (\x,{2+0.01*(\x-5)^3 +0.3*sin(70*\x)}) {};
\draw[green, domain=4:4.7] plot (\x,{2+0.01*(4-5)^3 +0.3*sin(70*4) - (\x-4)^2}) {};
\draw[green, domain=4.7:5.3] plot (\x,{2+0.01*(4-5)^3 +0.3*sin(70*4) - (4.7-4)^2 +1.5*(\x-4.7)^2)}) {};

\draw[green, domain=5.3:7.5] plot (\x,{2+0.01*(4-5)^3 +0.3*sin(70*4) - (4.7-4)^2 +1.5*(5.3-4.7)^2) + 0.7*(\x-5.3)^(1/1.5) }) {};
\draw[green, domain=7.5:8.5] plot (\x,{2+0.01*(4-5)^3 +0.3*sin(70*4) - (4.7-4)^2 +1.5*(5.3-4.7)^2) + 0.7*(7.5-5.3)^(1/1.5) -0.3*(\x-7.5)^(1.8) }) {};
\draw[green, domain=8.5:9.5] plot (\x,{2+0.01*(4-5)^3 +0.3*sin(70*4) - (4.7-4)^2 +1.5*(5.3-4.7)^2) + 0.7*(7.5-5.3)^(1/1.5) -0.3*(8.5-7.5)^(1.8) - 0.25*sin(220*(\x-8.5) }) {};

\draw[-, green, dotted] (4.45,0)--(4.45,1.5);
\draw[-, green, dotted] (5.15,0)--(5.15,1.5);
\draw[-, green, dotted] (6.5,0)--(6.5,2.54);
\draw[-, green, dotted] (8.1,0)--(8.1,2.82);
\draw [green, thick] (4.45,0) -- (5.15,0);
\draw [green, thick] (6.5,0) -- (8.1,0);

\draw[|-|, red] (4,2)--(4,2.5) node[right, midway]{$\eta$};

\node [right, color=blue] at (1) {$b_0$};

\end{tikzpicture}
\caption{\small{An illustration of a function $b_0$ and its $\eta$-environment, together with some other drift function that generates the observed data. The green intervals on the horizontal line are the regions of deviation from similarity we aim to identify.}}\label{Fig_1}
\end{figure}
\end{center}
\normalsize

\subsection{Computability of the test statistic $T_T^\eta$}\label{SubSec_Imp1}
When implementing the test statistic $T_T^\eta(X)$ for an observed path $X$ one has to evaluate several Lebesgue integrals and the stochastic It\^{o} integral $\int_0^T K_{y,h}(X_s) dX_s$. Given a continuously differentiable kernel $K$, we have seen in Section~\ref{Sec_Stability} that by It\^{o}'s formula
\[ \int_0^T K_{y,h}(X_s) dX_s = \int_{X_0}^{X_T} K_{y,h}(z) dz - \frac{\sigma^2}{2} \int_0^T (K_{y,h})'(X_s) ds, \]
and based on this we established a version $\tilde\Psi_{T,y,h}^{b_0}(X)$ of $\Psi_{T,y,h}^{b_0}(X)$, given in \eqref{eq: tilde_psi}, that is defined pathwise. Working with this version directly yields a way to compute the local statistics for any given path $X$ and we have seen in Theorem~\ref{Upper_bound} that for a rate optimal procedure, we can indeed choose a kernel that is continuously differentiable.

\begin{remark}[Grid of bandwidths]\label{remark_bandwidth}
In practice, one has to restrict to a finite subset $\mathcal{T}'\subset\mathcal{T}$ when computing the statistic $T_T^\eta$. The proof of Theorem~\ref{Upper_bound} reveals that the minimal bandwidth used for detection is of order $(\log T/T)^{1/(2\beta+1)}$. Hence, it suffices to choose $h$ from the finite set
\[ \mathcal{H}_T := \left\{ k T^{-1} \mid k=1,2,\dots, \lfloor AT\rfloor \right\}. \]
The same distance of points should be applied to define a grid for $y$, i.e. we choose this parameter from
\[ \mathcal{Y}_T := \left\{ -A + kT^{-1} \mid k=1,2,\dots \lfloor 2AT\rfloor \right\} \]
and the finite set $\mathcal{T}'$ of location and bandwidth parameters for implementation may be chosen as $\mathcal{T}' = (\mathcal{Y}_T\times \mathcal{H}_T)\cap\mathcal{T}$, assuming $T\in\Q$.
\end{remark}

\begin{remark}[Quantiles for the simple null]\label{remark_quantile_simple}
From a practical point of view it is much more convenient to replace the quantile $\kappa_{T,\alpha}^{\neq b_0}$ in the test $\phi_T^{b_0}$ given in \eqref{eq: phi_0} by the quantile
\[ \kappa_\alpha^{b_0} := \min\left\{ r\in\R \mid \Pr\left( S_{b_0}\leq r\right)\geq 1-\alpha\right\} \]
of the limiting statistic $S_{b_0}$ from Theorem~\ref{weak_conv}, which can be easily simulated by Monte Carlo methods as the explicit form of the invariant density $q_{b_0}$ is known.
\end{remark}

\subsection{Numerical example}\label{SubSec_Imp2}

In this section we will give a numerical illustration of the testing procedure of Section~\ref{Sec_similarity} for the parameter specification $b_0(x) =-x$ together with $A=\sigma^2=1$.\\
The quantiles $\kappa_{\eta,\alpha}$ given in \eqref{eq: kappa_eta_alpha} are estimated by the empirical quantiles of $N=10000$ independent samples of $U_1\vee U_2+4\sqrt{A\eta/\sigma^2}$. Results for different $\eta$ and $\alpha$ are given in Table~\ref{Table_quant}. It can be seen that the influence of $\eta$ on the quantiles is mostly due to the additive correction term $4\sqrt{A\eta/\sigma^2}$. 
Note that $q_{b_0\pm\eta}$ is given explicitly for the specification $b_0(x)=-x$ by the stationary density of an Ornstein--Uhlenbeck process which is known to be Gaussian.

\vspace{0.3cm}

\begin{center}
\begin{table}[h]
\begin{tabular}{c|ccc}
\textbf{$\eta$} & \textbf{$\alpha=0.1$} & \textbf{$\alpha=0.05$} & \textbf{$\alpha = 0.01$} \\
\hline
$0$  & $1.3781$ & $1.6775$ & $2.2468$ \\
$0.05$ & $1.4222$ & $1.7311$ &  $2.3178$ \\
$0.1$ & $1.4473$ & $1.7391$ & $2.3726$ \\
$0.2$ & $1.4503$ & $1.7497$ & $2.3047$ \\
$0.3$ & $1.4646$ & $1.7458$ & $2.3403$ \\
$0.4$ & $1.4642$ & $1.7692$ & $2.3748$ \\
$0.5$ & $1.5021$ & $1.7670$ & $2.3114$ \\
\end{tabular}
\quad
\begin{tabular}{c|ccc}
\textbf{$\eta$} & \textbf{$\alpha=0.1$} & \textbf{$\alpha=0.05$} & \textbf{$\alpha = 0.01$} \\
\hline
$0$  & $1.3781$ & $1.6775$ & $2.2468$ \\
$0.05$ & $2.3166$ & $2.6256$ &  $3.2122$ \\
$0.1$ & $2.7122$ & $3.0040$ & $3.6375$ \\
$0.2$ & $3.2391$ & $3.5386$ & $4.0935$ \\
$0.3$ & $3.6555$ & $3.9367$ & $4.5312$ \\
$0.4$ & $3.9900$ & $4.2990$ & $4.9046$ \\
$0.5$ & $4.3305$ & $4.5954$ & $5.1398$ \\
\end{tabular}
\medskip
\caption{\small{Quantiles of $U_1\vee U_2$ (left) and $U_1\vee U_2+4\sqrt{A\eta/\sigma^2}$ (right) for different $\eta$ and $\alpha$. The simulated quantiles are taken from a sample of $N=10000$.
}}\label{Table_quant}
\end{table}
\end{center}

To demonstrate the detection power of our test, we simulate data that follows the SDE in \eqref{eq: SDE} with
\begin{align}\label{eq: b_alt}
b_{alt}(x) := -x  - 0.8 K\left(\frac{x+0.6}{0.15}\right) + 0.15 K\left(\frac{x}{0.2}\right) + 0.5 K\left(\frac{x-0.5}{0.1}\right),
\end{align} 
where
\begin{align}\label{eq: K_alt}
K(x) := \frac{15}{16} \left( 1-x^2\right)^2 \1_{\{|x|\leq 1\}}.
\end{align} 
By Remark~\ref{remark_fixed_point}, our results hold true in the case of a fixed starting point of the diffusion and we assume $X_0=0$ for simplicity in implementation. A display of $b_{alt}$ is included in Figure~\ref{Figure_2}. The simulation is done with a time horizon $T=10000$ on an equidistant grid with width $0.001$, i.e. we simulate $10^7$ values of the diffusion according to the Euler-Maruyama scheme (cf. Section~$9.1$ and $10.2$ in \cite{Kloeden}). In Figure~\ref{Figure_2} the minimal intervals of $D_{\alpha}^\eta$ given in Remark~\ref{remark_D_alpha} are depicted for $\alpha=0.05$ and the two values $\eta_1=0.05$ and $\eta_2=0.2$. Here, an interval $I\in D_\alpha^\eta$ is called minimal, if for any interval $I\neq J\in D_\alpha^\eta$ we have $J\nsubseteq I$. The test statistic is computed with the kernel given in \eqref{eq: K_alt}. \vspace{-1cm}

\begin{center}
\begin{figure}[h]
\begin{tabular}{cc}
\includegraphics[scale=0.23]{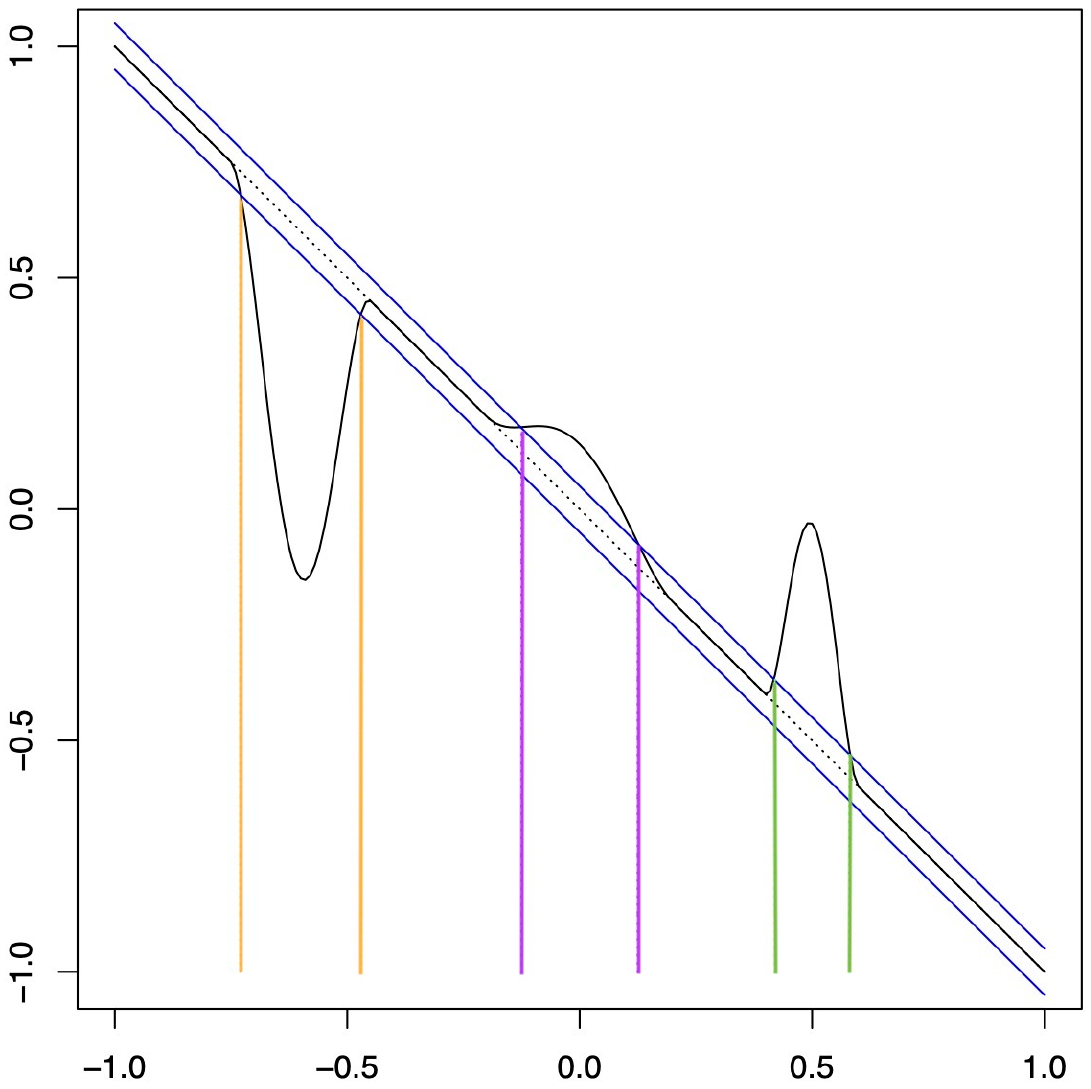} & \includegraphics[scale=0.23]{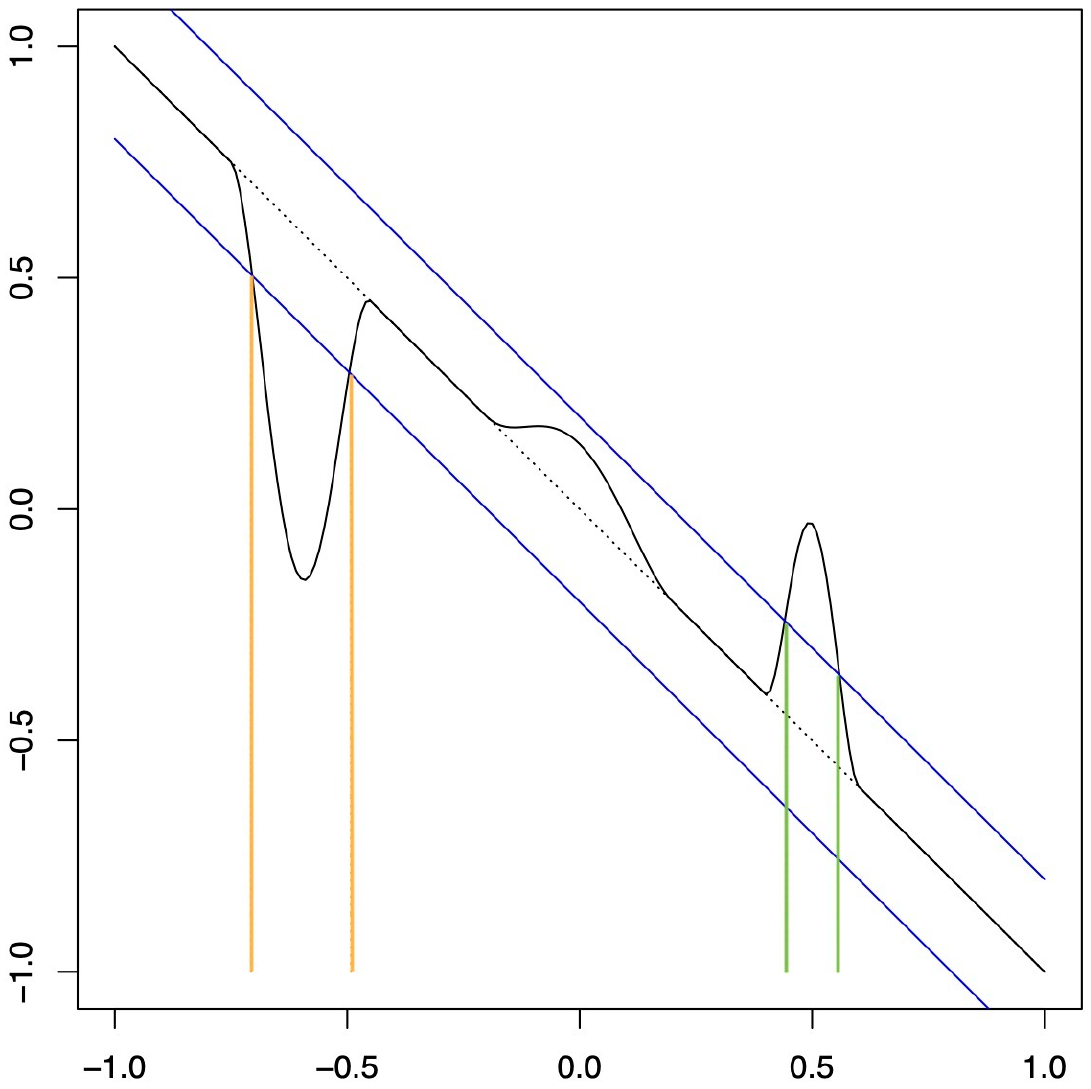} \vspace{-2cm} \\
\includegraphics[scale=0.23]{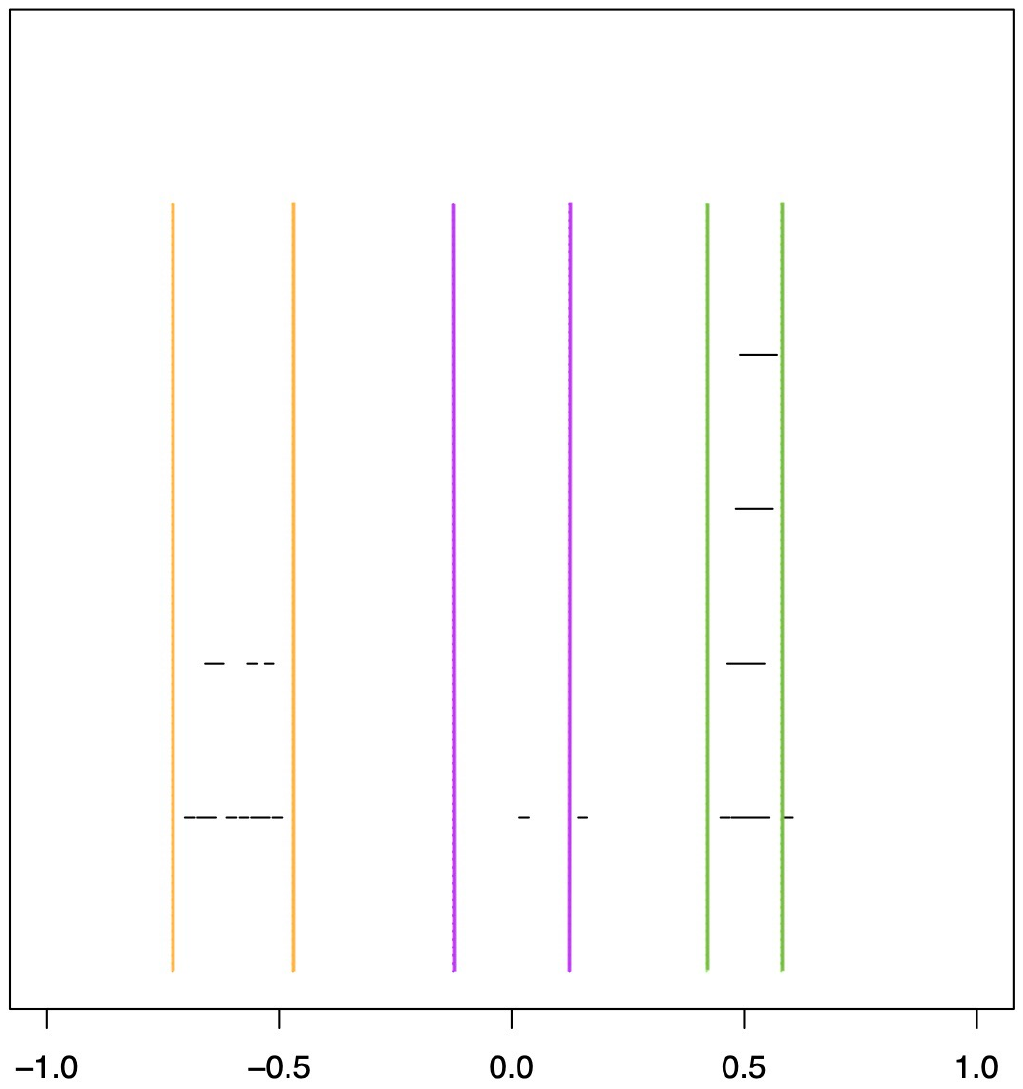} & \includegraphics[scale=0.23]{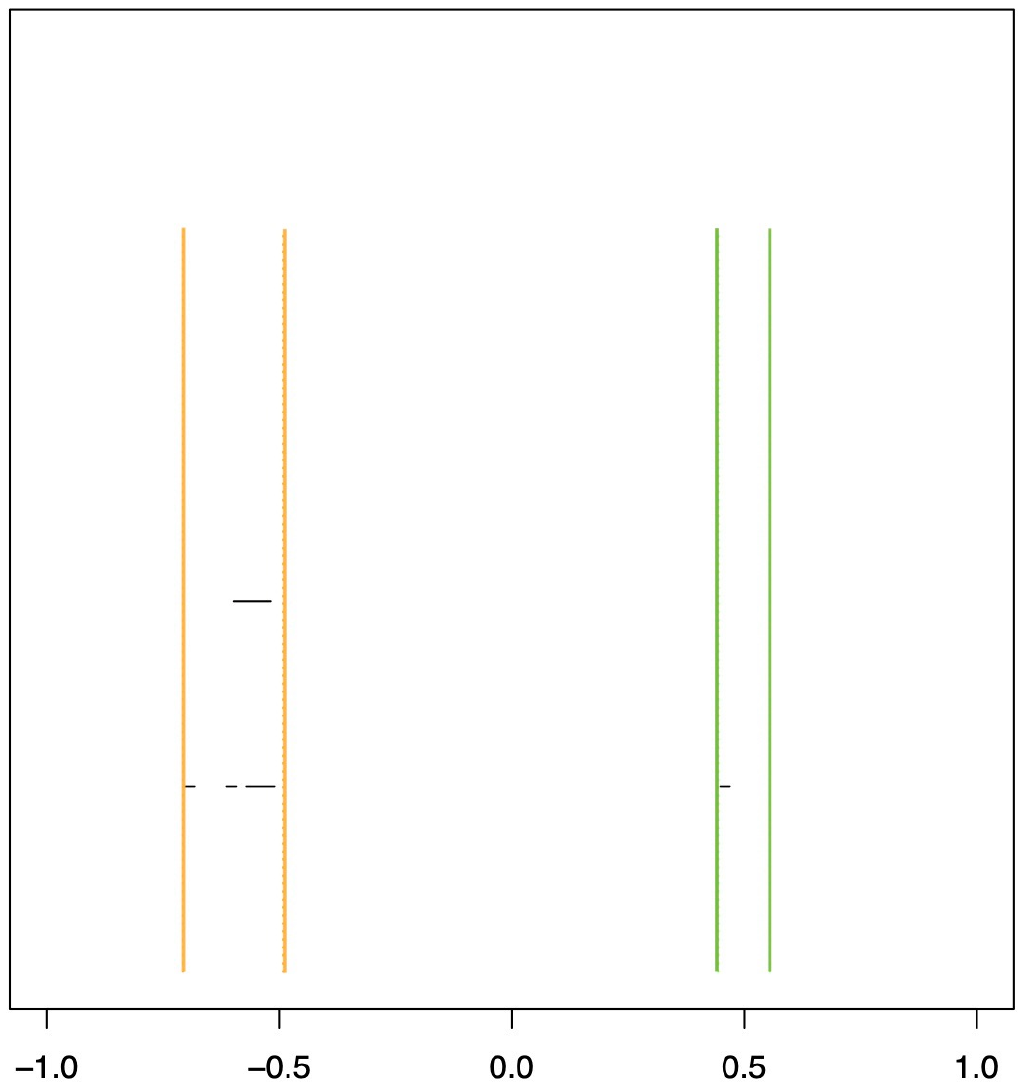} \vspace{-1cm}
\end{tabular}
\caption{\small{Both upper pictures show the drift $b_{alt}$ together with its $\eta$-environment for $\eta_1=0.05$ (left) and $\eta_2=0.2$ (right). In the bottom, the minimal intervals are depicted for both cases according to the parameter specification given in the text and $\alpha=0.05$.}}\label{Figure_2}
\end{figure}
\end{center}

\vspace{-0.8cm}
In Table~\ref{Table_prop_rejection}, the proportion of detections is given for various values of $\eta$ and $\alpha$ for $N=500$ simulated paths. The other simulation parameters and the grid specification are the same as given above. Table~\ref{Table_prop_rejection} also contains this proportion of detections 'locally' in order to show which of the three possible violations from $H_0(b_0,\eta)$ is detected. Note that it depends on $\eta$ if $b_{alt}$ lies outside an $\eta$-environment of $b_0$ or not. By our choice of $b_{alt}$ in \eqref{eq: b_alt} and $\|K\|_{[-1,1]}=\frac{15}{16}$ for $K$ in \eqref{eq: K_alt}, there is a deviation from $H_0(-x,\eta)$ for all 
\begin{align*}
&\eta< 0.75\textcolor{white}{0000} \quad \textrm{ in the interval }\ [-0.75,-0.45], \\
&\eta< 0.140625 \quad \textrm{ in the interval }\ [-0.2,0.2], \\
&\eta< 0.46875\textcolor{white}{0} \quad \textrm{ in the interval }\ [0.4,0.6]. \\
\end{align*}

\begin{center}
\begin{table}[h]
\begin{tabular}{c|ccc}
\textbf{$\eta$} & \textbf{$\alpha=0.1$} & \textbf{$\alpha=0.05$} & \textbf{$\alpha = 0.01$} \\
\hline
$0$  & $1.00$ & $1.00$ & $1.00$ \\
$0.05$ & $1.00$ & $1.00$ & $1.00$ \\
$0.1$ & $1.00$ & $1.00$ & $1.00$ \\
$0.2$ & $1.00$ & $1.00$ & $1.00$ \\
$0.3$ & $0.98$ & $0.97$ & $0.93$ \\
$0.4$ & $0.44$ & $0.33$ & $0.18$ \\
$0.5$ & $0.14$ & $0.11$ & $0.07$ \\
\end{tabular}
\quad
\begin{tabular}{c|ccc}
\textbf{$\eta$} & \textbf{$\alpha=0.1$} & \textbf{$\alpha=0.05$} & \textbf{$\alpha = 0.01$} \\
\hline
$0$  & $1.00$ & $1.00$ & $1.00$ \\
$0.05$ & $1.00$ & $1.00$ & $1.00$ \\
$0.1$ & $1.00$ & $1.00$ & $1.00$ \\
$0.2$ & $1.00$ & $1.00$ & $1.00$ \\
$0.3$ & $0.98$ & $0.96$ & $0.88$ \\
$0.4$ & $0.40$ & $0.29$ & $0.16$ \\
$0.5$ & $0.12$ & $0.09$ & $0.06$ \\ 
\end{tabular}
\hfill
\begin{tabular}{c|ccc}
\multicolumn{4}{c}{} \\
\textbf{$\eta$} & \textbf{$\alpha=0.1$} & \textbf{$\alpha=0.05$} & \textbf{$\alpha = 0.01$} \\
\hline
$0$  & $0.99$ & $0.98$ & $0.93$ \\
$0.05$ & $0.75$ & $0.64$ & $0.46$ \\
$0.1$ & $0.49$ & $0.41$ & $0.24$ \\
$0.2$ & $0.14$ & $0.10$ & $0.06$ \\
$0.3$ & $0.04$ & $0.03$ & $0.02$ \\
$0.4$ & $0.01$ & $0.01$ & $0.00$ \\
$0.5$ & $0.00$ & $0.00$ & $0.00$ \\
\end{tabular}
\quad
\begin{tabular}{c|ccc}
\multicolumn{4}{c}{} \\
\textbf{$\eta$} & \textbf{$\alpha=0.1$} & \textbf{$\alpha=0.05$} & \textbf{$\alpha = 0.01$} \\
\hline
$0$  & $1.00$ & $1.00$ & $1.00$ \\
$0.05$ & $0.99$ & $0.99$ & $0.96$ \\
$0.1$ & $0.91$ & $0.84$ & $0.60$ \\
$0.2$ & $0.32$ & $0.27$ & $0.18$ \\
$0.3$ & $0.16$ & $0.12$ & $0.07$ \\
$0.4$ & $0.06$ & $0.05$ & $0.03$ \\
$0.5$ & $0.02$ & $0.01$ & $0.00$ \\
\end{tabular}
\medskip
\caption{\small{Proportions of rejection of the null for different $\eta$ and $\alpha$ and $N=500$ simulated paths. We count, if there is any detected interval (upper left), a minimal interval intersecting with $[-0.75,-0.45]$ (upper right), $[-0.2,0.2]$ (bottom left) or $[0.4,0.6]$ (bottom right).}}\label{Table_prop_rejection}
\end{table}
\end{center}

\section{Preliminaries on scalar ergodic diffusion processes}\label{App_prelim}

Following Section~\ref{Sec_Notation}, we consider the stochastic differential equation~\eqref{eq: SDE} of the form
\[ dX_t = b(X_t) dt + \sigma dW_t, \quad X_0=\xi\]
for a standard Brownian motion $W$, initial condition $\xi\sim\mu_b$ independent of $W$ and a drift function $b$ belonging to the class $\Sigma(C,A,\gamma,\sigma)$. We recall its definition for fixed constants $A,\gamma, \sigma >0$ and $C\geq 1$,
\begin{align*}
\Sigma(C,A,\gamma, \sigma) &:=\left\{ b\in \textrm{Lip}_\textrm{loc}(\R):\ |b(x)|\leq C(1+|x|)\ \forall x\in\R \textcolor{white}{\frac12} \right.\\
&\hspace{3cm}\left.\textrm{ and } \frac{b(x)}{\sigma^2}\textrm{sign}(x)\leq -\gamma\ \forall |x|\geq A\right\}.
\end{align*}
Here, $\textrm{Lip}_\textrm{loc}(\R)$ denotes the class of all functions $f:\R\rightarrow\R$ such that for every $n\in\N$ there exists a constant $L_n>0$ such that 
\[ |f(x) - f(y)| \leq L_n |x-y|\quad \textrm{ for all } x,y\textrm{ with } |x|,|y|\leq n. \]
As noted in Section~\ref{Sec_Notation}, for each $b\in\Sigma(C,A,\gamma,\sigma)$ the diffusion $X$ admits the invariant density
\[ q_b(x) := \frac{1}{C_{b,\sigma} } \exp\left( \int_0^x \frac{2b(u)}{\sigma^2} du\right) \quad\textrm{ for all } x\in\R,\]
with the normalizing constant 
\[ C_{b,\sigma}:= \int_\R \exp\left(\int_0^x \frac{2b(u)}{\sigma^2} du\right)dx,\]
where for $x<0$ the integrals should be read as $\int_0^x f(u)du = -\int_{x}^0 f(u) du$.
Some regularity properties of $q_b$ are already determined by $b$. In particular, it is easily seen that $q_b$ is differentiable and for all $x\in\R$,
\[q_b'(x) = 2\sigma^{-2}b(x) q_b(x).\]
Moreover, an important property that is made use of in several steps, is that $q_b$ and $q_b'$ can be uniformly upper bounded over $\Sigma(C,A,\gamma,\sigma)$, whereas on the other hand on the inverval $[-A,A]$, the invariant density $q_b$ is uniformly bounded away from zero.

\begin{lemma}[\cite{Brutsche}, Lemma 3.2.1 and 3.2.2]\label{bound_invariant_density}
There exist two constants $L^* = L^*(C,A,\gamma,\sigma)<\infty$ and $L_*=L_*(C,A,\gamma,\sigma)>0$ such that
\[ \sup_{b\in\Sigma(C,A,\gamma,\sigma)} \max\left\{ \|q_b\|_\infty, \| q_b'\|_\infty\right\} \leq L^*\]
and
\[ \inf_{b\in \Sigma(C,A,\gamma,\sigma)} \inf_{x\in [-A,A]} q_b(x)\geq L_*.\]
\end{lemma}

Another very important property is that the normalized local time $\frac{1}{\sigma^2 T} L_T^\cdot (X)$ approximates the invariant density $q_b$ in the following sense.

\begin{proposition}[\cite{Strauch}, Corollary $14$]\label{moments_local-invariant}
Let $b\in\Sigma(C,A,\gamma,\sigma)$. Then there exist constants $c_1, c_2>0$ such that, for any $p, T\geq 1$, 
\begin{align*} 
&\sup_{b\in\Sigma(C,A,\gamma,\sigma)} \left(\E_{b}\left[ \left\| \frac{1}{T\sigma^2} L_T^\cdot(X) - q_b\right\|_\infty^p\right]\right)^\frac1p \\
&\hspace{2.5cm} \leq c_1\left( \frac{p}{T} + \frac{1}{\sqrt{T}}\left( 1+ \sqrt{p} + \sqrt{\log T}\right) + Te^{-c_2T}\right).
\end{align*}
\end{proposition}

\begin{remark}\label{remark_moment_invariant}
The preceding Proposition~\ref{moments_local-invariant} was proven for our setup where the diffusion $X$ is started in the invariant density. In Section~\ref{Sec_Stability} we need to consider the case where it is started in some fixed $x_0\in\R$ to compare it with a fractional diffusion. For this, it is crucial that our results of Section~\ref{Sec_Power} are still true under this assumption, which is also interesting on its own, see Remark~\ref{remark_fixed_point}. This can be established by showing that Proposition~\ref{moments_local-invariant} also works for $X_0=x_0\in [-A,A]$. A proof of this is given in detail in Section~$3.2.3$ of \cite{Brutsche}.
\end{remark}

By Markov's inequality we get as a direct corollary of Proposition~\ref{moments_local-invariant} that for every $\epsilon>0$, we have
\begin{align}\label{eq:conv_emprical_density_uniform}
\sup_{b\in\Sigma(C,A,\gamma,\sigma)} \Pr_b\left( \left\| \frac{1}{\sigma^2 T} L_T^\cdot (X) - q_b\right\|_\infty >\epsilon\right) \stackrel{T\to\infty}{\longrightarrow} 0.
\end{align}

With this result, we can deduce a uniform version of the weak law of large numbers for a bounded class of functions having compact support.

\begin{proposition}\label{uniform_ergodic}
Let $\mathcal{F}$ be a class of functions that are bounded uniformly by some constant $C_\mathcal{F}$ and supported in $[-A,A]$. Then for every $\epsilon>0$,
\[ \lim_{T\to\infty} \sup_{b\in\Sigma(C,A,\gamma,\sigma)} \Pr_b\left( \sup_{f\in\mathcal{F}}\left|\frac1T \int_0^T f(X_s) ds - \int_\R f(z) q_b(z) dz \right| >\epsilon\right) =0.\]
\end{proposition}
\begin{proof}
By the occupation times formula,
\begin{align*}
&\sup_{b\in\Sigma(C,A,\gamma,\sigma)}\Pr_b\left(\sup_{f\in\mathcal{F}} \left|\frac1T \int_0^T f(X_s) ds - \int_\R f(z) q_b(z) dz \right| >\epsilon\right)\\
&\hspace{1cm} =  \sup_{b\in\Sigma(C,A,\gamma,\sigma)}\Pr_b\left( \sup_{f\in\mathcal{F}}\left|\int_\R f(z) \left( \frac{1}{\sigma^2 T}L_T^z(X) -q_b(z)\right) dz \right| >\epsilon\right) \\
&\hspace{1cm} \leq \sup_{b\in\Sigma(C,A,\gamma,\sigma)}\Pr_b\left(  \left\| \frac{1}{\sigma^2 T}L_T^\cdot(X) -q_b\right\|_{[-A,A]}\ \sup_{f\in\mathcal{F}}\int_\R |f(z)| dz  >\epsilon\right) \\
&\hspace{1cm} \leq \sup_{b\in\Sigma(C,A,\gamma,\sigma)}\Pr_b\left(  \left\| \frac{1}{\sigma^2 T}L_T^\cdot (X) -q_b\right\|_{[-A,A]}  >\frac{\epsilon}{2A C_\mathcal{F}}\right),
\end{align*}
where we used that by assumption $\int_\R |f(z)|dz \leq 2A C_{\mathcal{F}}$ uniformly in $f\in\mathcal{F}$. Now, the last probability converges to zero for $T\to\infty$ by \eqref{eq:conv_emprical_density_uniform} and the claim follows.
\end{proof}

\section{Proof of Theorem \ref{Multiscale_Lemma} and general multiscale theory}\label{App_Multiscale}
This part of the appendix is organized in the following way: In Subsection~\ref{B1_SubSec} we present the multiscale theory that is used in our context, where the key result is Theorem~\ref{Multiscale_general}. In Subsection~\ref{B2_SubSec} we then provide the proof of Theorem~\ref{Multiscale_Lemma} as an application of this Theorem~\ref{Multiscale_general}. In addition, we use it to prove finiteness of the random variable $S_b$ given in Theorem~\ref{weak_conv}.

\subsection{Auxiliary results from empirical process theory}\label{B1_SubSec}
The key ingredient for the proof of Theorem~\ref{Multiscale_Lemma} will by a general result about the supremum of a stochastic process $Z=(Z(s))_{s\in\mathcal{S}}$ that is defined on a totally bounded metric space $(\mathcal{S},\rho)$. Totally bounded means that for arbitrary $u>0$ the capacity number
\[ D(u,\mathcal{S},\rho) := \max\left\{ \#\mathcal{S}_0: \mathcal{S}_0\subset\mathcal{S}, \rho(s,t)>u\textrm{ for different }s,t\in\mathcal{S}_0\right\}\]
is finite, where $\#S$ denotes the cardinality of a finite set $S$. Additionally, we consider a function $\sigma:\mathcal{S}\rightarrow (0,1]$, where $\sigma(s)$ measures the spread of $X(s)$. We assume that
\begin{align}\label{eqp: B2}
|\sigma(t) - \sigma(s)| \leq \rho(s,t)\quad \textrm{ for all }\ s,t\in\mathcal{S}
\end{align} 
and
\begin{align}\label{eqp: B12}
\left\{ t\in\mathcal{S}:\ \sigma(t)\geq \delta\right\}
\end{align}
is compact for any $\delta\in (0,1]$. A result of this kind was first established by Dümbgen and Spokoiny in \cite{Duembgen/Spokoiny} to derive a multiscale test in a regression setting and substantially refined by Dümbgen and Walther in \cite{Walther}. A setup with random metrics and time-dependent metric spaces $\mathcal{S}_T$ was treated by Rohde in \cite{Rohde}. In our scenario, the semimetric $\rho$ and the spread measure $\sigma$ are random and depend on the time horizon $T$, but are defined on the same fixed semimetric space $\mathcal{S}$ for all $T$. We follow the proofs given in technical report~\cite{Walther} to weaken the assumption on the tails of $Z$, by allowing for an additional $\log$-factor compared to Theorem~$8$ and Corollary~$9$ in this reference \cite{Walther}.

\begin{proposition}\label{sub-exp_differences}
Let $\mathcal{S}$ be countable, $L>0$ be some constant, and for $\delta>0$ let $G(\cdot, \delta)$ be a nondecreasing function on $[0,\infty)$ such that for all $\lambda\geq 0$ and $s,t\in\mathcal{S}$ with $\rho(s,t)\geq\delta$ we have
\[ \Pr\left( \frac{|Z(s)-Z(t)|}{\rho(s,t)}> G(\lambda,\delta)\right) \leq L e^{-\lambda}.\]
Then for arbitrary $\delta >0$ and $a\geq 1$
\[ \Pr\left( |Z(s)-Z(t)| \geq 12 J(\rho(s,t),a)\textrm{ for some } s,t\in\mathcal{S}\textrm{ with }\rho(s,t)\leq \delta\right)\leq \frac{L\delta}{2a},\]
where 
\[ J(\epsilon,a) := \int_0^\epsilon G\left(\log\left(\frac{aD(u,\mathcal{S},\rho)^2}{u}\right),u\right) du.\]
\end{proposition}
\begin{proof}
For finite $\mathcal{S}$ the statement follows by Theorem~$7$ in the technical report \cite{Walther} since the topology induced by the metric then coincides with the discrete topology with respect to which $Z$ is always continuous. For countable $\mathcal{S}$, the proposition is then a consequence of the theorem of monotone convergence.
\end{proof}

The following theorem is an extension of Theorem~$8$ in the technical report~\cite{Walther} that establishes this result for random metrics and spread measures and weakens the tail assumption by an additional log-term.

\begin{thm}\label{Multiscale_preparation}
Let $\mathcal{S}$ be countable and suppose that the following three conditions are satisfied:
\begin{enumerate}
\item[(a)] There exists a constant $c_0>0$ such that for (possibly) random function $\tilde\sigma:\Omega\times\mathcal{S}\rightarrow (0,1]$ and random metric $\tilde\rho:\Omega\times \mathcal{S}^2\rightarrow\R_{\geq 0}$ which satisfy~\eqref{eqp: B2}, 
\[ \frac{1}{c_0}\leq \frac{\tilde\sigma}{\sigma}\leq c_0\quad \textrm{ and }\quad \frac{1}{c_0}\leq \frac{\tilde\rho}{\rho}\leq c_0. \] 
\item[(b)] There exists a constant $L>0$ such that for arbitrary $\lambda \geq 1$ we have
\[ \Pr\left( |Z(s)| \geq \tilde\sigma(s) \lambda \right)\leq L\exp\left( -\frac{\lambda^2}{2} + \log(\lambda)\right).\]
\item[(c)] There exist positive constants $A_0,B,V>0$ such that
\[ D\left( u\delta, \{ t\in\mathcal{S}: \tilde\sigma(s)\leq\delta\}, \tilde\rho\right) \leq A_0 u^{-B}\delta^{-V}\quad \textrm{ for all }u,\delta\in (0,1].\]
\end{enumerate}
For a constant $Q>0$ define the events
\[ \mathcal{A}(Q,\delta) :=\left\{  \sup_{s,t\in\mathcal{S}: \tilde\rho(s,t)\leq\delta}\frac{|Z(s)-Z(t)|}{\tilde\rho(s,t) \log(e/\tilde\rho(s,t))} \leq Q\right\}.\]
Then there exist constants $C_0=C_0( A_0,B,V,Q,c_0)>0$ and $0<\delta_0\leq 1$ such that for $0<\delta\leq \delta_0$ the probability of the event \newpage
\begin{align*}
&\left\{ |Z|\leq \tilde\sigma \sqrt{ 2V\log(1/\tilde\sigma) + C_0\log\log(e/\tilde\sigma)} + C_0\tilde\sigma\log(e/(C_0\tilde\sigma))^{-1} \right.\\
&\hspace{7cm}\left. \textcolor{white}{\sqrt{V}}\textrm{ on } \{ s\in\mathcal{S}: \tilde\sigma(s)\leq\delta\}\right\}
\end{align*} 
is at least $\Pr(\mathcal{A}(Q,2\delta)) - C_0\log(e/(2c_0^2\delta))^{-\frac12}$.
\end{thm}

\begin{remark}\label{remark_multiscale_II}
The proof will show that an additional term $+c\log(\lambda)$ in condition (b) of Theorem~\ref{Multiscale_preparation} can only be added to the exponent $-\lambda^2/2$ for $c <2$.
\end{remark}

\begin{proof}[Proof of \ref{Multiscale_preparation}]
We define a set $\mathcal{S}^*$ inductively as follows: Let $s_1$ be any point in $\mathcal{S}$ maximizing $\sigma$. Next, let $u$ be some continuous, non-decreasing function from $(0,1]$ into iteself that will be specified later. Suppose we already picked $s_1,\dots, s_m$ and the set 
\[\mathcal{S}':= \left\{ s\in\mathcal{S}:\ \min_{i=1,\dots, m} \rho(s,s_i)\geq u(\sigma(s)) \sigma(s)\right\}\]
is non-empty. Then we define $s_{m+1}$ to be an element of it with maximal value $\sigma(s)$, which is well-defined as the displayed set $\mathcal{S}'$ is closed and $\{ \sigma\geq\delta\}$ is compact for any $\delta>0$ by \eqref{eqp: B12}. Following this procedure we end up with a finite or countable set $\mathcal{S}^*=\{ s_1,s_2,\dots\}$ and its construction entails that
\[ \sigma(s_1)\geq \sigma(s_2)\geq \dots\]
An important property of this set $\mathcal{S}^*$ is that for any $s\in\mathcal{S}$ there exists a point $t\in\mathcal{S}^*$ such that
\begin{align}\label{eqp: B18}
\sigma(s)\leq \sigma(t)\quad \textrm{ and }\quad \rho(s,t) < u(\sigma(s))\sigma(s).
\end{align} 
This can be seen in the following way: for $s\in\mathcal{S}$ let $m$ be a maximal index such that $\sigma(s_m)\geq \sigma(s)$. If $\rho(s,s_i)\geq u(\sigma(s))\sigma(s)$ for every $i\leq m$, then $s$ would belong to the set $\mathcal{S}'$, whence $\sigma(s_{m+1})\geq \sigma(s)$ contradicting the definition of $m$.
For $0<\delta\leq 1$ define
\[ \mathcal{S}^*(\delta) := \left\{ s\in\mathcal{S}^*:\ \frac{c_0\delta}{2}<\sigma(s)\leq c_0\delta\right\}, \]
which is contained in $\{ s\in\mathcal{S}: \sigma(s)\leq c_0\delta\}$ with $\rho(s,t)\geq u(c_0\delta/2)\frac{c_0\delta}{2}$ for different $s,t\in\mathcal{S}^*(\delta)$. Taking into account that
\begin{align*}
D\left( u\delta, \{ t\in\mathcal{S}: \sigma(t)\leq\delta\}, \rho\right) &\leq D\left( u\delta, \{ t\in\mathcal{S}: \tilde\sigma(t)\leq c_0\delta\}, c_0\tilde\rho\right)\\
&= D\left( c_0^{-1} u\delta, \{ t\in\mathcal{S}: \tilde\sigma(t)\leq c_0\delta\}, \tilde\rho\right),
\end{align*}\newpage

\noindent
we consequently have from assumption (c) that
\begin{align}\label{eqp: B1}
\begin{split}
\#\mathcal{S}^*(\delta) &\leq D\left( \frac{1}{2c_0} u(c_0\delta/2) c_0\delta, \{ s\in\mathcal{S}: \tilde\sigma(s)\leq c_0^2\delta\}, \tilde\rho\right)\\
&\leq A_0 \left(2c_0^2\right)^B c_0^{-2V} u(c_0\delta/2)^{-B}\delta^{-V}.
\end{split}
\end{align}
In order to bound $\frac{|Z(t)|}{\tilde\sigma(t)}$ for all $t\in\mathcal{S}^*$ we define
\[ H(t) := \sqrt{ 2V\log(1/\sigma(t)) + 2B\log(1/u(\sigma(t))) +4\log\log(e/\sigma(t)) }\]
and choose
\[ u(\delta) := \log(e/\delta)^{-2}.\]
Furthermore, we set
\[ \hat{H}(t) := \sqrt{2V\log(1/\tilde\sigma(t)) + (2V+4B+5)\log\log(e/\tilde\sigma(t))}.\]
With our choice of $u$ and $\delta_0'$ small enough such that $\log\log(e/\tilde\sigma(t))\geq \log(c_0)\vee\log(2)$ for all $t$ satisfying $\sigma(t)\leq \delta\leq \delta_0'$ (this choice is possible by condition (a)), we have
\begin{align*}
H(t) &= \sqrt{2V\log(1/\sigma(t)) + (4B+4)\log\log(e/\sigma(t))}\\
&\leq \sqrt{2V\log(c_0/\tilde\sigma(t)) + (4B+4)\log\log(c_0 e/\tilde\sigma(t))}\\
&= \sqrt{2V\log(1/\tilde\sigma(t)) + 2V\log(c_0) + (4B+4)\log\log(c_0 e/\tilde\sigma(t))}\\
&\leq \sqrt{2V\log(1/\tilde\sigma(t)) + (2V+4B+5)\log\log(e/\tilde\sigma(t))} =\hat{H}(t),
\end{align*}
where we used
\begin{align*}
(4B+4)\log\log(c_0 e/\tilde\sigma(t)) &= (4B+4)\log\left(\log(e/\tilde\sigma(t)) +\log(c_0)\right) \\
&\leq (4B+4)\log\left(2\log(e/\tilde\sigma(t))\right)\\
&= (4B+4)\log\log(e/\tilde\sigma(t)) + (4B+4)\log(2) \\
&\leq (4B+5)\log\log(e/\tilde\sigma(t)).
\end{align*}
Note further that
\[ \left\{t\in\mathcal{S}^*: \tilde\sigma(t)\leq\delta\right\} \subset \left\{ t\in\mathcal{S}^*: \sigma(t)\leq c_0\delta\right\}.\]
We now may choose $\delta_0'$ even small enough such that $H(t)\geq 1$ for $\sigma(t)\leq c_0^{-1}\delta_0'$.
Then with assumption (b) we have for $0<\delta\leq \delta_0'$,
\begin{align*}
&\Pr\left( \sup_{t\in\mathcal{S}^*: \tilde\sigma(t)\leq\delta}\left(\frac{|Z(t)|}{\tilde\sigma(t)} - \hat{H}(t)\right) >0\right)\\
&\hspace{0.2cm}\leq\Pr\left( \sup_{t\in\mathcal{S}^*: \sigma(t)\leq c_0\delta}\left(\frac{|Z(t)|}{\tilde\sigma(t)} - H(t)\right) >0\right)\\
&\hspace{0.2cm}\leq \sum_{t\in\mathcal{S}^*: \sigma(t)\leq c_0\delta} \Pr\left( \frac{|Z(t)|}{\tilde\sigma(t)}\geq H(t)\right)\\
&\hspace{0.2cm}\leq \sum_{t\in\mathcal{S}^*: \sigma(t)\leq c_0\delta} L\exp\left( -V\log(1/\sigma(t))  -B\log(1/u(\sigma(t))) - 2\log\log(e/\sigma(t))\right)\\
&\hspace{1.1cm} \cdot\exp\left(\frac12\log\left(2V\log(1/\sigma(t)) + 2B\log(1/u(\sigma(t))) + 4\log\log(e/\sigma(t))\right)\right)\\
& \hspace{0.2cm}= L\sum_{k=0}^\infty \sum_{t\in\mathcal{S}^*(2^{-k}\delta)} \sigma(t)^V u(\sigma(t))^B\log(e/\sigma(t))^{-2}\\
&\hspace{2.5cm}\cdot\sqrt{2V\log(1/\sigma(t)) + 2B\log(1/u(\sigma(t))) + 4\log\log(e/\sigma(t))}.
\end{align*}
The choice of $u$ implies $2B\log(1/u(\sigma(t))) = 4B\log\log(e/\sigma(t))$ and since we have additionally $\log\log(e/\sigma(t))\leq \log(e/\sigma(t))$, the square root factor is bounded by $\sqrt{2(V+2B+2)\log(e/\sigma(t))}$. We proceed by
\begin{align*}
&\Pr\left( \sup_{t\in\mathcal{S}^*: \tilde\sigma(t)\leq\delta}\left(\frac{|Z(t)|}{\tilde\sigma(t)} - \hat{H}(t)\right) >0\right)\\
&\hspace{0.2cm}\leq L\sum_{k=0}^\infty \sum_{t\in\mathcal{S}^*(2^{-k}\delta)} \sigma(t)^V u(\sigma(t))^B\log(e/\sigma(t))^{-2} \sqrt{2(V+2B+2)\log(e/\sigma(t))}\\
&\hspace{0.2cm}\leq C_1\sum_{k=0}^\infty \sum_{t\in\mathcal{S}^*(2^{-k}\delta)} (2^{-k}\delta)^V u(2^{-k}\delta)^B\log(e/(2^{-k}\delta))^{-\frac32}\\
&\hspace{0.2cm}\leq  C_1\sum_{k=0}^\infty \left(\#\mathcal{S}^*(2^{-k}\delta)\right)(2^{-k}\delta)^V u(2^{-k}\delta)^B \left( \log(e/\delta)+ \log(2)k\right)^{-\frac32},
\end{align*}
where $C_1 := L\sqrt{2(V+2B+2)}$. Now, by \eqref{eqp: B1} we have
\begin{align*}
\left(\#\mathcal{S}^*(2^{-k}\delta)\right)(2^{-k}\delta)^V u(2^{-k}\delta)^B  &\leq A_02^{B}c_0^{2B-2V}\left(\sup_{0<x\leq 1}\frac{u(x)}{u(c_0 x/2)} \right)^B \\
& \leq A_02^{B}c_0^{2B-2V}\log(2e/c_0)^{2B} =:C_2.
\end{align*} 
Here we used that 
\[ \frac{u(x)}{u(c_0 x/2)} = \left(\frac{\log(2e/(c_0x))}{\log(e/x)} \right)^2= \left(\frac{\log(2/c_0)}{\log(e/x)} +1\right)^2,\]
which is increasing for growing $x$ and equals $\log(2e/c_0)$ for $x=1$. By the integral test for convergence, \pagebreak
\begin{align*}
\sum_{k=0}^\infty (\log(e/\delta) + \log(2)k)^{-\frac32}&\leq \log(e/\delta)^{-\frac23} + \int_0^\infty (\log(e/\delta) + \log(2)x)^{-\frac32} dx\\
&\leq \left(1+\frac{2}{\log(2)}\right) \log(e/\delta)^{-\frac12}.
\end{align*}
Finally, for $C_3=C_1C_2(1+2/\log(2))$ and $\delta\leq \delta_0'$ we have shown so far 
\begin{align}\label{eqp: B3}
\Pr\left( \sup_{t\in\mathcal{S}^*: \tilde\sigma(t)\leq\delta}\left(\frac{|Z(t)|}{\tilde\sigma(t)} - \hat{H}(t)\right) >0\right) \leq C_3 \log(e/\delta)^{-\frac12}.
\end{align} 
To proceed, let $s\in\mathcal{S}$ be arbitrary and choose $t\in\mathcal{S}^*$ to satisfy $\sigma(s)\leq\sigma(t)$ and $\rho(s,t)<\log(e/\sigma(s))^{-2}\sigma(s)$, which is possible by \eqref{eqp: B18} and our specific choice of $u$. Then 
\begin{align}\label{eqp: B17}
\left|\frac{\tilde\sigma(t)}{\tilde\sigma(s)}-1 \right| \leq \frac{\tilde\rho(s,t)}{\tilde\sigma(s)} < c_0^2\log(e/(c_0\tilde\sigma(s))^{-2},
\end{align} 
where we used \eqref{eqp: B2} for the first inequality and
\begin{align*}
\tilde\rho(s,t)\leq c_0\rho(s,t) &< c_0\log(e/\sigma(s))^{-2} \sigma(s) \\
&\leq c_0 \log(e/(c_0\tilde\sigma(s)))^{-2} c_0\tilde\sigma(s)
\end{align*}
for the second.
Now we have on the set $\mathcal{A}(Q,2\tilde\sigma(s))$, 
\begin{align*}
\frac{|Z(s)|}{\tilde\sigma(s)} - \frac{|Z(t)|}{\tilde\sigma(t)} &\leq \frac{|Z(s)-Z(t)|}{\tilde\sigma(s)} + \frac{|Z(t)|}{\tilde\sigma(t)}\left|\frac{\tilde\sigma(t)}{\tilde\sigma(s)}-1\right|\\
&\leq \frac{Q\tilde\rho(s,t)\log(e/\tilde\rho(s,t))}{\tilde\sigma(s)} + \frac{|Z(t)|}{\tilde\sigma(t)}\frac{\tilde\rho(s,t)}{\tilde\sigma(s)}\\
&\leq C_4\log(e/(C_5\tilde\sigma(s)))^{-1} + \frac{|Z(t)|}{\tilde\sigma(t)} c_0^2\log(e/(c_0\tilde\sigma(s)))^{-2}
\end{align*}
for suitable constants $C_4=C_4(Q,c_0), C_5=C_5(c_0)$. Using that $x\log(e/x)$ is increasing on $(0,1]$, this follows from the estimate
\begin{align*}
&\frac{Q\tilde\rho(s,t)}{\tilde\sigma(s)}\log(e/\tilde\rho(s,t))\\
&\hspace{0.5cm} \leq c_0^2 \log(e/(c_0\tilde\sigma(s)))^{-2} \log\left(e/(c_0^2\log(e/(c_0\tilde\sigma(s))^{-2}\tilde\sigma(s))\right)\\
&\hspace{0.5cm} = c_0^2\log(e/(c_0\tilde\sigma(s)))^{-2} \left(\log\left(e/(c_0^2\tilde\sigma(s)\right) + 2\log\log\left(e/(c_0\tilde\sigma(s)\right) \right)  \\
&\hspace{0.5cm}\leq C_4\log(e/(C_5\tilde\sigma(s)))^{-1}.
\end{align*}
Consequently, if in addition $\frac{|Z(t)|}{\tilde\sigma(t)}\leq \hat H(t)$, then
\begin{align*}
\frac{|Z(s)|}{\tilde\sigma(s)}&\leq \hat H(t) + C_4\log(e/C_5\tilde\sigma(s))^{-1} + \hat H(t)c_0^2\log(e/(c_0\tilde\sigma(s))^{-2}
\end{align*}
Now we use that by choice of $s,t$ and assumption (a) we have
\[ \tilde\sigma(t) \geq c_0^{-1}\sigma(t) \geq c_0^{-1}\sigma(s)\geq c_0^{-2}\tilde\sigma(s)\]
and with this inequality
\begin{align*}
\hat H(t) &\leq \sqrt{2V\log( 1/(c_0^{-2}\tilde\sigma(s))) + (2V+4B+4)\log\log(e/(c_0^{-2}\tilde\sigma(s)))}\\
&\leq \sqrt{2V\log( 1/\tilde\sigma(s)) + C_6\log\log(e/\tilde\sigma(s))},
\end{align*}
with a constant $C_6=C_6(V,B,c_0)$. Using this and $\sqrt{x}\leq 1+x$ we find
\begin{align*}
\frac{|Z(s)|}{\tilde\sigma(s)}&\leq  \sqrt{2V\log( 1/\tilde\sigma(s)) + C_6\log\log(e/\tilde\sigma(s))} + C_4\log(e/C_5\tilde\sigma(s))^{-1} \\
&\hspace{0.5cm} + \left( 1+ 2V\log( 1/\tilde\sigma(s)) + C_6\log\log(e/\tilde\sigma(s))\right)c_0^2\log(e/(c_0\tilde\sigma(s)))^{-2}\\
&\leq \sqrt{2V\log( 1/\tilde\sigma(s)) + C_6\log\log(e/\tilde\sigma(s))} + C_7\log(e/(C_8\tilde\sigma(s)))^{-1}
\end{align*}
for suitable constants $C_7=C_7(V,B,Q,c_0)$ and $C_8=C_8(c_0)$.
Finally, note that $\tilde\sigma(s)\leq\delta$ implies $\tilde\sigma(t)\leq c_0^2\delta$ because
\[ \tilde\sigma(t) \leq c_0\sigma(t) \leq c_0\left( \sigma(s) + \rho(s,t)\right) \leq 2c_0\sigma(s) \leq 2c_0^2 \tilde\sigma(s), \] 
where we used $\rho(s,t)\leq \sigma(s)$ by choice of $t$. Consequently, on $\{\tilde\sigma(s)\leq\delta\}$ the ratio $\frac{|Z(s)|}{\tilde\sigma(s)}$ is not greater than
\[ \sqrt{2V\log( 1/\tilde\sigma(s)) + C_6\log\log(e/\tilde\sigma(s))} + C_7\log(e/(C_8\tilde\sigma(s)))^{-1}\]
with probability at least $\Pr(\mathcal{A}(Q,2\delta) )- C_3\log(e/(c_0^2\delta))^{-\frac12}$ by \eqref{eqp: B3} for every $\delta\leq c_0^{-2}\delta_0'$. This yields the statement of the theorem if we take a suitable $C_0=C_0(A_0,B,V,Q,c_0)$ and set $\delta_0 := c_0^{-2}\delta_0'$.
\end{proof}

Next, we consider a family $(Z_T)_{T\geq 0}$ of stochastic processes, where for each $T\geq 0$, $Z_T=(Z_T(s))_{s\in\mathcal{S}}$ is a stochastic process on some totally bounded countable semimetric space $(\mathcal{S},\rho_T)$, where $\rho_T:\Omega\times\mathcal{S}^2\rightarrow \R_{\geq 0}$ itself is allowed to be random. Moreover, there exists a spread measure $\sigma_T:\Omega\times\mathcal{S}\rightarrow (0,1]$ for $Z_T$ such that $\sigma_T$ and $\rho_T$ satisfy \eqref{eqp: B2} and \eqref{eqp: B12} almost surely.

\begin{thm}\label{Multiscale_general}
 Let $\sigma_T$, $\rho_T$ be as above together with deterministic metric $\rho$ and $\sigma$ with \eqref{eqp: B2} and \eqref{eqp: B12}. Suppose there exists a family of random sets $(\mathcal{C}_T)_{T\geq 0}$ with $\lim_{T\to\infty}\Pr(\mathcal{C}_T)=1$ such that the following three conditions are satisfied:
\begin{enumerate}
\item[(a)] There exists a constant $c_0>0$ such that on $\mathcal{C}_T$ we have
\[ \frac{1}{c_0} \leq \frac{\sigma_T}{\sigma} \leq c_0\quad \textrm{ and }\quad \frac{1}{c_0} \leq \frac{\rho_T}{\rho} \leq c_0.\]
\item[(b)] For arbitrary $s\in\mathcal{S}$ and $\lambda\geq 1$ there exists a constant $L>0$ such that
\[ \Pr\left(\{ |Z_T(s)| \geq \sigma_T(s)\lambda\}\cap\mathcal{C}_T\right) \leq L\exp\left(-\frac{\lambda^2}{2} + \log(\lambda)\right).\]
\item[(c)] There exists a constant $L'\geq 1$ such that for arbitrary $s,t\in\mathcal{S}$ and $\lambda\geq 1$,
\[ \Pr\left(\{ |Z_T(s) - Z_T(t)| \geq \rho_T(s,t)\lambda\}\cap \mathcal{C}_T\right) \leq L'\exp\left( -\frac{\lambda^2}{2} + \log(\lambda)\right). \]
\item[(d)] For some constants $A_0,B,V>0$ we have the following bound for the capacity numbers on $\mathcal{C}_T$:
\[ D\left( u\delta, \{ s\in\mathcal{S}: \sigma_T(s)\leq \delta\}, \rho_T\right)\leq A_0 u^{-B}\delta^{-V}\quad \textrm{ for all } u,\delta\in (0,1].\]
\end{enumerate}
Then for every $\epsilon>0$ there exists $T_\epsilon$ and $\lambda_\epsilon>0$ such that for $T\geq T_\epsilon$,
\[ \Pr\left( \sup_{s\in\mathcal{S}} \frac{|Z_T(s)|/\sigma_T(s) - \sqrt{2V\log(1/\sigma_T(s))}}{D(\sigma_T(s))} \geq\lambda_\epsilon\right)\leq \epsilon, \]
where $D(\delta):= \log(e/\delta)^{-\frac12}\log\log(e^e/\delta)$. Moreover, the sequence of random variables
\[ \left( \sup_{s\in\mathcal{S}} \frac{|Z_T(s)|/\sigma_T(s) - \sqrt{2V\log(1/\sigma_T(s))}}{D(\sigma_T(s))} \right)_{T\geq 0} \]
is asymptotically tight, provided that for some $T_0>0$, 
\begin{align}\label{eqp: B16}
\inf_{T\geq T_0} \sup_{s\in\mathcal{S}}\left( \sigma_T(s)|_{\mathcal{C}_T} \right)>0.
\end{align} 
\end{thm}
\begin{proof}[Proof]
First of all, we define the set
\[ \mathcal{A}_T(Q,\delta) := \left\{ \sup_{s,t\in\mathcal{S}: \rho_T(s,t)\leq\delta}\frac{|Z_T(t) - Z_T(s)|}{\rho_T(s,t)\log(e/\rho_T(s,t))} \leq Q\right\}. \]
Because $x\log(e/x)$ is monotonously increasing on $(0,1]$, we have by assumption (a) that
\[\mathcal{A}\left(Q,c_0\delta\right)  \subset \mathcal{A}_T(Q,\delta)  \]
where
\[ \mathcal{A}\left(Q,\delta\right) := \left\{ \sup_{s,t\in\mathcal{S}: \rho(s,t)\leq \delta}\frac{|Z_T(t) - Z_T(s)|}{c_0^{-1}\rho(s,t)\log(e/(c_0^{-1}\rho(s,t)))} \leq Q\right\}.\]
Again by assumption (a) we find that 
\begin{align*}
&\Pr\left(\{ |Z_T(s) - Z_T(t)| \geq \rho(s,t)\lambda\}\cap \mathcal{C}_T\right) \\
&\hspace{2cm} \leq \Pr\left(\{ |Z_T(s) - Z_T(t)| \geq c_0^{-1}\rho_T(s,t)\lambda\}\cap \mathcal{C}_T\right),
\end{align*}
which implies together with the sub-Gaussian bound of assumption (c) that the process $Z_T$ with semimetric $\rho(\cdot,\cdot)$ satisfies assumption (a) of Proposition~\ref{sub-exp_differences} for a suitable constant $L''$.  
Now, note the following: For $\tilde{q},q>0$ and $G(\lambda, \delta) = \tilde{q}\lambda^q$, Remark~$3$ in the technical report~\cite{Walther} entails that under assumption (d) for $\delta=1$,
\[ J(\epsilon,1) \leq C'\epsilon \log(e/\epsilon)^q \quad \textrm{ for all } 0<\epsilon \leq 1,\]
for $C'=\tilde{q}\max(1+2B,\log(A_0^2))^q \int_0^1\log(e/z)^q dz$ and $J(\cdot,\cdot)$ given in Proposition~\ref{sub-exp_differences}.
Then, with $\tilde{q}=q=1$ Proposition~\ref{sub-exp_differences} implies for $Q=1+12C'$ that 
\begin{align}\label{eqp: B4}
\Pr\left(\mathcal{A}_T(Q,\delta)^c \cap\mathcal{C}_T\right) \leq \Pr\left(\mathcal{A}(Q,c_0\delta)^c\right)\leq \frac{L''c_0\delta}{2}.
\end{align}  
For the rest of the proof define
\[ \Psi_T(s) := \frac{|Z_T(s)|/\sigma_T(s) - \sqrt{2V\log(1/\sigma_T(s))}}{D(\sigma_T(s))}. \]
Next, we bound from above, 
\begin{align*}
\Pr\left( \sup_{s\in\mathcal{S}}\Psi_T(s)\geq\lambda\right) 
&\leq \Pr\left(\left\{ \sup_{s\in\mathcal{S}}\Psi_T(s)\geq\lambda\right\}\cap\mathcal{C}_T\right) + \Pr\left(\mathcal{C}_T^c\right)\\
&\leq \Pr\left(\left\{\sup_{s\in\mathcal{S}: \sigma_T(s)<\delta} \Psi_T(s)\geq \lambda\right\}\cap\mathcal{C}_T\right)  +\Pr\left(\mathcal{C}_T^c\right) \\
&\hspace{1cm} +  \Pr\left(\left\{\sup_{s\in\mathcal{S}: \sigma_T(s)\geq\delta} \Psi_T(s)\geq \lambda\right\}\cap\mathcal{C}_T\right) 
\end{align*}
for some $0<\delta<1$. Let $\epsilon>0$. We are now going to show that for appropriate $\delta=\delta_\epsilon>0$ there exists $T_\epsilon$ and $\lambda_\epsilon$ such that each summand is $<\frac{\epsilon}{3}$ for $T>T_\epsilon$ and $\lambda=\lambda_\epsilon$.
\begin{enumerate}
\item[(i)] We have
\begin{align*}
&\sqrt{2V\log(1/\delta) + C_0\log\log(e/\delta)} + C_0\log(e/(C_0\delta))^{-1} - \sqrt{2V\log(1/\delta)} \\
&\hspace{0.2cm} = \frac{C_0\log\log(e/\delta)}{\sqrt{2V\log(1/\delta) + C_0\log\log(e/\delta)}+\sqrt{2V\log(1/\delta)}} + \frac{C_0}{\log(e/(C_0\delta))}\\
&\hspace{0.2cm} = \mathcal{O}\left( \log\log(e/\delta) \log(e/\delta)^{-\frac{1}{2}}\right)
\end{align*}
for $\delta\searrow 0$. Conditions (a), (b) and (d) are clearly the same as those of Theorem~\ref{Multiscale_preparation} and we may apply it to get for $\lambda$ large enough, 
\begin{align*}
&\Pr\left(\left\{\sup_{s\in\mathcal{S}: \sigma_T(s)<\delta} \Psi_T(s)\geq \lambda\right\}\cap\mathcal{C}_T\right)\\
&\hspace{1cm} \leq 1-\Pr(\mathcal{A}_T(Q,2\delta)\cap\mathcal{C}_T) + C_0\log(e/(2c_0^2\delta))^{-\frac12} \\
&\hspace{1cm}  \leq \Pr(\mathcal{A}_T(Q,2\delta)^c\cap\mathcal{C}_T) +\Pr\left( \mathcal{C}_T^c\right)+ C_0\log(e/(2c_0^2\delta))^{-\frac12} \\
&\hspace{1cm} \leq L''c_0\delta +\Pr\left(\mathcal{C}_T^c\right)+ C_0\log(e/(2c_0^2\delta))^{-\frac12},
\end{align*} 
where we used
\begin{align*}
1-\Pr(A\cap B) &= \Pr((A\cap B)^c) = \Pr( A^c\cup B^c)\\
&= \Pr( (A^c\cap B) \cup B^c) \leq \Pr(A^c\cap B) + \Pr(B^c)
\end{align*}
in the second step and \eqref{eqp: B4} in the last one. Now choose $\delta_\epsilon>0$ fullfilling $L''c_0\delta_\epsilon + C_0\log(e/(2c_0^2 \delta_\epsilon))^{-\frac12}<\frac{\epsilon}{6}$ and $T_2$ large enough such that $\Pr(\mathcal{C}_T^c)<\frac{\epsilon}{6}$ for $T\geq T_2$.

\item[(ii)] It is clear by $\lim_{T\to\infty}\Pr(\mathcal{C}_T)=1$ that there exists $T_1$ such that for all $T\geq T_1$ we have $\Pr(\mathcal{C}_T^c)<\frac{\epsilon}{3}$.

\item[(iii)] Lastly, we need to bound the remaining probability by $\frac{\epsilon}{3}$ where we have to use $\delta=\delta_\epsilon$ from step (i). On $\{\sigma_T(s)\geq \delta\}$ we have the bound
\begin{align}\label{eqp: B6}
\Psi_T(s) \leq \frac{C''}{\delta} \sup_{s\in\mathcal{S}} |Z_T(s)|,
\end{align}
where the constant $C''$ emerges from minimizing $D(\cdot)$ on $[\delta, 1]$.
Now fix $T$ and choose $\kappa>0$ together with a maximal subset $\{s_1, \dots, s_N\}\subset\mathcal{S}$ that fullfills $\rho(s_i, s_j)>\kappa/c_0$ for all $i\neq j$. From assumptions (a) and (d) we know that $N\leq A_0c_0^{2B}\kappa^{-B}$. Now, on the set $\mathcal{A}_T(Q,\kappa)$,
\[ \sup_{s\in\mathcal{S}}|Z_T(s)| \leq Q\kappa \log(e/\kappa) + \max_{i=1,\dots, N} |Z_T(s_i)|.\]
Hence, by the union bound and assumption (a) we then have for $\lambda':= \lambda \delta/C''$ with $\lambda$ large enough and $\kappa$ small enough to ensure $\lambda'-Q\kappa \log(e/\kappa)\geq 1$,
\begin{align*}
&\Pr\left(\left\{ \sup_{s\in\mathcal{S}} |Z_T(s)| >\lambda'\right\}\cap\mathcal{C}_T\cap\mathcal{A}_T(Q,\kappa)\right)\\
&\hspace{0.7cm}\leq \Pr\left(\left\{ Q\kappa\log(e/\kappa) + \max_{i=1,\dots, N} |Z_T(s_i)| >\lambda' \right\}\cap\mathcal{C}_T\right)\\
&\hspace{0.7cm}\leq \Pr\left(\left\{ \max_{i=1,\dots, N} |Z_T(s_i)| >\lambda' - Q\kappa\log(e/\kappa)\right\}\cap\mathcal{C}_T\right)\\
&\hspace{0.7cm}\leq \exp\left( \log(N) - \frac{(\lambda'-Q\kappa\log(e/\kappa))^2}{2} + \log(\lambda' - Q\kappa\log(e/\kappa))\right)\\
&\hspace{0.7cm} =  \exp\left( \log\left(A_0c_0^{2B}\right) -B\log(\kappa) - \frac{(\lambda'-Q\kappa\log(e/\kappa))^2}{2} \right.\\[-2\jot]
&\hspace{6cm} \left. \textcolor{white}{\frac12}+ \log(\lambda' - Q\kappa\log(e/\kappa))\right).
\end{align*}

Thus, with \eqref{eqp: B4} and \eqref{eqp: B6}, we have
\begin{align*}
&\Pr\left(\left\{ \sup_{s\in\mathcal{S}: \sigma_T(s)\geq\delta}\Psi_T(s)>\lambda \right\}\cap\mathcal{C}_T\right)\\
&\hspace{0.5cm} \leq \Pr\left(\left\{ \sup_{s\in\mathcal{S}} |Z_T(s)| >\lambda'\right\}\cap\mathcal{A}_T(Q,\kappa)\cap\mathcal{C}_T\right) + \Pr\left(\mathcal{A}_T(Q,\kappa)^c \cap \mathcal{C}_T\right)\\ 
&\hspace{0.5cm}\leq \exp\left( \log\left(A_0c_0^{2B}\right) -B\log(\kappa) - \frac12 (\lambda'-Q\kappa\log(e/\kappa))^2\right.\\[-3\jot]
&\hspace{5.5cm} \left.\textcolor{white}{\frac12}+ \log(\lambda' - Q\kappa\log(e/\kappa))\right) + \frac{L''\kappa}{2}.
\end{align*}
Note that we accomplished to bound the probability independent of $T$. Now, choose $\kappa_\epsilon$ small enough such that $\frac{L''\kappa_\epsilon}{2} <\frac{\epsilon}{6}$ and after that $\lambda_\epsilon$ large enough in order to ensure 
\begin{align*}
&\exp\left( \log\left(A_0c_0^{2B}\right) -B\log(\kappa_\epsilon) - \frac{(\lambda_\epsilon \delta/C -Q\kappa_\epsilon\log(e/\kappa_\epsilon))^2}{2} \right. \\
&\hspace{5cm} \left.\textcolor{white}{\frac12} + \log(\lambda_\epsilon \delta/C - Q\kappa_\epsilon\log(e/\kappa_\epsilon))\right) <\frac{\epsilon}{6}.
\end{align*} 
For those choices we then have
\[ \Pr\left(\left\{ \sup_{s\in\mathcal{S}: \sigma_T(s)\geq\delta}\Psi_T(s)>\lambda\right\}\cap\mathcal{C}_T\right) \leq \frac{\epsilon}{3}.\]
\end{enumerate}
In conclusion, for all $T>\max\{T_1,T_2\}$,
\[ \Pr\left(\sup_{s\in\mathcal{S}} \Psi_T(s) >\lambda_\epsilon\right)<\epsilon,\]
which gives the first claim of the theorem. For the asymptotic tightness, it remains to show that for any $\epsilon>0$ there exists $\lambda_\epsilon'>0$ and $T_\epsilon'$ such that for $T\geq T_\epsilon'$
\begin{align}\label{eqp: B9}
\Pr\left( \sup_{s\in\mathcal{S}} \Psi_T(s) < - \lambda_\epsilon'\right) \leq \epsilon. 
\end{align} 
By non-negativity of $|Z_T(s)|/\sigma_T(s)$ and $D(\cdot)$, it is enough to show that
\begin{align*}
&\Pr\left( - \inf_{s\in\mathcal{S}}\sqrt{2V\log(1/\sigma_T(s))} < - \lambda_\epsilon'\right) \\
& \hspace{1cm} \leq \Pr\left( -\inf_{s\in\mathcal{S}}\sqrt{2V\log(1/\sigma_T(s))} < - \lambda_\epsilon', \mathcal{C}_T\right) + \Pr\left( \mathcal{C}_T^c\right)
\end{align*} 
is bounded from above by $\epsilon$ for $T\geq T_\epsilon'$. The second term is bounded by $\frac{\epsilon}{2}$ for $T\geq T_1'$ large enough by assumption on $\mathcal{C}_T$. For the first one, this bound follows for $T\geq T_0$ directly from the assumption \eqref{eqp: B16} and $\lambda_\epsilon'$ large enough. Now \eqref{eqp: B9} follows for $T\geq T_\epsilon' := T_1'\vee T_0$.
\end{proof}

\subsection{Proof of Theorem \ref{Multiscale_Lemma}}\label{B2_SubSec}
The proof of Theorem~\ref{Multiscale_Lemma} will be an application of Theorem~\ref{Multiscale_general}. One major part is establishing the exponential inequalities on appropriate sets - the other bounding the covering numbers. Our proof of the first part relies substantially on the following result.

\begin{proposition}[\cite{Spokoiny}, Proposition A.1]\label{Exp_ineq}
Let $(M_t)_{t\geq 0}$ be a continuous martingale with $M_0=0$. Then for every $T>0$, $\theta>0$, $S\geq 1$ and $\lambda\geq 1$ we have
\[ \Pr\left( \left| M_T\right| > \lambda\sqrt{\langle M\rangle_T}, \theta\leq \sqrt{\langle M\rangle_T}\leq \theta S\right) \leq 4\lambda\sqrt{e}(1+\log(S))\exp\left(-\frac{\lambda^2}{2}\right).\]
\end{proposition}

\begin{proof}[Proof of Theorem \ref{Multiscale_Lemma}]
The proof is an application of Theorem~\ref{Multiscale_general} for the process
\[ Z_T(y,h) := \frac{1}{\sqrt{T}\hat\sigma_{T,\max}} \int_0^T K_{y,h}(X_s) dW_s\]
and the set
\begin{align}\label{eqp: B13}
\overline{\mathcal{T}} := \left\{ (y,h)\in\R^2\mid 0<h\leq A\textrm{ and } -A+h \leq y\leq A-h\right\}
\end{align} 
with dense subset $\mathcal{T}=\overline{\mathcal{T}}\cap\Q^2$. The (random) spread measure of $Z_T$ is given by $\overline\sigma_T(y,h) := \frac{\hat\sigma_T(y,h)}{\hat\sigma_{T,\max}}\leq 1$ and the corresponding semimetric by
\[ \overline\rho_T((y,h), (y',h'))^2 := \frac{1}{T\hat\sigma_{T,\max}^2} \int_0^T \left( K_{y,h}(X_s) - K_{y',h'}(X_s)\right)^2 ds.\]
The random map $\hat\sigma_T^2:\overline{\mathcal{T}}\rightarrow (0,1]$ is continuous, as
\[ \left| \hat\sigma_T(y_n,h_n)^2-\hat\sigma_T(y,h)^2\right| \leq \left\|K_{y_n,h_n}^2-K_{y,h}^2\right\|_{[-A,A]} \]
and the right-hand side converges to zero for $(y_n,h_n)\to (y,h)$. In particular, it follows that $\overline\sigma_T$ is continuous in $(y,h)$ and hence the level sets in \eqref{eqp: B12} are compact because they are bounded and closed as the preimage of the closed set $[\delta,1]$. The inequality $|\overline\sigma_T(y,h) -\overline\sigma_T(y',h')|\leq \overline\rho_T((y,h),(y',h'))$ follows by Cauchy-Schwarz' inequality.
Moreover, for the application of Theorem~\ref{Multiscale_general}, we define the set 
\begin{align*}
\mathcal{C}_{T,b_0} := \left\{ \sup_{z\in [-A,A]} \left| \frac{1}{\sigma^2 T}L_T^z(X) - q_{b_0}(z)\right| <\frac12 L_*\right\}
\end{align*} 
where $L_*$ is the lower bound of the invariant density given in Lemma~\ref{bound_invariant_density}. By \eqref{eq:conv_emprical_density_uniform} we have $\lim_{T\to\infty}\Pr(\mathcal{C}_{T,b_0})= 1$ and on $\mathcal{C}_{T,b_0}$,
\begin{align}\label{eqp: B10}
\frac12 L_*\leq \frac12 q_{b_0}(z) \leq \frac{1}{\sigma^2 T} L_T^z(X)\leq \frac32 q_{b_0}(z) \leq \frac32 L^*
\end{align} 
with the upper bound $L^*$ from Lemma~\ref{bound_invariant_density}. Now we check each of the conditions (a)-(d) of Theorem \ref{Multiscale_general} and define the deterministic counterparts
\[ \overline{\sigma}_b(y,h)^2 := \sigma_{b,\max}^{-2}\int_\R K_{y,h}(z)^2q_b(z) dz\]
and
\[ \overline{\rho}_b((y,h),(y',h'))^2 :=\sigma_{b,\max}^{-2} \int_\R \left( K_{y,h}(z) - K_{y',h'}(z)\right)^2 q_b(z) dz, \]
of $\sigma_T$ and $\rho_T$, where in both cases $\sigma_{b,\max}^2 := \int_\R \1_{[-A,A]}(z) q_b(z)dz$. The validity of \eqref{eqp: B2} and compactness of \eqref{eqp: B12} are checked in the same way as for the random metrics.

\begin{enumerate}
\item[(a)] Using the occupation times formula, we have
\[ \hat\sigma_{T,\max}^2 = \frac1T \int_0^T \1_{[-A,A]}(X_s) ds= \int_\R \1_{[-A,A]}(z) \frac{1}{\sigma^2 T}L_T^z(X) dz.\]
Hence, on $\mathcal{C}_{T,b_0}$ we have by \eqref{eqp: B10}
\[ \frac12 \sigma_{b,\max}^2\leq \hat\sigma_{T,\max}^2 \leq \frac32 \sigma_{b,\max}^2. \]
In the same way we see
\[ \frac12\sigma_{b,\max}^2\overline{\sigma}_b(y,h)^2\leq \hat{\sigma}_T(y,h)^2 \leq \frac32\sigma_{b,\max}^2\overline{\sigma}_b(y,h)^2  \]
and (omitting the argument $((y,h),(y',h'))$)
\[ \frac12\sigma_{b,\max}^2\overline{\rho}_b^2\leq \hat{\rho}_T^2 \leq \frac32\sigma_{b,\max}^2\overline{\rho}_b^2.  \]
Combining these estimates yields assumption (a) of Theorem~\ref{Multiscale_general} for the quotients $\overline{\sigma}_T/\overline{\sigma}_b$ and $\overline{\rho}_T/\overline{\rho}_b$ with $c_0=\sqrt{3}$.

\item[(b)] By the occupation times formula and \eqref{eqp: B10} we have on $\mathcal{C}_{T,b_0}$
\[ \sqrt{\frac12 L_*} \|K_{y,h}\|_{L^2} \leq \hat\sigma_T(y,h) \leq \sqrt{\frac32 L^*} \|K_{y,h}\|_{L^2}. \]
Note furthermore that $Z'_t(y,h)=\hat\sigma_{t,\max} \sqrt{t}Z_t(y,h)$ is a martingale (as a process indexed in $t$) with quadratic variation $\langle Z'(y,h)\rangle_t = t\hat\sigma_t(y,h)^2$. Setting $\theta=\sqrt{TL_*/2}\|K_{y,h}\|_{L^2}$ and $S=\sqrt{3L^*/L_*}\geq 1$ in Proposition~\ref{Exp_ineq} gives for $\lambda\geq 1$,
\begin{align*}
&\Pr\left( |Z_T(y,h)|\geq \lambda\overline\sigma_T(y,h), \mathcal{C}_{T,b_0}\right)\\
&\hspace{2cm} = \Pr\left( | Z'_T(y,h)|\geq \lambda\sqrt{T}\hat\sigma_T(y,h), \mathcal{C}_{T,b_0}\right)\\
&\hspace{2cm} \leq 4\sqrt{e}\left(1+\log\big(\sqrt{3L^*/L_*}\big)\right) \exp\left(-\frac{\lambda^2}{2} + \log(\lambda)\right).
\end{align*} 

\item[(c)] By the same arguments as in part (b) for $\lambda\geq 1$,
\begin{align*}
&\Pr\left( |\tilde{Z}_T| >\lambda\overline\rho_T((y,h), (y',h')), \mathcal{C}_{T,b_0}\right) \\
&\hspace{2cm}\leq 4\sqrt{e}\left(1+ \log\big(\sqrt{3L^*/L_*}\big)\right)\exp\left(-\frac{\lambda^2}{2} + \log(\lambda)\right),
\end{align*}
by an application of Proposition~\ref{Exp_ineq}. The sub-exponential tail follows then for $L'>0$ large enough because $\exp\left(-\lambda^2/2 +\log(\lambda)+\lambda\right)$ is bounded from above for $\lambda\geq 1$.

\item[(d)] We split the interval $[-A,A]$ into a partition $(M_k)_{1\leq k\leq N}$ such that 
\[ \lambda(M_k)\leq\frac{\hat\sigma_{T,\max}^2}{6L^*\|K\|_{TV}}(u\delta)^2\]
for $1\leq k\leq N$, where equality holds for $k\neq N$. Here and subsequently, $\|K\|_{TV}$ denotes the total variation of $K$ and $\lambda(\cdot)$ the Lebesgue measure on $\R$.
Then we have
\[ 2A=\lambda([-A,A]) = \sum_{k=1}^N \lambda(M_k)\geq (N-1)\frac{\hat\sigma_{T,\max}^2}{6L^*\|K\|_{TV}}(u\delta)^2, \]
or equivalently
\begin{align}\label{eqp: B15}
N\leq 1+ \frac{12AL^*\|K\|_{TV}}{\hat\sigma_{T,\max}^2}(u\delta)^{-2}.
\end{align} 
Take $(y,h)\in\overline{\mathcal{T}}$ such that $\overline\sigma_T(y,h)\leq\delta$, i.e.
\[ \frac{1}{T\hat\sigma_{T,\max}^2}\int_0^T K_{y,h}(X_s)^2 ds = \frac{1}{\hat\sigma_{T,\max}^2}\int_\R K_{y,h}(z)^2 \frac{1}{\sigma^2 T} L_T^z(X) dz \leq \delta^2\]
by the occupation times formula. Now using the lower bound~\eqref{eqp: B10} for the averaged local time on $\mathcal{C}_{T,b_0}$ we get
\[ \delta^2\geq \frac{L_*}{2\hat\sigma_{T,\max}^2}\int_\R K_{y,h}(z)^2 dz  = \frac{hL_*\| K\|_{L^2}^2}{2\hat\sigma_{T,\max}^2},\]
which implies 
\[ h\leq 2\hat\sigma_{T,\max}^2\delta^2 (L_*\|K\|_{L^2}^2)^{-1}. \] 
Now suppose $y-h\in M_i$ and $y+h\in M_j$. As we know the length of the intervals $M_k$, we conclude
\[ (i-j-1)\frac{\hat\sigma_{T,\max}^2(u\delta)^2}{6L^*\|K\|_{TV}} \leq \frac{4\delta^2\hat\sigma_{T,\max}^2}{L_*\|K\|_{L^2}^2}\quad\Leftrightarrow\quad i-j\leq 1+\frac{24 L^*\|K\|_{TV}}{L_*\|K\|_{L^2}^2} u^{-2}.\]
From this and \eqref{eqp: B15}, we conclude that there are at most 
\begin{align}\label{eqp: B8}
\begin{split}
&\left( 1+\frac{24 L^*\|K\|_{TV}}{L_*\|K\|_{L^2}^2} u^{-2}\right)\left( 1+\frac{12AL^*\|K\|_{TV}}{\hat\sigma_{T,\max}^2}(u\delta)^{-2}\right)\\
&\hspace{0.5cm} \leq \left( 1+ 12L^* \|K\|_{TV} \left( \frac{1}{L_*}+\frac{2}{ L_* \|K\|_{L^2}^2}+\frac{24 L^* \|K\|_{TV}}{(L_*)^2\|K\|_{L^2}^2}\right)\right) u^{-4}\delta^{-2}
\end{split}
\end{align} 
such pairs $(i,j)$, where we used the lower bound $\hat\sigma_{T,\max}^2\geq AL_*$ on $\mathcal{C}_{T,b_0}$. Let $\mathcal{T}'\subset\overline{\mathcal{T}}$ be a maximal subset of $\{ (y,h)\in\overline{\mathcal{T}}: \overline\sigma_T(y,h)\leq \delta\}$ with $\overline\rho_T((y,h),(y',h'))>u\delta$ for arbitrary $(y,h),(y',h')\in\mathcal{T}'$.
Then, the proof is finished by \eqref{eqp: B8} if we could show that for all $(i,j)$ there is at most one such point $(y_{ij},h_{ij})$ in $\mathcal{T}'$ with $y_{ij}-h_{ij}\in M_i$ and $y_{ij}+h_{ij}\in M_j$. To this aim, we pick $(y,h), (y',h')\in\overline{\mathcal{T}}$ such that $y-h, y'-h'\in M_i$ and $y+h, y'+h'\in M_j$ and are done if we can show that $\overline\rho_T((y,h),(y',h'))^2\leq (u\delta)^2$.\\
As the kernel $K$ is of bounded variation, there exists a probability measure $\mu$ on $[-1,1]$ and a measurable function $g$ with $|g|\leq \|K\|_{TV}$ such that for almost all $x\in [-1,1]$ we have
\[ K(x) = \int_{-1}^x g d\mu.\]
Since $K$ is bounded by one, we have on $\mathcal{C}_{T,b_0}$,
\begin{align*}
&\frac1T\int_0^T \left( K_{y,h}(X_s)-K_{y',h'}(X_s)\right)^2 ds\\
&\hspace{1.5cm} \leq \frac2T\int_0^T \left| K_{y,h}(X_s)-K_{y',h'}(X_s)\right| ds\\
&\hspace{1.5cm} = 2\int_\R\left| K\left(\frac{z-y}{h}\right) - K\left(\frac{z-y'}{h'}\right)\right| \frac{1}{\sigma^2 T} L_T^z(X) dz\\
&\hspace{1.5cm} \leq 3L^* \int_\R \left| \int_{-1}^{\frac{z-y}{h}} g(u) d\mu(u)- \int_{-1}^{\frac{z-y'}{h'}} g(u) d\mu(u)\right| dz\\
&\hspace{1.5cm} \leq 3L^* \int_\R \int_{\min\left\{\frac{z-y}{h},\frac{z-y'}{h'}\right\}}^{\max\left\{\frac{z-y}{h}, \frac{z-y'}{h'}\right\}} |g(u)| d\mu(u) dz\\
&\hspace{1.5cm} \leq 3L^* \|K\|_{TV} \int_\R\int_{-1}^1 \1_{B_z((y,h),(y',h'))} (u) d\mu(u) dz\\
&\hspace{1.5cm} = 3L^* \|K\|_{TV} \int_{-1}^1 \int_\R \1_{B_z((y,h),(y',h'))} (u) dz d\mu(u),
\end{align*}
with
\[B_z((y,h),(y',h')) = \left[\min\left\{\frac{z-y}{h},\frac{z-y'}{h'}\right\}, \max\left\{\frac{z-y}{h},\frac{z-y'}{h'}\right\}\right]\]
and an application of Fubini's theorem in the last step. To proceed, note that $u\in B_z((y,h),(y',h'))$ if and only if
\[ \frac{z-y}{h}\leq u\leq \frac{z-y'}{h'}\quad \textrm{ or }\quad \frac{z-y'}{h'}\leq u\leq \frac{z-y}{h}.\]
The first condition is equivalent to $uh'+y'\leq z\leq uh+y$, the second one to $uh+y\leq z\leq uh'+y'$. Thus, 
\[ \int_\R \1_{B_z((y,h),(y',h'))} (u) dz = |y' + uh' - y - uh| \]
and we get
\begin{align*}
&\frac1T\int_0^T \left( K_{y,h}(X_s)-K_{y',h'}(X_s)\right)^2 ds\\
&\hspace{1cm}\leq 3L^* \|K\|_{TV}\int_{-1}^1 |y'+uh'-y-uh| d\mu(u)\\
&\hspace{1cm}\leq 3L^*  \|K\|_{TV}\max_{u=\pm 1} |y-y' +u( h-h')|\\
&\hspace{1cm}\leq 3L^* \|K\|_{TV} \lambda([y-h,y+h]\bigtriangleup [y'-h',y'+h']),
\end{align*}
where the second inequality follows from the fact, that $|y-y' + u(h-h')|$ is maximized for $u=\textrm{sign}(y-y')\textrm{sign}(h-h')$. Here, $\bigtriangleup$ denotes the symmetric difference of two sets.\\
As we chose $(y,h),(y',h')\in\mathcal{T}'$ to fullfill $y-h,y'-h'\in M_i$ and $y+h,y'+h'\in M_j$, we know that
\[\lambda([y-h,y+h]\bigtriangleup [y'-h',y'+h']) \leq 2\frac{(u\delta)^2\hat\sigma_{T,\max}^2}{6L^* \|K\|_{TV}}.\]
Hence,
\begin{align*}
\overline\rho_T((y,h),(y',h'))^2 &=\frac{1}{T\hat\sigma_{T,\max}^2}\int_0^T \left( K_{y,h}(X_s)-K_{y',h'}(X_s)\right)^2 ds\\
&\leq \frac{3L^*\|K\|_{TV}}{\hat\sigma_{T,\max}^2}\frac{2(u\delta)^2\hat\sigma_{T,\max}^2}{6L^* \|K\|_{TV}} = (u\delta)^2
\end{align*} 
and the bound of the covering numbers follows with $B=4$ and $V=2$.
\end{enumerate}
To conclude asymptotic tightness, it suffices by Theorem~\ref{Multiscale_general} to verify
\[ \inf_{T\geq T_0} \sup_{(y,h)\in\overline{\mathcal{T}}} \overline\sigma_T(y,h) >0 \]
on $\mathcal{C}_{T,b_0}$. But this follows directly by the occupations times formula, as \eqref{eqp: B10} holds on $\mathcal{C}_{T,b_0}$, and hence
\begin{align*}
\overline\sigma_T(y,h)^2 = \frac{\int_0^T K_{y,h}(X_s)^2 ds}{\int_0^T \1_{[-A,A]} (X_s) ds} = \frac{\int_\R K_{y,h}(z) \frac{1}{\sigma^2 T}L_T^z(X) dz}{\int_\R \1_{[-A,A]}(z) \frac{1}{\sigma^2 T}L_T^z(X) dz} \geq \frac{\|K_{y,h}\|_{L^2}^2 L_*}{6AL^*}.
\end{align*}
The assertion of Theorem~\ref{Multiscale_Lemma} now follows from the observation that $D(\cdot)$ is bounded and stricly positive on $(0,1]$ with $\lim_{r\searrow 0}D(r)=0$.
\end{proof}

The following result proves finiteness of the random variable $S_b$ given in Theorem~\ref{weak_conv} and can be proven along the lines of the proof of Theorem~\ref{Multiscale_Lemma}. This is fully worked out as Theorem~$3.3.8$ in \cite{Brutsche}.

\begin{thm}\label{Multiscale_Limit}
Let $K$ be a compactly supported kernel function of bounded variation with $\|K\|_\infty\leq 1$. For $b\in\Sigma(C,A,\gamma,\sigma)$ define
\[ \sigma_b(y,h)^2 := \int_\R K_{y,h}(z)^2 q_b(z) dz\]
and 
\[ \sigma_{b,\max}^2 := \int_\R \1_{[-A,A]}(z) q_b(z) dz.\]
Then for $\mathcal{T}$ given in \eqref{eq: cT}, $\Pr_b$-almost surely,
\[ \sup_{(y,h)\in\mathcal{T}} \left( \frac{\left| \int_\R K_{y,h}(z) \sqrt{q_b(z)} dW_z\right|}{\|K_{y,h}\sqrt{q_b}\|_{L^2}} - \Upsilon\left(\sigma_b(y,h)^2/\sigma_{b,\max}^2\right) \right)<\infty.\]
\end{thm}

\section{Proof of Theorem \ref{weak_conv} and uniform weak convergence}\label{App_weak}

This section presents the proof of the uniform weak convergence result in Theorem~\ref{weak_conv}. In Subsection~\ref{SubSec_C1} we introduce the notion of uniform weak convergence using the dual bounded Lipschitz metric. Furthermore, it comprises many auxiliar results on uniform weak convergence that are of interest on their own, in particular the important Proposition~\ref{prop_unif_weak}. The proof of Theorem~\ref{weak_conv}, which is a consequence of Lemma~\ref{lemma_weak_2} and this Proposition~\ref{prop_unif_weak}, is then given in Subsection~\ref{SubSec_C2}.

\subsection{Preliminaries on uniform weak convergence}\label{SubSec_C1}

In this part we deal with general results about uniform weak convergence. Here and subsequently, let $(E,d)$ denote some separable metric space and $\mathcal{S}$ a countable parameter space. To define weak convergence that is uniform over some class of parameters, we use the fact that weak convergence for probability measures (and for random variables via their corresponding measure) is metrized by the dual bounded Lipschitz metric. To introduce this metric, we define the bounded Lipschitz norm for a bounded Lipschitz function $f:E\rightarrow\R$ as
\[ \|f\|_{BL} := \left( \|f\|_\infty  + \sup_{x\neq y} \frac{|f(x) - f(y)|}{d(x,y)}\right)\]
and set
\begin{align}\label{eq: F_BL}
\mathcal{F}_{BL}(E) := \{ f:E\rightarrow \R \mid \|f\|_{BL}\leq 1\}.
\end{align} 
A sequence of $E$-valued random variables $(X_n)_{n\in\N}$ converges weakly to $X$ if and only if
\[ d_{BL}(X_n, X) := \sup_{f\in\mathcal{F}_{BL}(E)} \left| \E[f(X_n)] - \E[f(X)]\right| \stackrel{n\to\infty}{\longrightarrow} 0,\]
see for example \cite{Wellner}, p. 73. This can be used to define uniform weak convergence in the following way.

\begin{definition}\label{def_weak_conv}
Let $\Theta$ be a set of parameters and $(X_n)_{n\in\N}$ a sequence of $E$-valued random variables, where $(E,d)$ is a separable metric space. We say that $(X_n)_{n\in\N}$ converges uniformly (over $\Theta$) in distribution to $X$ if and only if
\[ \sup_{\theta\in\Theta} d_{BL}^\theta(X_n,X) :=\sup_{\theta\in\Theta} \sup_{f\in\mathcal{F}_{BL}(E)} \left| \E_\theta[f(X_n)] - \E_\theta[f(X)]\right| \stackrel{n\to\infty}{\longrightarrow} 0.\]
\end{definition}

\begin{lemma}\label{lemma_weak_quotient}
Let $(A_n(s))_{n\in\N}$ and $(B_n)_{n\in\N}$ be two real-valued stochastic processes, where for all $n\in\N$, $B_n\neq 0$  and $A_n(s)$ depends on some parameter $s\in\mathcal{S}$. Suppose that for all $\epsilon>0$,
\[ \sup_{\theta\in\Theta} \Pr_\theta\left( \sup_{s\in\mathcal{S}}\left| A_n(s) - A(s)\right|>\epsilon\right)\rightarrow 0 \]
and 
\[ \sup_{\theta\in\Theta} \Pr_\theta\left( \left| B_n - B\right|>\epsilon\right) \rightarrow 0,\]
where $A(s)$ and $B$ are real-valued random variables satisfying $1\geq B\geq \beta >0$ and $\sup_{s\in\mathcal{S}} |A(s)|\leq 1$. Then we have for all $\epsilon>0$,
\[ \sup_{\theta\in\Theta} \Pr_\theta\left( \sup_{s\in\mathcal{S}}\left| \frac{A_n(s)}{B_n} - \frac{A(s)}{B}\right|>\epsilon\right)\rightarrow 0. \]
\end{lemma}
\begin{proof}
This follows easily by using
\[ \left| \frac{A_n(s)}{B_n} - \frac{A(s)}{B}\right|\leq \left| \frac{A_n(s)-A(s)}{B_n}\right| + \left| \frac{A(s)}{BB_n}\right|\left|B-B_n\right|.\]
\end{proof}

Next, we establish that a uniform continuous mapping theorem holds true for Lipschitz functions. This result is stated in \cite{Kasy} as Theorem~$1$ for the real-valued case, but the result can be likewise shown for arbitrary metric spaces and a proof is therefore omitted. 

\begin{lemma}\label{uniform_continuous_mapping}
Let $(E,d)$ be a separable metric space. Suppose that the $E$-valued sequence $(X_n)_{n\in\N}$ converges uniformly (over $\Theta$) in distribution to $X$. If $h:E\rightarrow \R$ is Lipschitz continuous, then $h(X_n)$ converges uniformly (over $\Theta$) in distribution to $h(X)$.
\end{lemma}

\begin{lemma}\label{approx_P}
Let $(X_n)_{n\in\N}$ be a sequence of $\R^d$-valued random variables whose distribution depends on a parameter $\theta\in\Theta$ and the uniform convergence
\[ \sup_{\theta\in\Theta} \sup_{f\in \mathcal{F}_{BL}(\R^d)} \E_\theta\left[ | f(X_n) - f(X)|\right] \stackrel{n\to\infty}{\longrightarrow} 0 \]
holds true. Furthermore, assume the family $(\Pr_\theta^X)_{\theta\in\Theta}$ to be tight. Then for every $\epsilon>0$ there exists a compact set $K_\epsilon\subset \R^d$ and $n_0\in\N$ such that for every $n\geq n_0$,
\[ \inf_{\theta\in\Theta} \Pr_\theta\left( X_n\in K_\epsilon\right) \geq 1-\epsilon. \]
\end{lemma}
\begin{proof}
Let $\epsilon>0$. By assumption there exists a compact set $K\subset\R^d$ such that
\[ \inf_{\theta\in\Theta} \Pr_\theta (X\in K) \geq 1-\epsilon.\]
We assume $K$ to be a closed ball of radius $r$ centered at the origin, i.e. $K=\overline{B_r(0)}$. This can always be done, as any compact set in $\R^d$ is subset of such a ball by the theorem of Heine--Borel. Now pick $\delta>0$ and define $K^\delta := \overline{B_{r+\delta}(0)}$, together with
\[ f_{K,\delta}(x) := \begin{cases}
1& \textrm{ if } x\in K,\\
1-\frac{d(x,K)}{\delta} & \textrm{ if } x\in K^\delta\setminus K, \\
0 &\textrm{ else},
\end{cases} \]
where $d(x,K) := \inf_{z\in K} \|x-z\|_2$. The function $f_{K,\delta}$ is Lipschitz continuous and we have $\1_K\leq f_{K,\delta}\leq \1_{K^\delta}$. This gives
\begin{align*}
&\inf_{\theta\in\Theta} \Pr_\theta \left(X_n\in K^\delta\right) = \inf_{\theta\in\Theta}\E_\theta \left[ \1_{K^\delta}(X_n)\right] \geq \inf_{\theta\in\Theta} \E_\theta\left[ f_{K,\delta}(X_n)\right]\\
&\hspace{1cm} \stackrel{n\to\infty}{\longrightarrow} \inf_{\theta\in\Theta}\E_\theta\left[ f_{K,\delta}(X)\right] \geq \inf_{\theta\in\Theta} \E_\theta\left[ \1_K(X)\right] = \inf_{\theta\in\Theta}\Pr_\theta(X\in K).
\end{align*}
The last term is greater than $1-\epsilon$ and we find an integer $n_0$ such that for $n\geq n_0$, we have $\inf_{\theta\in\Theta} \Pr_\theta \left(X_n\in K^\delta\right)\geq 1-2\epsilon$.
\end{proof}

The following result is a modification of Theorem~$1.12.1$ in \cite{Wellner} for uniform weak convergence and its proof follows the same ideas. In our later application of this result, the bounded and equicontinuous class $\mathcal{F}$ used within this result will be chosen as $\mathcal{F}_{BL}(\R^d)$.

\begin{lemma}\label{uniform_in_function_class}
Let $\Theta$ be a parameter space and $X, X_1, X_2\dots$ be $\R^d$-valued random variables such that the family $(\Pr_\theta^{X})_{\theta\in\Theta}$ is tight, i.e. for every $\epsilon>0$ there exists a compact set $K=K_\epsilon\subset\R^d$  such that
\[ \inf_{\theta\in\Theta} \Pr_\theta ( X\in K) \geq 1-\epsilon.\]
Furthermore, we have for any bounded and continuous $f:\R^d\rightarrow\R$,
\[ \sup_{\theta\in\Theta} \left|\E_\theta\left[ f(X_n) \right] - \E_\theta\left[f(X)\right] \right|\stackrel{n\to\infty}{\longrightarrow} 0.\]
Let $\mathcal{F}$ be a bounded and equicontinuous class of functions. Then the last convergence even holds true uniformly in $\mathcal{F}$, i.e.
\[ \sup_{\theta\in\Theta} \sup_{f\in\mathcal{F}} \left|\E_\theta\left[ f(X_n) \right] - \E_\theta\left[f(X)\right] \right| \stackrel{n\to\infty}{\longrightarrow} 0. \]
\end{lemma}
\begin{proof}
Let $\epsilon>0$. Then choose $K'$ to suffice $\sup_{\theta\in\Theta} \Pr_\theta(X\in K')\geq 1-\epsilon$ according to the assumption. Using Lemma~\ref{approx_P}, there exists a compact set $K$ with $K'\subset K\subset\R^d$ and $n_0\in\N$ such that for $n\geq n_0$ we have
\[ \inf_{\theta\in\Theta} \Pr_\theta( X_n\in K)\geq 1-\epsilon.\]
Denote by $\mathcal{F}_K$ the set of all functions of $\mathcal{F}$ restricted to the compact set $K$. Then by the Arzel\`{a}--Ascoli theorem, $\mathcal{F}_K$ is totally bounded in $\mathcal{C}_b(K)$ and there exist finitely many balls of radius $\epsilon$ that cover $\mathcal{F}_K$. Denote their centers by $f_1, \dots, f_N$. By Tietze's theorem, we can extend them to elements of $\mathcal{C}_b(\R^d)$ which are again denoted by $f_1,\dots, f_N$. Then, by assumption, 
\begin{align}\label{eqp: C1}
\max_{i=1,\dots, N}\sup_{\theta\in\Theta} \left|\E_\theta\left[ f(X_n) \right] - \E_\theta\left[f(X)\right] \right|\stackrel{n\to\infty}{\longrightarrow} 0.
\end{align}
We split
\begin{align}\label{eqp: C13}
\begin{split}
\E_\theta\left[ f(X_n) \right] - \E_\theta\left[f(X)\right] &= \E_\theta\left[ f(X_n) \right] - \E_\theta\left[f_i(X_n) \right] \\
&\hspace{1cm}+ \E_\theta\left[ f_i(X_n) \right] - \E_\theta\left[f_i(X)\right]\\
&\hspace{1cm}+\E_\theta\left[ f_i(X) \right] - \E_\theta\left[f(X)\right].
\end{split}
\end{align} 
The first difference can be rewritten as
\begin{align*}
&\E_\theta\left[ f(X_n) \right] - \E_\theta\left[f_i(X_n) \right] \\
&\hspace{1cm}= \E_\theta\left[(f(X_n) - f_i(X_n))\1_K\right] + \E_\theta\left[ (f(X_n) - f_i(X_n))\1_{K^c}\right].
\end{align*} 
For the right choice of $i$ the first term is bounded by $\epsilon$, as we can choose $f_i$ to be the center of the $\epsilon$-ball (in supremum norm on $K$) that contains $f$. For the latter summand we have 
\begin{align*}
\sup_{\theta\in\Theta}\sup_{f\in\mathcal{F}}\E_\theta\left[ (f(X_n) - f_i(X_n))\1_{K^c}\right] \leq 2C' \sup_{\theta\in\Theta} \Pr_\theta\left( X_n\in K^c\right) \leq 2C'\epsilon,
\end{align*}
for $n\geq n_0$ and $C'$ being the uniform bound on $\mathcal{F}$. Additionally, we used
\[ \sup_{\theta\in\Theta} \Pr_\theta\left(X_n\in K^c\right) = \sup_{\theta\in\Theta} \left( 1- \Pr(X_n\in K)\right) = 1-\inf_{\theta\in\Theta} \Pr_\theta(X_n\in K) \leq \epsilon. \] 
The third difference in \eqref{eqp: C13} can be treated in the same way and the second one is bounded by $\epsilon$ for sufficiently large $n$ by \eqref{eqp: C1}. Thus, 
\[ \sup_{\theta\in\Theta}\sup_{f\in\mathcal{F}}\left|\E_\theta\left[ f(X_n) \right] - \E_\theta\left[f(X)\right] \right| \leq 3\epsilon + 4C'\epsilon, \]
and letting $\epsilon\searrow 0$ concludes the proof.
\end{proof}

The following result can be proven by standard arguments.

\begin{lemma}\label{lemma_weak_1}
Let $(E,d)$ be a separable metric space and $(X_n)_{n\in\N}$ and $(Y_n)_{n\in\N}$ be two sequences of $E$-valued random variables such that for all $\epsilon>0$,
\[ \sup_{\theta\in\Theta} \Pr_\theta\left( d(X_n ,Y_n)>\epsilon\right) \stackrel{n\to\infty}{\longrightarrow}\ 0\]
and $\sup_{\theta\in\Theta}d_{BL}^\theta(X_n, X)\rightarrow 0$ for some random variable $X$. Then,
\[\sup_{\theta\in\Theta}d_{BL}^\theta(Y_n, X) \stackrel{n\to\infty}{\longrightarrow}\ 0.\]
\end{lemma}

\begin{lemma}\label{lemma_weak_2}
Let sequences $(X_n(s))_{n\in\N}$, $(\alpha_n(s))_{n\in\N}$ and $(\beta_n(s))_{n\in\N}$ of real-valued stochastic processes  together with real-valued random variables $\alpha(s),\beta(s),X(s)$ depending on a parameter $s\in\mathcal{S}$ be given. Assume the following conditions to be true for a parameter space $\Theta$:
\begin{enumerate}
\item[(a)] We have the uniform weak convergence
\[ \sup_{\theta\in\Theta} d_{BL}^\theta\left(\sup_{s\in\mathcal{S}}\left(\frac{X_n(s)}{\alpha(s)} -\beta(s)\right),\sup_{s\in\mathcal{S}}\left(\frac{X(s)}{\alpha(s)} -\beta(s)\right)\right) \stackrel{n\to\infty}{\longrightarrow}\ 0.\]
\item[(b)] We have $\alpha(s)\in (0,\infty)$ and for every $\epsilon>0$ there exist compact sets $K(\epsilon), K'(\epsilon)\subset\R$ with
\[ \inf_{\theta\in\Theta} \Pr_\theta\left(\sup_{s\in\mathcal{S}} \beta(s)\in K(\epsilon)\right) \geq 1-\epsilon \]
and
\[ \inf_{\theta\in\Theta} \Pr_\theta\left(\sup_{s\in\mathcal{S}}\left(\frac{X(s)}{\alpha(s)} -\beta(s)\right)\in K'(\epsilon)\right)\geq 1-\epsilon.\]
\item[(c)] For all $n\in\N$, $\alpha_n(s)\in (0,\infty)$ and for every $\epsilon>0$,
\[ \sup_{\theta\in\Theta} \Pr_\theta\left(\sup_{s\in\mathcal{S}} \left|\frac{\alpha(s)}{\alpha_n(s)} - 1\right|>\epsilon\right)\stackrel{n\to\infty}{\longrightarrow} 0.\]
\item[(d)] For every $\epsilon>0$, 
\[ \sup_{\theta\in\Theta} \Pr_\theta\left( \sup_{s\in\mathcal{S}}|\beta_n(s) - \beta(s)|>\epsilon\right) \stackrel{n\to\infty}{\longrightarrow} 0 .\]
\end{enumerate}
Then we have
\[ \sup_{\theta\in\Theta} d_{BL}^\theta\left(\sup_{s\in\mathcal{S}}\left(\frac{X_n(s)}{\alpha_n(s)} -\beta_n(s)\right),\sup_{s\in\mathcal{S}}\left(\frac{X(s)}{\alpha(s)} -\beta(s)\right)\right) \stackrel{n\to\infty}{\longrightarrow}\ 0.\]
\end{lemma}
\begin{proof}
We want to apply Lemma~\ref{lemma_weak_1} and therefore need to show in addition to assumption (a) that for all $\epsilon>0$,
\[ \sup_{\theta\in\Theta}\Pr_\theta\left( \left|\sup_{s\in\mathcal{S}}\left( \frac{X_n(s)}{\alpha_n(s)} - \beta_n(s)\right) - \sup_{s\in\mathcal{S}}\left( \frac{X_n(s)}{\alpha(s)}- \beta(s)\right) \right|>\epsilon\right) \stackrel{n\to\infty}{\longrightarrow} 0.\]
By upper bounding the random variable within the probability, we find
\begin{align*}
&\sup_{\theta\in\Theta}\Pr_\theta\left( \left|\sup_{s\in\mathcal{S}}\left( \frac{X_n(s)}{\alpha_n(s)} - \beta_n(s)\right) - \sup_{s\in\mathcal{S}}\left( \frac{X_n(s)}{\alpha(s)}- \beta(s)\right) \right|>\epsilon\right) \\
&\hspace{0.1cm} \leq \sup_{\theta\in\Theta}\Pr_\theta\left( \sup_{s\in\mathcal{S}}\left|\left( \frac{X_n(s)}{\alpha_n(s)} - \beta_n(s)  -\frac{X_n(s)}{\alpha(s)}+ \beta(s)\right) \right|>\epsilon\right) \\
&\hspace{0.1cm} = \sup_{\theta\in\Theta}\Pr_\theta\left(\sup_{s\in\mathcal{S}} \left| \left(\frac{X_n(s)}{\alpha(s)} - \beta(s) \right)\left( \frac{\alpha(s)}{\alpha_n(s)} -1\right) +\frac{\alpha(s)}{\alpha_n(s)}\beta(s) - \beta_n(s)  \right|>\epsilon\right) \\
&\hspace{0.1cm} \leq \sup_{\theta\in\Theta}\Pr_\theta\left( \sup_{s\in\mathcal{S}} \left|\frac{X_n(s)}{\alpha(s)} - \beta(s) \right|\sup_{s\in\mathcal{S}}\left| \frac{\alpha(s)}{\alpha_n(s)} -1\right| \right. \\
&\hspace{6cm} \left. + \sup_{s\in\mathcal{S}}\left|\frac{\alpha(s)}{\alpha_n(s)}\beta(s) - \beta_n(s)  \right|>\epsilon\right) \\
&\hspace{0.1cm} \leq \sup_{\theta\in\Theta}\Pr_\theta\left( \sup_{s\in\mathcal{S}} \left| \frac{X_n(s)}{\alpha(s)} - \beta(s) \right|\sup_{s\in\mathcal{S}}\left| \frac{\alpha(s)}{\alpha_n(s)} -1\right| >\frac{\epsilon}{2}\right) \\
&\hspace{4cm} + \sup_{\theta\in\Theta}\Pr_\theta\left( \sup_{s\in\mathcal{S}}\left|\frac{\alpha(s)}{\alpha_n(s)}\beta(s) - \beta_n(s)\right|>\frac{\epsilon}{2}\right) .
\end{align*}
Those last two summands will be treated separately. Let $\delta>0$ and choose a corresponding $K=K(\delta)>0$ such that
\[ \sup_{\theta\in\Theta} \Pr_\theta \left( \sup_{s\in\mathcal{S}} \left|\frac{X(s)}{\alpha(s)}- \beta(s)\right| >K\right) < \delta,\]
which is possible by assumption (b). Using $\sup_\theta\Pr_\theta(A)= 1-\inf_\theta\Pr_\theta(A^c)$
and Lemma~\ref{approx_P}, for some $n\geq n_0$ and $K'\geq K>0$,
\[ \sup_{\theta\in\Theta} \Pr_\theta \left( \sup_{s\in\mathcal{S}} \left|\frac{X_n(s)}{\alpha(s)}- \beta(s)\right| >K'\right) < \delta. \]
Proceeding with this,
\begin{align*}
&\sup_{\theta\in\Theta}\Pr_\theta\left( \sup_{s\in\mathcal{S}} \left| \frac{X_n(s)}{\alpha(s)} - \beta(s) \right|\sup_{s\in\mathcal{S}}\left| \frac{\alpha(s)}{\alpha_n(s)} -1\right| >\frac{\epsilon}{2}\right) \\
&\hspace{0.2cm} \leq \sup_{\theta\in\Theta}\Pr_\theta\left( K'\sup_{s\in\mathcal{S}}\left| \frac{\alpha(s)}{\alpha_n(s)} -1\right| >\frac{\epsilon}{2}\right)  + \sup_{\theta\in\Theta} \Pr_\theta \left( \sup_{s\in\mathcal{S}} \left|\frac{X_n(s)}{\alpha(s)}- \beta(s)\right| >K'\right)\\
&\hspace{0.2cm} \leq \sup_{\theta\in\Theta}\Pr_\theta\left(\sup_{s\in\mathcal{S}}\left| \frac{\alpha(s)}{\alpha_n(s)} -1\right| >\frac{\epsilon}{2K'}\right)  + \delta
\end{align*}
and the probability converges to zero for $n\to\infty$ by condition (c). For the other summand at the end of our initial estimate, we proceed in a similar manner. We choose $L$ large enough such that
\[ \sup_{\theta\in\Theta}\Pr_\theta\left( \sup_{s\in\mathcal{S}}|\beta(s)|>L\right) <\delta,\]
which is possible with regard to our assumption (b). Then,
\begin{align*}
& \sup_{\theta\in\Theta}\Pr_\theta\left( \sup_{s\in\mathcal{S}}\left|\frac{\alpha(s)}{\alpha_n(s)}\beta(s) - \beta_n(s)\right|>\frac{\epsilon}{2}\right) \\
&\hspace{1cm} =  \sup_{\theta\in\Theta}\Pr_\theta\left( \sup_{s\in\mathcal{S}}\left|\left(\frac{\alpha(s)}{\alpha_n(s)} -1\right)\beta(s)  + \beta(s)- \beta_n(s)\right|>\frac{\epsilon}{2}\right) \\
&\hspace{1cm} \leq \sup_{\theta\in\Theta}\Pr_\theta\left( \sup_{s\in\mathcal{S}}\left|\left(\frac{\alpha(s)}{\alpha_n(s)} -1\right)\beta(s) \right| + \sup_{s\in\mathcal{S}}| \beta(s)- \beta_n(s)|>\frac{\epsilon}{2}\right)\\[10\jot]
&\hspace{1cm} \leq \sup_{\theta\in\Theta}\Pr_\theta\left( \sup_{s\in\mathcal{S}}\left|\left(\frac{\alpha(s)}{\alpha_n(s)} -1\right)\beta(s) \right| >\frac{\epsilon}{4}\right) \\
&\hspace{3.5cm} + \sup_{\theta\in\Theta}\Pr_\theta\left( \sup_{s\in\mathcal{S}}| \beta(s)- \beta_n(s)| >\frac{\epsilon}{4}\right) \\
&\hspace{1cm} \leq \sup_{\theta\in\Theta}\Pr_\theta\left( \sup_{s\in\mathcal{S}}\left|\frac{\alpha(s)}{\alpha_n(s)} -1\right| >\frac{\epsilon}{4L}\right) + \sup_{\theta\in\Theta}\Pr_\theta\left( \sup_{s\in\mathcal{S}}|\beta(s)|>L\right)  \\
&\hspace{3.5cm} + \sup_{\theta\in\Theta}\Pr_\theta\left( \sup_{s\in\mathcal{S}}| \beta(s)- \beta_n(s)| >\frac{\epsilon}{4}\right). 
\end{align*}
The middle summand here is bounded by $\delta$ and the others both converge to zero by condition (c) and (d). The assertion of the lemma follows by letting $\delta\searrow 0$.
\end{proof}

For a pseudometric space $(\mathcal{S},\rho)$ and $u>0$, the \textit{covering numbers} $N(u,\mathcal{S},\rho)$ are defined by
\begin{align}\label{eq: covering_number}
N(u,\mathcal{S},\rho) := \min\left\{ \#\mathcal{S}_0: \mathcal{S}_0\subset\mathcal{S} \textrm{ and } \inf_{s_0\in\mathcal{S}_0} \rho(s,s_0)\leq u \textrm{ for all } s\in\mathcal{S} \right\},
\end{align}
Subsequently, a proof of Proposition~\ref{prop_unif_weak} is given. This result is an extension of Theorem $8$ in \cite{Rohde_Diss} to uniform weak convergence.

\begin{proof}[Proof of Proposition \ref{prop_unif_weak}]
For every natural number $k\in\N$ let $\mathcal{S}_\theta^k$ be some maximal subset of $\mathcal{S}$ such that $\rho_\theta(s,s')\geq \frac1k$ for any $t,t'\in\mathcal{S}_\theta^k$, and $\mathcal{S}_\theta^1\subset\mathcal{S}_\theta^2\subset\cdots$. Now define
\[ \lambda_\theta^k(s,u) := \frac{\left( 1- k\rho_\theta(s,u)\right)_+}{\sum_{v\in\mathcal{S}_\theta^k} \left( 1- k\rho_\theta(s,v)\right)_+}\]
for all $s\in\mathcal{S}$ and $u\in\mathcal{S}_\theta^k$. It satisfies
\[ \lambda_\theta^k(s,u) =0  \quad \textrm{ if }\quad \rho_\theta(s,u))\geq\frac1k \]
and  
\[ \sum_{u\in\mathcal{S}_\theta^k} \lambda_\theta^k(\cdot, u) =1.\]
Moreover, $0\leq \lambda_\theta^k(\cdot, u) \in\mathcal{C}_u(\mathcal{S},\rho_\theta)$, where
\[ \mathcal{C}_u(\mathcal{S}, \rho_\theta) = \left\{ f\in l_\infty(\mathcal{S}):\ \lim_{\delta\searrow 0} \sup_{s,t\in\mathcal{S}: \rho_\theta(s,t)<\delta} |f(t)- f(s)| =0\right\}. \]
Define $\pi_\theta^k: l_\infty(\mathcal{S}) \rightarrow\mathcal{C}_u(\mathcal{S}, \rho_\theta)$ by
\[ \pi_\theta^k f := \sum_{u\in\mathcal{S}_\theta^k} f(u) \lambda_\theta^k(\cdot, u).\]
Then for all $f\in l_\infty(\mathcal{S})$,
\begin{align}\label{eqp: C16}
\|\pi_\theta^k f\|_\mathcal{S} \leq \|f\|_{\mathcal{S}_\theta^k}\quad \textrm{ and }\quad \| f- \pi_\theta^k f\|_{\mathcal{S}} \leq \sup_{\rho_\theta(s,s')\leq \frac1k} |f(s) - f(s')|,
\end{align} 
and $\pi_\theta^k$ is a linear map with Lipschitz constant one because for $f,g\in l_\infty(\mathcal{S})$,
\begin{align*}
\left| (\pi_\theta^k f)(s) - (\pi_\theta^k g)(s)\right| &= \left| \sum_{u\in\mathcal{S}_\theta^k} (f(u)-g(u))\lambda_\theta^k(s,u)\right| \\
&\leq \sup_{u\in\mathcal{S}_\theta^k}\left| f(u)-g(u)\right| \sum_{u\in\mathcal{S}_\theta^k} \lambda_\theta^k(s,u)\\
& \leq \|f-g\|_{\mathcal{S}}.
\end{align*}
In particular, for any $f: l_\infty(\mathcal{S}) \rightarrow [0,1]$ with Lipschitz constant $L$, the composition $f\circ\pi_\theta^k$ again takes values in $[0,1]$ and has Lipschitz constant $L$. Now we split
\begin{align*}
&\sup_{\theta\in\Theta} \sup_{f\in\mathcal{F}_{BL}(l_\infty(\mathcal{S}))} \left| \E_\theta\left[ f(X_n)\right] - \E_\theta\left[f(Y_n)\right]\right| \\
&\hspace{2cm}\leq  \sup_{\theta\in\Theta} \sup_{f\in\mathcal{F}_{BL}(l_\infty(\mathcal{S}))}  \E_\theta\left[ \left|f(X_n)-f(\pi_\theta^k X_n)\right|\right]\\
&\hspace{3cm} + \sup_{\theta\in\Theta} \sup_{f\in\mathcal{F}_{BL}(l_\infty(\mathcal{S}))} \E_\theta\left[ \left| f(\pi_\theta^k X_n)-f(\pi_\theta^k Y_n)\right|\right]\\
&\hspace{3cm} + \sup_{\theta\in\Theta} \sup_{f\in\mathcal{F}_{BL}(l_\infty(\mathcal{S}))}  \E_\theta\left[\left| f(\pi_\theta^k Y_n)-f(Y_n) \right|\right].
\end{align*}
We have $\{f\circ\pi_\theta^k\mid f\in\mathcal{F}_{BL}(l_\infty(\mathcal{S}))\} \subset \mathcal{F}_{BL}(l_\infty(\mathcal{S}_\theta^k))$ and assumption (c) on the covering numbers yields $\sup_\theta|\mathcal{S}_\theta^k|<\infty$. By assumption (a),
\begin{align*}
&\sup_{\theta\in\Theta} \sup_{f\in\mathcal{F}_{BL}(l_\infty(\mathcal{S}))} \E_\theta\left[ \left| f(\pi_\theta^k X_n)-f(\pi_\theta^k Y_n)\right|\right]\stackrel{n\to\infty}{\longrightarrow}\ 0.
\end{align*} 
Let $\epsilon>0$. By assumption (b) there exists a natural number $l=l(\epsilon)$ such that for $Z_n= X_n, Y_n$ we have
\[  \limsup_{n\to\infty} \sup_{\theta\in\Theta} \Pr_\theta\left( \sup_{\rho_\theta(s,s')\leq\frac{1}{l}} |Z_n(s) - Z_n(s')| >\epsilon \right) <\epsilon. \]
By the bounded Lipschitz property of $f$ and \eqref{eqp: C16} we get for this number $l$,\pagebreak
\begin{align*}
&\limsup_{n\to\infty}\sup_{\theta\in\Theta} \sup_{f\in\mathcal{F}_{BL}(l_\infty(\mathcal{S}))}  \E_\theta\left[ \left|f(Z_n)-f(\pi_\theta^{l} Z_n)\right|\right]\\
&\hspace{0.5cm}\leq \limsup_{n\to\infty}\sup_{\theta\in\Theta} \sup_{f\in\mathcal{F}_{BL}(l_\infty(\mathcal{S}))}  \E_\theta\left[ \left|f(Z_n)-f(\pi_\theta^{l} Z_n)\right| \1_{\{\|Z_n-\pi_\theta^{l} Z_n\|_\mathcal{S}\leq\epsilon\}}\right]\\
&\hspace{1cm} +\limsup_{n\to\infty}\sup_{\theta\in\Theta} \sup_{f\in\mathcal{F}_{BL}(l_\infty(\mathcal{S}))}  \E_\theta\left[ \left|f(Z_n)-f(\pi_\theta^{l} Z_n)\right| \1_{\{\|Z_n-\pi_\theta^{l} Z_n\|_\mathcal{S}>\epsilon\}}\right]\\
&\hspace{0.5cm} \leq  \epsilon +2\limsup_{n\to\infty}\sup_{\theta\in\Theta}\Pr_\theta\left( \|Z_n-\pi_\theta^l Z_n\|_\mathcal{S}>\epsilon\right) \\
&\hspace{0.5cm}\leq \epsilon + 2 \limsup_{n\to\infty} \sup_{\theta\in\Theta} \Pr_\theta\left( \sup_{\rho_\theta(s,s')\leq\frac1l} |Z_n(s) - Z_n(s')| >\epsilon \right) <3\epsilon.
\end{align*}
Letting $\epsilon\searrow 0$ completes the proof.
\end{proof}

\subsection{Proof of Theorem \ref{weak_conv}}\label{SubSec_C2}

Throughout the proof we will abbreviate $\Sigma:= \Sigma(C,A,\gamma,\sigma)$. Moreover, we define 
\[Z_T(y,h) := \frac{1}{\sqrt{T}} \int_0^T K_{y,h}(X_s) dW_s \]
and
\[ Z_b(y,h) := \int_\R K_{y,h}(z) \sqrt{q_b(z)} dW_z.\]
We remember the notation   
\[ \hat\sigma_T(y,h) ^2= \frac1T\int_0^T K_{y,h}(X_s)^2 ds\quad \textrm{ and } \quad \hat\sigma_{T,\max}^2= \frac1T\int_0^T \1_{[-A,A]}(X_s) ds \]
from Section~\ref{App_Multiscale}, together with $\overline\sigma_T(y,h) = \frac{\hat\sigma_T(y,h)}{\hat\sigma_{T,\max}}$ and the semimetric
\[ \overline\rho_T((y,h), (y',h'))^2 = \frac{1}{T\hat\sigma_{T,\max}^2} \int_0^T \left( K_{y,h}(X_s) - K_{y',h'}(X_s)\right)^2 ds.\]
We define the deterministic counterparts
\begin{align}\label{eqp: C5}
\sigma_b(y,h) ^2 := \int_\R K_{y,h}(z)^2 q_b(z) dz\ \ \textrm{ and } \ \ \sigma_{b,\max}^2 := \int_\R \1_{[-A,A]}(z)q_b(z) dz
\end{align} 
on the set $\mathcal{T}$ given in \eqref{eq: cT}, and the corresponding semimetric on $\mathcal{T}\times\mathcal{T}$ by
\[ \rho_b((y,h), (y',h'))^2 := \int_\R \left( K_{y,h}(z) - K_{y',h'}(z)\right)^2 q_b(z) dz. \]
Last but not least, we set 
\begin{align}\label{eqp: C6}
\overline\sigma_b(y,h) := \frac{\sigma_b(y,h)}{\sigma_{b,\max}}\ \ \textrm{ and }\ \ \overline\rho_b((y,h), (y',h')) := \frac{\rho_b((y,h), (y',h'))}{\sigma_{b,\max}}. 
\end{align}

By an application of Proposition~\ref{uniform_ergodic} one directly gets the following result about these quantities.

\begin{lemma}\label{lemma_unif_sigma}
For the norms and seminorms defined above, we have for every $\epsilon>0$ that
\[ \lim_{T\to\infty}\sup_{b\in\Sigma} \Pr_b\left( \| \overline\sigma_T - \overline\sigma_b\|_\mathcal{T} >\epsilon\right) =0 \]
and
\[ \lim_{T\to\infty}\sup_{b\in\Sigma} \Pr_b\left( \| \overline\rho_T - \overline\rho_b\|_{\mathcal{T}\times\mathcal{T}} >\epsilon\right) =0. \]
\end{lemma}

\begin{proof}[Proof of Theorem \ref{weak_conv}]
The proof follows the steps (i)--(iii) given in the sketch of proof in Subsection~\ref{SubSec_weak} which are given here again with some additional details:
\begin{enumerate}
\item[(i)] First, we establish for $0<c\leq A$ and 
\begin{align}\label{eqp: C9}
\mathcal{T}_{c}=\{(y,h)\in\mathcal{T}\mid h\geq c\} 
\end{align}
the uniform weak convergence of $(Z_T(y,h))_{(y,h)\in\mathcal{T}_{c}}$ to $(Z_b(y,h))_{(y,h)\in\mathcal{T}_{c}}$ with Proposition~\ref{prop_unif_weak}.
\item[(ii)] Secondly, we define
\begin{align}\label{eqp: T_T}
& T_T(\delta, \delta') := \sup_{\substack{(y,h)\in\mathcal{T}:\\ \delta < \overline\sigma_b(y,h)\leq \delta'}} \left(\frac{\left|Z_T(y,h)\right|}{\hat\sigma_T(y,h)} - \Upsilon\bigg( \frac{\hat\sigma_T(y,h)^2}{\hat\sigma_{T,\max}^2}\bigg) \right) 
\end{align} 
as well as
\begin{align*}
&S_b(\delta, \delta') := \sup_{\substack{(y,h)\in\mathcal{T}:\\ \delta < \overline\sigma_b(y,h)\leq \delta'}} \left(\frac{\left|Z_b(y,h)\right|}{\sigma_b(y,h) } - \Upsilon\bigg( \frac{\sigma_b(y,h)^2}{\sigma_{b,\max}^2}\bigg) \right)
\end{align*} 
and prove uniform weak convergence of $T_T(\delta,1)$ to $S_b(\delta,1)$ for every $c'>\delta>0$ and some constant $c'>0$. This is based on the uniform continuous mapping result in Lemma~\ref{uniform_continuous_mapping} and Lemma~\ref{lemma_weak_2}.
\item[(iii)] Finally, we improve the result from (ii) to establish the desired weak convergence of $T_T(0,1)$ to $S_b(0,1)$.
\end{enumerate}

In the following, we will work through step (i), (ii) and (iii).\\

\textit{Step (i).} This result is established by using Proposition~\ref{prop_unif_weak}. To prove prerequisite (a) of it, we first pick $(y_1,h_1),\dots, (y_N, h_N)\in\overline{\mathcal{T}}_c$ where $\overline{\mathcal{T}}_c$ is given by
\[\overline{\mathcal{T}}_c := \left\{(y,h)\in\R^2 \mid c\leq h\leq A \textrm{ and } -A+h \leq y \leq A-h\right\}. \]
For all $\epsilon>0$ and $1\leq i,j\leq N$ we have by Proposition \ref{uniform_ergodic},
\begin{align*}
&\lim_{T\to\infty}\sup_{b\in\Sigma} \Pr\left( \left| \frac1T\int_0^T K_{y_i,h_i}(X_s) K_{y_j, h_j}(X_s)  ds \right.\right. \\
&\hspace{3.5cm} \left.\left. - \int_\R K_{y_i, h_i}(z) K_{y_j, h_j}(z) q_{b}(z) dz \right| >\epsilon\right) =0.
\end{align*} 
In consequence, by Proposition $1.21$ in \cite{Kutoyants},
\[ \lim_{T\to\infty}\sup_{b\in\Sigma} \left| \E_b\left[ f((Z_T(y_k,h_k))_{1\leq k\leq N}) \right] - \E_b\left[ f((Z_b(y_k,h_k))_{1\leq k\leq N})\right]\right| =0\]
for any bounded and continuous function $f:\R^N\rightarrow\R$.
In order to apply Lemma~\ref{uniform_in_function_class} and conclude
\begin{align}\label{eqp: C10}
\begin{split}
&\lim_{T\to\infty}\sup_{b\in\Sigma}\sup_{f\in \mathcal{F}_{BL}(\R^N)}  \left| \E_b\left[ f((Z_T(y_k,h_k))_{1\leq k\leq N}) \right] \right. \\
&\hspace{4cm} \left. - \E_b\left[ f((Z_b(y_k,h_k))_{1\leq k\leq N})\right]\right| =0,
\end{split}
\end{align}
we have to show uniform tightness of $(Z_b(y_k,h_k))_{1\leq k\leq N}$. Using Markov's inequality and It\^{o}'s isometry, we find
\begin{align*}
\sup_{b\in\Sigma} \Pr_b\left( \| Z_b(y_k,h_k)\|_2^2 \geq r\right) &= \sup_{b\in\Sigma} \Pr_b\left( \sum_{k=1}^N Z_b(y_k,h_k)^2 \geq r\right)\\
&\leq \frac1r \sup_{b\in\Sigma}  \E_b\left[ \sum_{k=1}^N Z_b(y_k,h_k)^2\right] \\
&= \frac1r \sup_{b\in\Sigma} \sum_{k=1}^N \int_\R K_{y_k,h_k}(z)^2 q_b(z) dz\\
&\leq \frac{NA\|K\|_{L^2}^2 L^*}{r},
\end{align*}
where we used $h_k\leq A$ and Lemma~\ref{bound_invariant_density}. For $r\to\infty$, this expression tends to zero and therefore \eqref{eqp: C10} holds true by Lemma~\ref{uniform_in_function_class}. Next, we have to consider an additional $\sup_{B\subset\mathcal{T}_c: \#B\leq N}$ instead of fixed $N$, but for the statement of the theorem, it suffices to consider $B\subset\mathcal{T}_c=\overline{\mathcal{T}}_c\cap\Q^2$. Suppose that for some $n_0$,
\begin{align*}
&\sup_{B\subset\mathcal{T}_c: \#B\leq n_0} \sup_{b\in\Sigma} \sup_{f\in\mathcal{F}_{BL}(\R^{\#B})} \left| \E_b\left[ f((Z_T(y,h))_{(y,h)\in B}) \right]  \right.\\
&\hspace{5.5cm} \left. - \E_b\left[ f((Z_b(y,h))_{(y,h)\in B})\right]\right| 
\end{align*} 
does not tend to zero for $T\to\infty$. Then there exists a sequence $(T_n)_{n\in\N}$ with $T_n\stackrel{n\to\infty}{\longrightarrow}\infty$ together with $(B_n)_{n\in\N}\subset \mathcal{T}_c$, $\#B_n\leq n_0$, and $\epsilon>0$ such that for all $n\in\N$,
\begin{align}\label{eqp: C11}
\begin{split}
&\sup_{b\in\Sigma} \sup_{f\in\mathcal{F}_{BL}(\R^{\#B_n})} \left| \E_b\left[ f((Z_{T_n}(y,h))_{(y,h)\in B_n}) \right] \right.\\
&\hspace{4cm} \left. - \E_b\left[ f((Z_b(y,h))_{(y,h)\in B_n})\right]\right| >\epsilon.
\end{split}
\end{align} 
We will show that this is not possible. First of all, there exists an integer $\tilde n_0\leq n_0$ and a subsequence $(n_k)_{k\in\N}$ such that $\#B_{n_k}= \tilde n_0$, as there are only finitely many integers $\leq n_0$. Furthermore, as $\mathcal{T}_c=\overline{\mathcal{T}}_c\cap\Q^2$ and $\overline{\mathcal{T}}_c^{\tilde n_0}$ is compact, 
there again exists a subsubsequence $(n_{k_l})_{l\in\N}$ such that the vector of locations $((y,h))_{(y,h)\in B_{n_{k_l}}}$ converges to a vector 
\[ \left( (y_0^1, h_0^1) , \dots, (y_0^{\tilde n_0}, h_0^{\tilde n_0})\right) =: ((y,h))_{(y,h)\in B}\]
with $B\subset\overline{\mathcal{T}}_c$. We reindex this subsubsequence just by $n$ for ease of notation and have
\begin{align*}
& \sup_{b\in\Sigma} \sup_{f\in\mathcal{F}_{BL}(\R^{\tilde n_0})} \left| \E_b\left[ f\left((Z_{T_n}(y,h))_{(y,h)\in B_n}\right) \right] - \E_b\left[ f((Z_b(y,h))_{(y,h)\in B_n})\right]\right| \\
&\hspace{0.1cm}\leq \sup_{b\in\Sigma} \sup_{f\in\mathcal{F}_{BL}(\R^{\tilde n_0})} \E_b\left[ \left| f\left((Z_{T_n}(y,h))_{(y,h)\in B_n}\right) - f\left((Z_{T_n}(y,h))_{(y,h)\in B}\right)\right|\right]\\
&\hspace{0.5cm} + \sup_{b\in\Sigma} \sup_{f\in\mathcal{F}_{BL}(\R^{\tilde n_0})} \E_b\left[ \left| f\left((Z_{T_n}(y,h))_{(y,h)\in B}\right) - f\left((Z_b(y,h))_{(y,h)\in B}\right)\right|\right]\\
&\hspace{0.5cm} + \sup_{b\in\Sigma} \sup_{f\in\mathcal{F}_{BL}(\R^{\tilde n_0})} \E_b\left[ \left| f\left((Z_b(y,h))_{(y,h)\in B}\right) - f\left((Z_b(y,h))_{(y,h)\in B_n}\right)\right|\right].
\end{align*}
The second summand tends to zero due to our previous work on uniform weak convergence of a finite-dimensional marginal, see \eqref{eqp: C10}. For the first term we use the fact that $f\in\mathcal{F}_{BL}(\R^{\tilde n_0})$ is Lipschitz continuous with constant one and that all norms on $\R^{\tilde n_0}$ are equivalent. Then we have \pagebreak
\begin{align*}
&\sup_{b\in\Sigma} \sup_{f\in\mathcal{F}_{BL}(\R^{\tilde n_0})} \E_b\left[ \left| f\left((Z_{T_n}(y,h))_{(y,h)\in B_n}\right) - f\left((Z_{T_n}(y,h))_{(y,h)\in B}\right)\right|\right] \\
&\hspace{0.5cm}\leq C' \sup_{b\in\Sigma} \E_b\left[ \left\| \left((Z_{T_n}(y,h))_{(y,h)\in B_n}\right) - \left((Z_{T_n}(y,h))_{(y,h)\in B}\right)\right\|_2\right]\\
&\hspace{0.5cm}\leq C' \sup_{b\in\Sigma} \sqrt{\E_b\left[ \left\| \left((Z_{T_n}(y,h))_{(y,h)\in B_n}\right) - \left((Z_{T_n}(y,h))_{(y,h)\in B}\right)\right\|_2^2\right] }
\end{align*}
for some constant $C'>0$. Denote $B_n = \left\{ (y_n^1, h_n^1),\dots, (y_n^{\tilde n_0}, h_n^{\tilde n_0})\right\}$. Using It\^{o}'s isometry, we have
\begin{align*}
&\E_b\left[ \left\| \left((Z_{T_n}(y,h))_{(y,h)\in B_n}\right) - \left((Z_{T_n}(y,h))_{(y,h)\in B}\right)\right\|_2^2\right]\\
&\hspace{1cm} = \E_b\left[ \sum_{k=1}^{\tilde n_0} \left(\frac{1}{\sqrt{T_n}} \int_0^{T_n} \left(K_{y_n^k, h_n^k}(X_s) - K_{y_0^k, h_0^k}(X_s)\right) dW_s\right)^2\right]\\
&\hspace{1cm} \leq \sum_{k=1}^{\tilde n_0} \left\| K_{y_n^k, h_n^k} - K_{y_0^k, h_0^k}\right\|_{[-A,A]}^2.
\end{align*}
This last term tends to zero for $n\to\infty$ as the location parameters of the kernel function converge to each other. The third term can be treated in the same way, where we have
\begin{align*}
&\E_b\left[ \left\| \left((Z_b(y,h))_{(y,h)\in B_n}\right) - \left((Z_b(y,h))_{(y,h)\in B}\right)\right\|_2^2\right]\\
&\hspace{1cm} = \E_b\left[ \sum_{k=1}^{\tilde n_0} \left( \int_\R \left(K_{y_n^k, h_n^k}(z) - K_{y_0^k, h_0^k}(z)\right) \sqrt{q_b(z)} dW_z\right)^2\right]\\
&\hspace{1cm}\leq L^* \sum_{k=1}^{\tilde n_0} \int_\R \left(K_{y_n^k, h_n^k}(z) - K_{y_0^k, h_0^k}(z)\right)^2 dz.
\end{align*}
Here, $L^*$ is the uniform upper bound on the invariant density $q_b$ from Lemma~\ref{bound_invariant_density} and again this expression goes to zero due to convergence of the location parameters. Altogether we have convergence to zero of the term
\[ \sup_{b\in\Sigma} \sup_{f\in\mathcal{F}_{BL}(\R^{\tilde n_0})} \left| \E_b\left[ f\left((Z_{T_n}(y,h))_{(y,h)\in B_n}\right) \right] - \E_b\left[ f((Z_b(y,h))_{(y,h)\in B})\right]\right|, \]
contradicting our choice of $T_n$ and $B_n$ that makes this term be bounded from below by $\epsilon$, see \eqref{eqp: C11}. Hence, condition (a) of Proposition~\ref{prop_unif_weak} is fullfilled.\newpage

We turn to condition (b). We start proving it for $Z_b$ and the metric $\overline\rho_b$. Here we have (with suppressed  dependence of $\overline\rho_b$ on its parameters for shorter notation)
\begin{align*}
&\Pr_{b}\left( \sup_{\overline\rho_b((y,h),(y',h'))\leq \delta}\left| Z_b(y,h) - Z_b(y',h')\right|>\epsilon\right)\\
&\hspace{0.5cm}\leq \Pr_{b}\left( \sup_{\overline\rho_b((y,h),(y',h'))\leq \delta}\frac{\left| Z_b(y,h) - Z_b(y',h')\right|}{\overline\rho_b \log(e/\overline\rho_b)} \right. \\
&\hspace{5cm} \left. \cdot \sup_{\overline\rho_b((y,h),(y',h'))\leq \delta}\overline\rho_b \log(e/\overline\rho_b)>\epsilon \right)\\
&\hspace{0.5cm} \leq  \Pr_b\left( Q\cdot \sup_{\overline\rho_b((y,h),(y',h'))\leq \delta}\overline\rho_b \log(e/\overline\rho_b)>\epsilon \right)\\
&\hspace{2.5cm} + \Pr_b\left( \sup_{\overline\rho_b((y,h),(y',h'))\leq \delta}\frac{\left| Z_b(y,h) - Z_b(y',h')\right|}{\overline\rho_b \log(e/\overline\rho_b)}>Q\right)
\end{align*}
for any $Q>0$. The first term tends to zero uniformly in $b$ as $\delta\searrow 0$ because $x\log(e/x)$ does for $x\searrow 0$ monotonously on $(0,1)$. For a proper choice of the constant $Q$, we have seen in the proof of Theorem~\ref{Multiscale_general} in equation \eqref{eqp: B4} that the second term is bounded by $\frac{L''\delta}{2}$ for some constant $L''$ not depending on $\delta, b$. Note that Theorem~\ref{Multiscale_general} is applicable in this setting and in consequence, 
\[ \lim_{\delta\searrow}\sup_{b\in\Sigma}\Pr_{b}\left( \sup_{\overline\rho_b((y,h),(y',h'))\leq \delta}\left| Z_b(y,h) - Z_b(y',h')\right|>\epsilon\right) =0.\]
To show condition (b) for the process $Z_T$, we first switch to the semimetric $\overline\rho_T$ instead of $\overline\rho_b$ in the following way (where again the dependence of $\overline{\rho}_b, \overline\rho_T$ on its parameters is suppressed in the notation):
\begin{align*}
&\Pr_{b}\left( \sup_{\overline\rho_b\leq \delta}\left| Z_T(y,h) - Z_T(y',h')\right|>\epsilon\right)\\
& \hspace{0.3cm} \leq \Pr_{b}\left(\sup_{\overline\rho_b^2\leq \delta^2}\left| Z_T(y,h) - Z_T(y',h')\right|>\epsilon, \| \overline\rho_T^2 - \overline\rho_b^2\|_{\mathcal{T}_c\times \mathcal{T}_c} \leq \delta^2 \right)\\
&\hspace{1.3cm} + \Pr_{b}\left( \sup_{\overline\rho_b^2\leq \delta^2}\left| Z_T(y,h) - Z_T(y',h')\right|>\epsilon, \| \overline\rho_T^2 - \overline\rho_b^2\|_{\mathcal{T}_c\times \mathcal{T}_c} > \delta^2 \right)\\
&\hspace{0.3cm} \leq \Pr_{b}\left( \sup_{\overline\rho_T\leq \sqrt{2}\delta}\left| Z_T(y,h) - Z_T(y',h')\right|>\epsilon \right) + \Pr_b\left( \| \overline\rho_T^2 - \overline\rho_b^2\|_{\mathcal{T}_c\times \mathcal{T}_c} > \delta^2 \right).
\end{align*}
By Lemma \ref{lemma_unif_sigma}, 
\[ \lim_{T\to\infty}\sup_{b\in\Sigma} \Pr_{b}\left( \| \overline\rho_T^2 - \overline\rho_b^2\|_{\mathcal{T}_c\times \mathcal{T}_c} > \delta^2 \right) =0.\] 
The first summand can be treated in the same way as above where we established the uniform asymptotic stochastic equicontinuity for $Z_b$ and is seen to be bounded from above by
\begin{align*}
&\Pr_b\left( Q\cdot \sup_{\overline\rho_T((y,h),(y',h'))\leq \sqrt{2}\delta}\overline\rho_T \log(e/\overline\rho_T)>\epsilon \right)\\
&\hspace{2.5cm} + \Pr_b\left( \sup_{\overline\rho_T((y,h),(y',h'))\leq \sqrt{2}\delta}\frac{\left| Z_T(y,h) - Z_T(y',h')\right|}{\overline\rho_T \log(e/\overline\rho_T)}>Q\right).
\end{align*}
The first term again vanishes uniformly in $b$ as $\delta\searrow 0$ because $x\log(e/x)\to 0$ for $x\searrow 0$ monotonously on $(0,1)$. For a proper choice of the constant $Q$, we have seen in the proof of Theorem~\ref{Multiscale_general} in equation \eqref{eqp: B4} that the second summand is bounded by $\frac{L''\delta}{2} + \Pr_b (\mathcal{C}_{T,b}^c)$ for some constant $L''$ not depending on $\delta,T, b$ and the set $\mathcal{C}_{T,b}$ is given by
\begin{align}\label{eqp: C2}
\mathcal{C}_{T,b} := \left\{ \left\| \frac{1}{\sigma^2 T}L_T^\cdot(X) - q_b\right\|_{[-A,A]} \leq \frac12 L_*\right\}.
\end{align} 
Applicability of Theorem~\ref{Multiscale_general} for $Z_T$ with this choice of $\mathcal{C}_T$ was checked in the proof of Theorem~\ref{Multiscale_Lemma}. Since 
\[ \lim_{T\to\infty} \sup_{b\in\Sigma}\Pr_b\left( \mathcal{C}_{T,b}^c\right) =0 \]
by \eqref{eq:conv_emprical_density_uniform}, uniform stochastic equicontinuity of $Z_T$ follows, i.e.
\[ \lim_{\delta\searrow 0}\limsup_{T\to\infty}\sup_{b\in\Sigma}\Pr_{b}\left( \sup_{\overline\rho_b\leq \delta}\left| Z_T(y,h) - Z_T(y',h')\right|>\epsilon\right) =0 \]
and condition (b) of Proposition~\ref{prop_unif_weak} is fullfilled.\\

For condition~(c) we first note that $N(u,\mathcal{T}_c,\overline\rho_b)\leq D(u,\mathcal{T}_c,\overline\rho_b)$. Then a uniform bound on the capacity numbers follows similarly to that given in the verification of (d) in the proof of Theorem~\ref{Multiscale_Lemma} for $\mathcal{T}\supset\mathcal{T}_c$.\\ 

Consequently, by Proposition~\ref{prop_unif_weak}, 
\[ \sup_{b\in\Sigma} \sup_{f\in\mathcal{F}_{BL}(l_\infty(\mathcal{T}_c))} \left| \E\left[ f((Z_T(y,h)_{(y,h)\in\mathcal{T}_c})\right] -  \E\left[ f((Z_b(y,h)_{(y,h)\in\mathcal{T}_c})\right]\right|\]
tends to zero for $T\to\infty$ and step (i) is complete. \newpage

\textit{Step (ii).} We set $c = 2AL_*( L^*\|K\|_{L^2}^2)^{-1}\delta^2 $, which is smaller than $A$ for every $\delta<c':= \|K\|_{L^2}\sqrt{L^*/(2L_*)}$. Using the upper and lower bound of the invariant density from Lemma \ref{bound_invariant_density},
\[ \overline\sigma_b(y,h)^2 := \frac{\int_\R K_{y,h}(z)^2 q_b(z) dz}{\int_\R \1_{[-A,A]}(z) q_b(z)dz}  \leq \frac{ L^* h\int_\R K(z)^2 dz}{2AL_*}.\]
Consequently, $\overline\sigma_b(y,h) \geq \delta$ implies $h\geq c$ and
\begin{align*}
&\sup_{b\in\Sigma} \sup_{f\in\mathcal{F}_{BL}(\{(y,h)\in\mathcal{T}: \overline\sigma_b(y,h)>\delta\})} \left| \E\left[ f((Z_T(y,h))_{\{(y,h): \overline\sigma_b(y,h)> \delta\}}) \right.\right. \\
&\hspace{6.5cm} \left.\left.- f((Z_b(y,h))_{\{(y,h): \overline\sigma_b(y,h)> \delta\}})\right]\right|
\end{align*} 
converges to zero for $T\to\infty$ by step (i). By the reverse triangle inequality, for $a,b\in l_\infty(\mathcal{S})$ and $\delta>0$, the map $l_\infty(\mathcal{S}) \rightarrow \R$ with
\[ (x(s))_{s\in\mathcal{S}}\ \mapsto\ \sup_{s\in\mathcal{S}: a(s)>\delta}\left| \frac{x(s)}{a(s)} - b(s)\right| \]
is Lipschitz continuous with respect to the supremum metric $\|\cdot\|_\mathcal{S}$ on $l_\infty(\mathcal{S})$. Hence, by Lemma \ref{uniform_continuous_mapping} it follows that 
\begin{align}\label{eqp: C15}
\sup_{\substack{(y,h)\in\mathcal{T}: \\ \delta<\overline\sigma_b(y,h)\leq 1}} \left(\frac{\left| Z_T(y,h)\right|}{\sigma_b(y,h)} - \Upsilon\bigg( \frac{\sigma_b(y,h)^2}{\sigma_{b,\max}^2}\bigg)\right) 
\end{align} 
converges uniformly to $S_b(\delta,1)$. To establish the weak convergence of $T_T(\delta, 1)$ to $ S_b(\delta, 1)$ for any $0<\delta<c'$ from this finding, we apply Lemma \ref{lemma_weak_2} to $Z_T$ and $Z_b$ with the specifications
\[ \alpha_T(y,h) = \hat\sigma_T(y,h) \quad \textrm{ and }\quad \alpha(y,h) =\sigma_b(y,h), \]
and
\[ \beta_T(y,h) = \Upsilon\left(\hat\sigma_T(y,h)^2/\hat\sigma_{T,\max}^2\right) \ \ \textrm{ and }\ \ \beta(y,h) =\Upsilon\left(\sigma_b(y,h)^2/\sigma_{b,\max}^2\right). \]
Its first condition (a) is true by the convergence of the term in \eqref{eqp: C15} to $S_b(\delta,1)$. Uniform tightness of $S_b(\delta,1)$ follows from Theorem \ref{Multiscale_Limit} by noting that within its proof the capacity numbers can be bounded similarly as in \eqref{eqp: B8} with a bound that is independent of $b\in\Sigma(C,A,\gamma,\sigma)$. Moreover, tightness of $\sup_{(y,h)\in\mathcal{T}: \overline\sigma_b(y,h)>\delta}\Upsilon(\overline\sigma_b(y,h)^2)$ is clear, as $\Upsilon(\overline\sigma_b(y,h)^2)$ is bounded by $\Upsilon(\delta^2)$ for $\overline\sigma_b(y,h)>\delta$, because $\Upsilon(\cdot)$ is decreasing on $[0,1]$. In particular, assumption (b) of Lemma~\ref{lemma_weak_2} holds and it remains to establish assumption (c) and (d) of it, i.e. to prove that for any $\epsilon>0$,
\begin{align}\label{eqp: C7}
\sup_{b\in\Sigma} \Pr_b\left( \sup_{\overline\sigma_b(y,h)>\delta} \left| \frac{\sigma_b(y,h)}{\hat\sigma_T(y,h)} - 1\right|>\epsilon\right) \stackrel{T\to\infty}{\longrightarrow}\ 0
\end{align}
and 
\begin{align}\label{eqp: C8}
\sup_{b\in\Sigma}\Pr_b\left( \sup_{\overline\sigma_b(y,h)\geq\delta}\left| \Upsilon\left( \frac{\hat\sigma_T(y,h)^2}{\hat\sigma_{T,\max}^2}\right) - \Upsilon\left( \frac{\sigma_b(y,h)^2}{ \sigma_{b,\max}^2}\right) \right|>\epsilon\right) \stackrel{T\to\infty}{\longrightarrow}\ 0.
\end{align}
On the set $\mathcal{C}_{T,b}$ defined in \eqref{eqp: C2} we have $\frac{1}{\sigma^2 T}L_T^\cdot\geq \frac12 q_b$ and consequently
\begin{align*}
\sqrt{\frac1T \int_0^T K_{y,h}(X_s)^2 ds } = \sqrt{\int_\R K_{y,h}(z)^2 \frac{1}{\sigma^2 T}L_T^z(X) dz} \geq 2^{-\frac12}\|K_{y,h}\sqrt{q_b}\|_{L^2},
\end{align*} 
or equivalently $\hat\sigma_T(y,h) \geq 2^{-\frac12} \sigma_b(y,h)$. From $\overline\sigma_b(y,h)>\delta$, it follows that $\sigma_b(y,h)^2>\sigma_{b,\max}^2\delta^2 \geq 2AL_*\delta^2$. Thus, we get on $\mathcal{C}_{T,b}$, applying a first order Taylor expansion of $\sqrt{x}$,
\begin{align*}
&\sup_{\overline\sigma_b(y,h)>\delta} \left| \frac{\|K_{y,h}\sqrt{\rho_{b_0}}\|_{L^2}}{\sqrt{\frac1T\int_0^T K_{y,h}(X_s)^2ds}} - 1\right|\\
&\hspace{3cm}= \sup_{\overline\sigma_b(y,h)>\delta}\frac{1}{\hat\sigma_T(y,h)}\left| \sigma_b(y,h)- \hat\sigma_T(y,h)\right| \\
&\hspace{3cm} \leq \frac{1}{2AL_*\delta^2}\cdot \sup_{\overline\sigma_b(y,h)>\delta}\left| \hat\sigma_T(y,h)^2 - \sigma_b(y,h)^2\right|,
\end{align*}
where we used $1\geq \int_{-A}^A q_b(z) dz \geq 2AL_*$. This term converges to zero uniformly in probability by Proposition~\ref{uniform_ergodic}. Together with the fact that $\sup_{b\in\Sigma}\Pr_b(\mathcal{C}_{T,b}^c)\rightarrow 0$ by \eqref{eq:conv_emprical_density_uniform}, we have for any $\epsilon>0$ that
\begin{align*}
&\sup_{b\in\Sigma} \Pr_b\left( \sup_{\overline\sigma_b(y,h)>\delta} \left| \frac{\sigma_b(y,h)}{\hat\sigma_T(y,h)} - 1\right|>\epsilon\right)\\
&\hspace{0.3cm}\leq \sup_{b\in\Sigma} \Pr_b\left( \left\{\sup_{\overline\sigma_b(y,h)>\delta} \left| \frac{\sigma_b(y,h)}{\hat\sigma_T(y,h)} - 1\right|>\epsilon\right\} \cap\mathcal{C}_{T,b}\right) +  \sup_{b\in\Sigma}\Pr_b\left(\mathcal{C}_{T,b}^c\right)
\end{align*} 
converges to zero for $T\to\infty$ and \eqref{eqp: C7} holds.\\
To verify \eqref{eqp: C8}, we use that for $\epsilon>0$ and $\delta'>\delta>0$,
\begin{align}\label{eqp: C3}
\begin{split}
&\Pr_b\left( \sup_{\overline\sigma_b(y,h)\geq\delta}\left| \Upsilon\left( \overline\sigma_T(y,h)^2\right)- \Upsilon\left( \overline\sigma_b(y,h)^2\right) \right|>\epsilon\right)\\
&\hspace{1cm} \leq \Pr_b\left( \sup_{\delta<\overline\sigma_b(y,h)\leq\delta'}\left| \Upsilon\left( \overline\sigma_T(y,h)^2\right) - \Upsilon\left( \overline\sigma_b(y,h)^2\right)\right| >\epsilon\right) \\
&\hspace{2cm}+ \Pr_b\left( \sup_{\overline\sigma_b(y,h)\geq\delta'}\left| \Upsilon\left( \overline\sigma_T(y,h)^2\right) - \Upsilon\left( \overline\sigma_b(y,h)^2\right)\right| >\epsilon\right).
\end{split}
\end{align}
Next, we apply a Taylor expansion for $\Upsilon(r)=\sqrt{2\log(1/r)}$ with derivative $\Upsilon'(r)=-(r\Upsilon(r))^{-1}$.
The function $r\Upsilon(r)$ has its zeros in the endpoints of the interval $[0,1]$, but the absolute value of the derivative is bounded on the compact interval $[\delta, \delta']$ by some constant $c(\delta, \delta')$. Hence, 
\begin{align*}
&\sup_{\delta<\overline\sigma_b(y,h)\leq\delta'}\left| \Upsilon\left( \overline\sigma_T(y,h)^2\right)- \Upsilon\left( \overline\sigma_b(y,h)^2\right) \right| \\
&\hspace{1.5cm}\leq c(\delta, \delta') \cdot \sup_{\delta<\sigma_b(y,h)\leq\delta'} \left| \overline\sigma_T(y,h)^2 - \overline\sigma_b(y,h)^2\right|,
\end{align*}
which converges to zero in probability, uniformly in $\Sigma$, by Lemma~\ref{lemma_unif_sigma}. Consequently, the first summand of the right-hand side of \eqref{eqp: C3} converges to zero, uniformly in $b$, for $T\to\infty$ and we now treat the second one. As $\Upsilon(\cdot)$ is monotonously decreasing on $[0,1]$, for any $\epsilon'>0$,
\begin{align*}
&\Pr_b\left(\sup_{\overline\sigma_b(y,h)\geq \delta'}\left| \Upsilon\left( \overline\sigma_T(y,h)^2\right)- \Upsilon\left( \overline\sigma_b(y, h)^2\right) \right| >\epsilon\right) \\
&\hspace{0.3cm} \leq \Pr_b\left( \sup_{\overline\sigma_b(y,h)\geq \delta'} \Upsilon\left( \overline\sigma_T(y,h)^2\right)+ \Upsilon\left( (\delta')^2 \right) >\epsilon\right)\\
&\hspace{0.3cm} \leq \Pr_b\left( \sup_{\overline\sigma_b(y,h)\geq \delta'} \Upsilon\left( \overline\sigma_b(y,h)^2 + (\overline\sigma_T(y,h)^2 - \overline\sigma_b(y,h)^2)\right)\right. \\
&\hspace{2cm} \left. \textcolor{white}{\frac12} + \Upsilon\left( (\delta')^2 \right) >\epsilon, \|\overline\sigma_T^2 -\overline\sigma_b^2\|_\mathcal{T}<\epsilon'\right)  +\Pr_b\left( \|\overline\sigma_T^2 -\overline\sigma_b^2\|_\mathcal{T}\geq \epsilon'\right)\\
&\hspace{0.3cm} \leq \Pr_b\left( \Upsilon\left( (\delta')^2 - \epsilon'\right) + \Upsilon\left( (\delta')^2 \right) >\epsilon\right) + \Pr_b\left( \|\overline\sigma_T^2 -\overline\sigma_b^2\|_\mathcal{T}\geq \epsilon'\right).
\end{align*}
The latter summand vanishes  uniformly over $\Sigma$ for $T\to\infty$ by Lemma \ref{lemma_unif_sigma}, the first one is neither depending on $T$ nor probabilistic. Letting $\epsilon'\searrow 0$ and $\delta'\nearrow 1$, it converges to zero because $\Upsilon(r)\stackrel{r\to 1}{\longrightarrow} 0$, and we have shown \eqref{eqp: C8}. In summary,  by Lemma~\ref{lemma_weak_2} we have proven so far that
\[ \sup_{b\in\Sigma(C,A,\gamma,\sigma)} d_{BL}^b\left( T_T(\delta,1), S_b(\delta,1)\right) \stackrel{T\to\infty}{\longrightarrow} 0.\]

\textit{Step (iii).} To improve the result from (ii) to uniform weak convergence of $T_T(0,1)$ to $S_b(0,1)$, we split 
\begin{align}\label{eqp: C12}
\begin{split}
&\sup_{b\in\Sigma(C,A,\gamma,\sigma)} d_{BL}^b\left( T_T(0,1), S_b(0,1)\right)\\
&\hspace{1.5cm} = \sup_{b\in\Sigma} \sup_{f\in\mathcal{F}_{BL}(\R)} \left|\E_\theta\left[ f(T_T(0,1)) - f(S_b(0,1))\right]\right|\\
&\hspace{1.5cm} \leq \sup_{b\in\Sigma} \sup_{f\in\mathcal{F}_{BL}(\R)} \E_\theta\left[\left| f(T_T(0,1)) - f(T_T(\delta,1))\right|\right]\\
&\hspace{2.5cm} + \sup_{b\in\Sigma} \sup_{f\in\mathcal{F}_{BL}(\R)} \left|\E_\theta\left[ f(T_T(\delta,1)) - f(S_b(\delta,1))\right]\right| \\
&\hspace{2.5cm} + \sup_{b\in\Sigma} \sup_{f\in\mathcal{F}_{BL}(\R)} \E_\theta\left[\left| f(S_b(\delta,1)) - f(S_b(0,1))\right|\right].
\end{split}
\end{align}
The second summand equals $ \sup_{b\in\Sigma(C,A,\gamma,\sigma)} d_{BL}^b\left( T_T(\delta,1), S_b(\delta,1)\right)$ and we have seen in step (ii) that it vanishes for $T\to\infty$. The first summand is bounded from above by 
\begin{align*}
&\sup_{b\in\Sigma} \sup_{f\in\mathcal{F}_{BL}(\R)} \E_\theta\left[\left| f(T_T(0,1)) - f(T_T(\delta,1))\right| \1_{\{ T_T(0,1) - T_T(\delta,1) >\epsilon\}}\right] \\
&\hspace{1cm} + \sup_{b\in\Sigma} \sup_{f\in\mathcal{F}_{BL}(\R)} \E_\theta\left[\left| f(T_T(0,1)) - f(T_T(\delta,1))\right| \1_{\{ T_T(0,1) - T_T(\delta,1) \leq \epsilon\}}\right] \\
&\hspace{0.5cm}\leq \epsilon + 2\sup_{b\in\Sigma} \Pr_b\left( T_T(0,1) - T_T(\delta,1) >\epsilon\right).
\end{align*}
Writing $T_T(\delta,\delta') = \sup_{\delta<\overline\sigma(y,h)\leq \delta'} A_T(y,h)$, we can proceed with
\begin{align*}
&\sup_{b\in\Sigma} \Pr_b\left( T_T(0,1) - T_T(\delta,1) >\epsilon\right)\\
&\hspace{0.2cm} = \sup_{b\in\Sigma} \Pr_b\left( \max\left\{ \sup_{0< \overline\sigma(y,h)\leq \delta} A_T(y,h), \sup_{\delta<\overline\sigma(y,h)\leq 1} A_T(y,h)\right\} \right. \\
&\hspace{6.5cm} \left. - \sup_{\delta<\overline\sigma(y,h)\leq 1} A_T(y,h)>\epsilon\right)\\
&\hspace{0.2cm} = \sup_{b\in\Sigma} \Pr_b\left( \max\left\{ \sup_{0< \overline\sigma(y,h)\leq \delta} A_T(y,h)- \sup_{\delta<\overline\sigma(y,h)\leq 1} A_T(y,h), 0\right\} >\epsilon\right)\\
&\hspace{0.2cm} \leq \sup_{b\in\Sigma} \Pr_b\left( \max\{\dots\} >\epsilon, \sup_{\delta<\overline\sigma(y,h)\leq 1} A_T(y,h) \geq -\frac{\epsilon}{2}\right) \\
&\hspace{2cm} + \sup_{b\in\Sigma} \Pr_b\left( \max\{\dots\} >\epsilon, \sup_{\delta<\overline\sigma(y,h)\leq 1} A_T(y,h) < -\frac{\epsilon}{2}\right)\\
&\hspace{0.2cm} \leq  \sup_{b\in\Sigma} \Pr_b\left( \sup_{0<\overline\sigma(y,h)\leq \delta} A_T(y,h) >\frac{\epsilon}{2}\right) + \sup_{b\in\Sigma} \Pr_b\left( \sup_{\delta<\overline\sigma(y,h)\leq 1} A_T(y,h) <-\frac{\epsilon}{2}\right).
\end{align*}
We start by treating the first summand and show that for all $\epsilon>0$,
\begin{align}\label{eqp: C4}
\lim_{\delta\searrow 0} \limsup_{T\to\infty}\sup_{b\in\Sigma} \Pr_b\left( T_T(0,\delta) \geq \epsilon\right) =0.
\end{align} 
This is based on Theorem~\ref{Multiscale_general}. First of all, 
\begin{align*}
\Pr_{b}\left( T_T(0,\delta) \geq \epsilon\right) &= \Pr_{b}\left( T_T(0,\delta) \geq \epsilon, \|\overline\sigma_T^2 - \overline\sigma_b^2 \|_{\mathcal{T}}\leq\delta \right) \\
&\hspace{1.5cm} +\Pr_{b}\left( T_T(0,\delta) \geq \epsilon , \|\overline\sigma_T^2 - \overline\sigma_b^2 \|_{\mathcal{T}}>\delta \right),
\end{align*} 
where the second summand is bounded by $\Pr_{b}(\|\overline\sigma_T^2 - \overline\sigma_b^2 \|_{\mathcal{T}}>\delta)$ and vanishes uniformly in $b$ for $T\to\infty$ as shown in Lemma~\ref{lemma_unif_sigma}. For the other one, we have
\begin{align*}
&\Pr_{b}\left( T_T(0,\delta) \geq \epsilon, \|\overline\sigma_T^2 - \overline\sigma_b^2 \|_{\mathcal{T}}\leq\delta^2 \right) \\
&\hspace{1cm}= \Pr_{b}\left( \sup_{\overline\sigma_b(y,h)^2 \leq \delta^2} A_T(y,h) >\epsilon, \|\overline\sigma_T^2 - \overline\sigma_b^2 \|_{\mathcal{T}}\leq\delta^2 \right)\\
&\hspace{1cm}\leq  \Pr_{b}\left( \sup_{\overline\sigma_T(y,h)^2 \leq 2\delta^2} \frac{A_T(y,h)}{D(\overline\sigma_T(y,h))}\cdot D(\overline\sigma_T(y,h)) >\epsilon\right)\\
&\hspace{1cm} \leq \Pr_{b}\left( \sup_{\overline\sigma_T(y,h) \leq \sqrt{2}\delta} \frac{A_T(y,h)}{D(\overline\sigma_T(y,h))} \cdot \sup_{\overline\sigma_T(y,h) \leq \sqrt{2}\delta} D(\overline\sigma_T(y,h)) >\epsilon\right)\\
&\hspace{1cm} \leq \Pr_{b}\left( L \cdot \sup_{\overline\sigma_T(y,h) \leq \sqrt{2}\delta} D(\overline\sigma_T(y,h)) >\epsilon\right) \\
&\hspace{3cm}+ \Pr_{b}\left(\sup_{\overline\sigma_T(y,h) \leq \sqrt{2}\delta} \frac{A_T(y,h)}{D(\overline\sigma_T(y,h))} >L\right)
\end{align*}
for any constant $L>0$. Now, the first summand here tends to zero for $\delta\searrow 0$ as $D(r)\searrow 0$ monotonously for $r\searrow 0$. The random variable in the second probability is uniformly asymptotically tight by Theorem~\ref{Multiscale_general} and thus vanishes in the limit $\lim_{L\to\infty}\limsup_{T\to\infty}\sup_b$. Note that the applicability of Theorem~\ref{Multiscale_general} was checked in the proof of Theorem~\ref{Multiscale_Lemma}. For the uniformity mentioned before, note that all bounds used in the proof of Theorem~\ref{Multiscale_Lemma}, in particular the last one of \eqref{eqp: B8}, are independent of $b$.\\
It remains to deal for all $\epsilon>0$ with the term
\[ \lim_{\delta\searrow 0}\limsup_{T\to\infty}\sup_{b\in\Sigma} \Pr_b\left( T_T(\delta,1) <-\frac{\epsilon}{2}\right). \]
Define the Lipschitz function
\[ g_{\epsilon}(y) := \begin{cases}
1 &\textrm{ if } x\leq -\frac{\epsilon}{2},\\
1-4\frac{y+\frac{\epsilon}{2}}{\epsilon} &\textrm{ if } -\frac{\epsilon}{2}<y<-\frac{\epsilon}{4},\\
0 &\textrm{ if } x\geq -\frac{\epsilon}{4}.
\end{cases}\]
Then we have
\[  \1_{\left\{x < -\frac{\epsilon}{2}\right\}} \leq g_\epsilon(x) \leq \1_{\left\{x < -\frac{\epsilon}{4}\right\}} \]
and consequently,
\begin{align*}
&\lim_{\delta\searrow 0}\limsup_{T\to\infty}\sup_{b\in\Sigma} \Pr_b\left( T_T(\delta,1) <-\frac{\epsilon}{2}\right) \\
&\hspace{0.3cm} \leq \lim_{\delta\searrow 0}\limsup_{T\to\infty}\sup_{b\in\Sigma} \E_b\left[ g_\epsilon(T_T(\delta, 1))\right] \\
&\hspace{0.3cm} = \lim_{\delta\searrow 0}\sup_{b\in\Sigma} \E_b\left[ g_\epsilon(S_b(\delta, 1))\right] \\
&\hspace{1cm} + \lim_{\delta\searrow 0}\limsup_{T\to\infty}\sup_{b\in\Sigma} \E_b\left[ g_\epsilon(T_T(\delta, 1))\right] -\lim_{\delta\searrow 0}\limsup_{T\to\infty}\sup_{b\in\Sigma} \E_b\left[ g_\epsilon(S_b(\delta, 1))\right].
\end{align*}
By the already established uniform weak convergence of $T_T(\delta, 1)$ to $S_b(\delta, 1)$ and the fact that $g_\epsilon$ is a bounded Lipschitz function, we get for the difference
\begin{align*}
&\lim_{\delta\searrow 0}\limsup_{T\to\infty}\sup_{b\in\Sigma} \E_b\left[ g_\epsilon(T_T(\delta, 1))\right] -\lim_{\delta\searrow 0}\limsup_{T\to\infty}\sup_{b\in\Sigma} \E_b\left[ g_\epsilon(S_b(\delta, 1))\right] \\
& \hspace{1cm}\leq \lim_{\delta\searrow 0} \limsup_{T\to\infty} \sup_{b\in\Sigma} \left| \E_b\left[ g_\epsilon(T_T(\delta, 1))\right]  - \E_b\left[ g_\epsilon(S_b(\delta, 1))\right] \right| = 0
\end{align*}
and we have
\begin{align*}
\lim_{\delta\searrow 0}\limsup_{T\to\infty}\sup_{b\in\Sigma} \Pr_b\left( T_T(\delta,1) <-\frac{\epsilon}{2}\right) &\leq \lim_{\delta\searrow 0}\sup_{b\in\Sigma} \E_b\left[ g_\epsilon(S_b(\delta, 1))\right]  \\
&\leq \lim_{\delta\searrow 0}\sup_{b\in\Sigma} \Pr_b\left( S_b(\delta,1) <-\frac{\epsilon}{4}\right).
\end{align*}
Set $\tilde{\delta} := 4\delta^2 (2AL^*) (L_*\|K\|_{L^2}^2)^{-1}$
and define for $-\frac{A}{\tilde\delta}\leq k\leq \frac{A}{\tilde\delta}$,
\[Z_k(b) := \|K_{k\tilde\delta, \frac{\tilde\delta}{2}}\sqrt{\rho_b}\|_{L^2}^{-1} \int_\R K_{k\tilde\delta, \frac{\tilde\delta}{2}} (z)\sqrt{\rho_b(z)} dW_z.\]
Those random variables are independent, as the support of the functions $K_{k\tilde\delta, \tilde\delta/2}$ and $K_{k'\tilde\delta, \tilde\delta/2}$ is disjoint for $k\neq k'$. Moreover, they are standard Gaussians, in particular their distribution does not depend on $b$, and hence their maximum is $(2\log (2A\tilde{\delta}^{-1}))^\frac12 + o_{\Pr}(1)$ for $\delta\searrow 0$ (see \cite{Leadbetter}, Theorem~$1.5.3$). Using the uniform bound on the invariant density from Lemma~\ref{bound_invariant_density},
\[ \frac{\|K_{k\tilde\delta, \frac{\tilde\delta}{2}}\sqrt{\rho_b}\|_{L^2}^2}{\| \1_{[-A,A]}\sqrt{q_b}\|_{L^2}^2} \geq \frac{L_*\|K\|_{L^2}^2}{2AL^*}\frac{\tilde\delta}{2} =2\delta^2. \]
As $\Upsilon(\cdot)$ is monotonously decreasing on $[0,1]$,
\begin{align*}
S_b(\delta,1)&\geq \max_{-A\tilde\delta^{-1}\leq k\leq A\tilde\delta^{-1}}\left( Z_k(b) - \Upsilon\left( \frac{\|K_{k\tilde\delta, \frac{\tilde\delta}{2}}\sqrt{\rho_b}\|_{L^2}^2}{\| \1_{[-A,A]}\sqrt{q_b}\|_{L^2}^2}\right)\right)\\
&\geq \max_{-A\tilde\delta^{-1}\leq k\leq A\tilde\delta^{-1}} Z_k(b) - \Upsilon(2\delta^2)\\
& = \sqrt{2\log( 2A\tilde\delta^{-1})} -\sqrt{2\log ((2\delta^2)^{-1})} + o_\Pr(1)\\
&= o_\delta(1) + o_\Pr(1)\quad \textrm{ for } \delta\searrow 0,
\end{align*}
Thus, we finally get
\begin{align}\label{eqp: C14}
\begin{split}
&\lim_{\delta\searrow 0}\sup_{b\in\Sigma} \Pr_b\left( S_b(\delta,1) <-\frac{\epsilon}{4}\right)\\
& \hspace{0.3cm} \leq \lim_{\delta\searrow 0}\sup_{b\in\Sigma} \Pr_b\left( \max_{-A\tilde\delta^{-1}\leq k\leq A\tilde\delta^{-1}}\left( Z_k(b) - \Upsilon\left( \frac{\|K_{k\tilde\delta, \frac{\tilde\delta}{2}}\sqrt{\rho_b}\|_{L^2}^2}{\| \1_{[-A,A]}\sqrt{q_b}\|_{L^2}^2}\right)\right) <-\frac{\epsilon}{4}\right)\\
&\hspace{0.3cm}\leq \lim_{\delta\searrow 0} \sup_{b\in\Sigma}\Pr\left( \max_{-A\tilde\delta^{-1}\leq k\leq A\tilde\delta^{-1}} Z_k(b) - \Upsilon(2\delta^2) <-\frac{\epsilon}{4}\right) =0
\end{split}
\end{align}
and the convergence of the first summand on the right-hand side in \eqref{eqp: C12} to zero for $\lim_{\delta\searrow 0}\limsup_{T\to\infty}$ is shown. It remains to study the third summand in \eqref{eqp: C12}, i.e. 
\[ \sup_{b\in\Sigma} \sup_{f\in\mathcal{F}_{BL}} \E_\theta\left[\left| f(S_b(\delta,1)) - f(S_b(0,1))\right|\right].\]
This is done in the same way as for the first summand and is only easier, as the $\limsup_{T\to\infty}$ is omitted and the argument reduces to additionally showing that for any $\epsilon>0$,
\[ \lim_{\delta\searrow 0} \sup_{b\in\Sigma} \Pr_b\left( S_b(0,\delta) \geq \epsilon\right) =0.\]
This follows along the lines of the previous part as Theorem~\ref{Multiscale_general} is also applicable to $S_b$.
\end{proof}

\begin{remark}\label{remark_weak_Tc}
Let $c>0$ and $\mathcal{T}_c$ be the set given in \eqref{eqp: C9}. For any subset $\mathcal{T}'\subset \mathcal{T}_c$ define
\begin{align*}
& T_T(\mathcal{T}') := \sup_{(y,h)\in\mathcal{T}'} \left(\frac{\left|Z_T(y,h)\right|}{\hat\sigma_T(y,h)} - \Upsilon\bigg( \frac{\hat\sigma_T(y,h)^2}{\hat\sigma_{T,\max}^2}\bigg) \right) 
\end{align*} 
as well as
\begin{align*}
&S_b(\mathcal{T}') := \sup_{(y,h)\in\mathcal{T}'} \left(\frac{\left|Z_b(y,h)\right|}{\sigma_b(y,h) } - \Upsilon\bigg( \frac{\sigma_b(y,h)^2}{\sigma_{b,\max}^2}\bigg) \right).
\end{align*} 
Then our proof of Theorem~\ref{weak_conv} even shows that
\[ \sup_{b\in\Sigma(C,A,\gamma,\sigma)} d_{BL}^b \left( T_T(\mathcal{T}'), S_b(\mathcal{T}')\right)\stackrel{T\to\infty}{\longrightarrow} 0.\]
This can be seen as follows: Step (i) directly transfers and in step (ii) we have to show \eqref{eqp: C7} and \eqref{eqp: C8} for the supremum over $\mathcal{T}'$, which is straightforward by the same lines. Step (iii) can be omitted since $\mathcal{T}'\subset\mathcal{T}_c$.
\end{remark}

\section{Proofs of Section~\ref{Sec_similarity}}\label{App_least}

In Subsection~\ref{SubSecD1} the proof of Theorem~\ref{worst_case_delta} is given. This proof is sketched in Section~\ref{Sec_weak} and strongly relies on the uniform weak convergence result for the test statistic $T_T^{b}$ given in  Theorem~\ref{weak_conv} that was proven in Subsection~\ref{SubSec_C2}. In the subsequent Subsection~\ref{SubSecD2} we present a proof of Remark~\ref{remark_Piterbarg}, i.e. we prove that the distribution of the dominating random variable $U_1\vee U_2 + 4\sqrt{A\eta/\sigma^2}$ from Theorem~\ref{worst_case_delta} has no point mass.

\subsection{Proof of Theorem \ref{worst_case_delta}}\label{SubSecD1}

Let $b\in H_0(b_0,\eta)$. For ease of notation we abbreviate $H_0(b_0,\eta)$ by $H_0$ in what follows. As already indicated in Section~\ref{Sec_weak}, the key idea of the proof is to split the set $\mathcal{T}$ in three parts (two of them depending on $b$). For a rigorous description of those, we choose $\epsilon'>\epsilon>0$ and define the set\vspace{-0.05cm}
\[ U_{\epsilon,y,h}(b) := \left\{ x\in [y-h,y+h]: q_b(x)\left( |b(x)-b_0(x)|-\eta\right) < -\epsilon\right\}\]
and the partition \vspace{-0.05cm}
\begin{align*}
\mathcal{T}_1 &:= \left\{ (y,h)\in\mathcal{T} \mid h <\epsilon'\right\}, \\ 
\mathcal{T}_2 &:= \left\{ (y,h)\in\mathcal{T} \mid h\geq\epsilon'\textrm{ and } \lambda\left( U_{\epsilon,y,h}(b)\right)\geq \epsilon\right\}, \\
\mathcal{T}_3 &:= \left\{ (y,h)\in\mathcal{T} \mid h\geq\epsilon'\textrm{ and } \lambda\left( U_{\epsilon,y,h}(b)\right)< \epsilon \right\}.
\end{align*}
Note the dependences $\mathcal{T}_1=\mathcal{T}_1(\epsilon')$, $\mathcal{T}_2=\mathcal{T}_2(\epsilon,\epsilon',b)$ and $\mathcal{T}_3=\mathcal{T}_3(\epsilon,\epsilon',b)$, which most times will be suppressed in the notation to allow for shorter displays. With these definitions and the union bound,
\begin{align}\label{eqp: D2}
\begin{split}
&\Pr_b(T_T^\eta(X) \geq r)\\
&\hspace{0.2cm}= \Pr_b\left( \sup_{(y,h)\in\mathcal{T}}  \left(  |\Psi_{T,y,h}^{b_0}(X)| -\Lambda_{T,y,h}^\eta(X) - \Upsilon\bigg(\frac{\hat\sigma_T(y,h)^2}{\hat\sigma_{T,\max}^2}\bigg)  \right) \geq r \right) \\
&\hspace{0.2cm}\leq \sum_{i=1}^3\Pr_b\left( \sup_{(y,h)\in\mathcal{T}_i}  \left(   |\Psi_{T,y,h}^{b_0}(X)|-\Lambda_{T,y,h}^\eta(X)  - \Upsilon\bigg(\frac{\hat\sigma_T(y,h)^2}{\hat\sigma_{T,\max}^2}\bigg)  \right) \geq r \right).
\end{split}
\end{align}
Of course, the inequality remains valid if we put 
\[ \limsup_{\epsilon'\searrow 0}\limsup_{\epsilon\searrow 0}\limsup_{T\to\infty}\sup_{b\in H_0(b_0,\eta)}\]
before both sides. The left-hand side of \eqref{eqp: D2} does not depend on $\epsilon,\epsilon'$ and $\limsup_T \sup_b \Pr_b(T_T^\eta(X)\geq r)$ equals the left-hand side of the statement of the theorem. For the rest of the proof, we define the set 
\[ \mathcal{C}_{T,b} := \left\{ \left\| \frac{1}{\sigma^2 T}L_T^\cdot(X) - q_b\right\|_{[-A,A]} \leq \frac12 L_*\right\}.\]
Note that on $\mathcal{C}_{T,b}$ we have 
\[ \frac12 q_b(z) \leq \frac{1}{\sigma^2 T} L_T^z (X) \leq \frac32 q_b(z), \quad z\in [-A,A], \]
for each drift $b\in\Sigma(C,A,\gamma,\sigma)$ and $\lim_{T\to\infty} \sup_{b\in\Sigma(C,A,\gamma,\sigma)}\Pr_b( \mathcal{C}_{T,b}^c) =0$ holds by \eqref{eq:conv_emprical_density_uniform}.
Note that on $\mathcal{C}_{T,b}$, the term $|\Psi_{T,y,h}^{b_0}(X)| - \Lambda_{T,y,h}^\eta(X) $ is bounded from above by
\begin{align}\label{eqp: D1}
\begin{split}
&\left|\frac{\frac{1}{\sqrt{T}} \int_0^T K_{y,h}(X_s) dW_s}{\sqrt{\frac1T \int_0^T K_{y,h}(X_s)^2 ds}} \right| \\
&\hspace{1cm} + \frac{\left|\frac{1}{\sqrt{T}}\int_0^T K_{y,h}(X_s) ( b(X_s) - b_0(X_s)) ds\right|-\frac{\eta}{\sqrt{T}} \int_0^T K_{y,h}(X_s) ds}{\sigma \sqrt{\frac1T \int_0^T K_{y,h}(X_s)^2 ds}}. 
\end{split}
\end{align} 
Using the triangle inequality, the numerator of the second summand is smaller than
\[ \frac{1}{\sqrt{T}} \int_0^T K_{y,h}(X_s) \left( |b(X_s) -b_0(X_s)| -\eta\right) ds\]
and this term is smaller than or equal to zero for each $b\in H_0(b_0,\eta)$. Consequently, the whole second summand in \eqref{eqp: D1} is non-positive. Here, we used the assumption $K\geq 0$. With this preliminary observation, we give the heuristic about the decomposition of $\mathcal{T}$ into $\mathcal{T}_i$, $i=1,2,3$, and the route of the rest of the proof:
\begin{itemize}
\item The probability in \eqref{eqp: D2} with $\mathcal{T}_1$ is neglectable in the limit $\lim_{\epsilon'\searrow 0}$ which can be seen using the same techniques from empirical process theory as for the proof of Theorem~\ref{Multiscale_Lemma}.
\item For $(y,h)\in\mathcal{T}_2$, $|b-b_0|-\eta$ is sufficiently bounded away from zero on a subinterval of $[y-h,y+h]$ such that the second summand in \eqref{eqp: D1} tends to $-\infty$, whereas the first one remains stochastically bounded, uniformly in $b$. Hence, the corresponding probability in \eqref{eqp: D2} tends to zero.
\item The probability with $\mathcal{T}_3$ in \eqref{eqp: D2} does not vanish in the limit and is estimated against $\Pr( U_1\vee U_2+ 4\sqrt{A\eta/\sigma^2}\geq r)$. In particular, it is the origin of the correction term $4\sqrt{A\eta/\sigma^2}$.
\end{itemize}
From \eqref{eqp: C14} in step (iii) of the proof of Theorem~\ref{weak_conv} it can be seen that both random variables $U_1$ and $U_2$ are $\geq 0$ almost surely and consequently $U_1\vee U_2 + 4\sqrt{A\eta/\sigma^2}> 0$ for $\eta>0$. Therefore, it suffices to consider $r >0$ in the following, where we will estimate each summand on the right-hand side of \eqref{eqp: D2} separately.\\

\textit{Step (i).} By neglecting the non-positive second summand in \eqref{eqp: D1}, for $b\in H_0$ we may estimate 
\begin{align*}
&\Pr_b\left( \sup_{(y,h)\in\mathcal{T}_1} \left( |\Psi_{T,y,h}^{b_0}(X)| -\Lambda_{T,y,h}^\eta(X) - \Upsilon\bigg(\frac{\hat\sigma_T(y,h)^2}{\hat\sigma_{T,\max}^2}\bigg)\right)   \geq r\right)\\
&\hspace{1cm}\leq \Pr_b\left( \sup_{(y,h)\in\mathcal{T}_1} \left( \left|\frac{\frac{1}{\sqrt{T}} \int_0^T K_{y,h}(X_s) dW_s}{\sqrt{\frac1T \int_0^T K_{y,h}(X_s)^2 ds}} \right| - \Upsilon\bigg(\frac{\hat\sigma_T(y,h)^2}{\hat\sigma_{T,\max}^2}\bigg)\right) \geq r,  \mathcal{C}_{T,b}\right)\\
&\hspace{2cm} + \Pr_b\left(\mathcal{C}_{T,b}^c\right).
\end{align*}
The second summand vanishes in the limit $T\to\infty$ uniformly over the drift $b$. For the first one, note that $h<\epsilon'$ implies
\[ \frac{\|K_{y,h}\sqrt{q_b}\|_{L^2}^2}{\|\1_{[-A,A]}\sqrt{q_b}\|_{L^2}^2} = \frac{\int_\R K_{y,h}(z)^2 q_b(z)dz}{\int_{-A}^A q_b(z) dz} \leq \frac{L^* h\|K\|_{L^2}^2}{2AL_*} < \frac{L^* \|K\|_{L^2}^2 \epsilon'}{2AL_*} =:c\epsilon' \]
by using the uniform upper and lower bound of the invariant density $q_b$ given in Lemma~\ref{bound_invariant_density}.
Then, with $T_T(\delta,\delta')$ given in \eqref{eqp: T_T} within the proof of Theorem~\ref{weak_conv}, 
\begin{align*}
&\limsup_{\epsilon'\searrow 0}\limsup_{\epsilon\searrow 0}\limsup_{T\to\infty}\sup_{b\in H_0} \Pr_b\left( \sup_{(y,h)\in\mathcal{T}_1} \left( \Psi_{T,y,h}^{b_0}(X) - \Lambda_{T,y,h}^\eta(X)\right.\right.\\
&\hspace{7cm}\left.\left. - \Upsilon\bigg(\frac{\hat\sigma_T(y,h)^2}{\hat\sigma_{T,\max}^2}\bigg)  \right) \geq r,  \mathcal{C}_{T,b}\right) \\
&\hspace{1cm}\leq \limsup_{\epsilon'\searrow 0} \limsup_{T\to\infty}\sup_{b\in H_0}\Pr\left( T_T\left(0,\sqrt{c\epsilon'}\right) \geq r\right) =0,
\end{align*} 
where we used that $\mathcal{T}_1$ does not depend on $\epsilon$ and the last equality was proven in step (iii) of the proof of Theorem~\ref{weak_conv}, see \eqref{eqp: C4}. \\

\textit{Step (ii).} With \eqref{eqp: D1}, we bound the supremum taken over $\mathcal{T}_2$  on $\mathcal{C}_{T,b}$ by
\begin{align*}
&\sup_{(y,h)\in\mathcal{T}_2}  \left(   |\Psi_{T,y,h}^{b_0}(X)|-\Lambda_{T,y,h}^\eta(X)  - \Upsilon\bigg(\frac{\hat\sigma_T(y,h)^2}{\hat\sigma_{T,\max}^2}\bigg)  \right)\\
&\hspace{1cm} \leq \sup_{(y,h)\in\mathcal{T}_2}\left( \left|\frac{\frac{1}{\sqrt{T}} \int_0^T K_{y,h}(X_s) dW_s}{\sqrt{\frac1T \int_0^T K_{y,h}(X_s)^2 ds}} \right| - \Upsilon\bigg(\frac{\hat\sigma_T(y,h)^2}{\hat\sigma_{T,\max}^2}\bigg)\right)\\
&\hspace{2cm} + \sup_{(y,h)\in\mathcal{T}_2} \frac{ \frac{1}{\sqrt{T}} \int_0^T K_{y,h}(X_s) \left( |b(X_s) -b_0(X_s)| -\eta\right) ds}{\sigma \sqrt{\frac1T \int_0^T K_{y,h}(X_s)^2 ds}}.
\end{align*} 
The limiting distribution of the first term is given by Theorem~\ref{weak_conv} together with Remark~\ref{remark_weak_Tc}, in particular it is asymptotically tight uniformly in $b\in H_0$. We are going to show convergence of the second term to $-\infty$ in probability uniformly in $b\in H_0$ on $\mathcal{C}_{T,b}$. To proceed, we first write with the occupation times formula,
\begin{align}\label{eqp: D10}
\begin{split}
&\frac{\frac{1}{\sqrt{T}} \int_0^T K_{y,h}(X_s)(|b(X_s)-b_0(X_s)|-\eta) ds}{\sigma \sqrt{\frac{1}{T} \int_0^T K_{y,h}(X_s)^2 ds}} \\
&\hspace{1cm}= \sqrt{T}\frac{\int_\R K_{y,h}(z) (|b(z)-b_0(z)|-\eta) \frac{1}{\sigma^2 T} L_T^z(X) dz}{\sigma\sqrt{\int_\R K_{y,h}(z)^2 \frac{1}{\sigma^2 T} L_T^z(X) dz}} \\
&\hspace{1cm}\leq \sqrt{T}\frac{\int_\R K_{y,h}(z) (|b(z)-b_0(z)|-\eta) \frac12 q_b(z) dz}{\sigma\sqrt{\int_\R K_{y,h}(z)^2 \cdot\frac32 q_b(z) dz}}
\end{split}
\end{align} 
In the third step we used that $|b-b_0|-\eta$ is non-positive and therefore inserted the lower bound of the empirical density in the nominator and the upper bound in the denominator. 
Keeping in mind that the last expression in \eqref{eqp: D10} is negative, bounding it from above is the same as finding lower bounds of the absolute value. In particular, we may choose the denominator as large as possible by using the upper bound $L^*$ of the invariant density from Lemma \ref{bound_invariant_density}. 
Thus, the right-hand side of \eqref{eqp: D10} is bounded from above by

\[ \frac{\sqrt{T}}{\sqrt{6 L^* h}\sigma\|K\|_{L^2}} \int_\R K_{y,h}(z)\left( |b(z)-b_0(z)|-\eta\right) q_b(z) dz.\]
From our condition on the set $\mathcal{T}_2$, we have
\begin{align*}
\int_\R K_{y,h}(z)\left( |b(z)-b_0(z)|-\eta\right) q_b(z) dz &\leq - \epsilon \int_{U_{\epsilon,y,h}} K_{y,h}(z) dz\\
&= -\epsilon h \int_{(U_{\epsilon,y,h}-y)/h} K(z) dz.
\end{align*} 
Here, for a set $A$ we denote $(A-y)/h:=\{(x-y)/h\mid x\in A\}$. We have $\lambda((U_{\epsilon,y,h}-y)/h)\geq \epsilon/h \geq \epsilon/A$. Hence, the last integral is bounded away from zero by a constant $C_0(K)>0$. This gives the bound
\[ \frac{\frac{1}{\sqrt{T}} \int_0^T K_{y,h}(X_s)(|b(X_s)-b_0(X_s)|-\eta) ds}{\sigma \sqrt{\frac{1}{T} \int_0^T K_{y,h}(X_s)^2 ds}} \leq -\frac{\sqrt{T\epsilon'}}{\sqrt{6 L^*}\sigma\|K\|_{L^2}} C_0\epsilon \]
where we used that $h\geq \epsilon'$ on $\mathcal{T}_2$. 
This bound does not depend on the location parameters $(y,h)$ and we can put the supremum before the expression on the left-hand side without changing the right-hand side. 
In particular, we have the desired uniform convergence in probability
\[ \sup_{(y,h)\in\mathcal{T}_2} \frac{ \frac{1}{\sqrt{T}} \int_0^T K_{y,h}(X_s) \left( |b(X_s) -b_0(X_s)| -\eta\right) ds}{\sigma \sqrt{\frac1T \int_0^T K_{y,h}(X_s)^2 ds}} \longrightarrow_{\Pr_b, \textrm{unif}}\  -\infty\]
for $T\to\infty$ on $\mathcal{C}_{T,b}$. Then we can finalize step (ii) by bounding
\begin{align*}
&\Pr_b \left(  \sup_{(y,h)\in\mathcal{T}_2}  \left( |\Psi_{T,y,h}^{b_0}(X)| - \Upsilon\bigg(\frac{\hat\sigma_T(y,h)^2}{\hat\sigma_{T,\max}^2}\bigg)-\Psi_T^\eta(y,h) \right) \geq r \right)  \\
&\hspace{0.cm} \leq \Pr_b \left( \left\{ \sup_{(y,h)\in\mathcal{T}_2}  \left( |\Psi_{T,y,h}^{b_0}(X)| - \Upsilon\bigg(\frac{\hat\sigma_T(y,h)^2}{\hat\sigma_{T,\max}^2}\bigg)-\Psi_T^\eta(y,h) \right) \geq r\right\}\cap\mathcal{C}_{T,b} \right)\\
&\hspace{1.5cm} + \Pr_b\left( \mathcal{C}_{T,b}^c\right)
\end{align*} 
and applying $\limsup_{T\to\infty}\sup_{b\in H_0}$ on both sides, which allows to conclude convergence to zero of the right-hand side, and consequently, of the left-hand side, too.\\

\textit{Step (iii).} In this case, we consider the test statistic with the supremum taken over $\mathcal{T}_3$ without the non-positive second summand in \eqref{eqp: D1}, i.e. we use
\begin{align*}
&\Pr_b\left( \sup_{(y,h)\in\mathcal{T}_3}  \left(   |\Psi_{T,y,h}^{b_0}(X)|-\Lambda_{T,y,h}^\eta(X)  - \Upsilon\bigg(\frac{\hat\sigma_T(y,h)^2}{\hat\sigma_{T,\max}^2}\bigg)  \right) \geq r \right) \\
&\hspace{0.7cm} \leq \Pr_b\left( \sup_{(y,h)\in\mathcal{T}_3}\left(\frac{\left| \frac{1}{\sqrt{T}}\int_0^T K_{y,h}(X_s) dW_s\right|}{\sqrt{\frac1T\int_0^T K_{y,h}(X_s)^2 ds}} - \Upsilon\bigg(\frac{\hat\sigma_T(y,h)^2}{\hat\sigma_{T,\max}^2}\bigg)\right) \geq r, \mathcal{C}_{T,b}\right)\\
&\hspace{2cm} +\Pr_b\left(\mathcal{C}_{T,b}^c\right).
\end{align*}
The second summand vanishes in the limit $T\to\infty$ uniformly in $b$. In the rest of the proof, we will find an upper bound for the first summand in the right-hand side of this estimate in the limit $\limsup_T \sup_b$. As a starting point, we make use of the weak limit result of Theorem~\ref{weak_conv} in combination with Remark~\ref{remark_weak_Tc} and get
\begin{align}\label{eqp: D3}
\begin{split}
&\limsup_{T\to\infty}\sup_{b\in H_0}\Pr_b\hspace{-0.1cm}\left(\hspace{-0.05cm} \sup_{(y,h)\in\mathcal{T}_3}\hspace{-0.05cm}\left(\frac{\left| \frac{1}{\sqrt{T}}\int_0^T K_{y,h}(X_s) dW_s\right|}{\sqrt{\frac1T\int_0^T K_{y,h}(X_s)^2 ds}}\hspace{-0.05cm} -\hspace{-0.05cm} \Upsilon\bigg(\frac{\hat\sigma_T(y,h)^2}{\hat\sigma_{T,\max}^2}\bigg)\hspace{-0.1cm}\right)\hspace{-0.05cm} \geq\hspace{-0.05cm} r, \mathcal{C}_{T,b}\right)\\
&\hspace{0.1cm} \leq \sup_{b\in H_0}\Pr_b\left( \sup_{(y,h)\in\mathcal{T}_3}\left( \frac{\left|\int_\R K_{y,h}(z) \sqrt{q_b(z)} dW_z\right|}{\| K_{y,h}\sqrt{q_b}\|_{L^2}} - \Upsilon\bigg(\frac{ \sigma_b(y,h)^2}{\sigma_{b,\max}^2}\bigg)\hspace{-0.05cm}\right) \hspace{-0.05cm}\geq r-\delta\right)\hspace{-0.05cm},
\end{split}
\end{align}
where $\delta>0$ is arbitrary. In what follows, we determine $\limsup_{\epsilon\searrow 0 }$ of the right-hand side. Remember that $\mathcal{T}_3 = \mathcal{T}_3(\epsilon, \epsilon',b)$, where this dependence is usually suppressed for smaller displays. The above inequality in \eqref{eqp: D3} can be seen as follows: Abbreviate both suprema by $A_T$ and $A$, respectively, and set for $\delta>0$ the function $g_\delta$ as
\begin{align*}
g_\delta(x) := \begin{cases}
0, &\textrm{ if } x\leq r-\delta,\\
\frac{x-r+\delta}{\delta}, &\textrm{ if } r-\delta < x <r,\\
1, &\textrm{ if } x\geq r.
\end{cases}
\end{align*}\pagebreak
Then we have,
\begin{align*}
&\limsup_{T\to\infty}\sup_{b\in H_0}\Pr_b\left( A_T\geq r, \mathcal{C}_{T,b}\right)\\
&\hspace{1.5cm} = \limsup_{T\to\infty}\sup_{b\in H_0}\E_b\left[ \1_{\{A_T\geq r\}} \1_{\mathcal{C}_{T,b}}\right]\\
&\hspace{1.5cm} \leq \limsup_{T\to\infty}\sup_{b\in H_0}\E_b\left[ g_\delta(A_T)\1_{\mathcal{C}_{T,b}}\right] \\
&\hspace{1.5cm} \leq \sup_{b\in H_0}\E_b\left[ g_\delta(A)\right] +\limsup_{T\to\infty}\sup_{b\in H_0}\E_b\left[ g_\delta(A_T) - g_\delta(A)\right]\\
&\hspace{2.5cm} +\limsup_{T\to\infty}\sup_{b\in H_0}\E_b\left[ g_\delta(A_T)(\1_{\mathcal{C}_{T,b}}-1)\right].
\end{align*}
The first term is bounded from above by $\sup_{b\in H_0}\Pr_b(A\geq r-\delta)$ and the second one vanishes by Theorem~\ref{weak_conv} and Remark~\ref{remark_weak_Tc} as $g_\delta$ is a bounded Lipschitz function. As the maximum value of $g_\delta$ is one, the third summand can be bounded by
\[ \limsup_{T\to\infty}\sup_{b\in H_0}\E_b\left[ \left|\1_{\mathcal{C}_{T,b}}-1\right|\right]\leq\limsup_{T\to\infty}\sup_{b\in H_0}\Pr_b\left(\mathcal{C}_{T,b}^c\right) = 0. \]

For the following arguments we define
\begin{align*}
U_{\epsilon,y,h}^+(b) &:= \left\{ x\in [y-h,y+h]\mid q_b(x)\left( b(x)-b_0(x)-\eta\right) < -\epsilon\right\},\\
U_{\epsilon,y,h}^-(b) &:= \left\{ x\in [y-h,y+h]\mid q_b(x)\left( b_0(x)-b(x)-\eta\right) < -\epsilon\right\}
\end{align*} 
and the corresponding sets
\begin{align*}
\mathcal{T}_3^+(\epsilon,\epsilon',b) &:= \mathcal{T}_3(\epsilon,\epsilon',b)\cap \left\{ (y,h)\in\mathcal{T}\mid \lambda\left(U_{\epsilon,y,h}^+(b)\right)<\epsilon\right\}, \\
\mathcal{T}_3^-(\epsilon,\epsilon',b) &:= \mathcal{T}_3(\epsilon,\epsilon',b)\cap \left\{ (y,h)\in\mathcal{T}\mid \lambda\left(U_{\epsilon,y,h}^-(b)\right)<\epsilon\right\},
\end{align*}
for which we have $\mathcal{T}_3(\epsilon,\epsilon',b) = \mathcal{T}_3^+(\epsilon,\epsilon',b) \cup \mathcal{T}_3^-(\epsilon,\epsilon',b)$. 
We now proceed with the right-hand side of \eqref{eqp: D3}. Remember the notation
\[ \overline\sigma_b(y,h) = \frac{\sigma_b(y,h)}{\sigma_{b,\max}}\]
from \eqref{eqp: C5} and \eqref{eqp: C6} and abbreviate
\[ A_{y,h}(b) := \frac{\left|\int_\R K_{y,h}(z) \sqrt{q_b(z)} dW_z\right|}{\| K_{y,h}\sqrt{q_b}\|_{L^2}} \]
in what follows. Omitting the dependence of $\mathcal{T}_3^\pm$ on $\epsilon, \epsilon'$ and $b$ for notational convenience,\pagebreak
\begin{align*}
&\Pr\left( \sup_{(y,h)\in\mathcal{T}_3} \left( A_{y,h}(b) - \Upsilon(\overline\sigma_b(y,h)^2)\right) \geq r-\delta\right)\\
& =  \Pr\left( \max\left\{\sup_{(y,h)\in\mathcal{T}_3^+} \left( A_{y,h}(b) - \Upsilon(\overline\sigma_b(y,h)^2)\right), \right.\right.\\
&\hspace{4cm}\left.\left. \sup_{(y,h)\in\mathcal{T}_3^-} \left( A_{y,h}(b) - \Upsilon(\overline\sigma_b(y,h)^2)\right) \right\}  \geq r- \delta\right)\\
& = \Pr\left(\hspace{-0.05cm} \max\left\{\sup_{(y,h)\in\mathcal{T}_3^+} \hspace{-0.1cm}\left( A_{y,h}(b_0+\eta)\hspace{-0.05cm} - \Upsilon(\overline\sigma_b(y,h)^2) + A_{y,h}(b) - A_{y,h}(b_0+\eta)\right), \right.\right.\\
&\hspace{0.2cm}\left.\left. \sup_{(y,h)\in\mathcal{T}_3^-}\hspace{-0.1cm} \left( A_{y,h}(b_0-\eta) - \Upsilon(\overline\sigma_b(y,h)^2) + A_{y,h}(b) - A_{y,h}(b_0-\eta)\right)\hspace{-0.05cm} \right\} \hspace{-0.05cm}\geq r-\delta\right)\hspace{-0.1cm}.
\end{align*}
Using the inequality $\max\{ a+b, c+d\} \leq \max\{a,c\} + \max\{b,d\}$, we can proceed upper bounding with
\begin{align*}
&\Pr\left( \sup_{\mathcal{T}_3^+} \left( A_{y,h}(b_0+\eta) - \Upsilon(\overline\sigma_b(y,h)^2)\right)\vee\sup_{\mathcal{T}_3^-} \left( A_{y,h}(b_0-\eta) - \Upsilon(\overline\sigma_b(y,h)^2)  \right) \right.\\
&\hspace{0.7cm}\left. +  \sup_{\mathcal{T}_3^+} \left( A_{y,h}(b) - A_{y,h}(b_0+\eta)\right) \vee \sup_{\mathcal{T}_3^-} \left( A_{y,h}(b) - A_{y,h}(b_0-\eta)\right)  \geq r-\delta\right)\\
&\leq \Pr\left( \max\left\{\sup_{\mathcal{T}_3^+} \left( A_{y,h}(b_0+\eta) - \Upsilon(\overline\sigma_b(y,h)^2)\right), \right.\right. \\
&\hspace{2.5cm} \left.\left. \sup_{\mathcal{T}_3^-} \left( A_{y,h}(b_0-\eta) - \Upsilon(\overline\sigma_b(y,h)^2)  \right) \right\} \geq r-2\delta\right)\\
&\hspace{0.5cm} +  \Pr\left(\sup_{\mathcal{T}_3^+} \left( A_{y,h}(b) - A_{y,h}(b_0+\eta)\right) \vee \sup_{\mathcal{T}_3^-} \left( A_{y,h}(b) - A_{y,h}(b_0-\eta)\right)  >\delta\right),
\end{align*}
where the last line is true for any $\delta>0$ because for any real-valued random variables $X,Y$ and $a,b\in\R$, 
\begin{align*}
\Pr(X+Y >a) &= \Pr(X+Y>a, Y\leq b) + \Pr(X+Y>a, Y>b)\\
&\leq \Pr(X>a-b) + \Pr(Y>b).
\end{align*}
We will consider the asymptotic $\limsup_{\delta\searrow 0}$ at the end of the proof. But first, we turn to the second summand. Splitted up by the union bound and written out, it is bounded from above by
\begin{align*}
&\Pr\left( \sup_{\mathcal{T}_3^+} \left| A_{y,h}(b) - A_{y,h}(b_0+\eta)\right| >\delta\right)\hspace{-0.05cm} +\hspace{-0.05cm} \Pr\left( \sup_{\mathcal{T}_3^-} \left| A_{y,h}(b) - A_{y,h}(b_0-\eta)\right| >\delta\right)\\
&\leq  \Pr\left(\sup_{(y,h)\in\mathcal{T}_3^+} \left|\int_\R K_{y,h}(z) \left( \frac{\sqrt{q_b(z)}}{\|K_{y,h}\sqrt{q_b}\|_{L^2}} -\frac{\sqrt{q_{b_0+\eta}(z)}}{\|K_{y,h}\sqrt{q_{b_0+\eta}}\|_{L^2}}\right) dW_z \right| >\delta\right)\\
&\hspace{0.2cm} + \Pr\left(\sup_{(y,h)\in\mathcal{T}_3^-}\hspace{-0.05cm} \left|\int_\R K_{y,h}(z) \left( \frac{\sqrt{q_b(z)}}{\|K_{y,h}\sqrt{q_b}\|_{L^2}}\hspace{-0.05cm} -\hspace{-0.05cm}\frac{\sqrt{q_{b_0-\eta}(z)}}{\|K_{y,h}\sqrt{q_{b_0-\eta}}\|_{L^2}}\right) dW_z\right| >\delta\right)\hspace{-0.05cm},
\end{align*} 
where the inequality used the reverse triangle inequality. Note, that the preceding inequality still holds true with $\sup_{b\in H_0(b_0,\eta)}$ in front of it. Our aim is to show that both probabilities vanish for $\limsup_{\epsilon\searrow 0}\sup_{b\in H_0(b_0,\eta)}$. As the procedure is the same for both, we focus on the first one. By It\^{o}'s formula applied to the function $g(x,y) = xy$ and the processes $(W_z)_{a\leq z\leq b}$ and $(f(z))_{a\leq z\leq b}$ for a continuous function $f$ of bounded variation with $\|f\|_{[a,b]}<\infty$, we get
\[ f(b)W_b - f(a)W_a = \int_a^b W_zdf(z) + \int_a^b f(z) dW_z\]
as the covariation $[f,W]$ vanishes. Consequently,
\begin{align}\label{eqp: D4}
\left| \int_a^b f(z) dW_z \right| \leq \|W\|_{[a,b]}\left( 2\|f\|_\infty + V_a^b(f)\right).
\end{align} 
Here, $V_a^b(f)$ denotes the variation of $f$ on the interval $[a,b]$ and we used the classical bound $|\int_a^b f(x)dg(x)| \leq \|f\|_{[a,b]} V_a^b(g)$ for Stieltjes-integrals (cf.~\cite{Rudin}, Corollary to Theorem~$6.29$). We want to apply \eqref{eqp: D4} to the function
\[ f_{y,h}(z) = K_{y,h}(z) \left( \frac{\sqrt{q_b(z)}}{\|K_{y,h}\sqrt{q_b}\|_{L^2}} -\frac{\sqrt{q_{b_0+\eta}(z)}}{\|K_{y,h}\sqrt{q_{b_0+\eta}}\|_{L^2}}\right) .\]
This function $f_{y,h}$ is continuous as $q_b$ and $K$ are continuous. Moreover, the function in brackets is of bounded variation because it is differentiable with bounded derivative (cf. \cite{Rudin}, Example~$6.23$(b)) and it follows that $f_{y,h}$ is of bounded variation as a product of two such functions (cf. \cite{Rudin}, Theorem~$6.24$).\\


To proceed, an intermediate step is necessary. From the explicit representation of the invariant density $q_b$ we get for $z\in [y-h, y+h]\subset\R_{\geq 0}$,\vspace{-0.05cm}
\begin{align*}
&\frac{\sqrt{q_b(z)}}{\|K_{y,h}\sqrt{q_b}\|_{L^2}} =\frac{\sqrt{\exp\left( \int_{y-h}^{z} 2\sigma^{-2} b(u) du \right)}}{\sqrt{\int_\R K_{y,h}(z)^2\exp\left( \int_{y-h}^{z} 2\sigma^{-2} b(u) du \right) dz}}.
\end{align*}
A similar equality holds for $[y-h,y+h]\subset\R_{\leq 0}$, where $y-h$ as a boundary value of the integral is replaced by $y+h$. In the third case that $y-h<0$ and $y+h>0$, there is no need for splitting the argument of $\exp(\cdot)$ in a part depending on $z$ and one that does not. In either way, we derived that $q_b(z) = c_{y,h} \tilde{q}_b(y,h)(z)$ on the interval $[y-h, y+h]$ for some constant $c_{y,h}$ that cancels out in the above fraction and $\tilde{q}_b(y,h)$ only depends on the values $b(z)$ for $z\in [y-h, y+h]$. In particular, this implies that for $b_1$ and $b_2$ that are the same on $[y-h,y+h]$, we have $\tilde{q}_{b_1}(y,h) = \tilde{q}_{b_2}(y,h)$ although $q_{b_1}\neq q_{b_2}$ in general. We also have seen
\[ \frac{\sqrt{q_b(z)}}{\|K_{y,h}\sqrt{q_b}\|_{L^2}}  = \frac{\sqrt{\tilde{q}_b(y,h)(z)}}{\| K_{y,h}\sqrt{\tilde{q}_b(y,h)}\|_{L^2}}.\]
In particular, we have 
\begin{align*}
f_{y,h}(z) &= K_{y,h}(z) \left( \frac{\sqrt{\tilde q_b(y,h)(z)}}{\|K_{y,h}\sqrt{\tilde q_b(y,h)}\|_{L^2}} -\frac{\sqrt{\tilde q_{b_0+\eta}(y,h)(z)}}{\|K_{y,h}\sqrt{\tilde q_{b_0+\eta}(y,h)}\|_{L^2}}\right) \\
& =: K_{y,h}(z) g_{y,h}(z).
\end{align*} 
Since $\|K\|_\infty\leq 1$, we have $\|f_{y,h}\|_{[y-h,y+h]} \leq \|g_{y,h}\|_{[y-h,y+h]}$ and for the variation (cf. \cite{Rudin}, proof of Theorem~$6.24$),
\begin{align*}
V_{y-h}^{y+h}(f_{y,h}) &= V_{y-h}^{y+h}( K_{y,h}g_{y,h})\\
&\leq \|K_{y,h}\|_{[y-h,y+h]}V_{y-h}^{y+h}(g_{y,h}) + \|g_{y,h}\|_{[y-h,y+h]} V_{y-h}^{y+h}( K_{y,h})\\
& \leq V_{y-h}^{y+h}(g_{y,h}) + \|g_{y,h}\|_{[y-h,y+h]} V_{-1}^1(K). 
\end{align*} 
So far, \eqref{eqp: D4} gives the estimate
\begin{align*}
& \sup_{(y,h)\in\mathcal{T}_3^+}\left|\int_\R K_{y,h}(z) \left( \frac{\sqrt{q_b(z)}}{\|K_{y,h}\sqrt{q_b}\|_{L^2}} -\frac{\sqrt{q_{b_0+\eta}(z)}}{\|K_{y,h}\sqrt{q_{b_0+\eta}}\|_{L^2}}\right) dW_z \right| \\
&\hspace{0.3cm} \leq \sup_{(y,h)\in\mathcal{T}_3^+}\left(\|W\|_{[y-h, y+h]} \right.\\
&\hspace{2.3cm} \left. \cdot \left( 2\|g_{y,h}\|_{[y-h,y+h]} + V_{y-h}^{y+h}(g_{y,h}) + \|g_{y,h}\|_{[y-h, y+h]}V_{-1}^1(K)\right)\right) .
\end{align*}
Clearly, $\sup_{(y,h)\in\mathcal{T}_3^+} \|W\|_{[y-h, y+h]}\leq \|W\|_{[-A,A]}$. The bracket on the right-hand side converges to zero for $\epsilon\searrow 0$ if we can establish
\begin{enumerate}
\item[(1)] $\sup_{(y,h)\in\mathcal{T}_3^+(\epsilon,\epsilon',b)}\|g_{y,h}\|_{[y-h,y+h]} \stackrel{\epsilon\searrow 0}{\longrightarrow} 0$ and 
\item[(2)] $\sup_{(y,h)\in\mathcal{T}_3^+(\epsilon,\epsilon',b)}V_{y-h}^{y+h}(g_{y,h}) \stackrel{\epsilon\searrow 0}{\longrightarrow}  0$.
\end{enumerate}
For both convergences the following fact is important. For each $(y,h)\in\mathcal{T}_3^+$ we have
\[ \int_{y-h}^{y+h} \left| b(y) - b_0(y) -\eta\right|dy \leq \frac{2h\epsilon}{L_*} + 2\eta\epsilon, \]
because we have $0\geq q_b( b-b_0-\eta) \geq -\epsilon$ on $[y-h,y+h]$ except for a set $A\subset [y-h,y+h]$ with $\lambda(A)<\epsilon$ on which $|b-b_0-\eta|\leq 2\eta$ as $b\in H_0(b_0,\eta)$. Furthermore, we have bounded $q_b\geq L_*$ by Lemma~\ref{bound_invariant_density}. This implies 
\begin{align}\label{eq_proof:uniform_T3_L1}
\sup_{(y,h)\in\mathcal{T}_3^+} \left\| b-(b_0+\eta)\right\|_{L^1([y-h,y+h])} \leq \frac{2A\epsilon}{L_*} + 2\eta\epsilon\ \stackrel{\epsilon\searrow 0}{\longrightarrow} 0.
\end{align}
Now we move on to prove (1) and (2).

\begin{enumerate}
\item[(1)] For this claim we will show that both the nominator and denominator of $g_{y,h}$ converge to each other uniformly in $(y,h)\in\mathcal{T}_3^+$. Then the claim follows by noting that 
\[ \inf_{(y,h)\in\mathcal{T}_3^+}\left\|K_{y,h}\sqrt{\tilde{q}_{b_0+\eta}(y,h)}\right\|_{L^2} >0, \]
which follows from the fact that $h>\epsilon'$ on $\mathcal{T}_3^+$ and that $\tilde{q}_{b_0+\eta}(y,h)$ is bounded away from zero for those $h>\epsilon'$.\\
In the following, we suppress the dependence of $\tilde{q}_b$ on $(y,h)$ for better readability. For the denominators we have
\begin{align*}
&\sup_{(y,h)\in\mathcal{T}_3^+} \left| \|K_{y,h} \sqrt{\tilde{q}_b}\|_{L^2}^2 - \| K_{y,h}\sqrt{\tilde{q}_{b_0+\eta}}\|_{L^2}^2\right|\\
&\hspace{1cm} \leq \sup_{(y,h)\in\mathcal{T}_3^+} \left| \int_\R K_{y,h}(z)^2 \left( \tilde{q}_b(z) - \tilde{q}_{b_0+\eta}(z)\right) dz \right| \\
&\hspace{1cm} \leq \|K\|_{L^2}^2\sup_{(y,h)\in\mathcal{T}_3^+} \left(  h \left\| \tilde{q}_b - \tilde{q}_{b_0+\eta}\right\|_{[y-h,y+h]}\right) \\
&\hspace{1cm} \leq A\|K\|_{L^2}^2\sup_{(y,h)\in\mathcal{T}_3^+}  \left\| \tilde{q}_b - \tilde{q}_{b_0+\eta}\right\|_{[y-h,y+h]}.
\end{align*}
To estimate the last term, an explicit representation of $\tilde{q}_b$ is helpful. We consider, as above, the case where $z\in [y-h,y+h]\subset\R_{>0}$, the others follow along the same lines. In the above case we have $\tilde q_b(y,h)(z) =\int_{y-h}^z 2\sigma^{-2} b(u)du$ and by Taylor's formula,
\begin{align*}
&\left\| \exp\left( \int_{y-h}^\cdot 2\sigma^{-2} b(y) dy\right) - \exp\left( \int_{y-h}^\cdot 2\sigma^{-2} (b_0+\eta)(y) dy\right)\right\|_{[y-h,y+h]}\\
& \hspace{1cm} =\exp(\xi)\left\| \int_{y-h}^\cdot 2\sigma^{-2} \left( b-(b_0+\eta)\right)(y) dy\right\|_{[y-h,y+h]}
\end{align*}
for some $\xi$ between the two integral terms. As the drift is only evaluated for $y$ lying in the interval $[-A,A]$, by the at most linear growth condition on $b$ we have $ |\xi|\leq 2A\sigma^{-2}C(1+A)$. 
The remaining supremum can be bounded by \eqref{eq_proof:uniform_T3_L1} which gives
\[ \sup_{(y,h)\in\mathcal{T}_3^+}  \left\| \tilde{q}_b - \tilde{q}_{b_0+\eta}\right\|_{[y-h,y+h]} \leq 2\sigma^{-2} e^{2A\sigma^{-2}C(1+A)} \left( \frac{2A}{L_*} + 2\eta\right)\epsilon \]
and this term goes to zero for $\epsilon\searrow 0$. This finishes the proof of convergence of the denominators in supremum norm and also proves this convergence for the nominators by noting that $\tilde{q}_b(y,h)$ is bounded from below as was noted at the beginning of this step (1).

\item[(2)] We turn to the second statment and use that $g$ is differentiable and we therefore have $V_a^b(g) = \int_a^b |g'(z)|dz$ (cf. \cite{Rudin}, Theorem~$6.35$). Noting that
\[ \frac{d}{dz} \sqrt{\tilde{q}_b(y,h)(z) } = \sigma^{-2}b(z)\sqrt{\tilde{q}_b(y,h)(z)}\]
by $\sigma^2 q_b' = 2bq_b$ and $q_b(z) = c_{y,h}\tilde q_b(y,h)(z)$, we thus have
\begin{align*}
V_{y-h}^{y+h}(g) = \int_{y-h}^{y+h} \left| \frac{b(z)\sqrt{\tilde{q}_b(y,h)(z)}}{\sigma^2 \|K_{y,h}\sqrt{\tilde{q}_b}\|_{L^2}} - \frac{(b_0(z)+\eta)\sqrt{\tilde{q}_{b_0+\eta}(y,h)(z)}}{\sigma^2 \|K_{y,h}\sqrt{\tilde{q}_{b_0+\eta}}\|_{L^2}}\right| dz 
\end{align*}
By the arguments of part (1) and \eqref{eq_proof:uniform_T3_L1} we see that
\[ \sup_{(y,h)\in\mathcal{T}_3^+}\int_{y-h}^{y+h} \left| b(z)\sqrt{\tilde{q}_b(y,h)(z)} - (b_0(z)+\eta)\sqrt{\tilde{q}_{b_0+\eta}(y,h)(z)}\right| dz \rightarrow 0\] 
and
\[ \sup_{(y,h)\in\mathcal{T}_3^+}\left(\frac{1}{\sigma^2 \|K_{y,h}\sqrt{\tilde{q}_b}\|_{L^2}} - \frac{1}{\sigma^2 \|K_{y,h}\sqrt{\tilde{q}_{b_0+\eta}}\|_{L^2}}\right) \rightarrow 0\]
for $\epsilon\searrow 0$ which then gives the desired claim (2).
\end{enumerate}

To this end, we can conclude that we find $\tilde\epsilon(\epsilon)>0$ converging to zero for $\epsilon\searrow 0$ such that $\|b-b_0-\eta\|_I <\epsilon$ implies
\begin{align*} 
& \sup_{(y,h)\in\mathcal{T}_3^+(\epsilon)}\left| \int_\R K_{y,h}(z) \left( \frac{\sqrt{\rho_b(z)}}{\|K_{y,h}\sqrt{\rho_b}\|_{L^2}} -\frac{\sqrt{\rho_{b_0+\eta}(z)}}{\|K_{y,h}\sqrt{\rho_{b_0+\eta}}\|_{L^2}}\right) dW_z\right|\\
&\hspace{1.5cm} \leq \|W\|_{[-A,A]}(2+2A+V_{-1}^1(K))\tilde\epsilon(\epsilon). 
\end{align*}
Note that $\tilde\epsilon(\epsilon)$ is independent of $b$. Hence, for $C'=2+2A+V_{-1}^1(K)$,
\begin{align}\label{eqp: D11}
\begin{split}
&\limsup_{\epsilon\searrow 0 }\sup_{b\in H_0}\Pr \left( \sup_{(y,h)\in\mathcal{T}_3^+(\epsilon)}  \left| A_{y,h}(b) - A_{y,h}(b_0+\eta)\right| >\delta\right)\\
&\hspace{1cm}\leq \limsup_{\epsilon\searrow 0 }\sup_{b\in H_0}\Pr \left( \|W\|_{[-A,A]} C'\tilde\epsilon(\epsilon) >\delta\right)\\
&\hspace{1cm}\leq \lim_{L\to\infty} \limsup_{\epsilon\searrow 0 }\sup_{b\in H_0}\Pr \left( \|W\|_{[-A,A]} C'\tilde\epsilon(\epsilon) >\delta, \|W\|_{[-A,A]} \leq L \right) \\
&\hspace{3cm} + \lim_{L\to\infty} \limsup_{\epsilon\searrow 0}\sup_{b\in H_0}\Pr \left(\|W\|_{[-A,A]} > L \right).
\end{split}
\end{align}
In the first summand, we have for each $\delta, L$ that
\begin{align*}
&\limsup_{\epsilon\searrow 0 }\sup_{b\in H_0}\Pr \left( \|W\|_{[-A,A]} C'\tilde\epsilon(\epsilon) >\delta, \|W\|_{[-A,A]} \leq L \right)\\
&\hspace{1cm} \leq \limsup_{\epsilon\searrow 0 }\sup_{b\in H_0}\Pr \left( L C'\tilde\epsilon(\epsilon) >\delta\right),
\end{align*}
which equals zero, as for small enough $\epsilon$ we have $\tilde\epsilon(\epsilon)< \frac{\delta}{LC'}$. The latter probability does not depend on $b$ and $\epsilon$. With Markov's inequality and the Burkholder-Davis-Gundy inequality (cf. \cite{Kallenberg}, Theorem~$20.12$),
\[ \lim_{L\to\infty}\Pr \left(\|W\|_{[-A,A]} > L \right) \leq \lim_{L\to\infty} \frac{ \E\left[ \|W\|_{[-A,A]}\right]}{L} \leq \lim_{L\to\infty}\frac{c \sqrt{2A}}{L}  =0,\] where $c>0$ is some constant. Thus, the right-hand side of \eqref{eqp: D11} equals zero.
The whole procedure above can be repeated for the other boundary case $b_0-\eta$ in the same way, i.e. the left-hand side of \eqref{eqp: D11} also vanishes for $b_0+\eta$ replaced by $b_0-\eta$ and $\mathcal{T}_3^+(\epsilon)$ replaced by $\mathcal{T}_3^-(\epsilon)$. Summing up, we have shown so far that
\begin{align}\label{eqp: D7}
\begin{split}
&\limsup_{\epsilon\searrow 0 }\limsup_{T\to\infty}\sup_{b\in H_0}\Pr_b\left( \sup_{(y,h)\in\mathcal{T}_3}\left(\frac{\frac{1}{\sqrt{T}} \left|\int_0^T K_{y,h}(X_s) dW_s\right|}{\sqrt{\frac1T\int_0^T K_{y,h}(X_s)^2 ds}} \right.\right.\\
&\hspace{6cm} \left.\left. - \Upsilon\left(\frac{\frac1T\int_0^T K_{y,h}(X_s)^2 ds}{\frac1T\int_0^T \1_{[-A,A]}(X_s) ds}\right)\right) \geq r\right)\\
&\hspace{0.1cm}\leq \limsup_{\delta\searrow 0}\limsup_{\epsilon\searrow 0 }\sup_{b\in H_0} \Pr\left( \max\left\{\sup_{\mathcal{T}_3^+} \left( A_{y,h}(b_0+\eta) - \Upsilon(\overline\sigma_b(y,h)^2)\right),\right.\right. \\
&\hspace{4.2cm} \left.\left.\sup_{\mathcal{T}_3^-} \left( A_{y,h}(b_0-\eta) - \Upsilon(\overline\sigma_b(y,h)^2)  \right) \right\}\geq r-2\delta\right).
\end{split}
\end{align}
It remains to evaluate the term $\Upsilon(\cdot)$. Here, we have
\begin{align}\label{eqp: D5}
\overline\sigma_b(y,h)^2 = \overline\sigma_{b_0+\eta}(y,h)^2\cdot\frac{\overline\sigma_b(y,h)^2}{\overline\sigma_{b_0+\eta}(y,h)^2}, 
\end{align} 
and the fraction is given by
\begin{align}\label{eqp: D6}
\begin{split}
&\frac{\int_\R \1_{[-A,A]}(z) \frac{1}{C_{b_0+\eta,\sigma}} \exp\left( \int_0^z 2\sigma^{-2}(b_0(u)+\eta) du\right) dz}{\int_\R \1_{[-A,A]}(z) \frac{1}{C_{b,\sigma}} \exp\left( \int_0^z 2\sigma^{-2}b(u) du\right) dz} \\
&\hspace{1cm}\cdot \frac{\int_\R K_{y,h}(z)^2 \frac{1}{C_{b,\sigma}} \exp\left( \int_0^z 2\sigma^{-2}b(u) du\right) dz}{\int_\R K_{y,h}(z)^2 \frac{1}{C_{b_0+\eta,\sigma}} \exp\left( \int_0^z 2\sigma^{-2}(b_0(u)+\eta) du\right) dz}.
\end{split}
\end{align}
Clearly, the normalizing constants $C_{b,\sigma}^{-1}$ and $C_{b_0+\eta,\sigma}^{-1}$ cancel out. In both the denominator of the first factor and the nominator of the second one we have the term $\exp\left(\int_0^z 2\sigma^{-2} b(u)du\right)$, which may be written as
\[ \exp\left( \int_0^z 2\sigma^{-2} (b_0(u)+\eta)du\right) \exp\left( \int_0^z 2\sigma^{-2}(b(u)-b_0(u) - \eta) du\right). \]
On $[-A,A]$ we have by assumption on $b\in H_0(b_0,\eta)$ that 
\[ -2\eta \leq (b(u) - b_0(u)) - \eta \leq 0.\]
Remembering that we use the notation $\int_0^x f(y)dy = -\int_{x}^0 f(y) dy$ for $x<0$, we have for $z\in [-A,A]$,
\[ e^{-4\sigma^{-2}A\eta} \leq \exp\left( \int_0^z 2\sigma^{-2}(b(u)-b_0(u) - \eta) du\right) \leq e^{4\sigma^{-2}A\eta}.\]
Inserting all this in \eqref{eqp: D5} and \eqref{eqp: D6}, it follows that 
\[ \overline\sigma_b(y,h)^2 = \overline\sigma_{b_0+\eta}(y,h)^2\cdot\frac{\overline\sigma_b(y,h)^2}{\overline\sigma_{b_0+\eta}(y,h)^2}  \leq e^{8\sigma^{-2}A\eta}\overline\sigma_{b_0+\eta}(y,h)^2.\]
Since $\Upsilon(\cdot)$ is decreasing on $[0,1]$, we get 
\begin{align*}
\Upsilon\left( \overline\sigma_b(y,h)^2\right) &\geq \Upsilon\left( e^{8\sigma^{-2}A\eta}\overline\sigma_{b_0+\eta}(y,h)^2 \wedge 1\right) \\
&= \sqrt{2\log \left( \overline\sigma_{b_0+\eta}(y,h)^{-2}\right)  - 2\log\left(e^{8\sigma^{-2}A\eta}\wedge \overline\sigma_{b_0+\eta}(y,h)^{-2}\right)}\\
&\geq \sqrt{2\log \left( \overline\sigma_{b_0+\eta}(y,h)^{-2}\right)} - \sqrt{2\log\left(e^{8\sigma^{-2}A\eta}\wedge \overline\sigma_{b_0+\eta}(y,h)^{-2}\right)}\\
&= \Upsilon\left(\overline\sigma_{b_0+\eta}(y,h)^2\right) -\sqrt{2\log\left(e^{8\sigma^{-2}A\eta}\wedge\overline\sigma_{b_0+\eta}(y,h)^{-2}\right)}.
\end{align*}
In the third step we used $\sqrt{a-b}\geq \sqrt{a}-\sqrt{b}$ for $a>b>0$ 
together with the fact $\log(\overline\sigma_{b_0+\eta}(y,h)^{-2}) \geq \log(e^{8\sigma^{-2}A\eta}\wedge \overline\sigma_{b_0+\eta}(y,h)^{-2})$. Hence,
\begin{align*} 
-\Upsilon\left( \overline\sigma_b(y,h)^2\right)&\leq -\Upsilon\left(\overline \sigma_{b_0+\eta}(y,h)^2\right)+\sqrt{2\log\left(e^{8\sigma^{-2}A\eta}\wedge \overline\sigma_{b_0+\eta}(y,h)^{-2}\right)} \\
&\leq -\Upsilon\left(\overline\sigma_{b_0+\eta}(y,h)^2\right) + \sqrt{2\log\left(e^{8\sigma^{-2}A\eta}\right)}\\
&= -\Upsilon\left(\overline\sigma_{b_0+\eta}(y,h)^2\right) + 4\sqrt{A\eta/\sigma^2}.
\end{align*}
We thus have shown that
\begin{align}\label{eqp: D8}
\begin{split}
&\sup_{\mathcal{T}_3^+} \left( A_{y,h}(b_0+\eta) - \Upsilon(\overline\sigma_b(y,h)^2)\right)\\
&\hspace{1cm}\leq \sup_{\mathcal{T}_3^+} \left( A_{y,h}(b_0+\eta) - \Upsilon(\overline\sigma_{b_0+\eta}(y,h)^2) + 4\sqrt{A\eta/\sigma^2}\right)\\
&\hspace{1cm}\leq \sup_{\mathcal{T}} \left( A_{y,h}(b_0+\eta) - \Upsilon(\overline\sigma_{b_0+\eta}(y,h)^2) \right) + 4\sqrt{A\eta/\sigma^2} \\
&\hspace{1cm} = U_1 + 4\sqrt{A\eta/\sigma^2}
\end{split}
\end{align}
for $U_1$ given in the statement of the theorem. Repeating all this arguments for $b_0-\eta$ instead of $b_0+\eta$ gives
\begin{align}\label{eqp: D9}
\begin{split}
&\sup_{\mathcal{T}_3^-} \left( A_{y,h}(b_0-\eta) - \Upsilon(\tilde\sigma_b(y,h)^2)\right)\\
&\hspace{1cm} \leq  \sup_{\mathcal{T}} \left( A_{y,h}(b_0-\eta) - \Upsilon(\tilde\sigma_{b_0-\eta}(y,h)^2) \right) + 4\sqrt{A\eta/\sigma^2}\\
&\hspace{1cm} = U_2 + 4\sqrt{A\eta/\sigma^2}.
\end{split}
\end{align} 
Note that these upper bounds depend no longer on $b$ or $\epsilon$. In consequence, it follows from \eqref{eqp: D7}, the upper bounds \eqref{eqp: D8} and \eqref{eqp: D9} and monotonicity of measures that
\begin{align*}
&\limsup_{\delta\searrow 0}\limsup_{\epsilon\searrow 0 }\limsup_{T\to\infty}\sup_{b\in H_0} \Pr_b\left( \sup_{(y,h)\in\mathcal{T}_3(\epsilon)}\left(\frac{\frac{1}{\sqrt{T}}\left|\int_0^T K_{y,h}(X_s) dW_s\right|}{\sqrt{\frac1T\int_0^T K_{y,h}(X_s)^2 ds}} \right.\right. \\
&\hspace{5.5cm} \left.\left. - \Upsilon\left(\frac{\frac1T\int_0^T K_{y,h}(X_s)^2 ds}{\frac1T\int_0^T \1_{[-A,A]}(X_s) ds}\right)\right) \geq r-\delta\right)\\
&\hspace{1cm} \leq \limsup_{\delta\searrow 0} \Pr\left( U_1\vee U_2 + 4\sqrt{A\eta/\sigma^2} \geq r-2\delta \right) \\
&\hspace{1cm}= \Pr\left( U_1\vee U_2 + 4\sqrt{A\eta/\sigma^2} \geq r\right)
\end{align*}
for $U_1$ and $U_2$ given in the statement of the theorem and the proof is complete.

\subsection{Proof of Remark \ref{remark_Piterbarg}}\label{SubSecD2}

In this subsection, we start with the proof of Remark~\ref{remark_Piterbarg}. Two lemmas used within are stated and proven thereafter. Within this subsection we abbreviate 'infinitely often' by i.o.

\begin{proof}[Proof of Remark~\ref{remark_Piterbarg}]
We will show that for $b\in\Sigma(C,A,\gamma,\sigma)$ the random variable
\[ S_b = \sup_{(y,h)\in\mathcal{T}}\left( \frac{\int_{-A}^A K_{y,h}(z)\sqrt{q_b(z)} dW_z}{\|K_{y,h}\sqrt{q_b}\|_{L^2}} - \Upsilon\left( \frac{\sigma_b(y,h)^2}{\sigma_{b,\max}^2}\right)\right) \]
has no point mass. Then in particular, both $U_1$ and $U_2$ have no point mass. Consequently, the same is true for their maximum because for any $x\in\R$,
\[ \Pr\left( U_1\vee U_2 =x\right) \leq \Pr(U_1 =x) + \Pr(U_2=x) =0,\]
which finishes the claim.\\

First of all, $S_b$ has no positive atoms, which can be seen along the lines of the proof of Lemma~$2$ in \cite{Piterbarg} and we have to show that there is no point mass in zero, i.e. $\Pr(S_b>0) =1$. 
By using $q_b\geq L_*$ from Lemma~\ref{bound_invariant_density}, we find
\begin{align*}
\Upsilon\left(\frac{\sigma_b(y,h)^2}{\sigma_{b,\max}^2}\right)^2 &= 2\log\left(\frac{ \int_{-A}^A q_b(z) dz}{\int_{-A}^A K_{y,h}(z)^2 q_b(z) dz}\right)\leq 2\log\left(\frac{ \int_{-A}^A q_b(z)/L_* dz}{h \|K\|_{L^2}^2}\right)
\end{align*}
and 
\[ S_b \geq \sup_{(y,h)\in\mathcal{T}}\left( \frac{\int_{-A}^A K_{y,h}(z)\sqrt{q_b(z)} dW_z}{\|K_{y,h}\sqrt{q_b}\|_{L^2}} - \sqrt{2\log\left( \frac{L'}{h \|K\|_{L^2}^2}\right)}\right) =:\tilde{S}_b(L'),\]
for $L'\geq L:=1\vee \int_{-A}^A q_b(z)/L_* dz$. 
Now, we can write
\begin{align*}
X(y,h) := \int_{-A}^A \frac{K_{y,h}(z)\sqrt{q_b(z)}}{\|K_{y,h}\sqrt{q_b}\|_{L^2}^2} dW_z = Y_1(y,h) + h^\frac14 Y_2(y,h) 
\end{align*}
with
\begin{align*}
Y_1(y,h) &:= \frac{1}{\sqrt{h}\|K\|_{L^2}} \int_{-A}^A K_{y,h}(z) dW_z, \\
Y_2(y,h) &:= h^{-\frac14} \int_{-A}^A \left( \frac{K_{y,h}(z)\sqrt{q_b(z)}}{\|K_{y,h}\sqrt{q_b}\|_{L^2}} - \frac{K_{y,h}(z)}{\sqrt{h}\|K\|_{L^2}} \right) dW_z.
\end{align*}
By Lemma~$1$ in \cite{Piterbarg},
\[ \Pr\left( \sup_{ -A+h\leq y\leq A-h} Y_1(y,h) > \sqrt{2\log(L/h\|K\|_{L^2}^2)} \textrm{ i.o.  as } h\searrow 0 \right) =1.\]
We now apply Lemma~\ref{lemma_H3} to this and the process $Y=Y_2$, which is possible due to Lemma~\ref{lemma_H2}, and conclude
\begin{align*}
\Pr\left( \sup_{ -A+h\leq y\leq A-h} X(y,h) > \sqrt{2\log(L'/h\|K\|_{L^2}^2)} \textrm{ i.o.  as } h\searrow 0 \right) =1
\end{align*}
for some $L'\geq L$. Following the lines of the proof of Theorem~$1$ in \cite{Piterbarg} we then see
\[ \Pr(S_b >0) \geq \Pr(\tilde{S}_b(L') >0) =1\]
and we conclude that $S_b$ has no point mass in zero. 
\end{proof}

\begin{lemma}\label{lemma_H3}
Let $X(u)$ depend on a parameter $u>0$ and suppose that for all $K\geq 1$ we have
\[ \Pr\left( X(u) > \sqrt{2\log(K/u)} \textrm{ i.o.  as } u\searrow 0\right)=1. \]
Let $Y(u)$ be a family depending on $u>0$ and $Y^*(u_0) := \sup_{u\leq u_0 }Y(u)$ be finite almost surely for some $u_0>0$. Then for every $\alpha>0$ and some $K'=K'(\alpha)\geq 1$,
\[ \Pr\left( X(u) + u^\alpha Y(u) > \sqrt{2\log(K'/u)} \textrm{ i.o.  as } u\searrow 0\right)=1. \]
\end{lemma}
\begin{proof}
Let $\epsilon>0$. Choose $u_0>0$ small enough and $L>0$ large enough such that
\[ A_\epsilon := \left\{ \sup_{u\leq u_0} |Y(u)| \leq L\right\} >1-\epsilon.\]
For this set, 
\begin{align*}
&\Pr\left( X(u) +u^\alpha Y(u) > \sqrt{2\log(K/u)} \textrm{ i.o.  as } u\searrow 0\right) \\
&\hspace{0.2cm} \geq \Pr\left( X(u) +u^\alpha Y(u) > \sqrt{2\log(K/u)} \textrm{ i.o.  as } u\searrow 0, A_\epsilon\right)  \\
&\hspace{0.2cm} \geq \Pr\left( X(u) > \sqrt{2\log(K/u)}+u^\alpha L \textrm{ i.o.  as } u\searrow 0, A_\epsilon\right)  \\
&\hspace{0.2cm} \geq \Pr\left( X(u) > \sqrt{2\log(K/u)}+ u^\alpha L \textrm{ i.o.  as } u\searrow 0\right) - \Pr\left( A_\epsilon^c\right),
\end{align*}
where the last step uses
\[ \Pr(A\cap B) = \Pr(A) - \Pr(A\cap B^c) \geq \Pr(A) - \Pr(B^c).\]
Now choose $c>1$ large enough such that
\[ \sqrt{2\log(cK/u)} \geq \sqrt{2\log(K/u)} + u^\alpha L \]
for $u<u_0$, which is possible since the condition is equivalent to
\[ 2\log(c) \geq 2u^\alpha L \sqrt{2\log(K/u)} + u^{2\alpha} L^2 \]
and the right-hand side converges to zero for $u\searrow 0$. Then, for $K'=cK$, 
\begin{align*}
&\Pr\left( X(u) +u^\alpha Y(u) > \sqrt{2\log(K/u)} \textrm{ i.o.  as } u\searrow 0\right) \\
&\hspace{1.5cm} \geq \Pr\left( X(u) > \sqrt{2\log(K'/u)} \textrm{ i.o.  as } u\searrow 0\right) - \Pr\left( A_\epsilon^c\right).
\end{align*}
The first probability equals one by assumption and the claim follows by letting $\epsilon\searrow 0$.
\end{proof}

\begin{lemma}\label{lemma_H2}
Define 
\[ Y(y,h) := h^{-\frac14} \int_{-A}^A \frac{K_{y,h}(z)}{\sqrt{h}\|K\|_{L^2}} \left( \frac{\sqrt{q_b(z)} \sqrt{h}\|K\|}{\|K_{y,h}\sqrt{q_b}\|_{L^2}} - 1\right) dW_z.\]
Then there exists $0<h_0\leq A$ such that the random variable
\[ Y(h_0) := \sup_{(y,h)\in\mathcal{T}: h\leq h_0} Y(y,h)\]
is finite almost surely.
\end{lemma}
\begin{proof}
Define the functions
\[ f_{y,h}^1(z) := \frac{K_{y,h}(z)}{h^{3/4}\|K\|_{L^2}} \quad \textrm{ and }\quad f_{y,h}^2(z) := \frac{\sqrt{q_b(z)}\sqrt{h}\|K\|_{L^2}}{\|K_{y,h}\sqrt{q_b}\|_{L^2}} -1.\]
Then $f_{y,h}^1$ is of bounded variation as this holds for $K$, and $f_{y,h}^2$ is of bounded variation because $q_b$ is differentiable with bounded derivative (cf. \cite{Rudin}, Example~$6.23$(b) and Lemma~\ref{bound_invariant_density}). Consequently, $f_{y,h}^1\cdot f_{y,h}^2$ is of bounded variation as a product of two such functions (cf. \cite{Rudin}, Theorem~$6.24$). Then by \eqref{eqp: D4}, 
\[ |Y(y,h)| \leq \|W\|_{[-A,A]}\left( 2\|f_{y,h}^1\cdot f_{y,h}^2\|_\infty + V_{-A}^A\left(f_{y,h}^1\cdot f_{y,h}^2\right)\right), \]
where $V_a^b(g)$ denotes the variation of $g$ on the interval $[a,b]$. Because the random variable $\|W\|_{[-A,A]}$ is finite almost surely, the claim follows if we show that
\begin{itemize}
\item[(1)] $\sup_{(y,h)\in\mathcal{T}: h\leq h_0}\|f_{y,h}^1\cdot f_{y,h}^2\|_\infty <\infty$, and
\item[(2)] $\sup_{(y,h)\in\mathcal{T}: h\leq h_0} V_{-A}^A\left(f_{y,h}^1\cdot f_{y,h}^2\right)<\infty$
\end{itemize}
for some $0<h_0\leq A$. Subsequently, we will show these two items separately.\\

\textit{Item (1).} We choose $h_0$ small enough such that $4L^* h_0 <L_*/2$. We have,
\begin{align*}
\frac{\sqrt{q_b(z)}\sqrt{h} \|K\|_{L^2}}{\|K_{y,h}\sqrt{q_b}\|_{L^2}} 
= \sqrt{\frac{(q_b(y) + q_b(z) - q_b(y))h\|K\|_{L^2}^2}{q_b(y) h\|K\|_{L^2}^2 + \int_{-A}^A K_{y,h}(z)^2 (q_b(z)-q_b(y)) dz}}.
\end{align*}
Using $|q_b(z) - q_b(y)|\leq L^*h$ for $|z-y|\leq h$ by Lemma~\ref{bound_invariant_density}, we give an upper and lower bound for the last expression. We start with the upper bound
\begin{align*}
\sqrt{\frac{q_b(y) h\|K\|_{L^2}^2 + L^* \|K\|_{L^2}^2 h^2}{q_b(y) h\|K\|_{L^2}^2 - L^* \|K\|_{L^2}^2 h^2}} &= \sqrt{\frac{q_b(y) + L^*h}{q_b(y) - L^*h}} \\
& = \sqrt{ 1+ \frac{q_b(y) + L^*h - q_b(y) +L^*h}{q_b(y) - L^* h}}\\
&\leq \sqrt{1+4hL^*/L_* },
\end{align*}
where we used in the last line that $q_b(y)-L^*h \geq L_*-L^*h \geq L_*/2$ by choice of $h\leq h_0$. Similiarly, we find the lower bound
\begin{align*}
\sqrt{\frac{q_b(y) h\|K\|_{L^2}^2-  L^* \|K\|_{L^2}^2 h^2}{q_b(y) h\|K\|_{L^2} + L^* \|K\|_{L^2}^2 h^2}} 
&\geq \sqrt{1-4hL^*/L_* }.
\end{align*}
In consequence, we conclude by $\left(1-\sqrt{1+x}\right)^2 \leq x^2$ for $x\geq -1$ that
\begin{align}\label{eqp: H2}
\left( \frac{\sqrt{q_b(z)}\sqrt{h} \|K\|_{L^2}}{\|K_{y,h}\sqrt{q_b}\|_{L^2}}  - 1\right)^2 \leq
\max_{v=\pm 1} \left( \sqrt{1+ 4vhL^*/L_*}-1\right)^2 \leq c^2h^2
\end{align}
with $c := 4L^*/L_*$. This direcly implies that $\|f_{y,h}^2\|_\infty \leq ch$, in particular this bound is independent of $y$. Bounding $f_{y,h}^1$ is straightforward, as $\|K\|_\infty \leq 1$ and we thus have 
\[ \|f_{y,h}^1\cdot f_{y,h}^2\|_\infty \leq h^{-3/4} \|K\|_{L^2}^{-1} ch = c\|K\|_{L^2}^{-1} h^{1/4} \leq c\|K\|_{L^2}^{-1} A^{1/4},\]
which is independent of $y$ and $h$.\\

\textit{Item (2).} For the variation term we have (cf. \cite{Rudin}, proof of Theorem~$6.24$)
\begin{align}\label{eqp: Variation}
V_{-A}^A \left( f_{y,h}^1\cdot f_{y,h}^2\right) \leq \|f_{y,h}^1\|_\infty V_{y-h}^{y+h}\left( f_{y,h}^2\right) + \|f_{y,h}^2\|_\infty V_{y-h}^{y+h}\left( f_{y,h}^1\right). 
\end{align}
As seen in item (1), we have $\|f_{y,h}^1\|_\infty \leq h^{-3/4}\|K\|_{L^2}^{-1}$ and $\|f_{y,h}^2\|_\infty\leq ch$ for $h\leq h_0 <L_*/(8L^*)$. Moreover, 
\[ V_{y-h}^{y+h}\left( f_{y,h}^1\right) = h^{-3/4}\|K\|_{L^2}^{-1} V_{y-h}^{y+h}(K_{y,h}) = h^{-3/4}\|K\|_{L^2}^{-1} V_{-1}^1(K) \]
and 
\[ V_{y-h}^{y+h}\left( f_{y,h}^2\right) = \frac{\sqrt{h}\|K\|_{L^2}}{\|K_{y,h}\sqrt{q_b}\|_{L^2}} V_{y-h}^{y+h}\left( \sqrt{q_b}\right).\]
For differentiable $g$ we have $V_a^b(g) = \int_a^b |g'(z)|dz$ (cf. \cite{Rudin}, Theorem~$6.35$). Using $q_b'=2\sigma^{-2}bq_b$, the at most linear growth condition on $b$ and the uniform bound $q_b\leq L^*$ from Lemma~\ref{bound_invariant_density},
\begin{align*}
V_{y-h}^{y+h}\left( \sqrt{q_b}\right) =\sigma^{-2} \int_{y-h}^{y+h} |b(z)| \sqrt{q_b(z)} dz \leq 2h\sigma^{-2}C(1+A) \sqrt{L^*}. 
\end{align*} 
Using that $\|K_{y,h}\sqrt{q_b}\|_{L^2}\geq \sqrt{hL_*}\|K\|_{L_2}$ with the lower bound of Lemma~\ref{bound_invariant_density}, we then have
\[ V_{y-h}^{y+h}\left( f_{y,h}^2\right)\leq 2h\sigma^{-2}C(1+A) \sqrt{L^*/L_*}. \]
Inserting all these estimates into \eqref{eqp: Variation} yields
\[ V_{-A}^A \left( f_{y,h}^1\cdot f_{y,h}^2\right) \leq h^{1/4} \left( 2\sigma^{-2} \|K\|_{L^2}^{-1} C(1+A)\sqrt{L^*/L_*} + c\|K\|_{L^2}^{-1} V_{-1}^1(K)\right).\]
Finally, bounding $h\leq A$ gives a finite bound independent of $y$ and $h$.
\end{proof}

\section{Proof of the minimax results}\label{App_lower}

In Subsection~\ref{SubSec_lower}, a proof of Theorem~\ref{lower_bound} is given and in Subsection~\ref{SubSec_upper_sketch} a sketch of proof for the upper bound in Theorem~\ref{Upper_bound} is provided.

\subsection{Proof of Theorem~\ref{lower_bound}}\label{SubSec_lower}
In this section, a proof of the minimax lower bound given in Theorem~\ref{lower_bound} is provided. This follows the sketch of proof given in Section~\ref{Sec_supp}.\\

Like in most other proofs of lower bounds, we will reduce the problem to considering suitable challenging hypotheses -- in our case some small deviations of the drift $b_0\pm\eta$ -- that will not be detected by an arbitrary test. Their existence then implies that we cannot detect smaller deviations, giving the lower bound for the hypothesis testing problem. In order to construct those, we set
\begin{align}\label{eqp: h_T_w}
h_T^w := \left(\frac{c_*}{L}\right)^\frac{1}{\beta} \left( \frac{\sigma^2 \log T}{T w}\right)^\frac{1}{2\beta+1}
\end{align} 
for a parameter $w\in\R$ to be specified later. Assume we want to test for the $\eta$-environment of $b_0$. Then we pick some drift $b_0 \in\Sigma\left( \frac{C}{2}-\eta ,A, \gamma +\frac{\eta}{\sigma^2}, \sigma\right)$ for the zero hypothesis, consider 
\begin{align}\label{eqp: b_0_eta}
b_{0,\eta}:=b_0 +\eta\in H_0(b_0,\eta) 
\end{align} 
as a drift function on the boundary of our hypothesis and define an alternative by adding a small hat as
\begin{align}\label{eqp: b_w}
b^w(x) := b_{0,\eta}(x) + L(1-\epsilon_T)(h_T^w)^\beta  K_T^\beta\left(\frac{x-y^w}{h_T^w}\right),
\end{align} 
where 
\[ y: [L_*, L^*] \rightarrow \left[-A + h_T^{L_*}, A-h_T^{L_*}\right] =: [-A', A'] \]
is a function depending continuously on $w$ that describes the location our added hat is put on. $K_T^\beta$ is given as a truncated version of the optimal recovery kernel $K_\beta$, i.e.
\begin{align*}
K_T^\beta(x) := \begin{cases}
K_\beta(x) & \textrm{ if }\beta\geq 1,\\
K_\beta(x) \1_{\{\frac1T\leq |x|\leq 1\}} + (1-\frac{1}{T^\beta})\1_{\{|x|\leq \frac1T\}} &\textrm{ if } \beta<1.
\end{cases}
\end{align*}
To be a valid hypothesis we need three conditions on $b^w$: 
\begin{itemize}
\item The solution to SDE \eqref{eq: SDE} with drift $b^w$ should be ergodic, more precisely, we need $b^w\in \Sigma(C,A,\gamma, \sigma)$,
\item we need $ b^w-b_0\in\mathcal{H}(\beta, L)$, and
\item we want $\Delta_J(b)\geq (1-\epsilon_T)c_*\delta_T$. As the boundary case of equality is expected to be the most challenging, we require equality, i.e.
\begin{align*}
&\left|(b^w(y^w) - b_0(y^w)) \left(\frac{q_{b^w}(y^w)}{\sigma^2}\right)^\frac{\beta}{2\beta+1}\right| \\
&\hspace{1cm} = L(1-\epsilon_T)(h_T^w)^\beta K_T^\beta(0)\left( \frac{q_{b^w}(y^w)}{\sigma^2}\right)^\frac{\beta}{2\beta+1} \\
&\hspace{1cm}= (1-\epsilon_T) c_*\left(\frac{\log T}{T}\right)^\frac{\beta}{2\beta+1}.
\end{align*} 
Using our definition of $h_T^w$ in \eqref{eqp: h_T_w}, the last equality is equivalent to the fixed point equation
\begin{align}\label{eqp: fixed_point}
w = q_{b^w}(y^w) K_T^\beta(0)^\frac{2\beta+1}{\beta}.
\end{align} 
\end{itemize} 
As we are dealing with asymptotics, it is in fact enough to verify those conditions for $T\geq T_0$. The first two are standard and will be established in Lemma \ref{lemma_cond_2} and \ref{lemma_cond_1}. 
The first one is the reason why we modified the kernel $K_\beta$ for $\beta<1$. In fact, we have to check that $b^w\in \textrm{Lip}_\textrm{loc}(\R)$ which is easily obtained for $\beta\geq 1$, but not true for $\beta<1$ since the kernel is only Hölder continuous. As the proof of the lower bound will find the optimal constant, it is nonetheless necessary that $K_T^\beta$ is chosen in such a way that its $L^2$-norm approximates the $L^2$-norm of $K_\beta$ sufficiently fast. It is stated in Lemma \ref{lemma_K_T} that $K_T^\beta$ has all the desired properties.\\
The third condition will be dealt with in the proof of the main theorem by constructing the alternatives inductively. 
It has the nice feature that is gives rise to the fixed-point problem \eqref{eqp: fixed_point} : We associate to any $w$ in some suitable chosen interval a drift function $b^w$ which again is associated to a density function $q_{b^w}$ which is known in the ergodic case. This density can be evaluated at $y^w$ and we ask if it possible to chose $w$ in such a way that the outcome of the concatenation of these mappings is again $w$ times a constant. In Lemma~\ref{lemma_fixed_point} we prove this to be solvable.\\
Later in the proof of Theorem~\ref{lower_bound}, we have to use different hypotheses of this kind at different locations. The construction of those will be an iterative procedure of applying Lemma~\ref{lemma_fixed_point} for several locations leading to different $w_i$. But to keep notation simpler, we first show the important properties for a single $b^w$.

\begin{lemma}[\cite{Brutsche}, Lemma 3.6.1]\label{lemma_K_T}
Let $\beta<1$. Then $K_T^\beta\in\textrm{Lip}_\textrm{loc}(\R)$ and $K_T^\beta\in\mathcal{H}(\beta, 1)$. Furthermore, we have
\[ \left| \|K_T^\beta\|_{L^2}^2 - \|K_\beta\|_{L^2}^2\right| \leq \sqrt{8} \|K_\beta\|_{L^2} T^{-\frac12 - \beta}.\]
\end{lemma}

\begin{lemma}[\cite{Brutsche}, Lemma 3.6.2]\label{lemma_cond_2}
Let $b_0\in\Sigma(\frac{C}{2}-\eta,A,\gamma+\frac{\eta}{\sigma^2},\sigma)$ and $b^w$ be given as in \eqref{eqp: b_w}. Then we have $b^w-b_{0}\in\mathcal{H}(\beta, L)$.
\end{lemma}

\begin{lemma}[\cite{Brutsche}, Lemma 3.6.3]\label{lemma_cond_1}
For the drift $b^w$ given in \eqref{eqp: b_w}, we have $b^w\in \Sigma(C,A,\gamma, \sigma)$ for all $T$ large enough to ensure $L(h_T^w)^\beta\leq \frac{C}{2}$.
\end{lemma}

Next, we prove Lemma~\ref{lemma_fixed_point} that provides the solution to our fixed point problem.

\begin{proof}[Proof of Lemma \ref{lemma_fixed_point}]
Denote by $f$ the concatenation of the following maps
\begin{equation}\label{eqp: map}
\begin{matrix}
& [c_T L_*, c_T L^*] & \longrightarrow & \mathcal{C}([-A,A]) \times [-A,A]  &\longrightarrow &\R, \\
& w & \mapsto & \begin{pmatrix} q_{b^w}\\ y^w\end{pmatrix} &\mapsto  & c_T q_{b^w}(y^w),
\end{matrix}
\end{equation}
where $\mathcal{C}([-A,A])$ is equipped with $\|\cdot \|_{[-A,A]}$. Our proof will establish
\begin{itemize}
\item[(1)] continuity of $f:[c_T L_*, c_T L^*]\rightarrow\R_{\geq 0}$, 
\item[(2)] existence of a compact interval $I\subset\R_{>0}$ such that $f(I)\subset I$. 
\end{itemize}
Then it follows that there exists a fixed point $\tilde{w}$ within the interval $I$ by the intermediate value theorem and the proof is complete.\\
Statement $(2)$ is straightforward, as we know that each invariant density $q_{b^w}$ only takes values in $[L_*, L^*]$ by Lemma~\ref{bound_invariant_density}. With our specification of $c_T$ this implies $c_T q_{b^w}\in [c_T L_*, c_T L^*]$. It remains to deal with statement~$(1)$. The map $f$ consists of two parts and we will show continuity of each mapping in~\eqref{eqp: map} separately, each time for fixed $T$. \\

We start with the first mapping in \eqref{eqp: map}. Clearly, the map $w\mapsto h_T^w$ is continuous and hence $w\mapsto y+Rh_T^w =y^w$ is continuous, as well. For $w\mapsto q_b$ we note that the invariant density is of the form
\[ q_{b^w}(x) = \frac{1}{C_{b^w,\sigma}}\exp\left( \int_0^x \frac{2b^w(u)}{\sigma^2} du\right),\]
and we show continuity of $w\mapsto C_{b^w,\sigma}$ and $w\mapsto \exp\left(\int_0^x \frac{2b^w(u)}{\sigma^2} du \right)$, where the right-hand side of the latter is understood as a function restricted to the interval $[-A,A]$. 
For the second statement, we consider $w, w'>0$ and estimate 
\begin{align*}
&  \left| \exp\left(\int_0^x \frac{2b^w(u)}{\sigma^2} du \right)-\exp\left(\int_0^x \frac{2b^{w'}(u)}{\sigma^2} du \right)\right| \\
&\hspace{2cm} = \exp(\xi) \left| \int_0^x \frac{2b^w(u)}{\sigma^2} du-\int_0^x \frac{2b^{w'}(u)}{\sigma^2} du \right|,
\end{align*}
where $\xi$ lies between $\int_0^x 2\sigma^{-2}b^w(u)du$ and $\int_0^x 2\sigma^{-2} b^{w'}(u)du$.
In particular, as
\[ \left| \int_0^x \frac{2b^w(u)}{\sigma^2} du\right| \leq 2\sigma^{-2} \int_0^A C(1+u)du \leq 2C\sigma^{-2} (1+A)A \]
by the at most linear growth condition on $b$, we find
\begin{align*}
&  \left| \exp\left(\int_0^x \frac{2b^w(u)}{\sigma^2} du \right)-\exp\left(\int_0^x \frac{2b^{w'}(u)}{\sigma^2} du \right)\right| \\
&\hspace{2cm} \leq e^{2C\sigma^{-2}A(1+A)} \left| \int_0^x \frac{2b^w(u)}{\sigma^2} du-\int_0^x \frac{2b^{w'}(u)}{\sigma^2} du \right|.
\end{align*}
Writing $b^w = b_{0,\eta} + g^w$, 
\begin{align*}
\left|\int_0^x \frac{2b^w(u)}{\sigma^2} du-\int_0^x \frac{2b^{w'}(u)}{\sigma^2} du\right| &= 2\sigma^{-2} \left| \int_0^x b^w(u) - b^{w'}(u) du\right|\\
&\leq 2\sigma^{-2} \textrm{sgn}(x)\int_0^{x} \left| g^w(u) - g^{w'}(u)\right| du \\
&\leq 2\sigma^{-2} \int_{-A}^A\left| g^w(u) - g^{w'}(u)\right| du \\
& \leq 4A\sigma^{-2} \|g^w - g^{w'}\|_{[-A,A]} \\
& = 4A\sigma^{-2} \|b^w - b^{w'}\|_{[-A,A]},
\end{align*}
which gives the intermediate result 
\begin{align}\label{eqp: F3}
\begin{split} 
&\sup_{x\in[-A,A]} \left| \exp\left(\int_0^x \frac{2b^w(u)}{\sigma^2} du \right)-\exp\left(\int_0^x \frac{2b^{w'}(u)}{\sigma^2} du \right)\right| \\
&\hspace{2cm} \leq 4A\sigma^{-2} e^{2C\sigma^{-2}A(1+A)} \|b^w - b^{w'}\|_{[-A,A]}.
\end{split}
\end{align}
We move on to $w\mapsto C_{b^w,\sigma}$ and show that the difference of $C_{b^w,\sigma}$ and $C_{b^{w'},\sigma}$ can also be bounded by a multiple of $\|b^w-b^{w'}\|_{[-A,A]}$. By the triangle inequality we directly get
\begin{align*}
&\left| C_{b^w, \sigma} - C_{b^{w'}, \sigma}\right|\\
&\hspace{0.3cm} \leq  \int_\R\left|\exp\left( \int_0^x 2\sigma^{-2} b^w(u) du \right) - \exp\left(\int_0^x 2\sigma^{-2} b^{w'}(u) du \right)\right| dx \\
&\hspace{0.3cm} = \int_\R \exp\left(\int_0^x 2\sigma^{-2} b_{0,\eta}(u) du \right)\\
&\hspace{2cm} \cdot \left|\exp\left( \int_0^x 2\sigma^{-2} g^w(u) du \right) - \exp\left(\int_0^x 2\sigma^{-2} g^{w'}(u) du \right)\right| dx\\
&\hspace{0.3cm} \leq C_{b_{0,\eta},\sigma} \left\|\exp\left( \int_0^\cdot 2\sigma^{-2} g^w(u) du \right) - \exp\left(\int_0^\cdot 2\sigma^{-2} g^{w'}(u) du \right)\right\|_\infty \\
&\hspace{0.3cm}= C_{b_{0,\eta},\sigma} \hspace{-0.05cm}\left\|\exp\left( \int_0^\cdot 2\sigma^{-2} g^w(u) du \right)\hspace{-0.05cm} -\hspace{-0.05cm} \exp\left(\int_0^\cdot 2\sigma^{-2} g^{w'}(u) du \right)\right\|_{[-A,A]},
\end{align*}
where the last step follows from the fact that all $g^w$ are supported within $[-A,A]$. Then we can proceed as above. In consequence, by \eqref{eqp: F3},
\begin{align}\label{eqp: F4}
\left| C_{b^w, \sigma} - C_{b^{w'}, \sigma}\right|  \leq  4A\sigma^{-2} C_{b_{0,\eta},\sigma}  e^{2C\sigma^{-2}A(1+A)}\|b^w - b^{w'}\|_{[-A,A]}. 
\end{align} 
We now turn to the evaluation of $\| b^w - b^{w'}\|_{[-A,A]}$. First, we have pointwise
\begin{align}\label{eqp: F2}
\begin{split}
&\left| b^w(x) - b^{w'}(x)\right| \\
&\hspace{1cm}\leq L(1-\epsilon_T) \left| (h_T^w)^\beta K_T^\beta \left(\frac{x-y^w}{h_T^w}\right) - (h_T^{w'})^\beta K_T^\beta\left(\frac{x-y^{w'}}{h_T^{w'}}\right)\right|\\
&\hspace{1cm}\leq L(1-\epsilon_T)\left| \left( (h_T^w)^\beta - (h_T^{w'})^\beta\right) K_T^\beta\left(\frac{x-y^w}{h_T^w}\right) \right| \\
&\hspace{2.5cm}+ (h_T^{w'})^\beta \left| K_T^\beta\left(\frac{x-y^w}{h_T^w}\right) - K_T^\beta\left(\frac{x-y^{w'}}{h_T^{w'}}\right) \right|.
\end{split}
\end{align}
The first summand is clearly bounded by $L | (h_T^w)^\beta - (h_T^{w'})^\beta|$ as we have $\|K_T^\beta\|_\infty \leq 1$ and $1-\epsilon_T\leq 1$. For the second summand, we distinguish the cases $\beta\leq 1$ and $\beta>1$. In the first one, the Hölder property gives the upper bound
\[ (h_T^{w'})^\beta\left| \frac{x-y^{w'}}{h_T	^{w'}} - \frac{x-y^w}{h_T^w}\right|  ^\beta =(h_T^{w'})^\beta \left| \frac{h_T^w ( x-y^{w'}) - h_T	^{w'} (x-y^w)}{h_T^{w'}h_T^w}\right|^\beta. \]
For $\beta>1$ we know that $K_T^\beta=K_\beta$ is differentiable and by the Hölder condition on its derivative, this derivative is continuous. Since the support of $K_\beta$ is bounded, $K_\beta'$ attains its maximum and $\|K_\beta'\|_\infty<\infty$. In this case we have the bound 
\[  (h_T^{w'})^\beta \|K_\beta'\|_\infty \left| \frac{h_T^w ( x-y^{w'}) - h_T	^{w'} (x-y^w)}{h_T^{w'}h_T^w}\right|. \]
Using the definition of $y^w$ we find
\begin{align*}
&h_T^w ( x-y^{w'}) - h_T	^{w'} (x-y^w)\\
&\hspace{1cm}=  (h_T^w - h_T^{w'})x  - h_T^w ( y^{w'} - y^w) + y^w (h_T^{w'} - h_T^w) \\
&\hspace{1cm}= (h_T^w - h_T^{w'})\left( x  + Rh_T^w - y^w\right),
\end{align*}
which yields, using $x,y^w\in [-A,A]$ and $h_T^w\leq 1$ for large enough $T$,
\[ \left| \frac{h_T^w ( x-y^{w'}) - h_T	^{w'} (x-y^w)}{h_T^{w'}h_T^w}\right| \leq \frac{R+2A}{h_T^w h_T^{w'}} \left| h_T^w - h_T^{w'}\right|. \]
Consequently, equation \eqref{eqp: F2} yields for $\beta\leq 1$
\begin{align*}
\|b^w - b^{w'}\|_{[-A,A]} \leq L \left| (h_T^w)^\beta - (h_T^{w'})^\beta\right| +\left(\frac{R+2A}{h_T^w}\right)^\beta \left| h_T^w - h_T^{w'}\right|^\beta
\end{align*}
and for $\beta>1$
\begin{align*}
\|b^w - b^{w'}\|_{[-A,A]}\leq  L\hspace{-0.05cm} \left| (h_T^w)^\beta - (h_T^{w'})^\beta\right|\hspace{-0.05cm} +\hspace{-0.05cm}\|K_\beta'\|_\infty (h_T^{w'})^{\beta-1} \frac{R+2A}{h_T^w} \left| h_T^w - h_T^{w'}\right|.
\end{align*}
In this last expression the term $(h_T^{w'})^{\beta-1}$ can be upper bounded by one for large $T$ since $\beta-1>0$. Finally, let $\epsilon>0$. In both cases $\beta\leq 1$ and $\beta>1$, we find $\delta=\delta(w)>0$ by continuity of $w\mapsto h_T^w$ such that for $|w-w'|<\delta$,
\[ \|b^w- b^{w'}\|_{[-A,A]} \leq \frac{\epsilon}{4A\sigma^{-2}(1\vee C_{b_{0,\eta},\sigma}) e^{2C\sigma^{-2}A(1+A)}}, \]
which finalises the proof of continuity of $w\mapsto \exp\left( \int_0^\cdot 2\sigma^{-2} b^w(u) du\right)$ and $w\mapsto C_{b^w}$ as the right-hand side of both \eqref{eqp: F3} and \eqref{eqp: F4} is bounded from above by $\epsilon$ for $|w-w'|<\delta$.

We turn to the second mapping in \eqref{eqp: map}. We add zero and get \pagebreak
\begin{align*}
&\left| c_Tq_{b^w}(y^w) - c_T q_{b^{w'}}(y^{w'})\right| \\[-2\jot]
&\hspace{1cm}\leq  c_T \left| q_{b^w}(y^w) - q_{b^w}(y^{w'})\right| + c_T\left| q_{b^w}(y^{w'}) - q_{b^{w'}}(y^{w'})\right|.
\end{align*} 
For arbitrary $\epsilon>0$ and appropriate $\delta>0$, the first summand is bounded by $\epsilon$ for $|y^w - y^{w'}|<\delta$, as evaluation of a continuous function is a continuous mapping. The second summand is trivially bounded by $\epsilon>0$ whenever $\| q_{b^w}-q_{b^{w'}}\|_{[-A,A]}< \frac{\epsilon}{c_T}$.
\end{proof}

\begin{lemma}[\cite{Brutsche}, Lemma 3.6.4]\label{lemma_quotient_densities}
Let the drift functions $b_{0,\eta}$ and $b^w$ be given as in \eqref{eqp: b_0_eta} and \eqref{eqp: b_w} and let $w\in [L_*/2, L^*]$. Then we have
\[ \left\| \frac{q_{b^w}}{q_{b_{0,\eta}}} \right\|_\infty \leq \exp\left( 4\sigma^{-2} L (h_T^w)^\beta \|K_T^\beta\|_{L^1}\right) \leq \exp\left( c\left(\frac{\log T}{T}\right)^\frac{\beta}{2\beta+1}\right) \]
for a constant $c>0$ not depending on $T$ and $w$.
\end{lemma}

Next, we present the proof of the crucial Proposition~\ref{lemma_E_absvalue} that was already presented in the sketch of proof of Theorem~\ref{lower_bound} in Section~\ref{Sec_supp}.

\begin{proof}[Proof of Proposition \ref{lemma_E_absvalue}]
By assumption, we have
\[ 1=\E[Z_i]  = \E[Z_i \1_{\{Z_i \leq \epsilon m\}}] + \E[Z_i \1_{\{Z_i > \epsilon m\}}]. \]
Together with the fact that $Z_i > 0$,
\begin{align*}
&\E\left[\left| \frac1m \sum_{i=1}^m Z_i -1\right|\right] 
= \E\left[ \left| \frac1m\sum_{i=1}^m \left( Z_i \1_{\{Z_i \leq \epsilon m\}} - \E[ Z_i \1_{\{Z_i \leq  \epsilon m\}}]\right) \right.\right. \\
&\hspace{5cm} \left.\left. + \frac1m\sum_{i=1}^m  \left( Z_i \1_{\{Z_i > \epsilon m\}} - \E[Z_i \1_{\{Z_i > \epsilon m\}}]\right) \right|\right] \\
&\hspace{0.1cm}\leq \frac1m \textrm{Var}\left( \sum_{i=1}^m  Z_i \1_{\{Z_i \leq \epsilon m\}}\right)^\frac12 + \frac1m \E\left[ \sum_{i=1}^m \left| Z_i \1_{\{ Z_i>\epsilon m\}} \right| +  \left|\E[Z_i \1_{\{Z_i > \epsilon m\}}] \right| \right]\\
&\hspace{0.1cm}\leq \frac1m \left( \sum_{i=1}^m \textrm{Var}\left( Z_i \1_{\{ Z_i\leq \epsilon m\}}\right) + \sum_{i\neq j} \textrm{Cov}\left(  Z_i \1_{\{ Z_i\leq \epsilon m\}},  Z_j \1_{\{ Z_j\leq \epsilon m\}}\right)\right)^\frac12 \\
&\hspace{2.5cm} + \frac2m \sum_{i=1}^m \E\left[ Z_i \1_{\{ Z_i>\epsilon m\}}\right]\\
&\hspace{0.1cm} =: \frac1m\left( A_1 + A_2\right)^\frac12 + A_3.
\end{align*}
Now, we treat each term $A_i$, $i=1,2,3$, separately. As a variance, the term $A_1+A_2$ is non-negative and it suffices to find an upper bound for it without taking the absolute value. In particular, we may drop negative summands without double-checking that they are in absolute value smaller than the remaining positive ones. Using $\E[Z_i]=1$, for $A_1$ a bound is straightforward as
\[ A_1 \leq \sum_{i=1}^m \E\left[ Z_i^2 \1_{\{ Z_i\leq \epsilon m\}}\right] \leq \epsilon m \sum_{i=1}^m \E\left[ Z_i \1_{\{ Z_i\leq \epsilon m\}}\right] \leq \epsilon m\sum_{i=1}^m \E[Z_i] = \epsilon m^2.\]
For the second term $A_2$, we first note that $\1_{\{ Z_i\leq \epsilon m\}} = 1-\1_{\{ Z_i> \epsilon m\}}$ and similarly $\1_{\{ Z_i\leq \epsilon m\}}\1_{\{ Z_j\leq \epsilon m\}} = 1- A_{ij}$ with 
\begin{align*}
A_{ij} :=\1_{\{ Z_i\leq \epsilon m\}}\1_{\{ Z_j > \epsilon m\}} + \1_{\{ Z_i> \epsilon m\}}\1_{\{ Z_j\leq \epsilon m\}} +\1_{\{ Z_i> \epsilon m\}}\1_{\{ Z_j> \epsilon m\}}.
\end{align*} 
Now, we can start with 
\begin{align*}
A_2 &= \sum_{i\neq j}  \left( \E\left[ Z_iZ_j \1_{\{ Z_i\leq \epsilon m\}} \1_{\{ Z_j\leq \epsilon m\}}\right]  - \E\left[Z_i \1_{\{ Z_i\leq \epsilon m\}}\right]\E\left[Z_j \1_{\{ Z_j\leq \epsilon m\}}\right] \right) \\
&= \sum_{i\neq j} \left( \E\left[ Z_iZ_j ( 1- A_{ij})\right] - \E\left[ Z_i (1-\1_{\{ Z_i>\epsilon m\}})\right]\E\left[ Z_j (1-\1_{\{ Z_j> \epsilon m\}})\right] \right)\\
&= - \sum_{i\neq j} \left(\E\left[Z_iZ_j A_{ij}\right] +\E[Z_i \1_{\{ Z_i>\epsilon m\}}]\E[Z_j \1_{\{ Z_j>\epsilon m\}}] \right)\\
&\hspace{1cm}  + \sum_{i\neq j} \left( \E[Z_i \1_{\{ Z_i>\epsilon m\}}]+\E[Z_j \1_{\{ Z_j>\epsilon m\}}] \right) + \sum_{i\neq j} \left(\E[Z_iZ_j]-1 \right).
\end{align*}
As $Z_i> 0$ and $A_{ij}\geq 0$, the first summand is non-positive and by assumption,
\[ \sum_{i\neq j} \left(\E[Z_iZ_j]-1\right) \leq m(m-1) (C_0-1),\]
as well as
\[ \sum_{i\neq j} \E[Z_i \1_{\{ Z_i>\epsilon m\}}]+\E[Z_j \1_{\{ Z_j>\epsilon m\}}] = 2(m-1)\sum_{i=1}^m \E[Z_i \1_{\{ Z_i>\epsilon m\}}], \]
and we get
\[ A_2 \leq  2m\sum_{i=1}^m \E[Z_i \1_{\{ Z_i>\epsilon m\}}] + m^2 (C_0-1). \]
Now we use the indicator in the last remaining expectation to estimate for arbitrary $0<\nu\leq 1$,
\begin{align}\label{eqp: E4}
\E[Z_i \1_{\{ Z_i>\epsilon m\}}] \leq (\epsilon m)^{-\nu} \E\left[Z_i^{1+\nu} \1_{\{ Z_i>\epsilon m\}}\right]\leq  (\epsilon m)^{-\nu} \E\left[Z_i^{1+\nu}\right]. 
\end{align} 
Altogether, we arrive at 
\[ A_2 \leq 2\epsilon^{-\nu} m^{1-\nu} \sum_{i=1}^m\E\left[Z_i^{1+\nu}\right] + m^2 (C_0-1).\]
Using $\sqrt{a+b+c}\leq \sqrt{a}+\sqrt{b}+\sqrt{c}$ for $a,b,c\geq 0$, 
\begin{align*}
\frac1m(A_1 + A_2)^\frac12 &\leq \frac1m\left( \epsilon m^2 + 2\epsilon^{-\nu} m^{1-\nu} \sum_{i=1}^m\E\left[Z_i^{1+\nu}\right] + m^2 (C_0-1)\right)^\frac12\\
&\leq \frac1m\left( \sqrt{\epsilon} m + \left(2\epsilon^{-\nu} m^{1-\nu} \sum_{i=1}^m\E\left[Z_i^{1+\nu}\right]\right)^\frac12 + m\sqrt{C_0-1} \right)\\
&= \sqrt{\epsilon} + \left( 2\epsilon^{-\nu} m^{-(1+\nu)} \sum_{i=1}^m\E\left[Z_i^{1+\nu}\right]\right)^\frac12 +\sqrt{C_0-1},
\end{align*}
which are exactly the first to summands in the assertion of the proposition. We close the proof by the estimation of $A_3$. But here we simply use the same trick with $0<\nu\leq 1$ from \eqref{eqp: E4} and directly get
\[ A_3 \leq 2\epsilon^{-\nu} m^{-(1+\nu)} \sum_{k=1}^m \E\left[Z_i^{1+\nu}\right].\]
\end{proof}

\begin{remark}\label{remark_E_absvalue}
Proposition~\ref{lemma_E_absvalue} is an extension of a similar result proven by Dümbgen and Walther in \cite{Walther} (within the proof of Lemma~$7.4$ on p.~$1777$) and avoids the assumption of independent random variables. If we additionally assume that the random variables $Z_i$ are uncorrelated, we have $\E[Z_i Z_j]=1$ for all $1\leq i<j\leq m$ and the term $\sqrt{C_0-1}$ vanishes. Even this makes the argument trickier as the uncorrelatedness does not carry over to truncated versions of the random variables.
\end{remark}

\begin{proof}[Proof of Theorem \ref{lower_bound}]
We begin by building suitable hypotheses. Denote the support of the optimal recovery kernel solving \eqref{eq: optimal_recovery} by $[-R,R]$, in particular $R=1$ for $\beta\leq 1$. Then choose $b_0 \in\Sigma\left( \frac{C}{2}-\eta ,A, \gamma +\frac{\eta}{\sigma^2}, \sigma\right)$ for which $b_{0,\eta}=b_0+\eta\in H_0(b_0,\eta)$ and remember $h_T^w$ defined in~\eqref{eqp: h_T_w}, $b^w$ defined in \eqref{eqp: b_w} and $A'=A-h_T^{L_*}$. Setting $c_T :=K_T^\beta(0)^{(2\beta+1)/\beta}$, we proceed inductively:
\begin{itemize}
\item In the first step we set $y_1^w := -A' + Rh_T^w$. Then by Lemma~\ref{lemma_fixed_point} there exists $w_1>0$ with
\[ w_1=c_T q_{b^{w_1}}(y_1^{w_1})\]
and we set 
\[ b_1 := b^{w_1},\quad y_1:=y_1^{w_1}. \]
The support of $b_1-b_{0,\eta}$ is given by the interval $[-A', -A' + 2Rh_T^{w_1}] = [y_1-Rh_T^{w_1}, y_1+Rh_T^{w_1}]$.
\item If $b_j$ is constructed with $y_j$ we put $y_{j+1}^w := y_j + Rh_T^{w_j}+ Rh_T^w$. Again, Lemma~\ref{lemma_fixed_point} ensures existence of $w_{j+1}>0$ with 
\[ w_{j+1}=c_T q_{b^{w_{j+1}}}(y_{j+1}^{w_{j+1}})\]
and we set 
\[ b_{j+1} := b^{w_{j+1}}, \quad y_{j+1}:= y_{j+1}^{w_{j+1}}. \]
The support of $b_{j+1}-b_{0,\eta}$ is given by $[y_j + Rh_T^{w_j}, y_j + Rh_T^{w_j}+2Rh_T^{w_{j+1}}] = [y_{j+1}- Rh_T^{w_{j+1}}, y_{j+1}+Rh_T^{w_{j+1}}]$ and hence is disjoint of the support of all $b_j-b_{0,\eta}$ constructed before.
\item The construction ends when for the first time the right endpoint of the newly constructed interval $[y_N - Rh_T^{w_N}, y_N+Rh_T^{w_N}]$ is greater than $A'$. If this right endpoint is smaller than $A$, we consider $b_1,\dots, b_N$, otherwise, we only use $b_1,\dots, b_{N-1}$.
\end{itemize}

In what follows, we assume that we constructed $N$ hypotheses for notational simplicity. In fact, nothing changes if we only have $b_1,\dots, b_{N-1}$ in the last step of the construction process, as their number has the same order with respect to $T$.

%

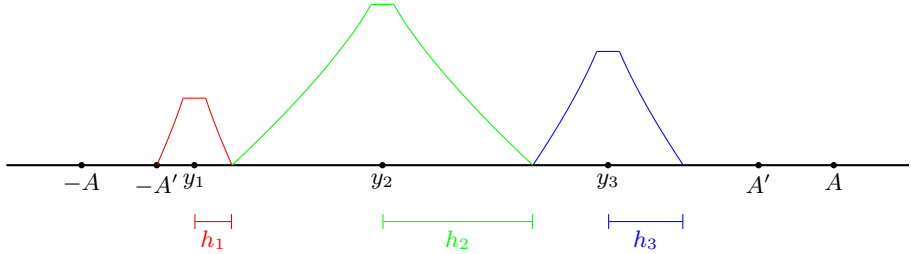
\begin{figure}[h]
\begin{tikzpicture}
\draw [thick] (0,0) -- (12,0);
\filldraw[black] (1,0) circle (1pt) node[anchor=north] {$-A$};
\filldraw[black] (11,0) circle (1pt) node[anchor=north] {$A$};
\filldraw[black] (2,0) circle (1pt) node[anchor=north] {$-A'$};
\filldraw[black] (10,0) circle (1pt) node[anchor=north] {$A'$};
\filldraw[black] (2.5,0) circle (1pt) node[anchor=north] {$y_1$};
\filldraw[black] (5,0) circle (1pt) node[anchor=north] {$y_2$};
\filldraw[black] (8,0) circle (1pt) node[anchor=north] {$y_3$};

\draw[|-|, red] (2.5,-0.75)--(3,-0.75) node[below,midway]{$h_1$};
\draw[|-|, green] (5,-0.75)--(7,-0.75) node[below,midway]{$h_2$};
\draw[|-|, blue] (8,-0.75)--(9,-0.75) node[below,midway]{$h_3$};

\draw[red, domain=2:2.35] plot (\x,{1.5*(1-(2*(2.5-\x))^(0.75))});
\draw[red, domain=2.35:2.65] plot (\x,{1.5*(1-0.3^(0.75))});
\draw[red, domain=2.65:3] plot (\x,{1.5*(1-(2*(\x-2.5))^(0.75))});

\draw[green, domain=3:4.85] plot (\x,{2.5*(1-((5-\x)/2)^(0.75))});
\draw[green, domain=4.85:5.15] plot (\x,{2.5*(1-0.075^(0.75))});
\draw[green, domain=5.15:7] plot (\x,{2.5*(1-((\x-5)/2)^(0.75))});

\draw[blue, domain=7:7.85] plot (\x,{2*(1-(8-\x)^(0.75))});
\draw[blue, domain=7.85:8.15] plot (\x,{2*(1-0.15^(0.75))});
\draw[blue, domain=8.15:9] plot (\x,{2*(1-(\x-8)^(0.75))});
\end{tikzpicture}
\caption{\small The construction principle of our hats illustrated with the kernel $K_{0.75}$. First, the red hat is constructed from $-A'$, then the green one starting at $y_1+h_1$ and finally the blue hat. One can see that depending on the solution $w$ of each fixed point problem, the height and width of each hat varies.}
\end{figure}

We have $w_j\leq c_T L^*\leq L^*$ for all $j=1,\dots, N$ by Lemma~\ref{lemma_fixed_point}. Next, we specify the order of $N$ that depends on $T$. From the construction of the hypotheses we have $2A \geq \sum_{j=1}^N 2Rh_T^{w_j}$. Hence, 
\begin{align*}
\frac{A}{R} \geq \sum_{j=1}^N h_T^{w_j}
\geq \left(\frac{c^*}{L}\right)^\frac{1}{\beta}\left(\frac{\log T}{T}\right)^{\frac{1}{2\beta +1}} N \left(\frac{\sigma^2}{L^*}\right)^{\frac{1}{2\beta+1}},
\end{align*}
which is equivalent to 
\[ N =N_T \leq  \frac{A}{R}\left(\frac{L^*}{\sigma^2}\right)^{\frac{1}{2\beta+1}} \left(\frac{L}{c^*}\right)^\frac{1}{\beta}\left(\frac{T}{\log T}\right)^{\frac{1}{2\beta +1}} .\]
Moreover, we know by construction that $A\leq \sum_{j=1}^N 2Rh_T^{w_j}$ and the same computation shows with upper bounding $w_j^{-1} \leq (c_T L_*)^{-1}\leq 2(L_*)^{-1}$ for $T$ large enough, that
\[ \frac{A}{2R} \leq \left(\frac{c^*}{L}\right)^\frac{1}{\beta}\left(\frac{\log T}{T}\right)^{\frac{1}{2\beta +1}} N \left(\frac{2\sigma^2}{L_*}\right)^{\frac{1}{2\beta+1}}, \]
or equivalently,
\[ N =N_T \geq \frac{A}{2R}\left(\frac{L_*}{2\sigma^2}\right)^{\frac{1}{2\beta+1}} \left(\frac{L}{c^*}\right)^\frac{1}{\beta}\left(\frac{T}{\log T}\right)^{\frac{1}{2\beta +1}} .\]
Hence, we know the order of $N_T$. \\

For the hypothesis $b_k$, we have by construction $\Delta_J(b_k)=(1-\epsilon_T)c_*\delta_T$ and in combination with Lemma \ref{lemma_cond_1} that $b_k\in H_1(b_0,\eta)$ for $k=1,\dots, N$. By Lemma \ref{lemma_cond_2}, $b_k-b_{0}\in\mathcal{H}(\beta,L)$ for $k=1,\dots, N$. Then we have for any test $\psi$ with $\sup_{b\in H_0(b_0,\eta)}\E_{b}[\psi]\leq\alpha$,
\begin{align*}
\inf_{\substack{b\in H_1(b_0,\eta) \cap\{ b-b_0\in \mathcal{H}(\beta, L)\}:\\ \Delta_J(b)\geq (1-\epsilon_T)c_*\delta_T}} \E_b[\psi] - \alpha &\leq \min_{1\leq k\leq N} \E_{b_k}[\psi] - \alpha\\
&\leq \min_{1\leq k\leq N} \E_{b_k}[\psi] - \sup_{b\in H_0(b_0,\eta)} \E_b[\psi]\\
&\leq \min_{1\leq k\leq N} \E_{b_k}[\psi] -  \E_{b_{0,\eta}}[\psi]\\
&\leq \E_{b_{0,\eta}}\left[ \left(\frac1N \sum_{k=1}^N \frac{d\Pr_{b_k}}{d\Pr_{b_{0,\eta}}}(X) - 1\right) \psi\right]\\
&\leq  \E_{b_{0,\eta}}\left[ \left| \frac1N \sum_{k=1}^N \frac{d\Pr_{b_k}}{d\Pr_{b_{0,\eta}}}(X) - 1\right|\right].
\end{align*}
Note that we used in the third step that $b_{0,\eta}\in H_0(b_0,\eta)$.
For ease of notation we introduce
\[ Z_k := \frac{d\Pr_{b_k}}{d\Pr_{b_{0,\eta}}} (X),  \]
which is given by Girsanov's theorem as
\[ \frac{q_{b_k}(X_0)}{q_{b_{0,\eta}}(X_0)} \exp\left( \int_0^T A_{T,k} K_{T,k}(X_s) dW_s - \frac12\int_0^T A_{T,k}^2 K_{T,k}(X_s)^2 ds\right), \]
where we used the notation 
\begin{align}\label{eqp: 5.1_notation}
A_{T,k} := \frac{(1-\epsilon_T) L\left( h_T^{w_k}\right)^\beta}{\sigma}\quad \textrm{ and }\quad K_{T,k}(\cdot) := K_{y_k, h_T^{w_k}}(\cdot).
\end{align} 
In the expression  $K_{y_k, h_T^{w_k}}(\cdot)$ we use the rescaled version of the kernel $K_T^\beta$, which is suppressed in the notation $K_{T,k}$. The dependence on $\beta$ is also not made explicit in this notation, but of course should be kept in mind.
As $Z_k>0$ and $\E_{b_{0,\eta}}[Z_k] = 1$ for each $1\leq k\leq N$, we now apply Proposition~\ref{lemma_E_absvalue}. To get 
\[\E_{b_{0,\eta}}\left[ \left| \frac1N \sum_{k=1}^N \frac{d\Pr_{b_k}}{d\Pr_{b_{0,\eta}}}(X) - 1\right|\right] \rightarrow 0\]
for $T\to\infty$ from this Proposition~\ref{lemma_E_absvalue}, it remains to show that for all $\epsilon>0$,
\begin{align}\label{eq_condition_lower}
\epsilon^{-\nu_T} N_T^{-(1+\nu_T)} \sum_{k=1}^{N_T} \E_{b_0}\left[ Z_k^{1+\nu_T}\right] \rightarrow 0\quad \textrm{ and }\quad \sqrt{C_T-1}\rightarrow 0
\end{align} 
for a suitable choice of $0<\nu=\nu_T\leq 1$ and $T\to\infty$, where $C_T$ is an upper bound of $\E[Z_i Z_j]$ for all $1\leq i<j\leq N$. The latter condition means that $Z_i$ and $Z_j$, $i\neq j$, are asymptotically uncorrelated.\\

We start with the second condition in \eqref{eq_condition_lower} about $C_T$ and estimate $\E_{b_{0,\eta}}[Z_iZ_j]$ for $i\neq j$. In this case we have $K_{T,i}(X_s) K_{T,j}(X_s) =0$ as the supports of $K_{T,i}$ and $K_{T,j}$ are disjoint. This will be used in the third step below. Moreover, we control the maximal value of the fraction of the invariant densities with Lemma~\ref{lemma_quotient_densities}, where $a_T=\left(\log T/T\right)^{\beta/(2\beta+1)}$ and $c>0$: 
\begin{align*}
&\E_{b_{0,\eta}}[Z_i Z_j]\\
&\hspace{0.3cm}= \E_{b_{0,\eta}}\left[ \frac{q_{b_i}(X_0)}{q_{b_{0,\eta}}(X_0)}\exp\left( \int_0^T A_{T,i} K_{T,i}(X_s) dW_s - \frac12\int_0^T A_{T,i}^2 K_{T,i}^2(X_s)^2 ds\right) \right.\\
&\hspace{1.2cm} \cdot\left. \frac{q_{b_j}(X_0)}{q_{b_{0,\eta}}(X_0)}\exp\left( \int_0^T A_{T,j}K_{T,j}(X_s) dW_s - \frac12\int_0^T A_{T,j}^2 K_{T,j}(X_s)^2 ds\right)\right]\\
&\hspace{0.3cm}\leq e^{2ca_T} \E_{b_{0,\eta}}\left[ \exp\left( \int_0^T \left(A_{T_i}K_{T,i}(X_s) + A_{T,j}K_{T,j}(X_s)\right) dW_s \right.\right.\\
&\hspace{2.5cm}  \left.\left. - \frac12\int_0^T \left( A_{T,i}^2 K_{T,i}(X_s)^2 + A_{T,j}^2 K_{T,j}(X_s)^2 \right) ds\right)\right]\\
&\hspace{0.3cm}= e^{2ca_T}  \E_{b_{0,\eta}}\left[ \exp\left( \int_0^T \left(A_{T,i}K_{T,i}(X_s) + A_{T,j}K_{T,j}(X_s)\right) dW_s \right.\right.\\
&\hspace{2.5cm}  \left.\left. - \frac12\int_0^T \left( A_{T,i} K_{T,i}(X_s) + A_{T,j} K_{T,j}(X_s) \right)^2 ds\right)\right] \\
&\hspace{0.3cm}= e^{2ca_T}  \E_{b_{0,\eta}}\left[\E_{b_{0,\eta}}\left[ \exp\left( \int_0^T \left(A_{T,i}K_{T,i}(X_s) + A_{T,j}K_{T,j}(X_s)\right) dW_s \right.\right.\right.\\
&\hspace{2.5cm}  \left.\left.\left.\left. - \frac12\int_0^T \left( A_{T,i} K_{T,i}(X_s) + A_{T,j} K_{T,j}(X_s) \right)^2 ds\right)\ \right|X_0\right]\right].
\end{align*}
The last expression within the expectation is again a Girsanov-type density of the diffusions started in $X_0$. By boundedness of $A_{T,i}K_{T,i}$ it follows that Novikov's condition (cf. \cite{Kallenberg}, Theorem~$19.24$) holds for the martingale
\[ \left( \int_0^t \left(A_{T,i}K_{T,i}(X_s) + A_{T,j}K_{T,j}(X_s)\right) dW_s\right)_{t\in [0,T]}\]
and hence its stochastic exponential, which coincides with the density process, is a uniformly integrable martingale. In particular, its expectation equals one. Consequently, $\E_{b_{0,\eta}}[Z_i Z_j]\leq e^{2ca_T}$. Note that this bound is $\geq 1$ and we have $\sqrt{e^{2ca_T}-1}\rightarrow 0$ for $T\to\infty$, as desired.\\

The rest of the proof is dedicated to show the first convergence in (\ref{eq_condition_lower}). In a first step, we investigate one single summand $\E_{b_{0,\eta}}[ Z_k^{1+\nu}]$ that is given by 
\begin{align}\label{eqp: E5}
\begin{split}
&\E_{b_{0,\eta}}\left[ \left(\frac{q_{b_k}(X_0)}{q_{b_{0,\eta}}(X_0)}\right)^{1+\nu} \exp\left( (1+\nu)\int_0^T A_{T,k} K_{T,k}(X_s) dW_s \right. \right. \\
&\hspace{5.0cm}\left.\left. - \frac{1+\nu}{2}\int_0^T A_{T,k}^2 K_{T,k}(X_s)^2 ds\right)\right].
\end{split}
\end{align} 
We now split the second integral in the exponent and get with the occupation times formula 
\begin{align*}
&-\frac{1+\nu}{2}\int_0^T A_{T,k}^2 K_{T,k}(X_s)^2 ds\\
& = -\frac{(1+\nu)^2}{2}\hspace{-0.05cm}\int_0^T A_{T,k}^2 K_{T,k}(X_s)^2 ds + \frac{T}{2} \nu(1+\nu)\hspace{-0.05cm}\int_\R A_{T,k}^2 K_{T,k}(z)^2 q_{b_{0,\eta}}(y_k) dz \\
&\hspace{2cm} + \frac{T}{2} \nu(1+\nu)\int_\R A_{T,k}^2 K_{T,k}(z)^2 \left(\frac{1}{\sigma^2 T} L_T^z(X) - q_{b_{0,\eta}}(z)\right) dz \\
&\hspace{2cm}+ \frac{T}{2} \nu(1+\nu)\int_\R A_{T,k}^2 K_{T,k}(z)^2 \left(q_{b_{0,\eta}}(z) - q_{b_{0,\eta}}(y_k)\right)dz.
\end{align*}
Inserting into the formula for $\E_{b_0}[Z_k^{1+\nu}]$ in \eqref{eqp: E5}, pulling the non-random factors out of the integral and bounding the fraction of the densities leads to
\begin{align*}
&\E_{b_{0,\eta}}[Z_k^{1+\nu}] \leq A_1\cdot A_2\cdot A_3 \cdot \E_{b_{0,\eta}}\left[ \exp\left( (1+\nu)\int_0^T A_{T,k} K_{T,k}(X_s) dW_s \right. \right. \\
&\hspace{4cm} - \frac{(1+\nu)^2}{2} \int_0^T A_{T,k}^2 K_{T,k}(X_s)^2 ds \\
&\hspace{2.2cm}\left.\left.+ \frac{T}{2}\nu(1+\nu)\int_\R A_{T,k}^2 K_{T,k}(z)^2 \left(\frac{1}{\sigma^2 T} L_T^z(X) - q_{b_{0,\eta}}(z)\right) dz\right)\right],
\end{align*}
where 
\begin{align*}
A_1 &:= \exp\left( \frac{T}{2} \nu(1+\nu)\int_\R A_{T,k}^2 K_{T,k}(z)^2 q_{b_{0,\eta}}(y_k) dz\right),\\
A_2 &:=\exp\left(  \frac{T}{2} \nu(1+\nu)\int_\R A_{T,k}^2 K_{T,k}(z)^2 \left(q_{b_{0,\eta}}(z) - q_{b_{0,\eta}}(y_k)\right)dz\right),\ \ \textrm{ and }\\
A_3 &:= \left\| q_{b_k}/q_{b_{0,\eta}}\right\|_\R^{1+\nu}.
\end{align*}
To further evaluate the term within the expectation, we again add a suitable zero. Therefore, pick $p=p_T>1$ that will be specified later and the corresponding $q=q_T$ with $\frac1p+\frac1q =1$. Then we rewrite the expectation in the following form:
\begin{align*}
&\E_{b_{0,\eta}}\hspace{-0.05cm}\left[ \exp\left(\hspace{-0.05cm} (1+\nu)\hspace{-0.05cm}\int_0^T \hspace{-0.05cm}A_{T,k}K_{T,k}(X_s) dW_s - \frac{p(1+\nu)^2}{2}\hspace{-0.05cm} \int_0^T A_{T,k}^2 K_{T,k}(X_s)^2 ds\right) \right.\\
&\hspace{1cm}  \cdot\exp\left( \frac12 (p-1)(1+\nu)^2 \int_0^T A_{T,k}^2 K_{T,k}(X_s)^2 ds \right.\\
&\hspace{2.5cm} \left.\left. + \frac{T}{2}\nu(1+\nu) \int_\R A_{T,k}^2 K_{T,k}(z)^2 \left(\frac1T L_T^z(X) - q_{b_{0,\eta}}(z)\right) dz\right)\right]\\
&= \E_{b_{0,\eta}}\left[ \exp\left( (1+\nu)\int_0^T A_{T,k}K_{T,k}(X_s) dW_s \right.\right. \\
&\hspace{5cm} \left.\left. - \frac{p(1+\nu)^2}{2}\int_0^T A_{T,k}^2 K_{T,k}(X_s)^2 ds\right) \right.\\
&\hspace{2cm} \cdot \exp\left( \frac{T}{2} (p-1) (1+\nu)^2  \int_\R A_{T,k}^2 K_{T,k}(z)^2 q_{b_{0,\eta}}(z) dz \right.\\
&\hspace{3.5cm} \left. \left. +\left( \frac{T}{2} (p-1)(1+\nu)^2 + \frac{T}{2}\nu(1+\nu)\right)  \right.\right. \\
&\hspace{4.5cm} \left.\left.  \cdot\int_\R A_{T,k}^2 K_{T,k}(z)^2\left(\frac1T L_T^z(X)-q_{b_{0,\eta}}(z)\right) dz\right)\right].
\end{align*}
Here, we first applied the occupation times formula and then splitted $\frac{1}{\sigma^2 T} L_T^z(X) = q_{b_{0,\eta}}(z) + \left(\frac{1}{\sigma^2 T} L_T^z(X) - q_{b_{0,\eta}}(z)\right)$. Pulling the non-random factor out of the expectation and applying Hölder's inequality yields
\[ \E_{b_{0,\eta}}[Z_k^{1+\nu}]  = A_1\cdot A_2\cdot A_3\cdot A_4\cdot A_5\cdot A_6,\]
where
\begin{align*}
A_4 &:= \E_{b_{0,\eta}}\left[ \exp\left( p(1+\nu)\int_0^T A_{T,k}K_{T,k}(X_s) dW_s \right.\right. \\
&\hspace{4cm} \left.\left. - \frac{p^2(1+\nu)^2}{2}\int_0^T A_{T,k}^2 K_{T,k}(X_s)^2 ds\right) \right]^\frac1p, \\
A_5&:= \exp\left( \frac{T}{2} (p-1) (1+\nu)^2  \int_\R A_{T,k}^2 K_{T,k}(z)^2 q_{b_{0,\eta}}(z) dz \right),\quad \textrm{ and } \\
A_6&:= \E_{b_{0,\eta}}\left[ \exp\left( q\left(\frac{T}{2} (p-1)(1+\nu)^2 + \frac{T}{2}\nu(1+\nu)\right)\right.\right. \\
&\hspace{2.5cm}\left.\left.\cdot  \int_\R A_{T,k}^2 K_{T,k}(z)^2\left(\frac{1}{\sigma^2 T} L_T^z(X)-q_{b_{0,\eta}}(z)\right) dz\right)\right]^\frac1q.
\end{align*}
The most interesting term will be $A_1$ and we consider it later within the sum (remember that we still just treat $Z_k^{1+\nu}$ for some fixed $k$ here). The terms $A_2,\dots, A_6$ on the other side will turn out to be bounded from above by constants independent of $k$. We will use the definition of $A_{T,k}$ and $K_{T,k}$ from \eqref{eqp: 5.1_notation} in the following estimates without further notice: 
\begin{itemize}
\item[$A_2:$] According to Lemma \ref{bound_invariant_density}, the invariant density is Lipschitz continuous with constant $L^*$, and we have $|q_{b_{0,\eta}}(y_k)-q_{b_{0,\eta}}(z)|\leq 2L^* h_T^{w_k}$ on the support of $K_{T,k}$. Hence,
\begin{align*}
A_2 &= \exp\left( \frac12\nu(1+\nu)T A_{T,k}^2 \int_\R K_{T,k}(z)^2 \left(q_{b_{0,\eta}}(z) - q_{b_{0,\eta}}(y_k)\right) dz\right)\\
&\leq \exp\left( \frac12\nu(1+\nu) T A_{T,k}^2 2L^*(h_T^{w_k})^2 \|K_T^\beta\|_{L^2}^2\right)\\
&\leq \exp\left(2\|K_\beta\|_{L^2}^2L^* L^2 (1-\epsilon_T)^2 \left(\frac{c_*}{L}\right)^{\frac{2\beta+1}{\beta}} Th_T^{w_k} \frac{\log T}{T w_k}\right).
\end{align*}
The last step used $\nu\leq 1$. This upper bound of $A_2$ converges to one, as the exponent goes to zero due to the fact that $w_k$ is bounded from below by $L_*/2$ for $T$ large enough, $(1-\epsilon_T)\rightarrow 1$ and $h_T^{w_k} \log T\rightarrow 0$ as for some $\tilde{c}>0$,
\[ h_T^{w_k} \log T \leq \tilde{c} \left( \frac{\log T}{T}\right)^{\frac{1}{2\beta+1}} \log T = T^{-\frac{1}{2\beta+1}} \left( \log T\right)^{\frac{2\beta +2}{2\beta +1}}.\]

\item[$A_3:$] Let $T$ be large enough such that $c_T=K_T^\beta(0)^{(2\beta+1)/\beta} \geq \frac12$. Then by Lemma~\ref{lemma_quotient_densities} and $\nu\leq 1$, 
\[ A_3 \leq \exp\left( 2c\left(\frac{\log T}{T}\right)^\frac{\beta}{2\beta+1}\right), \]
where $c$ does not depend on $T$ and $k$. This term obviously tends to one for $T\to\infty$.

\item[$A_4:$] For every $T>0$, the term $A_4$ is the expectation of a Girsanov-type density. By boundedness of the integrand $A_{T,k}K_{T,k}$, Novikov's condition (cf. \cite{Kallenberg}, Theorem~$19.24$) holds for the martingale
\[ \left( p(1+\nu)\int_0^t A_{T,k}K_{T,k}(X_s) dW_s\right)_{t\in [0,T]}\]
and hence its stochastic exponential, which coincides with the term inside the expectation in $A_4$, is a uniformly integrable martingale. In particular, its expectation equals one, i.e. $A_4 =1$.

\item[$A_5:$] Here, we use again Lemma~\ref{bound_invariant_density} to get $\|q_{b_{0,\eta}}\|_\infty\leq L^*$. Then we estimate
\begin{align*}
&A_5 = \exp\left( \frac12 (p-1) (1+\nu)^2 T A_{T,k}^2 \int_\R K_{T,k}(z)^2 q_{b_{0,\eta}}(z) dz\right)\\
&\leq \exp\left( \frac12 (p-1)(1+\nu)^2 L^* \|K_T^\beta\|_{L^2}^2 T A_{T,k}^2 h_T^{w_k}\right)\\
&= \exp\left( \frac12 (1+\nu)^2 L^* \|K_T^\beta\|_{L^2}^2 L^2 \left(\frac{c_*}{L}\right)^{\frac{2\beta+1}{\beta}} w_k^{-1} (1-\epsilon_T)^2 T(p-1)\frac{\log T}{T}\right)\\
&\leq \exp\left( 2 L^* \|K_\beta\|_{L^2}^2 L^2 \left(\frac{c_*}{L}\right)^{\frac{2\beta+1}{\beta}} w_k^{-1} (1-\epsilon_T)^2 (p-1)\log T\right),
\end{align*}
where we used $(1+\nu)\leq 2$ in the last step. Remember that $\epsilon_T\to 0$ and $w_k\in [L_*/2, L^*]$ for $T$ large enough. Consequently, the last exponent converges to zero if $(p-1)\log T\rightarrow 0$ which is for example true, if we set 
\[ p= p_T = 1+ T^{-\frac14}.\]
In particular, we have $p>1$, which is import as we applied the Hölder inequality for it. Of course, many other choices of $p$ would be equally reasonable here, but we will see in the next step, that $p$ must not be too close to one and our choice perfectly applies for both $A_5$ and $A_6$.

\item[$A_6:$] In this case we have
\begin{align*}
A_6 &= \E_{b_{0,\eta}}\left[ \exp\left( q\left(\frac12 (p-1)(1+\nu)^2 + \frac12\nu(1+\nu)\right) TA_{T,k}^2 \right. \right. \\
&\hspace{1.5cm} \left. \left. \cdot \int_\R  K_{T,k}(z)^2\left(\frac{1}{\sigma^2 T} L_T^z(X)-q_{b_{0,\eta}}(z)\right) dz\right)\right]^\frac1q\\
&\leq \E_{b_{0,\eta}}\left[ \exp\left(\frac12 (1+\nu)c_p \|K_T^\beta\|_{L^2}^2 L^2 \left(\frac{c_*}{L}\right)^\frac{2\beta+1}{\beta} w_k^{-1} \right.\right. \\
&\hspace{3.8cm}\left.\left.\cdot (1-\epsilon_T)^2 \log T\left\|\frac{1}{\sigma^2 T} L_T(X) - q_{b_{0,\eta}}\right\|_{\infty}\right)\right]^\frac1q\\
&\leq \E_{b_{0,\eta}}\left[ \exp\left( \tilde{c} c_p \log T \left\|\frac{1}{\sigma^2 T} L_T(X) - q_{b_{0,\eta}}\right\|_{\infty}\right)\right]^\frac1q,
\end{align*}
with
\begin{align*}
\tilde{c} &:= 2\|K_\beta\|_{L^2}^2 L^2 \left(\frac{c_*}{L}\right)^\frac{2\beta +1}{\beta}L_*^{-1}, \\
c_{p}&:= q((p-1)(1+\nu) + \nu).
\end{align*}
In the estimation, we bounded $\nu\leq 1$ and $\|K_T\|_{L^2}^2\leq \|K_\beta\|_{L^2}^2$ to get rid of the dependence on $T$ in the constant, used $(1-\epsilon_T)^2\leq 1$ and the bound $w_k\geq \frac12 L_*$ for $T$ large enough. Note further that $q$ is determined by $p$ as its conjugate. By a series expanison of $\exp(\cdot)$ and monotone convergence as well as an application of Proposition~\ref{moments_local-invariant},
\begin{align*}
&\E_{b_{0,\eta}}\left[ \exp\left( \tilde{c} c_p \log T \left\|\frac{1}{\sigma^2 T} L_T^z(X) - q_{b_{0,\eta}}(z)\right\|_{\infty}\right)\right]\\
&\hspace{0.5cm} = \sum_{k=0}^\infty \frac{(\tilde{c}c_p\log T)^k}{k!} \E_{b_{0,\eta}}\left[\left\|\frac{1}{\sigma^2 T} L_T(X) - \rho_{b_{0,\eta}}\right\|_\infty^k\right]\\
&\hspace{0.5cm}\leq \sum_{k=0}^\infty \frac{(\tilde{c}c_p\log T)^k}{k!} c_1^k\left( \frac{k}{T} + \frac{1}{\sqrt{T}}\left( 1+\sqrt{k} + \sqrt{\log T}\right) + Te^{-c_2T}\right)^k \\
&\hspace{0.5cm} \leq \sum_{k=0}^\infty\frac{k^k}{k!} (\tilde{c}c_1)^k (c_pB_T)^k,
\end{align*}
with
\[ B_T = \frac{\log T}{T} + \frac{\log T}{\sqrt{T}} \left( 2+\sqrt{\log T}\right) + T(\log T) e^{-c_2T}.\]
Considering the power series of $\exp(\cdot)$, we have $e^x >\frac{x^n}{n!}$ for $x\geq 0$ and hence for $x=n$, 
\[  \frac{n^n}{n!}< e^n.\]
Moreover, we remember $p=p_T= 1+T^{-\frac14}$ from the previous step of bounding $A_5$. This choice gives $q = \frac{p}{p-1} = 1+ T^{\frac14}$ and with $\nu\leq 1$, we have for $T\geq 1$,
\[c_p = (1+T^\frac14)\left( (1+\nu) T^{-\frac14} + \nu\right) \leq 3(1+T^\frac14).\]
Thus, after inserting all this, we end up with the upper bound
\begin{align*}
\sum_{k=0}^\infty\frac{k^k}{k!} (\tilde{c}c_1)^k (c_pB_T)^k& \leq \sum_{k=0}^\infty (\tilde{c}c_1 e)^k\left( 3(1+T^\frac14) B_T\right)^k\\
&= \sum_{k=0}^\infty (3\tilde{c}c_1 e)^k \tilde{B}_T^k,
\end{align*}
where $ \tilde{B}_T$ is defined as 
\[ \frac{\log T(1+T^\frac14)}{T} + \frac{\log T (1+T^\frac14)}{\sqrt{T}}\left( 2+\sqrt{\log T}\right) + T (1+T^\frac14) (\log T) e^{-c_2 T}.\]
The important feature now is that $\tilde{B}_T\rightarrow 0 $ for $T\to\infty$ and we have shown so far 
\[ A_6 \leq \left( \sum_{k=0}^\infty (3\tilde{c}c_1 e)^k \tilde{B}_T^k\right)^{\frac{1}{1+T^{1/4}}}.\]
Now it is time to choose $T$ large enough such that $\tilde{B}_T< (3\tilde{c}c_1 e)^{-1}$. Then we have 
\begin{align*}
\sum_{k=0}^\infty \left( 3\tilde{c}c_1 e \tilde{B}_T\right)^k = \frac{1}{1-3\tilde{c}c_1 e \tilde{B}_T}>1.
\end{align*}
As $(1+T^{1/4})^{-1}<1$ we get the bound
\[ A_6 \leq \frac{1}{1-3\tilde{c}c_1 e \tilde{B}_T} \stackrel{T\to\infty}{\longrightarrow} 1.\]
\end{itemize}
Summing up, we have shown so far that there exists $T_0$ and a constant $C'=C'(T_0)>1$ not depending on $k$ such that for $T\geq T_0$, 
\begin{align}\label{eqp: F5}
\E_{b_0}\left[ Z_k^{1+\nu}\right]\leq C' A_1 = C'\exp\left( \frac{T}{2}\nu(1+\nu) A_{T,k}^2 q_{b_{0,\eta}}(y_k) \int_\R K_{T,k}(z)^2 dz\right).
\end{align} 
Next, we analyse the deterministic exponent and plug in the definitions of $A_{T,k}$, $K_k$ from \eqref{eqp: 5.1_notation} and of the optimal constant $c_*$ given in \eqref{eq: c_optimal}. This yields 
\begin{align*}
&\frac12\nu(1+\nu) q_{b_{0,\eta}}(y_k) T A_{T,k}^2 \int_\R K_{T,k}(z)^2 dz \\
&\hspace{1.5cm}= \nu(1+\nu) \frac{1}{2\beta+1} \frac{\|K_T^\beta\|_{L^2}^2}{\|K_\beta\|_{L^2}^2}\frac{q_{b_{0,\eta}}(y_k)}{w_k}  (1-\epsilon_T)^2 \log T.
\end{align*}
It remains to show that both the fraction of the invariant density and $w_k$ and that of the $L^2$-norms are close enough to one. For the first one, things get easier, if we can express its order in $T$ without dependence on $k$. First of all, we use that $w_k = q_{b_k}(y_k) K_T^\beta(0)^{(2\beta+1)/\beta}$ and hence,
\[  \frac{q_{b_0,\eta}(y_k)}{w_k} = \frac{q_{b_{0,\eta}}(y_k)}{q_{b_k}(y_k)} K_T^\beta(0)^{-\frac{2\beta+1}{\beta} } = \left(1 + \mathcal{O}( a_T)\right) K_T^\beta(0)^{-\frac{2\beta+1}{\beta} }, \] 
where $a_T = (\log(T)/T)^{\beta/(2\beta+1)}$. This can be seen by following the lines of the proof of Lemma \ref{lemma_quotient_densities}. For $\beta\geq 1$ we have $K_T^\beta(0)=1$, for $\beta<1$ we have $K_T^\beta(0) = (1-T^{-\beta})$. By a Taylor expansion at $x=0$, we have for $\alpha>0$ that $(1-x)^{-\alpha} = 1+\mathcal{O}(x)$ as $\alpha (1-x)^{-\alpha-1}$ is bounded for $0\leq x\leq\frac12$ and hence for each $\beta>0$ and $T$ large enough such that $T^{-\beta}\leq\frac12$,
\[ K_T^\beta(0)^{-\frac{2\beta+1}{\beta}} = 1+  \mathcal{O}(T^{-\beta})  \leq 1+\mathcal{O}(a_T).\]
We conclude
\[ \frac{q_{b_{0,\eta}}(y_k)}{w_k} =(1+\mathcal{O}(a_T))^2 = 1+ \mathcal{O}(a_T).\]
For the $L^2$-norms we write
\[  \frac{\|K_T^\beta\|_{L^2}^2}{\|K_\beta\|_{L^2}^2} = \frac{\|K_\beta\|_{L^2}^2+\|K_T^\beta\|_{L^2}^2-\|K_\beta\|_{L^2}^2}{\|K_\beta\|_{L^2}^2} = 1 + \mathcal{O}(T^{-\frac12-\beta}) \]
by Lemma \ref{lemma_K_T}. Now, every ingredient is prepared and we finalize the proof. Remember that $N=N_T \geq c\left(T/\log T\right)^{\frac{1}{2\beta +1}}$ for some constant $c>0$. Then we have from \eqref{eqp: F5} that
\begin{align*}
&N^{-(1+\nu)} \sum_{k=1}^N \E_{b_{0,\eta}}\left[ Z_k^{1+\nu}\right] \\
&\hspace{0.1cm}\leq C'N^{-\nu} \exp\left( \nu(1+\nu) \frac{1}{2\beta+1} (1+\mathcal{O}(a_T)) (1+\mathcal{O}(T^{-\frac12-\beta}))  (1-\epsilon_T)^2 \log T\right)\\
&\hspace{0.1cm}= C'\exp\left( \nu(1+\nu) \frac{1}{2\beta+1} (1+\mathcal{O}(a_T))(1+\mathcal{O}(T^{-\frac12-\beta})) \right.\\
&\hspace{7.5cm}\cdot\left.  (1-\epsilon_T)^2 \log T - \nu\log N\right)\\
&\hspace{0.1cm} \leq C'\exp\left( \nu(1+\nu) \frac{1}{2\beta+1} (1+\mathcal{O}(a_T))(1+\mathcal{O}(T^{-\frac12-\beta}))   (1-\epsilon_T)^2 \log T \right.\\
&\hspace{3.5cm} \left. - \nu\frac{1}{2\beta+1}\left( \log T - \log\log T\right)  -\nu \log c\right).
\end{align*}
We note $(1+\mathcal{O}(a_T))(1+\mathcal{O}(T^{-\frac12-\beta})) =1+\mathcal{O}(a_T\vee T^{-\frac12-\beta})$. Setting $\nu=\epsilon_T$ we have $(\nu(1+\nu)(1-\epsilon_T)^2 = \epsilon_T - \epsilon_T^2 - \epsilon_T^3 + \epsilon_T^4$ and hence the last expression equals
\begin{align*}
&C' \exp\left(  \frac{1}{2\beta+1}\left( \left(\epsilon_T - \epsilon_T^2 + \mathcal{O}(\epsilon_T^3)\right)(1+\mathcal{O}(a_T\vee T^{-\frac12-\beta})) - \epsilon_T\right)\log T \right. \\
&\hspace{3cm}\left. + \frac{1}{2\beta+1}\epsilon_T\log\log T -\epsilon_T\log c\right)\\
&\hspace{1cm} = C' \exp\left(  \frac{1}{2\beta+1}\left( -\epsilon_T^2 (1+\mathcal{O}(\epsilon_T)) + \mathcal{O}(\epsilon_T)\mathcal{O}(a_T\vee T^{-\frac12-\beta})\right)\log T \right. \\
&\hspace{4cm}\left. + \frac{1}{2\beta+1}\epsilon_T\log\log T -\epsilon_T\log c\right).
\end{align*}
Looking after the definition of $a_T$ we see that $\mathcal{O}(a_T\vee T^{-\frac12-\beta})\log T = o_T(1)$ and from the choice of $\epsilon_T$ we also get $\mathcal{O}(\epsilon_T)\mathcal{O}(a_T\vee T^{-\frac12-\beta})\log T= o_T(1)$. It is clear by $\epsilon_T\to 0$ that $\epsilon_T \log c = o_T(1)$. Furthermore,
\begin{align*}
&-\epsilon_T^2 (1+\mathcal{O}(\epsilon_T))\log T + \epsilon_T \log\log T\\
&\hspace{2cm}=-\epsilon_T^2 \log T \left( 1+ \mathcal{O}(\epsilon_T) - \frac{\log\log T}{\epsilon_T \log T}\right)\\
&\hspace{2cm}=-\epsilon_T^2 \log T \left( 1+ \mathcal{O}(\epsilon_T) - \frac{\log\log T}{\epsilon_T \sqrt{\log T}\cdot \sqrt{\log T}}\right)\\
&\hspace{2cm}=-\epsilon_T^2 \log T \left( 1+  \mathcal{O}(\epsilon_T)  + o_T(1)\right), 
\end{align*}
where the last step used $\epsilon_T\sqrt{\log T} \rightarrow\infty$ by the choice of $\epsilon_T$. Using this again, this last term tends to $-\infty$ and thus, 
\begin{align*}
&C' \exp\left(  \frac{1}{2\beta+1}\left( -\epsilon_T^2 (1+\mathcal{O}(\epsilon_T)) + \mathcal{O}(\epsilon_T)\mathcal{O}(a_T\vee b_T)\right)\log T \right. \\
&\hspace{5cm}\left. + \frac{1}{2\beta+1}\epsilon_T\log\log T -\epsilon_T\log c\right) \stackrel{T\to\infty}{\longrightarrow}\ 0.
\end{align*}
In conclusion, the first convergence in \eqref{eq_condition_lower} holds true and the theorem is proven. 
\end{proof}

\begin{remark}\label{remark_fixed_start}
We have discussed in Remark~\ref{remark_fixed_point} how our results can be transferred from the stationary case $X_0\sim\mu_b$ to a diffusion started at $X_0=x_0\in [-A,A]$. For the lower bound in Theorem~\ref{lower_bound}, Proposition~\ref{moments_local-invariant} is used with the explicit moment bound, which is also available in the case $X_0=x_0$, see Remark~\ref{remark_moment_invariant}. Moreover, in the evaluation of the likelihood ratio term $\Pr_{b_k}/\Pr_{b_{0,\eta}}$, the stationarity assumption occurs as the likelihood of the initial values. But this term drops out in the case $X_0=x_0$, making the proof even slightly easier, as the factor $A_3$ equals one. In particular, Theorem~\ref{lower_bound} remains valid in the non-stationary case. This is used for our comparison with the fractional diffusion model in Section~\ref{Sec_Stability}, but is of course of independent interest.
\end{remark}

\subsection{The proofs of the minimax upper bounds in Section~\ref{Sec_Power}}\label{SubSec_upper_sketch}

In this subsection, we briefly sketch the proof of Theorem~\ref{Upper_bound}. Although it requires sound knowledge of stochastic calculus, the principle ideas do not differ substantially from similar other minimax upper bounds for nonparametric tests. A complete version of this proof can be found in Section~$3.7$ of \cite{Brutsche}, together with detailed proofs of Theorem~\ref{thm_adaptivity} and \ref{thm_simple_hypothesis}.
Subsequently we write $X_T = o_{\Pr_b,\textrm{unif}}(1)$ if $X_T$ converges uniformly for $b\in\Sigma(C,A,\gamma,\sigma)$ to zero in probability, i.e. for all $\epsilon>0$,
\[ \lim_{T\to\infty} \sup_{b\in\Sigma(C,A,\gamma,\sigma)} \Pr_b\left( \left|X_T\right| >\epsilon\right) = 0.\]
Moreover, we denote this uniform stochastic convergence by $\rightarrow_{\Pr_b,\textrm{unif}}$.

\begin{proof}[Sketch of proof of Theorem~\ref{Upper_bound}]  
The start of the proof is the following observation: For every $y\in J$ the probability $\Pr_b( T_T^\eta(X) >\kappa_{\eta,\alpha})$ of rejecting the null hypothesis is bounded from below by 
\[ \Pr_b\left( |\Psi_{T,y,h}^{b_0}(X)| > \kappa_{\eta,\alpha} + \Upsilon\left(\hat\sigma_T(y,h)^2/\hat\sigma_{T,\max}^2\right)+\Lambda_{T,y,h}^\eta(X)\right),\]
which in turn is bounded from below by
\begin{align*}
&\Pr_b\left( -\textrm{sign}\left(\int_0^T K_{y,h}(X_s)\left( b(X_s) - b_0(X_s)\right) ds\right)\frac{\frac{1}{\sqrt{T}} \int_0^T K_{y,h}(X_s) dW_s}{\sqrt{\frac1T \int_0^T K_{y,h}(X_s)^2 ds}} \right.\\
&\hspace{1.5cm}  <-\kappa_{\eta,\alpha} - \Upsilon\left(\frac{\hat\sigma_T(y,h)^2}{\hat\sigma_{T,\max}^2}\right) - \Lambda_{T,y,h}^\eta(X) \\
&\hspace{5cm} \left. + \frac{\frac{1}{\sqrt{T}} \left|\int_0^T K_{y,h}(X_s)\left(b(X_s)-b_0(X_s)\right) ds\right|}{\sigma\sqrt{\frac1T \int_0^T K_{y,h}(X_s)^2 ds}}  \right).
\end{align*}
The idea of the proof is to proceed as follows: First, we show that the left-hand side of the inequality forms an asymptotically tight sequence (uniformly in $b$). Then, the claim is established in a second step in which it is shown that
\begin{align*}
&\frac{\frac{1}{\sqrt{T}} \left|\int_0^T K_{y,h}(X_s)\left(b_T(X_s)-b_0(X_s)\right) ds\right| - \frac{\eta}{\sqrt{T}}\int_0^T K_{y,h}(X_s)ds}{\sigma \sqrt{\frac1T \int_0^T K_{y,h}(X_s)^2 ds}} \\
& \hspace{1cm} - \Upsilon\left(\frac{\frac1T \int_0^T K_{y,h}(X_s)^2 ds}{\frac1T \int_0^T \1_{[-A,A]}(X_s)ds}\right) \longrightarrow_{\Pr_b, \textrm{unif}}\ \infty
\end{align*}
for all $b_T\in H_1(b_0,\eta) \cap\{ b-b_0\in \mathcal{H}(\beta, L)\}$ with $\Delta(b_T)\geq (1+\epsilon_T) c\delta_T$ and some $(y,h)=(y_T,h_T)$, where $c$ is specified appropriately. For $\beta\leq 1$, this convergence follows from the lower bound
\[ c_*^\frac{2\beta+1}{2\beta} L^{-\frac{1}{2\beta}}  \|K_\beta\|_{L^2}\sqrt{\log T}(1+\epsilon_T)^\frac{2\beta+1}{2\beta} q_{b_T}(y_T)^{-\frac12} \frac{q_{b_T}(y_T)-a_T^{b_T}(y_T)}{\sqrt{q_{b_T}(y_T)+a_T^{b_T}(y_T)}}\]
of the first summand, where $K_\beta$ denots the optimal recovery kernel and
\[  a_T^{b_T}(y_T) :=\left\|(\sigma^2 T)^{-1} L_T^z(X) - q_{b_T}(y)\right\|_{[y_T-h_T, y_T+h_T]},\]
together with the fact that
\[ \Upsilon\left(\frac{\frac1T \int_0^T K_{y_T,h_T}(X_s)^2 ds}{\frac1T \int_0^T \1_{[-A,A]}(X_s) ds}\right) = \sqrt{\frac{2}{2\beta +1}\log T} + o_{\Pr_b, \textrm{unif}}(1).\]
The case $\beta >1$ can be treated similarly with a non-accurate constant in the lower bound. The reason for this is that we cannot work with the optimal recovery kernel for $\beta >1$ due to the fact that $K_\beta$ can take negative values in this case.
\end{proof}

\begin{remark}\label{remark_fixed_start_upper}
The results in Theorem \ref{Upper_bound}, \ref{thm_adaptivity} and \ref{thm_simple_hypothesis} and their auxiliary results (in particular Theorem \ref{weak_conv}) were derived under the assumption that the diffusion $X$ is started in the invariant density, i.e. $X_0\sim\mu_b$. As outlined in Remark \ref{remark_moment_invariant}, the moment inequality for the deviation of the normalized local time and invariant density given in Proposition \ref{moments_local-invariant} remains valid under the assumption that $X_0=x_0\in [-A,A]$ is fixed. Our results used stationarity, i.e. $X_0\sim \mu_b$, only via Proposition \ref{moments_local-invariant} and hence, Theorem \ref{Upper_bound}, \ref{thm_adaptivity} and \ref{thm_simple_hypothesis} remain valid if the diffusion $X$ is started at a fixed point $X_0=x_0\in [-A,A]$. \\
A closer look reveals that one only needs the uniform stochastic convergence 
\[ \sup_{b\in\Sigma(C,A,\gamma,\sigma)} \Pr_b\left( \log T \left\| \frac{1}{\sigma^2 T} L_T^\cdot(X) - q_b\right\|_{[-A,A]} >\epsilon\right) \rightarrow 0 \]
to derive these theorems, which is a weaker statement than the moment inequality in Proposition~\ref{moments_local-invariant}.
\end{remark}

\section{The proofs of the minimax upper bounds in Section~\ref*{Sec_Power}}\label{App_Upper}

This part is structured as follows: In Subsection \ref{SubSec_E1} we state and prove some preliminary lemmas. Then, in Subsection \ref{SubSec_E2}, the proofs of Theorem \ref*{Upper_bound}, \ref*{thm_adaptivity} and \ref*{thm_simple_hypothesis} are given. Within this section, we use the following additional notation: We write $X_T = o_{\Pr_b,\textrm{unif}}(a_T)$ for a family $(a_T)_{T\geq 0}$ of real numbers if $X_T/a_T$ converges uniformly for $b\in\Sigma(C,A,\gamma,\sigma)$ to zero in probability, i.e. for all $\epsilon>0$,
\[ \lim_{T\to\infty} \sup_{b\in\Sigma(C,A,\gamma,\sigma)} \Pr_b\left( \left|\frac{X_T}{a_T}\right| >\epsilon\right) = 0.\]
Moreover, we denote this uniform stochastic convergence by $\rightarrow_{\Pr_b,\textrm{unif}}$. Furthermore, $X_T=\mathcal{O}_{\Pr_b,\textrm{unif}}(a_T)$ means that $|X_T/a_T|$ is uniformly stochastically bounded, i.e. for every $\epsilon>0$ there exists $\kappa=\kappa(\epsilon)>0$ and $T_\epsilon$ such that 
\[ \sup_{b\in\Sigma(C,A,\gamma,\sigma)} \Pr_b\left( \left|\frac{X_T}{a_T}\right| >\kappa\right) <\epsilon\quad \textrm{ for all } T\geq T_\epsilon.\]

\subsection{Preliminary results}\label{SubSec_E1}

\begin{lemma}[\cite{Brutsche}, Lemma 3.7.1]\label{Prob_ineq}
Let $X,Y$ and $Z$ be real-valued random variables. Then
\begin{align*}
\Pr( |X+Y|> Z) \geq \Pr( -\textrm{sign}(Y)X < |Y|-Z).
\end{align*} 
\end{lemma}

\begin{lemma}\label{local_time-Invariant}
Let $h_T = c\left(\frac{\log T}{T}\right)^\alpha$ for $c,\alpha>0$ and $b\in\Sigma(C,A,\gamma,\sigma)$. Then 
\[ \left\| \frac{1}{\sigma^2 T} L_T^\cdot(X) - q_b(y)\right\|_{[y-h_T,y+h_T]} = o_{\Pr_b,\textrm{unif}}\left( (\log T)^{-1}\right).\]
\end{lemma}
\begin{proof}[Proof]
We first split in the obvious way:
\[ \frac{1}{\sigma^2 T} L_T^z(X) - q_b(y) = \left( \frac{1}{\sigma^2 T} L_T^z(X) - q_b(z)\right) + \left( q_b(z)-q_b(y)\right).\]
As $q_b$ is differentiable with bounded derivative, it is Lipschitz continuous. More precisely, by Lemma \ref{bound_invariant_density} we have for $z\in [y-h_T, y+h_T]$ that
\[ \left|q_b(z)-q_b(y)\right| \leq L^* |y-z| \leq 2L^* h_T.\]
Using the definition of $h_T$ it follows 
\[ \left\| q_b(y)-q_b\right\|_{[y-h_T, y+h_T]}\leq 2L^*h_T = o\left((\log T)^{-1}\right).\]
For the other term we use Proposition \ref{moments_local-invariant} and Markov's inequality to get
\begin{align*}
&\sup_{b\in\Sigma(C,A,\gamma,\sigma)} \Pr_b\left( \log T \left\|\frac{1}{\sigma^2 T} L_T^\cdot(X) -q_b\right\|_\infty >\epsilon\right) \\
&\hspace{1cm} \leq \sup_{b\in\Sigma(C,A,\gamma,\sigma)} \frac{\log T}{\epsilon} \E_b\left[ \left\| \frac{1}{\sigma^2 T} L_T^\cdot(X) - q_b\right\|_\infty\right] \\
&\hspace{1cm} \leq \frac{c_1}{\epsilon}\left( \frac{\log T}{T} + \frac{\log T}{\sqrt{T}}\left( 2+ \sqrt{\log T}\right) + T \log(T) e^{-c_2 T}\right),
\end{align*}
which converges to zero for $T\to\infty$ and any $\epsilon>0$, and we can conclude $\|(\sigma^2 T)^{-1} L_T^\cdot(X)-q_b\|_{[y-h_T, y+h_T]} = o_{\Pr_b,\textrm{unif}}((\log T)^{-1})$.
\end{proof}

The following result is a standard tool in the literature on multiscale testing with optimal constants. 

\begin{lemma}[\cite{Brutsche}, Lemma 3.7.3]\label{Lemma_rescaling}
Let $K=K_\beta$ be the optimal recovery kernel solving \eqref*{eq: optimal_recovery}.
\begin{enumerate}
\item[(i)] For any $\delta>0$ we have $\pm\delta K_{y,h}\in\mathcal{H}(\beta, \delta h^{-\beta})$. 
\item[(ii)] $\pm\delta K_{y,h}$ minimizes $\|f\|_{L^2}$ over $\mathcal{H}(\beta, \delta h^{-\beta})$ under the additional constraint $\pm f(y)\geq \delta$.
\item[(iii)] For any $f\in\mathcal{H}(\beta, \delta h^{-\beta})$ with $f(y)\geq \delta$, we have
\[ \left| \int_\R f(x) K_{y,h}(x) dx \right| \geq h\delta \|K_\beta\|_{L^2}^2.\]
\end{enumerate}
\end{lemma}

The following two lemmas are slight modifications and reformulations of results established by Rohde in \cite{Rohde_2} within the proof of Theorem $3$, page 1371-1372.

\begin{lemma}[see \cite{Rohde_2}, p. 1371]\label{lemma_f_beta>1_1}
Let $f\in\mathcal{H}(\beta,L)$ for $\beta>0$ and set $x^*:= \textit{argmax}_{x\in[-a,a]} |f(x)|$. Then there exists a constant $c=c(\beta,L)>0$ such that we have $|f(x)|\geq\frac12 |f(x^*)|$ on 
a closed interval $I(f)\subset [-a,a]$ with $\lambda(I(f))\geq c |f(x^*)|^{1/\beta}$.
\end{lemma}

\begin{lemma}[see \cite{Rohde_2}, p. 1372]\label{lemma_f_beta>1_2}
Let $(f_n)_{n\in\N}$ be a sequence of functions $f_n\in \mathcal{H}(\beta, L)$ defined on $[-a,a]$ with $|f(x^*)|=\|f\|_{[-a,a]} = \rho_n\leq 1$.
Then there exists a constant $c=c(\beta, L )>0$ such that 
\[ |f_n(x)|\geq \frac12 \rho_n \quad \textrm{ for }\quad x\in [-a,a]\cap\left[ x^*-c\rho_n^{1/\beta}, x^*+c\rho_n^{1/\beta}\right].\]
Moreover, considered as a function in $\beta$ and $L$, the constant $c(\beta,L)$ is continuous in its parameters.
\end{lemma}

\subsection{Proof of Theorem \ref*{Upper_bound}, \ref*{thm_adaptivity} and \ref*{thm_simple_hypothesis}}\label{SubSec_E2}

We give the proofs of the theorems in the order of their enumeration. 

\begin{proof}[Proof of Theorem \ref*{Upper_bound}]
The start of the proof is the following observation: For every $y\in J$ the probability $\Pr_b( T_T^\eta(X) >\kappa_{\eta,\alpha})$ of rejecting the null hypothesis is bounded from below by 
\[ \Pr_b\left( |\Psi_{T,y,h}^{b_0}(X)| > \kappa_{\eta,\alpha} + C\left(\hat\sigma_T(y,h)^2/\hat\sigma_{T,\max}^2\right)+\Lambda_{T,y,h}^\eta(X)\right),\]
which in turn, following Lemma \ref{Prob_ineq}, is bounded from below by
\begin{align*}
&\Pr_b\left( -\textrm{sign}\left(\int_0^T K_{y,h}(X_s)\left( b(X_s) - b_0(X_s)\right) ds\right)\frac{\frac{1}{\sqrt{T}} \int_0^T K_{y,h}(X_s) dW_s}{\sqrt{\frac1T \int_0^T K_{y,h}(X_s)^2 ds}} \right.\\
&\hspace{1.5cm}  <-\kappa_{\eta,\alpha} - C\left(\frac{\hat\sigma_T(y,h)^2}{\hat\sigma_{T,\max}^2}\right) - \Lambda_{T,y,h}^\eta(X) \\
&\hspace{5cm} \left. + \frac{\frac{1}{\sqrt{T}} \left|\int_0^T K_{y,h}(X_s)\left(b(X_s)-b_0(X_s)\right) ds\right|}{\sigma\sqrt{\frac1T \int_0^T K_{y,h}(X_s)^2 ds}}  \right).
\end{align*}
The idea of the proof is to proceed as follows: First, we show that the left-hand side of the inequality forms an asymptotically tight sequence (uniformly in $b$). Then, the claim is established in a second step, where we show that the right-hand side with appropriate $(y,h)$ diverges to infinity in probability, uniformly in $b$, i.e. that
\begin{align*}
&\frac{\frac{1}{\sqrt{T}} \left|\int_0^T K_{y,h}(X_s)\left(b(X_s)-b_0(X_s)\right) ds\right| - \frac{\eta}{\sqrt{T}}\int_0^T K_{y,h}(X_s)ds}{\sigma \sqrt{\frac1T \int_0^T K_{y,h}(X_s)^2 ds}}\\
&\hspace{6.5cm} - C\left(\frac{\frac1T \int_0^T K_{y,h}(X_s)^2 ds}{\frac1T \int_0^T \1_{[-A,A]}(X_s) ds }\right)
\end{align*}
tends to infinity in probability uniformly in $b$ for $T\to\infty$.\\

For $z\in [-A,A]$, we have $L_*\leq q_b(z)\leq L^*$ by Lemma~\ref{bound_invariant_density} and \ref{uniform_lower_bound_invariant_density}. Now define the set
\[\mathcal{C}_{T,b} :=\left\{ \sup_{z\in A} \left| \frac{1}{\sigma^2 T} L_T^z(X) - q_b(z)\right| <\frac{L_*}{2}\right\}.\] 
By Corollary~\ref{conv_emprical_density_uniform} we have $\sup_{b\in\Sigma(C,A,\gamma,\sigma)}\Pr_b(\mathcal{C}_{T,b}^c) \rightarrow 0$ for $T\to\infty$ and on $\mathcal{C}_{T,b}$ it holds that
\[ \frac{L_*}{2}\leq \frac12 q_b(z)\leq \frac1T L_T^z(X)\leq \frac32 q_b(z) \leq \frac{3L^*}{2}.\]
Now, let $\epsilon>0$ and take $T\geq T_0$ large enough such that the inequality $\sup_{b\in\Sigma(C,A,\gamma,\sigma)}\Pr_b(\mathcal{C}_{T,b})<\frac{\epsilon}{2}$ holds true. For $\kappa>0$ we can write
\begin{align*}
&\Pr_{b}\left(\left|\frac{\frac{1}{\sqrt{T}}\int_0^T K_{y,h}(X_s)dW_s}{\sqrt{\frac1T\int_0^T K_{y,h}(X_s)^2 ds}}\right| >\kappa \right)\\
&\hspace{1cm} \leq \Pr_{b}\left(\left|\frac{\frac{1}{\sqrt{T}}\int_0^T K_{y,h}(X_s)dW_s}{\sqrt{\frac1T\int_0^T K_{y,h}(X_s)^2 ds}}\right| >\kappa,\ \mathcal{C}_{T,b}\right) + \Pr_b\left( \mathcal{C}_{T,b}^c\right). 
\end{align*}
The second summand is bounded by $\frac{\epsilon}{2}$ uniformly in $b$ by the choice of $T$. For the first one we use the occupation times formula~\ref{occupation_times_formula} to get
\[ \frac1T \int_0^T K_{y,h}(X_s)^2 ds = \int_\R K_{y,h}(z)^2 \frac1T L_T^z(X) dz\]
and on $\mathcal{C}_{T,b}$ we can upper and lower bound the averaged local time to see
\[ \frac{L_*}{2} \int_\R K_{y,h}(z)^2 dz \leq \frac1T\int_0^T K_{y,h}(X_s)^2 ds \leq \frac{3L^*}{2} \int_\R K_{y,h}(z)^2 dz.\]
Now, we apply Proposition~\ref{Exp_ineq} (cf. the verification of (b) in the proof of Theorem~\ref*{Multiscale_Lemma}
) with constant $\theta = \sqrt{TL_*/2}\|K_{y,h}\|_{L^2}$ and $S=\sqrt{3L^*/L_*}$ to get the (uniform in $b$) bound
\begin{align*}
&\Pr_{b}\left(\left|\frac{\frac{1}{\sqrt{T}}\int_0^T K_{y,h}(X_s)dW_s}{\sqrt{\frac1T\int_0^T K_{y,h}(X_s)^2 ds}}\right| >\kappa,\ \mathcal{C}_{T,b}\right) \\
&\hspace{4cm} \leq 4\kappa\sqrt{e} (1+\log(\sqrt{3L^*/L_*}))\exp\left(-\frac{\kappa^2}{2}\right).
\end{align*} 
Clearly, we can choose $\kappa$ large enough to bound this term again by $\frac{\epsilon}{2}$ and the proof of the uniform asymptotically tightness is finished.\\

Now we want to show that
\begin{align}\label{eqp: E17}
\begin{split}
&\frac{\frac{1}{\sqrt{T}} \left|\int_0^T K_{y,h}(X_s)\left(b_T(X_s)-b_0(X_s)\right) ds\right| - \frac{\eta}{\sqrt{T}}\int_0^T K_{y,h}(X_s)ds}{\sigma \sqrt{\frac1T \int_0^T K_{y,h}(X_s)^2 ds}} \\
& \hspace{1cm} - C\left(\frac{\frac1T \int_0^T K_{y,h}(X_s)^2 ds}{\frac1T \int_0^T \1_{[-A,A]}(X_s)ds}\right) \longrightarrow_{\Pr_b, \textrm{unif}}\ \infty
\end{split}
\end{align}
for all $b_T\in H_1(b_0,\eta) \cap\{ b-b_0\in \mathcal{H}(\beta, L)\}$ with $\Delta(b_T)\geq (1+\epsilon_T) c\delta_T$ and some $(y,h)=(y_T,h_T)$, where $c$ will be specified later. Let $(b_T)_{T\geq 0}$ be such a sequence of functions with
\[ \Big(|b_T(y_T)-b_0(y_T)| - \eta\Big) \left(\frac{q_b(y_T)}{\sigma^2}\right)^\frac{\beta}{2\beta+1} \geq (1+\epsilon_T) c\delta_T \]
or equivalently
\begin{align}\label{eqp: E_delta}
|b_T(y_T)-b_0(y_T)| - \eta \geq (1+\epsilon_T) c \left( \frac{\sigma^2 \log T}{T q_b(y_T)}\right)^\frac{\beta}{2\beta+1} =:\tilde{\delta}_T.
\end{align}  
To proceed, we define the bandwidth $h_T$ to be given by
\begin{align}\label{eqp: E_hT}
h_T := \left( \frac{\tilde\delta_T}{L}\right)^\frac{1}{\beta} = \left(\frac{(1+\epsilon_T)c}{L}\right)^{\frac{1}{\beta}}\left(\frac{\sigma^2 \log T}{T q_{b_T}(y_T)}\right)^{\frac{1}{2\beta+1}}.
\end{align} 
Note that $(y_T,h_T)$ lies not necessarily in $\Q^2$, but as the left-hand side of \eqref{eqp: E17} is continuous in $(y,h)$ (cf. proof of Theorem~\ref*{thm_continuity} for a detailed proof), the claim still follows as $\Q\subset\R$ is dense.\\
First, $|b_T(y_T)-b_0(y_T)|$ is bounded away from zero by $\eta$ and hence by Lemma~\ref{lemma_f_beta>1_1} there exist intervals $I_T$ with $\liminf_{T\to\infty} \lambda(I_T)>0$ around $y_T$ on which $b_T-b_0$ has the same sign. As $h_T\searrow 0$, the sign of $b_T-b_0$ is the same on all of the support of $K_{y_T,h_T}$ for $T$ large enough. Subsequently, we will always consider only those $T$. Hence, by non-negativity of $K$,
\begin{align}\label{eqp: E8}
\begin{split}
&\left|\int_0^T K_{y_T,h_T}(X_s)\left(b_T(X_s)-b_0(X_s)\right) ds\right| \\
&\hspace{4cm}= \int_0^T K_{y_T,h_T}(X_s)\left|b_T(X_s)-b_0(X_s)\right| ds.
\end{split} 
\end{align} 
Next, we introduce
\[ d_{T,y_T}(z) := \frac{1}{\sigma^2 T} L_T^z(X) - q_{b_T}(y_T) \]
for ease of notation. With \eqref{eqp: E8} and the occupation times formula \ref{occupation_times_formula} we may write
\begin{align}\label{eqp: E9}
\begin{split}
&\frac{\left|\frac{1}{\sqrt{T}}\int_0^T K_{y_T,h_T}(X_s)\left(b_T(X_s) - b_0(X_s)\right)ds\right| - \frac{\eta}{\sqrt{T}}\int_0^T K_{y_T,h_T}(X_s) ds}{\sigma \sqrt{\frac1T\int_0^T K_{y_T,h_T}(X_s)^2 ds}}\\
&\hspace{0.1cm}= \frac{\sqrt{T}\int_\R K_{y_T,h_T}(z)\left( |b_T(z) - b_0(z)|-\eta\right) \frac{1}{\sigma^2 T} L_T^z(X)dz}{\sigma \sqrt{\int_\R K_{y_T,h_T}(z)^2 \frac{1}{\sigma^2 T} L_T^z(X) dz}}\\
&\hspace{0.1cm}=\frac{\sqrt{T}\int_\R K_{y_T,h_T}(z)\left( |b_T(z) - b_0(z)|-\eta\right) \left( q_{b_T}(y_T) + d_{T,y_T}(z)\right)dz}{\sigma \sqrt{\int_\R K_{y_T,h_T}(z)^2\left(q_{b_T}(y_T) + d_{T,y_T}(z)\right) dz}}\\
&\hspace{0.1cm} = \frac{\sqrt{T/h_T}\int_\R K_{y_T,h_T}(z)\left( |b_T(z) - b_0(z)|-\eta\right) \left( q_{b_T}(y_T) + d_{T,y_T}(z)\right)dz}{\sigma \sqrt{\int_\R \frac{1}{h_T}K_{y_T,h_T}(z)^2\left(q_{b_T}(y_T) + d_{T,y_T}(z)\right) dz}}.
\end{split}
\end{align}
Now it is time to distinguish the cases $\beta\leq 1$ and $\beta>1$.\\

We start with the case $\beta\leq 1$:\\
From the condition $b_T\in H_1(b_0,\eta) \cap\{ b-b_0\in \mathcal{H}(\beta, L)\}$ and the reverse triangle inequality,
\[||b_T(y_T)-b_0(y_T)|-|b_T(x)-b_0(x)||\leq Lh_T^\beta = \tilde\delta_T \]
for $|y_T-x|\leq h_T$.
Together with the fact that $|b_T(y_T)-b_0(y_T)|-\eta \geq \tilde\delta_T$, this implies that $|b_T-b_0|-\eta\geq 0$ on $[y_T-h_T, y_T+h_T]$. Since the optimal recovery kernel $K_\beta$ is non-negative for $\beta\leq 1$, we have 
\begin{align*}
&\sqrt{T/h_T}\int_\R K_{y_T,h_T}(z)\left( |b_T(z) - b_0(z)|-\eta\right) \left( q_{b_T}(y_T) + d_{T,y_T}(z)\right)dz \\
&\hspace{1cm} = \sqrt{T/h_T}\int_\R K_{y_T,h_T}(z)\left( |b_T(z) - b_0(z)|-\eta\right) q_{b_T}(y_T) dz \\
&\hspace{2cm} +\sqrt{T/h_T}\int_\R K_{y_T,h_T}(z)\left( |b_T(z) - b_0(z)|-\eta\right) d_{T,y_T}(z)dz \\
&\hspace{1cm} \geq \sqrt{T/h_T}\int_\R K_{y_T,h_T}(z)\left( |b_T(z) - b_0(z)|-\eta\right) q_{b_T}(y_T) dz \\
&\hspace{2cm} -\sqrt{T/h_T} \left|\int_\R K_{y_T,h_T}(z)\left( |b_T(z) - b_0(z)|-\eta\right) d_{T,y_T}(z)dz\right| \\
&\hspace{1cm} \geq \sqrt{T/h_T}\int_\R K_{y_T,h_T}(z)\left( |b_T(z) - b_0(z)|-\eta\right) q_{b_T}(y_T) dz \\
&\hspace{2cm} -\sqrt{T/h_T} \int_\R K_{y_T,h_T}(z)\left( |b_T(z) - b_0(z)|-\eta\right) \left|d_{T,y_T}(z)\right|dz \\
&\hspace{1cm} \geq \sqrt{T/h_T}q_{b_T}(y_T) \int_\R K_{y_T,h_T}(z)\left( |b_T(z) - b_0(z)|-\eta\right) dz \\
&\hspace{2cm} -\sqrt{T/h_T} \left\|\frac{1}{\sigma^2 T} L_T^z(X) - q_{b_T}(y_T)\right\|_{[y_T-h_T, y_T+h_T]}\\
&\hspace{5cm}\cdot\int_\R K_{y_T,h_T}(z)\left( |b_T(z) - b_0(z)|-\eta\right) dz \\
&\hspace{1cm} = \sqrt{T/h_T} \left( q_{b_T}(y_T) - a_T^{b_T}(y_T)\right) \int_\R K_{y_T,h_T}(z)\left( |b_T(z) - b_0(z)|-\eta\right) dz,
\end{align*}
where 
\[ a_T^{b_T}(y_T) :=\left\|(\sigma^2 T)^{-1} L_T^z(X) - q_{b_T}(y)\right\|_{[y_T-h_T, y_T+h_T]}.\]
Due to the condition $b_T-b_0\in\mathcal{H}(\beta,L)$, we have $|b_T-b_0|-\eta\in\mathcal{H}(\beta,L)$ because
\begin{align*}
&\left| (|b_T(x)-b_0(x)|-\eta) - (|b_T(y)-b_0(y)|-\eta)\right| \\
&\hspace{2.5cm}= \left| |b_T(x)-b_0(x)| - |b_T(y)-b_0(y)|\right|\\
&\hspace{2.5cm}\leq \left| b_T(x) - b_0(x) - (b_T(y)-b_0(y))\right|\\
&\hspace{2.5cm}\leq L|x-y|^\beta.
\end{align*}
By part (i) and (ii) of Lemma \ref{Lemma_rescaling}, $\tilde{\delta}_T K_{y_T,h_T}\in\mathcal{H}(\beta, \tilde{\delta}_T h_T^{-\beta}) = \mathcal{H}(\beta, L)$ and it minimizes the $L^2$-norm over this class for all functions $f$ satisfying additionally $f(y_T)\geq\tilde{\delta}_T$. In particular, by Lemma~\ref{Lemma_rescaling}~(iii) this implies
\[ \int_\R \tilde{\delta}_T K_{y_T,h_T}(z) (|b_T(z)-b_0(z)|-\eta) dz \geq \tilde{\delta}_T^2 \int_\R K_{y_T,h_T}(z)^2 dz.\]
From this it follows, 
\begin{align*}
&\sqrt{T/h_T}\int_\R K_{y_T,h_T}(z)\left( |b_T(z) - b_0(z)|-\eta\right) \left( q_{b_T}(y) + d_{T,y_T}(z)\right)dz \\
&\hspace{1cm} \geq \sqrt{T/h_T} \tilde\delta_T \left( q_{b_T}(y_T) - a_T^{b_T}(y_T)\right) \int_\R K_{y_T,h_T}(z)^2 dz\\
&\hspace{1cm} = \sqrt{Th_T}\tilde\delta_T \left( q_{b_T}(y_T) - a_T^{b_T}(y_T)\right) \|K_\beta\|_{L^2}^2,
\end{align*}
where the last step used substitution of $\tilde{x}=\frac{x-y_T}{h_T}$ in the integral. For the denominator on the right-hand side of \eqref{eqp: E9} we may proceed in a similar way. Here we have
\begin{align*}
& \int_\R \frac{1}{h_T}K_{y_T,h_T}(z)^2\left(q_{b_T}(y_T) + \frac{1}{\sigma^2 T} L_T^z(X) - q_{b_T}(y_T)\right) dz \\
&\hspace{1cm} = \int_\R \frac{1}{h_T}K_{y_T,h_T}(z)^2 q_{b_T}(y_T) dz + \int_\R \frac{1}{h_T}K_{y_T,h_T}(z)^2  d_{T,y_T}(z) dz\\
&\hspace{1cm} \leq q_{b_T}(y_T) \|K_\beta\|_{L^2}^2 + a_T^{b_T}(y_T)\int_\R \frac{1}{h_T} K_{y_T,h_T}(z)^2 dz\\
&\hspace{1cm} = (q_{b_T}(y_T) + a_T^{b_T}(y_T)) \|K_\beta\|_{L^2}^2.
\end{align*}
In consequence, by \eqref{eqp: E9} we have
\begin{align}\label{eqp: E10}
\begin{split}
&\frac{\sqrt{T/h_T}\int_\R K_{y_T,h_T}(z)\left( |b_T(z) - b_0(z)|-\eta\right) \left( q_{b_T}(y_T) + d_{T,y_T}(z)\right)dz}{\sigma \sqrt{\int_\R \frac{1}{h_T}K_{y_T,h_T}(z)^2\left(q_{b_T}(y_T) +d_{T,y_T}(z)\right) dz}}\\
& \hspace{1cm} \geq \frac{\sqrt{Th_T}\tilde\delta_T \left( q_{b_T}(y_T) - a_T^{b_T}(y_T) \right) \|K_\beta\|_{L^2}^2}{\sigma \sqrt{q_{b_T}(y_T) + a_T^{b_T}(y_T)} \|K_\beta\|_{L^2}}\\
&\hspace{1cm} = \sqrt{Th_T}\tilde\delta_T \frac{\|K_\beta\|_{L^2}}{\sigma} \frac{q_{b_T}(y_T)-a_T^{b_T}(y_T)}{\sqrt{q_{b_T}(y_T)+a_T^{b_T}(y_T)}}.
\end{split}
\end{align}
Now, by definition of $h_T$ and $\tilde{\delta}_T$ in \eqref{eqp: E_delta} and \eqref{eqp: E_hT} we have using $c=c_*$, where the optimal constant $c_*$ is given in \eqref*{eq: c_optimal}, that $\sqrt{Th_T}\tilde{\delta}_T$ equals
\begin{align}\label{eqp: E16}
\begin{split}
&T^\frac12 \left(\frac{(1+\epsilon_T)c_*}{L}\right)^{\frac{1}{2\beta}} \left( \frac{\sigma^2 \log T}{T q_{b_T}(y_T)}\right)^\frac{1}{2(2\beta+1)} (1+\epsilon_T) c_*\left(\frac{\sigma^2 \log T}{T q_{b_T}(y_T)}\right)^\frac{\beta}{2\beta+1}\\
&\hspace{1cm}= c_*^\frac{2\beta+1}{2\beta} L^{-\frac{1}{2\beta}} q_{b_T}(y_T)^{-\frac{1}{2(2\beta+1)} - \frac{\beta}{2\beta+1}} T^{\frac12 - \frac{1}{2(2\beta+1)} - \frac{\beta}{2\beta+1}} \\
&\hspace{5cm} \cdot(\sigma^2 \log T)^{\frac{1}{2(2\beta+1)} + \frac{\beta}{2\beta+1}}  (1+\epsilon_T)^\frac{2\beta+1}{2\beta}\\
&\hspace{1cm}= \sigma c_*^\frac{2\beta+1}{2\beta} L^{-\frac{1}{2\beta}}  q_{b_T}(y_T)^{-\frac12} (\log T)^{\frac12}(1+\epsilon_T)^\frac{2\beta+1}{2\beta},
\end{split}
\end{align}
as $\frac{1}{2(2\beta+1)} + \frac{\beta}{2\beta+1} = \frac{1+2\beta}{2(2\beta+1)} = \frac12$. Together with \eqref{eqp: E9} and \eqref{eqp: E10},
\begin{align*}
&\frac{\left|\frac{1}{\sqrt{T}}\int_0^T K_{y_T,h_T}(X_s)\left(b_T(X_s) - b_0(X_s)\right)ds\right| -\frac{\eta}{\sqrt{T}}\int_0^T K_{y_T,h_T}(X_s) ds}{\sigma\sqrt{\frac1T\int_0^T K_{y_T,h_T}(X_s)^2 ds}} \\
&\hspace{0.2cm}\geq c_*^\frac{2\beta+1}{2\beta} L^{-\frac{1}{2\beta}}  \|K_\beta\|_{L^2}\sqrt{\log T}(1+\epsilon_T)^\frac{2\beta+1}{2\beta} q_{b_T}(y_T)^{-\frac12} \frac{q_{b_T}(y_T)-a_T^{b_T}(y_T)}{\sqrt{q_{b_T}(y_T)+a_T^{b_T}(y_T)}}.
\end{align*}
Next, we treat the compensating term $C(\cdot)$ in a similar way. We have by its definition
\begin{align*}
&C\left(\frac{\frac1T \int_0^T K_{y_T,h_T}(X_s)^2 ds}{\frac1T \int_0^T \1_{[-A,A]}(X_s) ds}\right) = \left( 2\log\left(\frac{\frac1T \int_0^T \1_{[-A,A]}(X_s) ds}{\frac1T \int_0^T K_{y_T,h_T}(X_s)^2 ds}\right)\right)^\frac12\\
&\hspace{0.7cm} = \left( 2\log\left( \frac{1}{\frac1T \int_0^T K_{y_T,h_T}(X_s)^2 ds}\right)  + 2\log \left(\frac1T \int_0^T \1_{[-A,A]}(X_s) ds\right)\right)^\frac12.
\end{align*}
To upper bound the expression, we make the first summand larger, by making the denominator smaller. This is done in a similar way as above, i.e.
\begin{align}\label{eqp: E3}
\begin{split}
&\frac1T \int_0^T K_{y_T,h_T}(X_s)^2 ds\\
&\hspace{1cm}= \int_\R K_{y_T,h_T}(z)^2 \frac{1}{\sigma^2 T} L_T^z(X)  dz\\
&\hspace{1cm}= q_{b_T}(y_T)\int_\R K_{y_T,h_T}(z)^2 dz + \int_\R K_{y_T,h_T}(z)^2 d_{T,y_T}(z) dz \\
&\hspace{1cm} \geq q_{b_T}(y_T)h_T \|K_\beta\|_{L^2}^2 - \left| \int_\R K_{y_T,h_T}(z)^2 d_{T,y_T}(z)dz \right| \\
&\hspace{1cm} \geq q_{b_T}(y_T)h_T \|K_\beta\|_{L^2}^2 - a_T^{b_T}(y_T) \int_\R K_{y_T,h_T}(z)^2 dz \\\
&\hspace{1cm}= \left( q_{b_T}(y_T) - a_T^{b_T}(y_T)\right) h_T \|K_\beta\|_{L^2}^2.
\end{split}
\end{align}
Thus, the compensating term $C(\cdot)$ is upper-bounded by
\begin{align*}
&\left( 2\log\left(\frac{1}{h_T}\right) + 2\log\left( \frac{1}{(q_{b_T}(y_T) - a_T^{b_T}(y_T))\|K_\beta\|_{L^2}^2}\right) \right. \\
&\hspace{4.5cm} \left. + 2\log \left(\frac1T \int_0^T \1_{[-A,A]}(X_s) ds\right)\right)^\frac12.
\end{align*} 
Now, with our definition of $h_T$ in \eqref{eqp: E_hT} we have
\begin{align}\label{eqp: h_T} 
\begin{split}
\log(h_T) &= \frac{1}{\beta} \log\left( \frac{c_*}{L}\right) + \frac{1}{\beta} \log\left( 1+ \epsilon_T\right)\\
&\hspace{2cm} + \frac{1}{2\beta+1}\left( \log\log T - \log T - \log(\sigma^{-2} q_{b_T}(y_T)) \right)\\
&= \frac{1}{2\beta+1}\left( \log\log T - \log T\right) + \mathcal{O}_T(1),
\end{split}
\end{align}
since $\epsilon_T\searrow 0$ and hence $0\leq |\log(1+\epsilon_T)|\leq \log(2)$ and $L_*<q_{b_T}(y_T)<L^*$ for $y_T\in [-A,A]$ and $b_T\in\Sigma(C,A,\gamma,\sigma)$ due to Lemma \ref{bound_invariant_density} and \ref{uniform_lower_bound_invariant_density}.
Next,
\begin{align*}
&2\log\left( \frac1T \int_0^T \1_{[-A,A]}(X_s) ds \right)\\
&\hspace{1cm}= 2\log\left( \int_{-A}^A \frac{1}{\sigma^2 T} L_T^z(X) dz\right)\\
&\hspace{1cm}\leq 2\log\left( \int_{-A}^A q_b(z) dz + 2A\sup_{z\in [-A,A]} \left| \frac{1}{\sigma^2 T}L_T^z(X) - q_{b_T}(z)\right| \right)
\end{align*}
and this term is $\mathcal{O}_{\Pr_b,\textrm{unif}}(1)$, where the uniformity is considered with respect to $b\in\Sigma(C,A,\gamma,\sigma)$. This follows from the uniform bounds $L_*\leq q_{b_T}\leq L^*$ from Lemma~\ref{bound_invariant_density} and \ref{uniform_lower_bound_invariant_density} and the uniform convergence of the $\sup$-term to zero by Corollary~\ref{conv_emprical_density_uniform}. Together we then have
\begin{align}\label{eqp: E2}
\begin{split}
&C\left(\frac{\frac1T \int_0^T K_{y_T,h_T}(X_s)^2 ds}{\frac1T \int_0^T \1_{[-A,A]}(X_s) ds}\right)\\
&\hspace{1cm} = \left( \frac{2}{2\beta+1} \left(\log T -\log\log T \right) + \mathcal{O}_T(1) + \mathcal{O}_{\Pr_b,\textrm{unif}}(1) \right)^\frac12 \\
&\hspace{1cm}= \sqrt{\frac{2}{2\beta +1}\log T} + o_{\Pr_b, \textrm{unif}}(1),
\end{split}
\end{align}
where the last step follows from the fact that
\begin{align*}
&\left( \frac{2}{2\beta+1} \left(\log T -\log\log T \right) + \mathcal{O}_T(1) + \mathcal{O}_{\Pr_b,\textrm{unif}}(1) \right)^\frac12 - \sqrt{\frac{2}{2\beta +1}\log T}\\
&\hspace{0.8cm} = \frac{-2(2\beta+1)^{-1} \log\log T +\mathcal{O}_T(1) + \mathcal{O}_{\Pr_b,\textrm{unif}}(1)}{\left( \frac{2}{2\beta+1} \left(\log T -\log\log T \right) + \mathcal{O}_T(1) + \mathcal{O}_{\Pr_b,\textrm{unif}}(1) \right)^\frac12 + \sqrt{\frac{2}{2\beta +1}\log T}},
\end{align*}
which is $o_{\Pr_b,\textrm{unif}}(1)$ as the leading term in the denominator is $\sqrt{\log T}$.
Now we can put everything together and evaluate the whole expression:
\begin{align}\label{eqp: E18}
&\frac{\left|\frac{1}{\sqrt{T}}\int_0^T K_{y_T,h_T}(X_s)\left(b_T(X_s) - b_0(X_s)\right)ds\right|-\frac{\eta}{\sqrt{T}}\int_0^T K_{y_T,h_T}(X_s) ds}{\sqrt{\frac1T\int_0^T K_{y_T,h_T}(X_s)^2 ds}} \\ \notag
&\hspace{2cm} -C\left(\frac{\frac1T\int_0^T K_{y_T,h_T}(X_s)^2 ds}{\frac1T \int_0^T \1_{[-A,A]}(X_s) ds}\right)\\ \notag
&\hspace{0.5cm} \geq c_*^\frac{2\beta+1}{2\beta} L^{-\frac{1}{2\beta}}  \|K_\beta\|_{L^2}\sqrt{\log T}(1+\epsilon_T)^\frac{2\beta+1}{2\beta} q_{b_T}(y_T)^{-\frac12} \frac{q_{b_T}(y_T)-a_T^{b_T}(y_T)}{\sqrt{q_{b_T}(y_T)+a_T^{b_T}(y_T)}} \\ \notag
&\hspace{2cm} - \sqrt{\frac{2}{2\beta +1}\log T} + o_{\Pr_b, \textrm{unif}}(1) \\ \notag
&\hspace{0.5cm} = \left( \frac{2L^{\frac{1}{\beta}}}{(2\beta+1) \|K_\beta\|_{L^2}^2}\right)^\frac12 L^{-\frac{1}{2\beta}}  \|K_\beta\|_{L^2}\sqrt{\log T}(1+\epsilon_T)^\frac{2\beta+1}{2\beta} q_{b_T}(y_T)^{-\frac12}  \\ \notag
&\hspace{2cm} \cdot\frac{q_{b_T}(y_T)-a_T^{b_T}(y_T)}{\sqrt{q_{b_T}(y_T)+a_T^{b_T}(y_T)}}- \sqrt{\frac{2}{2\beta +1}\log T} + o_{\Pr_b, \textrm{unif}}(1) \\ \notag
&\hspace{0.5cm} = \sqrt{\frac{2}{2\beta+1} \log T} \left( (1+\epsilon_T)^\frac{2\beta+1}{2\beta} q_{b_T}(y_T)^{-\frac12} \frac{q_{b_T}(y_T)-a_T^{b_T}(y_T)}{\sqrt{q_{b_T}(y_T)+a_T^{b_T}(y_T)}}  - 1\right) \\ \notag
&\hspace{2cm}+ o_{\Pr_b,\textrm{unif}}(1).
\end{align}
By a Taylor expansion we write
\[  (1+\epsilon_T)^\frac{2\beta+1}{2\beta} = 1 + \frac{2\beta+1}{2\beta} \epsilon_T + \mathcal{O}(\epsilon_T^2).\]
Additionally, the equation
\[ \left(\frac{a-b}{\sqrt{a+b}}\right)^2 = \frac{a^2 - 2ab +b^2}{a+b} = \frac{(a+b)^2 - 4ab}{a+b} = a+b - \frac{4ab}{a+b}\]
implies
\begin{align*}
q_{b_T}(y_T)^{-\frac12} \frac{q_{b_T}(y_T)-a_T^{b_T}(y_T)}{\sqrt{q_{b_T}(y_T)+a_T^{b_T}(y_T)}}  &= \sqrt{1+ \frac{a_T^{b_T}(y_T)}{q_{b_T}(y_T)} - \frac{4a_T^{b_T}(y_T)}{q_{b_T}(y_T) + a_T^{b_T}(y_T)}}\\
&\hspace{-1cm}= 1 + \frac{1}{2\sqrt{\xi}} \left( \frac{a_T^{b_T}(y_T)}{q_{b_T}(y_T)} - \frac{4a_T^{b_T}(y_T)}{q_{b_T}(y_T) + a_T^{b_T}(y_T)}\right)
\end{align*}
for a point $\xi$ between $1$ and $1+ \frac{a_T^{b_T}(y_T)}{q_{b_T}(y_T)} - \frac{4a_T^{b_T}(y_T)}{q_{b_T}(y_T) + a_T^{b_T}(y_T)}$. As $a_T^{b}(y)$ converges to zero uniformly in $b$ we can take $\frac12\leq\xi\leq\frac32$ for large enough $T$. 
Then we have the lower bound
\begin{align*}
q_{b_T}(y_T)^{-\frac12} \frac{q_{b_T}(y_T)-a_T^{b_T}(y_T)}{\sqrt{q_{b_T}(y_T)+a_T^{b_T}(y_T)}}  &\geq 1+ \frac{1}{2\sqrt{\xi}} \left( \frac{a_T^{b_T}(y_T)}{q_{b_T}(y_T)} - \frac{4a_T^{b_T}(y_T)}{q_{b_T}(y_T)}\right)\\
&= 1- \frac{3}{\sqrt{2}}\frac{a_T^{b_T}(y_T)}{ q_{b_T}(y_T)} .
\end{align*}
Finally, our lower bound of interest, i.e. the right-hand side of \eqref{eqp: E18}, is given by
\begin{align*}
&\sqrt{\frac{2}{2\beta+1} \log T} \left(\hspace{-0.05cm}\left( 1 + \frac{2\beta+1}{2\beta} \epsilon_T + \mathcal{O}(\epsilon_T^2)\right)\hspace{-0.05cm}\left( 1- \frac{3}{\sqrt{2}}\frac{a_T^{b_T}(y_T)}{ q_{b_T}(y_T)} \right) -1 \right) \\
&\hspace{1.5cm} + o_{\Pr_b,\textrm{unif}}(1) \\
&\hspace{1cm} = \sqrt{\frac{2}{2\beta+1} \log T} \left( \frac{2\beta+1}{2\beta}\epsilon_T ( 1+ \mathcal{O}(\epsilon_T)) \right. \\
&\hspace{3cm}\left. - \frac{3}{\sqrt{2}}\frac{a_T^{b_T}(y_T)}{q_{b_T}(y_T)} \left( 1+ \frac{2\beta +1}{2\beta} \epsilon_T + \mathcal{O}(\epsilon_T^2)\right) \right) + o_{\Pr_b,\textrm{unif}}(1).
\end{align*}
Now, both $1+ \mathcal{O}(\epsilon_T)$ and $1+ \frac{2\beta +1}{2\beta} \epsilon_T + \mathcal{O}(\epsilon_T^2)$ converge to one. Furthermore, by the choice of $\epsilon_T$ we have $\sqrt{\log T}\epsilon_T\rightarrow \infty$ and by Lemma \ref{local_time-Invariant} we have $\sqrt{\log T} a_T^{b_T}(y_T)\rightarrow_{\Pr_b,\textrm{unif}} 0$ uniformly in $y_T$ and $b_T$. Hence, the whole expression diverges to $\infty$ and we are done with the case $\beta\leq 1$. \\

Now we turn to the case $\beta>1$:\\
Let $y_T$ maximize $|b_T(y) - b_0(y)|-\eta$ on the compact $J\subset (-A,A)$. As $b_T-b_0$ is bounded away from zero by $\eta$, by Lemma \ref{lemma_f_beta>1_1} we find intervals $I_T$ with $\liminf_{T\to\infty}\lambda(I_T)>0$ on which $b_T-b_0$ is bounded away from zero. Thus, for $T$ large enough, $b_T-b_0$ has the same sign on $[y_T - h_T, y_T+h_T]$ and on this interval $|b_T-b_0|-\eta\in\mathcal{H}(\beta, L)$.\\
First, we assume additionally that 
\begin{align}\label{eqp: upper_gamma_T}
\gamma_T:= \gamma_T(b_T) := |b_T(y_T)-b_0(y_T)| - \eta
\end{align} 
tends to zero for $T\to\infty$. We have $\gamma_T\geq \tilde\delta_T$ and by Lemma~\ref{lemma_f_beta>1_2} there exists a constant $c'=c'(\beta,L)>0$ such that $|b-b_0|-\gamma\geq \frac{\gamma_T}{2}\geq\frac{\tilde\delta_T}{2}$ on $[y_T-c'\gamma_T^{1/\beta}, y_T+c'\gamma_T^{1/\beta}]$. Especially, the same is true on the subset 
\[ \left[y_T-c'\tilde\delta_T^{1/\beta},y_T+c' \tilde\delta_T^{1/\beta}\right]  = \left[ y_T-c''h_T , y_T+c''h_T\right]\]
with $c'':=c''(\beta,L):= L^{1/\beta}c'(\beta,L)$. Then,
\begin{align}\label{eqp: E11}
\begin{split}
\int K_{y_T,c''h_T}(z)\left( |b_T(z) - b_0(z)|-\eta\right) dz &\geq \frac{\tilde{\delta}_T}{2}\int_{y_T-c''h_T}^{y_T+c''h_T} K_{y_T,h_T}(z) dz \\
&= \frac{1}{2} \tilde\delta_T c''h_T \int_{-1}^{1} K(z) dz.
\end{split}
\end{align} 
As $|b-b_0|-\eta\geq 0$ on the interval $[y_T-c''h_T, y_T+c''h_T]$ we can follow the lines of the case $\beta\leq 1$ to see that 
\begin{align*}
&\frac{\left| \frac{1}{\sqrt{T}} \int_0^T K_{y_T,c''h_T}(X_s) (b_T(X_s) - b_0(X_s)) ds\right| - \frac{\eta}{\sqrt{T}} \int_0^T K_{y_T,c''h_T}(X_s) ds}{\sigma \sqrt{\frac1T \int_0^T K_{y_T,c''h_T}(X_s)^2 ds}}\\
&\geq \frac{\sqrt{T/(c''h_T)}\left( q_{b_T}(y_T) - a_T^{b_T}(y_T)\right) \int_\R K_{y_T,c''h_T}(z) (|b_T(z)-b_0(z)|-\eta) dz}{\sigma \sqrt{(q_{b_T}(y_T) + a_T^{b_T}(y_T))\|K\|_{L^2}^2}} \\
&\hspace{0.1cm} \geq \frac{\sqrt{T/(c''h_T)}\left( q_{b_T}(y_T) - a_T^{b_T}(y_T)\right) \frac{1}{2} \tilde\delta_T c''h_T \|K\|_{L^1}}{\sigma \sqrt{(q_{b_T}(y_T) + a_T^{b_T}(y_T))\|K\|_{L^2}^2}} \\
& = \sqrt{T h_T} \tilde\delta_T \frac{\sqrt{c''} \|K\|_{L^1} }{2\sigma \|K\|_{L^2}^2} \frac{q_{b_T}(y_T) - a_T^{b_T}(y_T)}{\sqrt{q_{b_T}(y_T) + a_T^{b_T}(y_T)}}\\
&= c^{\frac{2\beta+1}{2\beta}} \hspace{-0.05cm}\underbrace{\frac{\sqrt{c''} \|K\|_{L^1} }{2\|K\|_{L^2}^2} L^{-\frac{1}{2\beta}}  }_{=:\tilde{c}(\beta,L,K)} \hspace{-0.1cm}\sqrt{\log T} (1+\epsilon_T)^{\frac{2\beta+1}{2\beta}}  q_{b_T}(y_T)^{-\frac12} \frac{q_{b_T}(y_T) - a_T^{b_T}(y_T)}{\sqrt{q_{b_T}(y_T) + a_T^{b_T}(y_T)}}.
\end{align*}
For the correction term $C(\cdot)$ we only have to use $c''h_T$ instead of $h_T$. Then \eqref{eqp: h_T} turns to
\begin{align}\label{eqp: E12}
\begin{split}
\log(c''h_T) &= \frac{1}{2\beta +1}\left( \log\log T - \log T\right) + \mathcal{O}_T(1) + \log(c'')\\
&= \frac{1}{2\beta +1}\left( \log\log T - \log T\right) + \mathcal{O}_T(1),
\end{split}
\end{align}
which does not change the result
\[ C(\cdot) = \sqrt{2/(2\beta+1)\log T} + o_{\Pr_b,\textrm{unif}}(1)\]
in \eqref{eqp: E2}, and we get the upper bound
\begin{align*}
&\frac{\left|\frac{1}{\sqrt{T}}\int_0^T K_{y_T,c''h_T}(X_s)\left(b_T(X_s) - b_0(X_s)\right)ds\right|-\frac{\eta}{\sqrt{T}}\int_0^T K_{y_T,c''h_T}(X_s) ds}{\sigma \sqrt{\frac1T\int_0^T K_{y_T,c''h_T}(X_s)^2 ds}} \\
&\hspace{5cm} -C\left(\frac{\frac1T\int_0^T K_{y_T,c''h_T}(X_s)^2 ds}{\frac1T \int_0^T \1_{[-A,A]}(X_s) ds}\right)\\
&\hspace{0.2cm}\geq c^{\frac{2\beta+1}{2\beta}} \tilde{c}(\beta, L, K) \sqrt{\log T} (1+\epsilon_T)^{\frac{2\beta+1}{2\beta}}  q_{b_T}(y_T)^{-\frac12}\frac{q_{b_T}(y_T) - a_T^{b_T}(y_T)}{\sqrt{q_{b_T}(y_T) + a_T^{b_T}(y_T)}}\\
&\hspace{5cm} - \sqrt{\frac{2}{2\beta+1} \log T} + o_{\Pr_b,\textrm{unif}}(1)\\
&\hspace{0.2cm} = \left( c^{\frac{2\beta+1}{2\beta}} \tilde{c}(\beta, L, K) (1+\epsilon_T)^{\frac{2\beta+1}{2\beta}}  q_{b_T}(y_T)^{-\frac12}\frac{q_b^{b_T}(y_T) - a_T^{b_T}(y_T)}{\sqrt{q_{b_T}(y_T)+ a_T^{b_T}(y_T)}} \right. \\
& \hspace{6.5cm}\left. - \sqrt{\frac{2}{2\beta+1}}\right)\sqrt{\log T} + o_{\Pr_b,\textrm{unif}}(1).
\end{align*}
Treating the term with $\epsilon_T$ and $a_T^{b_T}(y_T)$ in the same way as in the case $\beta\leq 1$, the bracket equals  
\begin{align*}
&c^{\frac{2\beta+1}{2\beta}} \tilde{c}(\beta, L, K) \left( 1+ \frac{2\beta+1}{2\beta}\epsilon_T ( 1+ \mathcal{O}(\epsilon_T)) - \frac{3}{\sqrt{2}}\frac{a_T^{b_T}(y_T)}{ q_{b_T}(y_T)} \right. \\
&\hspace{5cm}\left. \cdot\left( 1+ \frac{2\beta +1}{2\beta} \epsilon_T + \mathcal{O}(\epsilon_T^2)\right)\right) -\sqrt{\frac{2}{2\beta+1}}
\end{align*}
and by the same arguments as above this term multiplied by $\sqrt{\log T}$ diverges if we choose $c$ large enough to suffice
\[ c^{\frac{2\beta+1}{2\beta}} \tilde{c}(\beta, L, K)\geq \sqrt{\frac{2}{2\beta+1}}. \]

To complete the proof, we assume there exists a sequence $(b_T)_{T\geq 0}$ with $\|b_T - b_0\|_J-\eta\geq \tilde\delta_T$ and
\[ \liminf_{T\to\infty} \Pr_{b_T}\left( T_{b_T}^\eta(X) > \kappa_{\eta,\alpha}\right) = c<1.\]
Then there exists a subsequence $(b_{T_n})_n$ with $\lim_n \Pr_{b_{T_n}}( T_{b_{T_n}}^\eta(X) >\kappa_{\eta,\alpha}) =c$. By the item shown above this implies that $\gamma_{T_n}(b_{T_n})$ does not converge to zero, whence there exists a subsubsequence $(b_{T_{n_k}})_k$ such that $\gamma_{T_{n_k}}(b_{T_{n_k}})$ is uniformly bounded away from zero, for $\gamma_T$ being given in \eqref{eqp: upper_gamma_T}. If we can show in such a case that our test attains asymptotic power one, we have a contradiction and are done with the proof.\\
We again index the subsubsequence in $T$ for notational clarity. 
By Lemma~\ref{lemma_f_beta>1_1} there exists a sequence of intervals $I_T$ with $\liminf_{T\to\infty}\lambda(I_T)>0$ and $|b_T-b_0|-\eta\geq \frac{\gamma_T}{2}$ on $I_T$. In particular, this holds true on each $[y_T-h_T, y+h_T]$ for $T$ large enough and we can repeat the argument from the first part, where 
\[ \int K_{y_T,h_T}(z)\left( |b_T(z) - b_0(z)|-\eta\right) dz \geq \frac{\gamma_T}{2}\int_{y_T-h_T}^{y_T+h_T} K_{y_T,h_T}(z) dz\]
is of larger order than the right-hand side of \eqref{eqp: E11}. Hence, the whole term
\[\frac{\left| \frac{1}{\sqrt{T}} \int_0^T K_{y_T,h_T}(X_s) (b_T(X_s) - b_0(X_s)) ds\right| - \frac{\eta}{\sqrt{T}} \int_0^T K_{y_T,h_T}(X_s) ds}{\sigma \sqrt{\frac1T \int_0^T K_{y_T,h_T}(X_s)^2 ds}}\]
dominates the correction term $C(\cdot)$ in order and therefore their difference tends to infinity.
\end{proof}

\begin{proof}[Proof of Theorem \ref*{thm_adaptivity}]
The proof follows the lines of Theorem~\ref*{Upper_bound}. For rate adaptivity in both $\beta$ and $L$ we can no longer choose the kernel $K_\beta$ that depends on $\beta$. Hence, the constant $c(\beta,L,K)$ has to be modified in the case $\beta\leq 1$, which is done in the same spirit as for $\beta>1$ in the proof of Theorem~\ref*{Upper_bound}.
Then we note that
\begin{align}\label{eqp: E13}
\frac{1}{\beta}\log\left(\frac{c(\beta,L,K)}{L}\right) +\frac{1}{\beta}\left(1+\mathcal{O}(\epsilon_T)\right)  + \log(c''(\beta,L))
\end{align} 
is uniformly bounded by Lemma~\ref{lemma_f_beta>1_2} over $(\beta,L)\in [\beta_1,\beta_2]\times [L_1,L_2]$ and hence \eqref{eqp: h_T} and \eqref{eqp: E12} are still valid. The constant $c''(\beta,L)$ was introduced at the beginning of the case $\beta>1$ in the proof of Theorem~\ref*{Upper_bound} and we have $c''(\beta,L)=1$ for $\beta\leq 1$. With boundedness of \eqref{eqp: E13} in $(\beta,L)$, the rest of the proof is then the same as for Theorem~\ref*{Upper_bound}. \\
For sharp adaptivity in $L$, one uses that $\log(c_*/L)$ is uniformly bounded for $L\in [L_1,L_2]$ and again \eqref{eqp: h_T} and \eqref{eqp: E2} remain valid.
\end{proof}

\newpage
\begin{proof}[Proof of Theorem \ref*{thm_simple_hypothesis}]
Remember that for the quantiles $\kappa_{T,\alpha}^{\neq b_0}$ given in \eqref*{eq: quantile_simple} we have $\limsup_{T\to\infty}|\kappa_{T,\alpha}^{\neq b_0}|<\infty$ by the asymptotically tightness of $T_T^{b_0}$. Because of this, the proof works as that for Theorem \ref{Upper_bound} with $\eta=0$, i.e. we show that uniformly in probability
\begin{align}\label{eqp: E15}
\frac{\frac{1}{\sqrt{Th}} \left|\int_0^T K_{y,h}(X_s)\left(b(X_s)-b_0(X_s)\right) ds\right|}{\sigma \sqrt{\frac{1}{Th} \int_0^T K_{y,h}(X_s)^2 ds}} - C\left(\frac{\frac{1}{T} \int_0^T K_{y,h}(X_s)^2 ds}{\frac1T \int_0^T \1_{[-A,A]}(X_s) ds}\right)
\end{align}
tends to infinity for all $b_T-b_0\in H_{\neq} \cap\{ b-b_0\in \mathcal{H}(\beta, L)\}$ such that 
\begin{align}\label{eqp: E14}
\|(b_T-b_0) (q_{b_T}/\sigma^2)^{\beta/(2\beta+1)}\|_J \geq (1+\epsilon_T)c_*\delta_T, 
\end{align} 
with a suitable choice of $(y,h)=(y_T,h_T)$.
Denote by $y_T$ a location where the supremum in \eqref{eqp: E14} is attained and define $\tilde\delta_T$ and $h_T$ by \eqref{eqp: E_delta} and \eqref{eqp: E_hT}. This means, $\gamma_T:=|b_T(y_T) - b_0(y_T)|\geq \tilde\delta_T$. By the maximum condition of $y_T$ we have for all $z\in J$,
\begin{align*}
\left| b_T(y_T) - b_0(y_T)\right| q_{b_T}(y_T)^\frac{\beta}{2\beta+1} &\geq \left| b_T(z) - b_0(z)\right| q_{b_T}(z)^\frac{\beta}{2\beta+1} \\
& \geq \left| b_T(z) - b_0(z)\right| (L_*)^\frac{\beta}{2\beta+1},
\end{align*}
where the second step holds true by Lemma~\ref{uniform_lower_bound_invariant_density}. In particular,
\begin{align}\label{eqp: E1}
\begin{split}
\left| b_T(z) - b_0(z)\right|&\leq (L_*)^{-\frac{\beta}{2\beta+1}} \left| b_T(y_T) - b_0(y_T)\right| q_{b_T}(y_T)^\frac{\beta}{2\beta+1}\\
& \leq \left(\frac{L^*}{L_*}\right)^\frac{\beta}{2\beta+1} \gamma_T .
\end{split}
\end{align} 
With $d_{T,y_T}(z) := \frac{1}{\sigma^2 T} L_T^z(X)-q_{b_T}(y_T)$, we have as in the proof of Theorem~\ref*{Upper_bound}, 
\begin{align*}
&\frac{1}{\sqrt{Th_T}} \left|\int_0^T K_{y_T,h_T}(X_s)\left(b(X_s)-b_0(X_s)\right) ds\right| \\
&\hspace{1cm} = \sqrt{T/h_T} \left|\int_\R K_{y_T,h_T}(z) \left( b_T(z)-b_0(z)\right) \left( q_{b_T}(y_T) + d_{T,y_T}(z)\right) dz\right| \\
&\hspace{1cm}\geq \sqrt{T/h_T} q_{b_T}(y_T) \left| \int_\R K_{y_T,h_T}(z) \left( b_T(z) - b_0(z)\right) dz\right| \\
&\hspace{2cm} - \sqrt{T/h_T} a_T^{b_T}(y_T) \int_\R \left| K_{y_T,h_T}(z)\right| \left| b_T(z)- b_0(z)\right| dz.
\end{align*}
Again, $a_T^{b_T}(y_T) := \left\| \frac{1}{\sigma^2 T} L_T^\cdot - q_{b_T}(y_T)\right\|_{[y_T-h_T, y_T+h_T]} $.
From the bound on $|b_T-b_0|$ in \eqref{eqp: E1} we directly get
\begin{align*}
&- \sqrt{T/h_T} a_T^{b_T}(y_T) \int_\R \left| K_{y_T,h_T}(z)\right| \left| b_T(z)- b_0(z)\right| dz \\
&\hspace{3cm} \geq -\sqrt{Th_T} a_T^{b_T}(y_T) \gamma_T \left(\frac{L^*}{L_*}\right)^\frac{\beta}{2\beta+1} \|K_\beta\|_{L^1}.
\end{align*} 
To find a lower bound for the other term, note that $\frac{\tilde\delta_T}{\gamma_T}|b_T(y_T) - b_0(y_T)|=\tilde\delta_T$ and as $\gamma_T\geq\tilde\delta_T$, by Lemma \ref{Lemma_rescaling}~(i) we have
\[ \frac{\tilde\delta_T}{\gamma_T} \left( b_T(z) - b_0(z)\right) \in \mathcal{H}\left( \beta, (\tilde\delta_T/\gamma_T) L\right) \subset\mathcal{H}(\beta, L)\]
and
\[ \tilde\delta_T K_{y_T,h_T} \in \mathcal{H}\left(\beta, \tilde\delta_T h_T^{-\beta}\right) = \mathcal{H}(\beta, L).\]
Thus, by Lemma~\ref{Lemma_rescaling}~(iii), 
\begin{align*}
&\left| \int_\R K_{y_T,h_T}(z) \left( b_T(z) - b_0(z)\right) dz\right|\\
&\hspace{2cm}= \frac{\gamma_T}{\tilde\delta_T^2} \left| \int_\R \tilde\delta_T K_{y_T,h_T}(z) \frac{\tilde\delta_T}{\gamma_T} \left( b_T(z) - b_0(z)\right) dz\right| \\
&\hspace{2cm}\geq \frac{\gamma_T}{\tilde\delta_T^2} \left\| \tilde\delta_T K_{y_T,h_T}\right\|_{L^2}^2\\
&\hspace{2cm}= h_T\gamma_T \|K_\beta\|_{L^2}^2.
\end{align*}
Together, we have
\begin{align*}
&\frac{1}{\sqrt{Th_T}} \left|\int_0^T K_{y_T,h_T}(X_s)\left(b(X_s)-b_0(X_s)\right) ds\right| \\
&\hspace{1cm} \geq \sqrt{Th_T}\gamma_T \|K_\beta\|_{L^2}^2 \underbrace{\left( q_{b_T}(y_T) - a_T^{b_T}(y_T) \left(\frac{L^*}{L_*}\right)^\frac{\beta}{2\beta+1} \frac{\|K_\beta\|_{L^1}}{\|K_\beta\|_{L^2}^2}\right)}_{=: q_{b_T}(y_T) - a_T^{b_T}(y_T) C'}.
\end{align*}
In the proof of Theorem~\ref*{Upper_bound} we have seen in \eqref{eqp: E3} that
\[ \sqrt{\frac{1}{Th_T} \int_0^T K_{y_T,h_T}(X_s)^2 ds} \geq  \|K_\beta\|_{L^2} \sqrt{ q_{b_T}(y_T) + a_T^{b_T}(y_T)}\]
and in \eqref{eqp: E2} that 
\[ C\left(\frac{\frac1T \int_0^T K_{y_T,h_T}(X_s)^2 ds}{\frac1T \int_0^T \1_{[-A,A]}(X_s) ds}\right) \leq \sqrt{\frac{2}{2\beta+1}\log T} + o_{\Pr_b, \textrm{unif}}(1).\]
Consequently, we have the following lower bound for \eqref{eqp: E15}, which uses the evaluation of the term $\sqrt{Th_T}\tilde\delta_T$ in \eqref{eqp: E16}:
\begin{align*}
&\sqrt{Th_T}\tilde\delta_T \frac{\gamma_T}{\tilde\delta_T} \frac{\|K_\beta\|_{L^2}}{\sigma} \frac{ q_{b_T}(y_T) - C' a_T^{b_T}(y_T)}{\sqrt{q_{b_T}(y_T) + a_T^{b_T}(y_T)}} - \sqrt{\frac{2}{2\beta+1}\log T} + o_{\Pr_b,\textrm{unif}}(1)\\
&\hspace{1cm} \geq \sqrt{\frac{2}{2\beta+1} \log T}\left( \frac{\gamma_T}{\tilde\delta_T} (1+\epsilon_T)^\frac{2\beta+1}{2\beta} q_{b_T}(y_T)^{-\frac12} \right. \\
&\hspace{5cm} \left. \cdot \frac{ q_{b_T}(y_T) - C' a_T^{b_T}(y_T)}{\sqrt{q_{b_T}(y_T) + a_T^{b_T}(y_T)}} - 1\right) + o_{\Pr_b,\textrm{unif}}(1).
\end{align*}
Again as in the proof of Theorem~\ref*{Upper_bound}, we use a Taylor expansion to obtain
\[ (1+\epsilon_T)^\frac{2\beta+1}{2\beta} = 1+ \frac{2\beta+1}{2\beta} \epsilon_T + \mathcal{O}(\epsilon_T^2).\]
For $x\geq 0$ we have
\[ \frac{d}{dx}\frac{a-C'x}{\sqrt{a+x}} = -\frac{1}{\sqrt{a+x}}\left( C'+\frac12\frac{a-C'x}{a+x}\right) \geq -\frac{1}{\sqrt{a}} \left( C'+2^{-1}\right)\]
and conclude by a Taylor expansion that
\[ \frac{ q_{b_T}(y_T) - C' a_T^{b_T}(y_T)}{\sqrt{q_{b_T}(y_T) + a_T^{b_T}(y_T)}} \geq \sqrt{q_{b_T}(y_T)} - c a_T^{b_T}(y_T)\]
for some constant $c>0$. Our lower bound now turns into $o_{\Pr_b,\textrm{unif}}(1)$ plus
\begin{align*}
&\sqrt{\frac{2}{2\beta+1}\log T}\left( \frac{\gamma_T}{\tilde\delta_T}\left( 1+\frac{2\beta+1}{2\beta}\epsilon_T +\mathcal{O}(\epsilon_T^2)\right) \left( 1 -\frac{c a_T^{b_T}(y_T)}{\sqrt{q_{b_T}(y_T)}}\right) -1\right) \\
&\hspace{0.5cm}\geq\sqrt{\frac{2}{2\beta+1}\log T}\left( \frac{2\beta+1}{2\beta}\epsilon_T \left( 1+\mathcal{O}(\epsilon_T)\right) - \frac{c a_T^{b_T}(y_T)}{\sqrt{q_{b_T}(y_T)}} (1+\mathcal{O}(\epsilon_T))\right).
\end{align*}
This term tends to infinity, as by choice of $\epsilon_T$ we have $\epsilon_T\sqrt{\log T} \rightarrow\infty$ and $\sqrt{\log T} a_T^{b}(y) \rightarrow_{\Pr_{b},\textrm{unif}} 0$ by Lemma~\ref{local_time-Invariant}.
\end{proof}

\begin{remark}\label{remark_fixed_start_upper}
The results in Theorem \ref*{Upper_bound}, \ref*{thm_adaptivity} and \ref*{thm_simple_hypothesis} and their auxiliary results (in particular Theorem \ref*{weak_conv}) were derived under the assumption that the diffusion $X$ is started in the invariant density, i.e. $X_0\sim\mu_b$. As outlined in Remark \ref{remark_moment_invariant} and given in detail in Subsection~\ref{A3_SubSec}, the moment inequality for the deviation of the normalized local time and invariant density given in Proposition \ref{moments_local-invariant} remains valid under the assumption that $X_0=x_0\in [-A,A]$ is fixed. Our results used stationarity, i.e. $X_0\sim \mu_b$, only via Proposition \ref{moments_local-invariant} and hence, Theorem \ref*{Upper_bound}, \ref*{thm_adaptivity} and \ref*{thm_simple_hypothesis} remain valid if the diffusion $X$ is started at a fixed point $X_0=x_0\in [-A,A]$. \\
A closer look reveals that one only needs the uniform stochastic convergence 
\[ \sup_{b\in\Sigma(C,A,\gamma,\sigma)} \Pr_b\left( \log T \left\| \frac{1}{\sigma^2 T} L_T^\cdot(X) - q_b\right\|_{[-A,A]} >\epsilon\right) \rightarrow 0 \]
to derive these theorems, which is a weaker statement than the moment inequality in Proposition~\ref{moments_local-invariant}.
\end{remark}

\vspace{0.5cm}

\section{Proofs for Section \ref{Sec_Stability}}\label{App_stability}

This Section contains all proofs of Section~\ref{Sec_Stability}. In Subsection~\ref{SubSec_G1} a proof of the continuity result in Theorem~\ref{thm_continuity} is given. In Subsection~\ref{SubSec_G2} some details about fractional Brownian motion and the fractional diffusion model are given, in particular a Girsanov-type formula and some preliminaries on fractional calculus together with the proof of Proposition~\ref{lemma_X^H-X}. Those results are used in the subsequent Subsection~\ref{SubSec_G4} and~\ref{SubSec_G5} that contain the proofs of Theorem~\ref{prop_stability} and Theorem~\ref{Continuity_lower_H}.

\subsection{Proof of Theorem \ref{thm_continuity}}\label{SubSec_G1}

We begin with a preliminary lemma.

\begin{lemma}\label{continuity_ds-integral}
Let $s\in \mathcal{S}$ be an element of a compact metric space $(\mathcal{S},d_\mathcal{S})$ and $f:\mathcal{S}\times\R\rightarrow\R$ be a function such that 
\begin{enumerate}
\item[(a)] $|f|\leq c$ is bounded,
\item[(b)] $f(s,\cdot)$ is uniformly continuous with $\textrm{supp}(f(s,\cdot))\subset [-A,A]$ for every $s\in\mathcal{S}$, and
\item[(c)] $s\mapsto f(s,y)$ is continuous for almost all $y\in\R$.
\end{enumerate}
Then for fixed $T>0$, the map 
\begin{align*}
\mathcal{C}([0,T])\times \mathcal{S}\ \ &\rightarrow\ \ \ \ \ \R,\\
((x_u)_{0\leq u\leq T}, s)\ \  &\mapsto\ \int_0^T f(s, x_u) du,
\end{align*}
is continuous in $(x,s)$ for all $x\in\mathcal{C}([0,T])$ and $s\in\mathcal{S}$, where the left-hand side is equipped with the product topology on $\mathcal{C}([0,T])\times \mathcal{S}$, where $\mathcal{C}([0,T])$ is considered as a normed vector space with $\|\cdot\|_{[0,T]}$.
\end{lemma}
\begin{proof}[Proof]
Let $\epsilon>0$. We have
\begin{align*}
&\left| \int_0^T f(s,x_u) du - \int_0^T f(s', x_u')du\right|\\
&\hspace{1.5cm} \leq \left| \int_0^T f(s,x_u) - f(s',x_u) du\right| + \left| \int_0^T f(s',x_u)-f(s',x_u') du \right|.
\end{align*} 
By dominated convergence and assumption (a) and (c) there exists a $\delta_1>0$ such that 
\[ \left| \int_0^T f(s,x_u) - f(s',x_u) du\right| \leq \frac{\epsilon}{2} \]
for $d_\mathcal{S}(s,s')<\delta_1$. 
By uniform continuity of $f(s,\cdot)$, for each $s\in\mathcal{S}$ there exists $\delta_s>0$ such that
\[ \left| f(s,y) - f(s,y')\right| <\frac{\epsilon}{2T} \quad \textrm{ for all }\ |y-y'|<\delta_s.\]
As $\mathcal{S}$ is compact, we have $\delta_2 := \inf_{s\in\mathcal{S}} \delta_s>0$ because the infimum is attained for some $s_0$. Then for $\|x-x'\|_{[0,T]}<\delta_2$, 
\begin{align*}
\left| \int_0^T  f(s', x_u) - f(s', x_u') du\right| \leq T \sup_{u\in [0,T]}\left| f(s', x_u) - f(s', x_u')\right| <\frac{\epsilon}{2}.
\end{align*}
Putting together both estimates, the claim follows.
\end{proof}

\begin{proof}[Proof of Theorem \ref{thm_continuity}]
Denote 
\[ \overline{\mathcal{T}}_T:= \{(y,h)\in\R^2\mid h_{\min}(T)\leq h\leq A \textrm{ and } -A+h\leq y\leq A-h\}.  \]
With Lemma \ref{continuity_ds-integral}, we establish continuity of the mapping
\begin{align}\label{eqp: Gm1}
\begin{split}
\mathcal{C}([0,T])\times \overline{\mathcal{T}}_T\ \ &\longrightarrow\ \ \ \ \ \ \R,\\
( (x_s)_{s\in [0,T]}, (y,h))\  &\ \mapsto\ \  \int_0^T f_i((y,h), x_s) ds,
\end{split}
\end{align}
for the functions $f_1((y,h),z) := K_{y,h}(z)^2$, $f_2((y,h),z) := K_{y,h}(z) b_0(z)$ and $f_3((y,h),z) := \frac{1}{h} K'((y-z)/h)$ and additionally of
\begin{align}\label{eqp: Gm2}
\begin{split}
\mathcal{C}([0,T])\times \overline{\mathcal{T}}_T\ \ &\longrightarrow\ \ \ \ \ \ \R,\\
( (x_s)_{s\in [0,T]}, (y,h))\ \ &\ \mapsto\ \ \int_{x_0}^{x_T} K_{y,h}(z) dz.
\end{split}
\end{align}
Then we can conclude that \vspace{-0.1cm}
\begin{align*}
\mathcal{C}([0,T])\times \overline{\mathcal{T}}_T\ \ &\longrightarrow\ \ \ \ \R,\\
( (x_s)_{s\in [0,T]}, (y,h))\ \ &\ \mapsto\ \ \tilde\Psi_{T,y,h}^{b_0}(x),
\end{align*}
is continuous as a concatenation of continuous mappings and afterwards the claim follows by continuity of the map $\Upsilon(\cdot)$ on $\R_{>0}$, compactness of $\overline{\mathcal{T}}_T$, denseness of $\mathcal{T}_T\subset \overline{\mathcal{T}}_T$ and the fact that $g(y) := \sup_{x\in X} f(x,y)$ is continuous for continuous $f:X\times Y\rightarrow\R$, where $X$ and $Y$ are metric spaces and $X$ is compact (cf. Theorem~$14.30$ in \cite{Aliprantis}).\\
For continuity of \eqref{eqp: Gm1}, we first note that for $(y,h), (y',h')\in\overline{\mathcal{T}}_T$ we have $K_{y,h}(x)\rightarrow K_{y',h'}(x)$ for all $x\in\R$ in the limit $(y',h')\to (y,h)$ by continuity of $K$ and the same holds true for $K'$. Therefore, for $f_i$, $i=1,2,3$, condition~(c) of Lemma~\ref{continuity_ds-integral} is fullfilled and condition~(b) is clear, as we work with continuous functions having compact support, which is a subset of $[-A,A]$. Moreover, $f_1$ and $f_2$ are obviously bounded and $f_3$ is bounded, as $h$ is bounded away from zero on $\overline{\mathcal{T}}_T$, so condition~(a) of Lemma~\ref{continuity_ds-integral} holds as well. Consequently, Lemma~\ref{continuity_ds-integral} is applicable and the mapping in \eqref{eqp: Gm1} is continuous for $f_i$, $i=1,2,3$. \\
Lastly, we show continuity of \eqref{eqp: Gm2} in $((x_s)_{s\in [0,T]}, (y,h))$ and pick some $((x_s')_{s\in [0,T]}, (y',h'))$. Suppose $x_0\leq x_T$ and $x_0'\leq x_T'$ (the other cases work the same), then 
\begin{align*}
&\left| \int_{x_0}^{x_T} K_{y,h}(z) dz - \int_{x_0'}^{x_T'} K_{y',h'}(z) dz\right| \\
&\hspace{0.5cm} = \left|\int_{x_0}^{x_T} K_{y,h}(z) - K_{y',h'}(z) dz + \int_{x_0}^{x_T} K_{y',h'}(z) dz - \int_{x_0'}^{x_T'} K_{y',h'}(z) dz\right|\\
&\hspace{0.5cm} \leq \left( x_T - x_0\right) \|K_{y,h} - K_{y',h'}\|_{[-A,A]} \\
&\hspace{1.5cm} + \left| \int_{\min\{x_0,x_0'\}}^{\max\{x_0,x_0'\}} K_{y',h'}(z) dz + \int_{\min\{x_T,x_T'\}}^{\max\{x_T,x_T'\}} K_{y',h'}(z) dz\right|\\
&\hspace{0.5cm} \leq \left( x_T - x_0\right) \|K_{y,h} - K_{y',h'}\|_{[-A,A]}  + 2\|x-x'\|_{[0,T]} \|K\|_{[-1,1]}.
\end{align*}
The second summand vanishes for $\|x-x'\|_{[0,T]}\to 0$, the first one for $(y',h')\to (y,h)$ because
\begin{align*}
\| K_{y,h} - K_{y',h'}\|_{[-A,A]} &\leq \sup_{z \in [-A,A]} \|K'\|_{[-1,1]}\left| \frac{z-y}{h} - \frac{z-y'}{h'}\right|\\
& =\|K'\|_{[-1,1]}\cdot \sup_{z \in [-A,A]} \left| \frac{z(h-h') - yh'+y'h}{hh'}\right|
\end{align*}
tends to zero in this limiting scenario. Continuity of \eqref{eqp: Gm2} follows.
\end{proof}

\subsection{Preliminaries}\label{SubSec_G2}
Here, we first give some important preliminary results on the fractional calculus that are frequently used in the proofs of the results of Section~\ref{Sec_Stability} in the following subsections. In the second part we introduce a change of measure formula for fractional Brownian motion and present how the density process can be expressed in terms of fractional integrals.

\subsubsection{Fractional calculus and hypergeometric function}\label{SubSub_fBM1}
First, we give some elementary definitions and results about fractional integrals and derivatives. The standard reference on this subject is \cite{Samko}.  \\
Let $\Gamma$ denote the Gamma function.
The \textit{(left-handed) Riemann--Liouville fractional integral} $I_{0+}^\alpha$ of order $\alpha>0$ for any function $f\in L^1([0,T])$ is defined as
\begin{align}\label{eqp: I_+}
\left( I_{0+}^\alpha f\right)(x) := \frac{1}{\Gamma(\alpha)} \int_0^x \frac{f(y)}{(x-y)^{1-\alpha}} dy.
\end{align}  
The fractional derivative can be introduced as its inverse operator. Assuming $0<\alpha<1$ and $p\geq1$, we denote by $I_{0+}^\alpha(L^p([0,T]))$ the image of $L^p([0,T])$ with respect to the operator $I_{0+}^\alpha$. Then for $p>1$ and for each $f\in I_{0+}^\alpha(L^p([0,T]))$, the function $g$ with $f=I_{0+}^\alpha g$ is unique in $L^p([0,T])$ and coincides with the \textit{(left-handed) Riemann--Liouville fractional derivative} $D_{0+}^\alpha f$ of $f$ of order $\alpha$, which is for any $f\in L^1([0,T])$ given almost everywhere by
\[\left( D_{0+}^\alpha f\right)(x) :=  \frac{1}{\Gamma(1-\alpha)} \frac{d}{dx}\int_0^x \frac{f(y)}{(x-y)^\alpha} dy.\]
On $I_{0+}^\alpha(L^1([0,T]))$, the fractional derivative has the so-called Weyl representation (cf. Corollary subsequent to Theorem $13.1$ in \cite{Samko})
\begin{align}\label{eq: Weil_rep}
\left( D_{0+}^\alpha f\right)(x) =  \frac{1}{\Gamma(1-\alpha)} \left( \frac{f(x)}{x^\alpha} + \alpha\int_0^x \frac{f(x)-f(y)}{(x-y)^{1+\alpha}} dy\right)\quad a.e.
\end{align} 
Last but not least, we give two examples that are frequently used in the following subsections and can be found as $(2.26)$ and $(2.44)$ in \cite{Samko}. For the function $f(x) = x^{-\mu}$ and $0<\mu<1$ we have
\begin{align}\label{eq: D_x^mu}
\left( D_{0+}^\alpha f\right)(x) = \frac{\Gamma(1-\mu)}{\Gamma(1-\mu-\alpha)} x^{-\mu-\alpha}
\end{align}
and for $f(x)= x^{\beta-1}$ with $\beta>0$ it holds that
\begin{align}\label{eq: I_x^beta}
\left( I_{0+}^\alpha f\right)(x) = \frac{\Gamma(\beta)}{\Gamma(\alpha+\beta)} x^{\alpha+\beta-1}.
\end{align}

The Gaussian hypergeometric function
\[ F:\C\times \C\times \left(\C\setminus\{0,-1,-2,\dots \}\right) \times \left(\C\setminus\{x\in\R: x\geq 1\}\right)\longrightarrow\C \]
is the analytic continuation of the power series
\begin{align}\label{eq: F_power}
(a,b,c;z)\ \mapsto\ \sum_{n=0}^\infty \frac{(a)_n (b)_n}{(c)_n n!} z^n, 
\end{align} 
with the Pochhammer symbol $(z)_0=1$ and $(z)_n := z(z+1)\cdots (z+n-1)$, $z\in\C$, $n\in\N$ (cf. \cite{Decreusefond}, Section $2$ and \cite{Nikiforov}, $IV.20.3$). 

\begin{remark}[Some properties of the hypergeometric function]\label{remark_property_F}
We give some properties of the Gaussian hypergeometric function $F(a,b,c;z)$  with $(a,b,c,z)\in\C\times \C\times \left(\C\setminus\{0,-1,-2,\dots \}\right) \times \left(\C\setminus\{x\in\R: x\geq 1\}\right)$ that can be found in the literature and are used below.
\begin{enumerate}
\item[(i)] For $\Re(c)>\Re(b)>0$, we have Euler's formula
\[ F(a,b,c;z) = \frac{\Gamma(a)}{\Gamma(b)\Gamma(c-b)}\int_0^1 t^{b-1}(1-t)^{c-b-1} (1-tz)^{-a} dt, \]
cf. $(10)$ in \cite{Erdelyi}, Section~$2.1.3$, where $\Re(z)$ denotes the real part of $z\in\C$.
\item[(ii)] $F$ is commutative in its first two arguments i.e. 
\[ F(a,b,c;z) = F(b,a,c;z). \]
This is given in $2.1.2$ in \cite{Erdelyi} for $|z|<1$ and follows for the analytic continuation by the identity theorem for analytic functions, cf. Theorem~$III.3.2$ in \cite{Freitag}, applied to $z\mapsto F(a,b,c;z)-F(b,a,c;z)$.
\item[(iii)] We have $F(a,0,c;z)=1$, which is immediate by the power series representation \eqref{eq: F_power} inside the convergence disk $\{z\in\C: |z|<1\}$ and follows for the analytic continuation by the identity theorem for analytic functions, cf. Theorem $III.3.2$ in \cite{Freitag}.
\item[(iv)] We have $F(H-1/2, 1/2-H,H+1/2,x)\geq 0$ for real arguments $x\leq 0$. This follows for $H<\frac12$ by Euler's formula in (i) and for $H>\frac12$ correspondingly when first applying the commutativity relationship (ii). The case $H=\frac12$ is obtained from (iii).
\end{enumerate}
\end{remark}

\subsubsection{Fractional Brownian motion}\label{SubSub_fBM}

Let $\Gamma$ denote the Gamma function and define the kernel $K_H$ for $s\leq t$ and $H\in (0,1)$ by
\begin{align}\label{eq: K_H}
K_H(t,s) := \frac{(t-s)^{H-\frac12}}{\Gamma\left(H+\frac12\right)} F\left( H-\frac12, \frac12-H, H+\frac12; 1-\frac{t}{s}\right),
\end{align} 
where $F$ denotes the Gaussian hypergeometric function, see Subsection~\ref{SubSub_fBM1}.
The covariance function $R_H$ of a fractional Brownian motion in \eqref{eq: R_H} has a representation of the form (see \cite{Nualart}, p. $106$ and \cite{Decreusefond}, Lemma $3.1$)
\begin{align}\label{eqp: R_H}
R_H(t,s) = \int_0^{t\wedge s} K_H(t,u) K_H(s,u) du
\end{align}
and it follows from this that the process $W^H=(W_t^H)_{t\in [0,T]}$ with
\[ W_t^H := \int_0^t K_H(t,s) dW_s, \quad t\in [0,T],\]
is a fractional Brownian motion on the interval $[0,T]$. By $R_{1/2}(s,t) = s\wedge t$, we see that $W^H|_{H=1/2} = W$ is a standard Brownian motion. \\
To move on to a change of measure formula for the fractional Brownian motion, we introduce the operator $K_H$ associated with the kernel $K_H(\cdot, \cdot)$ on $L^2([0,T])$ that is given by
\[ \left( K_Hf\right)(t) := \int_0^t K_H(t,s) f(s) ds. \]
This operator is an isomorphism from $L^2([0,T])$ onto $I_{0+}^{H+\frac12}(L^2([0,T])$ (cf.~\cite{Samko}, Theorem~$10.4$).
It can be expressed for $f\in L^2([0,T])$ in terms of fractional integrals in the following way:
\begin{align}\label{eq: K_H_rep}
\left( K_H f\right) (t)  =\begin{cases}
I_{0+}^{2H}\left( (\cdot)^{\frac12-H} I_{0+}^{\frac12-H}\left( (\cdot)^{H-\frac12} f(\cdot)\right)(\cdot)\right)(t) & \quad \textrm{ for } H < \frac12, \\
I_{0+}^{1}\left( (\cdot)^{H-\frac12} I_{0+}^{H-\frac12}\left( (\cdot)^{\frac12-H} f(\cdot)\right)(\cdot)\right)(t) & \quad \textrm{ for } H > \frac12,
\end{cases}
\end{align}
see equation $(7)$ and $(8)$ in \cite{Nualart} and Theorem $10.4$ in \cite{Samko}. For absolutely continuous $f\in I_{0+}^{H+1/2}\left(L^2([0,T])\right)$, its inverse operator $K_H^{-1}$ can be expressed as 
\begin{align}\label{eq: K_H^-1_rep}
\left(K_H^{-1}f\right)(s) = \begin{cases}
s^{H-\frac12} I_{0+}^{\frac12-H} \left( (\cdot)^{\frac12-H} f'(\cdot)\right)(s) & \quad \textrm{ for } H<\frac12, \\
s^{H-\frac12} D_{0+}^{H-\frac12} \left( (\cdot)^{\frac12-H} f'(\cdot)\right)(s) & \quad \textrm{ for } H>\frac12,
\end{cases}
\end{align}
with weak derivative $f'$, see equation $(11)$ and $(13)$ in \cite{Nualart} where a derivation is given.

With the help of this inverse operator $K_H^{-1}$ we can now give a formulation of Girsanov's theorem for the fractional Brownian motion. Other sources on change of measure for fractional Brownian motion are \cite{Decreusefond} and \cite{Nualart}.

\begin{proposition}[Theorem $1$ in \cite{Tudor}]\label{Girsanov_fBM}
Let $u=(u_t)_{t\in [0,T]}$ be an adapted process with integrable trajectories and set
\[ \tilde{W}_t^H := W_t^H + \int_0^t u_s ds.\]
Assume that
\begin{enumerate}
\item[(a)] $\int_0^\cdot u_s ds\in I_{0+}^{H+\frac12}(L^2([0,T]))$ almost surely,
\item[(b)] $\E[Z_T^H(u)]=1$, where
\begin{align*}
&Z_T^H(u) = \exp\left( -\int_0^T K_H^{-1}\left(\int_0^\cdot u_r dr\right)(s) dW_s \right. \\
&\hspace{5cm} \left. - \frac12\int_0^T K_H^{-1}\left(\int_0^\cdot u_r dr\right)^2 (s) ds\right).
\end{align*} 
\end{enumerate}
Then the shifted process $\tilde{W}^H$ is a fractional Brownian motion with Hurst parameter $H$ under the new probability measure $\tilde{\Pr}$ defined by $\frac{d\tilde{\Pr}}{d\Pr} = Z_T^H(u)$ and the process
\[ W_t + \int_0^t K_H^{-1}\left( \int_0^\cdot u_r dr\right)(s) ds \]
is a Brownian motion.
\end{proposition}

\begin{remark}\label{remark_vg_adapted}
Note that from the representation \eqref{eq: K_H^-1_rep} and the form of the fractional operators in \eqref{eqp: I_+} and \eqref{eq: Weil_rep} it can be seen that the density process $Z_T^H(u)$ in Proposition~\ref{Girsanov_fBM} is adapted for $u_s:= g(X_s^{H,b})$, where $X^{H,b}$ solves \eqref{eq: SDE_H} and $g$ is a measurable function. If $g$ is in particular Lipschitz continuous, then Proposition~$1$ in \cite{Tudor} reveals that 
\[ \left( K_H^{-1}\left( \int_0^\cdot g(X_s^{H,b}) ds\right)(t)\right)_{t\in [0,T]} \in L^2([0,T]). \]
Thus, the stochastic integral in Proposition \ref{Girsanov_fBM} is well-defined in the classical It\^{o} sense. 
\end{remark}

\begin{remark}
In our paper we work under the condition $u_s = g(X^{H,b}_s)$ for a bounded and Lipschitz continuous function $g$ and need condition (a) in Proposition~\ref{Girsanov_fBM} to be satisfied for those.\\
For $H<\frac12$, this condition (a) follows from $\int_0^T u_s^2 ds <\infty$, which was derived in \cite{Bai} within the proof of Lemma~$4.1$ on p. 313--314. Since $g$ is bounded, this follows trivially for $u=g(X_\cdot^{H,b})$. For $H>\frac12$, at the beginning of Section~$5$ on p. 316--317 it has been derived that condition (a) is satisfied if $u$ has trajectories that are Hölder continuous of order $H-\frac12+\epsilon$ for some $\epsilon>0$. This is true for $u=g(X_\cdot^{H,b})$ due to the Hölder continuity of order $H-\epsilon$ with arbitrary $\epsilon\in (0,H)$ of the fractional Brownian motion $W^H$ and the Lipschitz assumption on $g$.
\end{remark}


\begin{proof}[Proof of Proposition \ref{lemma_X^H-X}]
From the SDE \eqref{eq: SDE_H} we directly get
\[ X_t^{H,b} - X_t^b = \int_0^t  \big( b(X_s^{H,b}) - b(X_s^b)\big) ds + \sigma( W_t^H - W_t).\]
As each $b\in\Sigma(C,A,\gamma,\sigma)\cap\mathcal{H}(1,L)$ is Lipschitz with constant $L$, we derive
\[ |X_t^{H,b} - X_t^b| \leq L\int_0^t |X_s^{H,b} - X_s^b| ds + \sigma |W_t^H -W_t|.\]
From Gronwall's lemma (cf. $(2.11)$ in \cite{Karatzas/Shreve}, Chapter $5$ with corresponding proof on p. 387--388) we then get
\[ |X_t^{H,b} - X_t^b| \leq \sigma |W_t^H - W_t| + L \int_0^t |W_s^H - W_s| e^{L (t-s)} ds\]
and in consequence
\[ \|X^{H,b} - X^b\|_{[0,T]} \leq \|W^H -W\|_{[0,T]} \left( \sigma + L\int_0^t e^{L (t-s)} ds\right).\]
Note that the right-hand side does not depend on $b$ anymore. 
Because $\|W^H -W\|_{[0,T]}$ converges to zero in probability, which was shown in Section~$3.8.3$ of \cite{Brutsche}, the claim follows.
\end{proof}

\addtocontents{toc}{\protect\enlargethispage{\baselineskip}}
\subsection{Proof of Theorem \ref{prop_stability}}\label{SubSec_G4}

First, we provide some helpful preliminary results. At the end of this subsection, the proofs of Theorem~\ref{prop_stability} and Remark~\ref{remark_stability} are given.

\begin{lemma}[Uniform continuous mapping]\label{lemma_CMT_uniform}
Let $(E,d)$ be a separable metric space and $f:E\rightarrow\R$ a uniformly continuous function. Suppose for a parameter space $\Theta$ and $E$-valued random variables $X,X_1,X_2,\dots$ we have uniform stochastic convergence, i.e. for every $\epsilon>0$,
\[ \sup_{\theta\in\Theta} \Pr_\theta \left( d(X_n,X)>\epsilon \right) \stackrel{n\to\infty}{\longrightarrow} 0.\]
Then we also have for every $\epsilon>0$ that
\[ \sup_{\theta\in\Theta} \Pr_\theta \left( |f(X_n)-f(X)|>\epsilon \right) \stackrel{n\to\infty}{\longrightarrow} 0.\]
\end{lemma}
\begin{proof}
By the uniform continuity of $f$, for every $\epsilon>0$ we can find $\delta_\epsilon>0$ such that $|f(x)-f(y)|\leq\epsilon$ whenever $d(x,y)\leq\delta_\epsilon$. Then
\begin{align*}
\sup_{\theta\in\Theta} \Pr_\theta \left( |f(X_n)-f(X)|>\epsilon \right) \leq \sup_{\theta\in\Theta} \Pr_\theta\left( d(X_n,X) >\delta_\epsilon\right)
\end{align*}
and the right-hand side converges to zero by assumption.
\end{proof}

From here on, we define for $c\in\R$ the set
\begin{align}\label{eqp: G21}
D_{c,T} := \{ f\in\mathcal{C}([0,T])\mid -c,c \in \textrm{im}(f)\}. 
\end{align}

\begin{lemma}\label{lemma_D_A}
There exists $\epsilon>0$ such that
\[ \lim_{T\to\infty} \inf_{b\in \Sigma(C,A,\gamma,\sigma)\cap\mathcal{H}(1,L)} \Pr\left( X^b\in D_{A+\epsilon,T}\right) =1.\]
\end{lemma}
\begin{proof}
By $X_0^b=x_0\in [-A,A]$, continuity of paths and the occupation times formula, we have for any $\epsilon>0$
\begin{align*}
&\Pr\left( X^b\in D_{A+\epsilon,T}\right)\\
&\hspace{0.5cm} \geq \Pr\left( \int_0^T \1_{(A+\epsilon, A+2\epsilon]}(X_s^b) ds >0 \textrm{ and } \int_0^T \1_{[ -A-2\epsilon, -A-\epsilon)}(X_s^b) ds >0\right) \\
&\hspace{0.5cm} =  \Pr\left( T \int_{A+\epsilon}^{A+2\epsilon} q_b(z) dz + T\int_{A+\epsilon}^{A+2\epsilon} \frac{1}{T\sigma^2}L_T^z(X^b) - q_b(z) dz >0 \right. \\
&\hspace{2cm} \left. \textrm{ and }T \int_{-A-2\epsilon}^{-A-\epsilon} q_b(z) dz + T\int_{-A-2\epsilon}^{-A-\epsilon} \frac{1}{T\sigma^2}L_T^z(X^b) - q_b(z) dz >0 \right).
\end{align*} 
Now we specify $\epsilon := L_*/(4L^*)$. Lemma \ref{bound_invariant_density} reveals that for all $A\leq y\leq A+2\epsilon$ and all $-A-2\epsilon\leq z\leq -A$,
\[ \left| q_b(A) - q_b(y)\right| \leq 2L^*\epsilon \leq \frac{L_*}{2}\ \textrm{ and }\ \left| q_b(-A) - q_b(z)\right| \leq 2L^*\epsilon \leq \frac{L_*}{2}. \]
Consequently, as $q_b(A),q_b(-A)\geq L_*$ by Lemma~\ref{bound_invariant_density}, 
\[ \int_{A+\epsilon}^{A+2\epsilon} q_b(z) dz \geq \frac{(L_*)^2}{8L^*}\ \ \textrm{ and }\ \  \int_{-A-2\epsilon}^{-A-\epsilon} q_b(z) dz \geq \frac{(L_*)^2}{8L^*}\]
and we can conclude \pagebreak
\begin{align*}
\Pr\left( X^b\in D_{A+\epsilon,T}\right)  &\geq \Pr\left( \epsilon \left\|\frac{1}{T\sigma^2}L_T^\cdot(X) - q_b\right\|_\infty < \frac{(L_*)^2}{16L^*}\right)\\
& = \Pr\left( \left\|\frac{1}{T\sigma^2}L_T^\cdot(X) - q_b\right\|_\infty < \frac{L_*}{4}\right).
\end{align*} 
In particular, 
\begin{align*}
&\liminf_{T\to\infty} \inf_{b\in \Sigma(C,A,\gamma,\sigma)\cap\mathcal{H}(1,L)} \Pr\left( X^b\in D_{A+\epsilon,T}\right) \\
&\hspace{1cm}  \geq \liminf_{T\to\infty} \inf_{b\in \Sigma(C,A,\gamma,\sigma)\cap\mathcal{H}(1,L)} \Pr\left( \left\|\frac{1}{T\sigma^2}L_T^\cdot(X) - q_b\right\|_\infty < \frac{L_*}{4}\right)\\
&\hspace{1cm} = 1 - \limsup_{T\to\infty}\sup_{b\in \Sigma(C,A,\gamma,\sigma)\cap\mathcal{H}(1,L)} \Pr\left( \left\|\frac{1}{T\sigma^2}L_T^\cdot(X) - q_b\right\|_\infty \geq \frac{L_*}{4}\right)
\end{align*}
and the right-hand side equals one by \eqref{eq:conv_emprical_density_uniform}.
\end{proof}

\begin{lemma}\label{lemma_D_A_closed}
The set $D_{A,T}\subset\mathcal{C}([0,T])$ is closed with respect to the topology of uniform convergence.
\end{lemma}
\begin{proof}
Consider a sequence $(f_n)_{n\in\N}$ with $f_n\in D_{A,T}$ for all $n\in\N$, that converges to some function $f\in\mathcal{C}([0,T])$, i.e. $\|f_n-f\|_{[0,T]}\rightarrow 0$. Because of $f_n\in D_{A,T}$, there exists a sequence $(t_n)_{n\in\N}$ with $f_n(t_n)=A$. By compactness of $[0,T]$, we find a subsequence $(t_{n_k})_{k\in\N}$ with $t_{n_k}\rightarrow t$ for $k\to\infty$ and some $t\in [0,T]$. For this $t$, we have
\begin{align*}
| f(t) - A| = \left|f(t) - f_{n_k}(t_{n_k}) \right| \leq \left| f(t) - f(t_{n_k})\right| + \left| f(t_{n_k}) - f_{n_k}(t_{n_k})\right|.
\end{align*}
The first summand vanishes in the limit $k\to\infty$ by continuity of $f$ and the second one by $\|f_{n_k}-f\|_{[0,T]}\rightarrow 0$. Thus, $A\in \textrm{im}(f)$. By the same reasoning, we see that $-A\in \textrm{im}(f)$ and conclude $f\in D_{A,T}$.
\end{proof}

\begin{lemma}\label{lemma_an_bn}
Let $(a_n)_{n\in\N}$ be a sequence of real numbers, $a\in\R$ and $|a_n -a|\to 0$. Let $(b_n)_{n\in\N}$ be a sequence with $b_n\in \{ a_k : k\in\N\}$, $b\in\R$, and $|b_n-b|\to 0$. Then one of the following is true:
\begin{itemize}
\item[(i)] $a=b$, or
\item[(ii)] there exists $n_0\in\N$ such that $b_n=b$ for all $n\geq n_0$.
\end{itemize} 
\end{lemma}
\begin{proof}
Suppose (i) does not hold and choose $0<\epsilon<|b-a|$. As $b$ is the limit of $(b_n)_{n\in\N}$, there exists $n_0$ such that $|b-b_n|<\frac{\epsilon}{2}$ for all $n\geq n_0$, in particular $|b_n-a|>\frac{\epsilon}{2}$ for $n\geq n_0$. As there are only finitely many $a_{n_1},\dots, a_{n_l}$ with $|a_{n_i}-a|>\frac{\epsilon}{2}$, $i=1,\dots, l$, we have $b_{n}\in \{a_{n_1},\dots, a_{n_l}\}=:A$ for all $n\geq n_0$. We conclude that $(b_n)_{n\geq n_0}$ is a convergent $A$-valued sequence and as $A$ only contains isolated points, (ii) holds.
\end{proof}

The following result is well-known, yet we did not find a version simplified for our context. For completeness, the proof is given.

\begin{lemma}\label{lemma_Arzela}
The set $\Sigma(C,A,\gamma,\sigma)\cap\mathcal{H}(1,L)$ is relatively compact with respect to the topology of uniform convergence on compact sets.
\end{lemma}
\begin{proof}
Let $(f_n)_{n\in\N}$ be a sequence in $\Sigma(C,A,\gamma,\sigma)\cap\mathcal{H}(1,L)$. We have to prove that there exists a subsequence which is convergent in $\mathcal{C}(\R)$ with respect to the topology of uniform convergence on compact sets.
For any compact set $K\subset\R$, the restriction of $\Sigma(C,A,\gamma,\sigma)\cap\mathcal{H}(1,L)$ to $K$ is relatively compact with respect to the topology of uniform convergence on $K$ by the Arzel\`{a}--Ascoli theorem. For any $n\in\N$, let $K_n:= [-n,n]$. Note that 
\[ \bigcup_{n\in\N} K_n = \R. \]
For $K_1$ we find a subsequence $\big(n_k^{(1)}\big)_{k\in\N}$ and an $f_1\in \mathcal{C}(K_1)$ such that 
\[ \left\| f_{n_k^{(1)}} -f_1\right\|_{K_1} \rightarrow 0\quad \textrm{ for } k\to\infty. \]
We proceed inductively as follows: If $\big( n_k^{(i)}\big)_{k\in\N}$ and $f_i$ are chosen, we find a subsequence $\big( n_k^{(i+1)}\big)_{k\in\N}$ of $\big( n_k^{(i)}\big)_{k\in\N}$ and $f_{i+1}$ such that
\[ \left\| f_{n_k^{(i+1)}} -f_{i+1}\right\|_{K_{i+1}} \rightarrow 0\quad \textrm{ for } k\to\infty. \]
Note that $f_i$ restricted to $K_j$ with $j<i$ coincides with $f_j$ on $K_j$ because $K_j\subset K_i$. Hence, there exists $f\in\mathcal{C}(\R)$ with $f$ restricted to $K_i$ equals $f_i$ for all $i\in\N$. Next, as every compact set $K\subset\R$ is (closed and) bounded, there exists $n\in\N$ such that $K\subset K_n$. Finally, an application of Cantor's diagonal subsequence argument reveals the diagonal sequence $(m_k)_{k\in\N}$ with $m_k := n_k^{(k)}$ such that for every compact set $K\subset\R$,
\[ \left\| f_{m_k} - f\right\|_{K} \rightarrow 0 \quad \textrm{ for } k\to\infty. \] 
\end{proof}

Here and subsequently, the image measure of a $\mathcal{C}([0,T])$-valued random variable $Y$ is denoted by $\Pr^Y$.

\begin{lemma}\label{lemma_X^H_tight}
Let $L>0$ and $(h_n)_{n\in\N}$ be a sequence with $h_n\in (0,1)$ for all $n\in\N$ and $|h_n-\frac12|\to 0$. Then the family of measures 
\[ \left\{\Pr^{X^{H,b}}: b\in\Sigma(C,A,\gamma,\sigma)\cap\mathcal{H}(1,L), H\in \{h_n: n\in\N\}\right\}\]
is tight.
\end{lemma}
\begin{proof}
Consider a sequence $(b_n,H_n)_{n\in\N}$ with $b_n\in \Sigma(C,A,\gamma,\sigma)\cap\mathcal{H}(1,L)$ and $H_n\in \{h_k:k\in\N\}$ for all $n\in\N$. As $\Sigma(C,A,\gamma,\sigma)\cap\mathcal{H}(1,L)$ is relatively compact with respect to the topology of uniform convergence on compact sets as a consequence of the Arzel\`{a}--Ascoli theorem (see Lemma \ref{lemma_Arzela}) and $(H_n)_{n\in\N}$ contains a convergent subsequence by the Bolzano--Weierstrass theorem as it is a bounded sequence, we find a subsequence $(b_{n_k},H_{n_k})_{k\in\N}$ with $\|b_{n_k} -b\|_K\to 0$ for any compact $K\subset\R$ and $|H_{n_k}-H|\to 0$. We have
\begin{align}\label{eqp: G2} 
\sup_{f\in\mathcal{F}_{BL}(\mathcal{C}([0,T]))}\hspace{-0.05cm} \E\left[ \left| f(X^{H_{n_k},b_{n_k}}) - f(X^{H,b})\right|\right]\hspace{-0.05cm} \leq \E\left[ \|X^{H_{n_k},b_{n_k}} - X^{H,b}\|_{[0,T]} \wedge 2\right]
\end{align}
by the definition \eqref{eq: F_BL} of $\mathcal{F}_{BL}$ in Subsection \ref{SubSec_C1}. In order to conclude that the right-hand side converges to zero by the dominated convergence theorem, we have to show that $\|X^{H_{n_k},b_{n_k}} - X^{H,b}\|_{[0,T]} \wedge 2$ converges to zero pointwise, i.e. $\omega$-wise. For each $\omega\in\Omega$, there exists a compact set $K=K(\omega)\subset\R$ such that $X_t^{H,b}(\omega)\in K$ for all $t\in [0,T]$ by continuity of $X^{H,b}(\omega)$. Then for $t\leq T$,
\begin{align*}
&\left| X_t^{H_{n_k},b_{n_k}} - X_t^{H, b}\right|\\
&\hspace{1cm}\leq  \int_0^t \left| b_{n_k}(X_s^{H_{n_k},b_{n_k}}) - b( X_s^{H,b}) \right| ds + \sigma \left| W_t^{H_{n_k}} -W_t^H\right|\\
&\hspace{1cm}\leq \int_0^t \left|  b_{n_k}(X_s^{H_{n_k},b_{n_k}})  -b_{n_k}(X_s^{H,b}) \right|+\left| b_{n_k}( X_s^{H,b})- b( X_s^{H,b}) \right| ds \\
&\hspace{3cm} +  \sigma \left| W_t^{H_{n_k}} -W_t^H\right|\\
&\hspace{1cm}\leq  L \int_0^t  \left| X_s^{H_{n_k},b_{n_k}} - X_s^{H,b}\right| ds +t\|b_{n_k} - b\|_K + \sigma \left| W_t^{H_{n_k}} -W_t^H\right|.
\end{align*}
By Gronwall's lemma (cf. $(2.11)$ in \cite{Karatzas/Shreve}, Chapter $5$ with corresponding proof on p. 387--388) we get
\begin{align*}
\left| X_t^{H_{n_k},b_{n_k}} - X_t^{H, b}\right|&\leq \left( t\|b_{n_k} -b\|_K + \sigma \left| W_t^{H_{n_k}} -W_t^H \right|\right) \\
&\hspace{0.5cm} + L\int_0^t \left( s\|b_{n_k} -b\|_K + \sigma \left| W_s^{H_{n_k}} -W_s^H \right|\right) e^{L(t-s)} ds
\end{align*} 
and in consequence 
\begin{align}\label{eqp: G6}
\begin{split}
&\left\| X_t^{H_{n_k},b_{n_k}} - X_t^{H, b}\right\|_{[0,T]} \\
&\hspace{1cm} \leq \left( T\|b_{n_k} -b\|_K + \sigma \left\| W^{H_{n_k}} -W^H\right\|_{[0,T]}\right) e^{LT}. 
\end{split}
\end{align} 
Now, we distinguish cases. By Lemma~\ref{lemma_an_bn} we either have 
\begin{itemize}
\item[(i)] $H_{n_k}=H$ for all $k\geq k_0$ for some $k_0\in\N$, or
\item[(ii)] $H=\frac12$.
\end{itemize}
First, suppose (i) holds true. Then for $k\geq k_0$, \eqref{eqp: G6} simplifies to
\[ \left\| X_t^{H_{n_k},b_{n_k}} - X_t^{H, b}\right\|_{[0,T]}\leq  T\|b_{n_k} -b\|_K\cdot e^{LT}, \]
which converges to zero as $k\to\infty$. Hence, the right-hand side of \eqref{eqp: G2} vanishes in the limit $k\to\infty$ by dominated convergence and we have 
\[ \Pr^{X^{H_{n_k},b_{n_k}}}\Rightarrow \Pr^{X^{H,b}}. \]
Assume (ii) holds true. Then $\|W^{H_{n_k}} - W^H\|_{[0,T]}$ converges to zero in probability by Proposition~\ref{lemma_X^H-X}. Therefore, we find a subsequence $(H_{n_{k_l}})_{l\in\N}$ with
 \[ \|W^{H_{n_{k_l}}} - W^H\|_{[0,T]} \stackrel{l\to\infty}{\longrightarrow} 0\quad a.s. \] 
(cf. Lemma $5.2$ in \cite{Kallenberg}). Then we repeat the above arguments with $n_k$ replaced by $n_{k_l}$, in particular \eqref{eqp: G6} tends to zero for $l\to\infty$ almost surely and consequently the right-hand side in \eqref{eqp: G2} converges to zero, which implies\[ \Pr^{X^{H_{n_{k_l}},b_{n_{k_l}}}}\Rightarrow \Pr^{X^{H,b}}. \]

In either way, we have shown that any sequence $\{ \Pr^{X^{H_n,b_n}}:n\in\N \}$ contains a weakly convergent subsequence and hence the family 
\[ \left\{\Pr^{X^{H,b}}: b\in\Sigma(C,A,\gamma,\sigma)\cap\mathcal{H}(1,L), H\in \{ h_n: n\in\N\} \right\}\]
is relatively compact and by Prohorov's theorem (\cite{Billingsley}, Theorem $6.2$) it is tight.
\end{proof}

\begin{proof}[Proof of Theorem \ref{prop_stability}]
Let $D_{c,T} := \{ f\in\mathcal{C}([0,T])\mid -c,c\in \textrm{im}(f)\}$ be given as in \eqref{eqp: G21}. Unless stated otherwise, we denote $X^{H,b}=(X_t^{H,b})_{t\leq T}$ and $X^b=(X_t^b)_{t\leq T}$. First, we note that for any $\epsilon'>0$,\pagebreak
\begin{align*}
\Pr\left( X^{H,b}, X^b\in D_{A,T}\right) &\geq \Pr\left( X^b\in D_{A+\epsilon',T}\textrm{ and }\|X^{H,b}-X^b\|_{[0,T]} <\frac{\epsilon'}{2}\right) \\
&\geq \Pr\left( X^b\in D_{A+\epsilon', T}\right) - \Pr\left( \|X^{H,b}-X^b\|_{[0,T]} > \frac{\epsilon'}{2} \right)
\end{align*}
and in consequence
\begin{align}\label{eqp: G4}
\begin{split}
&\inf_b  \Pr\left( X^{H,b}, X^b \in D_{A,T}\right) \\
&\hspace{1cm} \geq \inf_b\Pr\left( X^b\in D_{A+\epsilon', T}\right)  -\sup_b \Pr\left( \|X^{H,b}-X^b\|_{[0,T]} > \frac{\epsilon'}{2} \right),
\end{split}
\end{align}
where $\inf_b$ and $\sup_b$ are taken over $\Sigma(C,A,\gamma,\sigma)\cap \mathcal{H}(1,L)$. By Lemma~\ref{lemma_D_A} we find $T_0$ and $\epsilon'>0$ such that the first term is bounded away from zero for all $T\geq T_0$ and consequently by Proposition~\ref{lemma_X^H-X} 
\begin{align}\label{eqp: G3}
\liminf_{H\to\frac12} \inf_b\Pr\left( X^{H,b}, X^b \in D_{A,T}\right) >0 
\end{align} 
for $T\geq T_0$. As 
\begin{align*}
&\Pr\left( \left. \left| \tilde T_T^\eta(X^{H,b}) - \tilde T_T^\eta(X^{b}) \right| >\epsilon \ \right| X^{H,b},X^{b}\in D_{A,T} \right) \\
&\hspace{1cm} = \frac{\Pr\left( \left| \tilde T_T^\eta(X^{H,b}) - \tilde T_T^\eta(X^{b}) \right| >\epsilon \textrm{ and } X^{H,b},X^{b}\in D_{A,T} \right)}{ \Pr\left(X^{H,b},X^b\in D_{A,T}\right)},
\end{align*} 
the first claim of the proposition follows from \eqref{eqp: G3} once we have shown that \vspace{-1ex}
\[  \lim_{H\to\frac12} \sup_b \Pr\left( \left| \tilde T_T^\eta(X^{H,b}) - \tilde T_T^\eta(X^{b}) \right| >\epsilon \textrm{ and } X^{H,b},X^{b}\in D_{A,T} \right) = 0,\]
which is equivalent to showing that for any sequence $(H_n)_{n\in\N}$ with $H_n\rightarrow\frac12$ \vspace{-3ex}
\begin{align}\label{eqp: G7}
\lim_{n\to\infty} \sup_b \Pr\left( \left| \tilde T_T^\eta(X^{H_n,b}) - \tilde T_T^\eta(X^{b}) \right| >\epsilon \textrm{ and } X^{H_n,b},X^{b}\in D_{A,T} \right) = 0.
\end{align}  
To see this, let $(H_n)_{n\in\N}$ be such a sequence and $\epsilon''>0$. By Lemma~\ref{lemma_X^H_tight} there exist compact sets $K_1=K_1((H_n)_{n\in\N})$ and $K_2$ such that 
\begin{align}\label{eqp: compacts}
\sup_{n\in\N}\sup_b\Pr\left( X^{H_n,b} \in K_1^c\right) \leq \epsilon'' \ \textrm{ and }\ \sup_b \Pr\left( X^{b} \in K_2^c\right) \leq \epsilon''.
\end{align} 
With the compact set $K=K_1\cup K_2$, \vspace{-1ex}
\begin{align*}
&\sup_b \Pr\left( \left| \tilde T_T^\eta(X^{H_n,b}) - \tilde T_T^\eta(X^{b}) \right| >\epsilon \textrm{ and } X^{H_n,b},X^{b}\in D_{A,T} \right) \\
&\hspace{0.8cm} \leq \sup_b \Pr\left( \left| \tilde T_T^\eta(X^{H_n,b}) - \tilde T_T^\eta(X^{b}) \right| >\epsilon \textrm{ and } X^{H_n,b},X^{b}\in D_{A,T}\cap K \right) \\
&\hspace{2cm} + \sup_b  \Pr\left( \left\{ X^{H_n,b}, X^b \in K\right\}^c\right).
\end{align*}
By Lemma~\ref{lemma_D_A_closed}, the set $D_{A,T}\cap K$ is compact as a closed subset of a compact set. By Theorem~\ref{thm_continuity}, the map $\tilde T_T^\eta$ is continuous on $D_{A,T}$ and hence, it is uniformly continuous on the compact set $D_{A,T}\cap K$. Therefore, Lemma~\ref{lemma_CMT_uniform} is applicable on this set and the limit $\lim_{n\to\infty}$ of the first term is zero by Proposition~\ref{lemma_X^H-X}. By the union bound and \eqref{eqp: compacts}, the second term is bounded by $2\epsilon''$ and \eqref{eqp: G7} follows as $\epsilon''>0$ was arbitrary.\\

For the second claim of the theorem, we start with inequality \eqref{eqp: G4}. As \vspace{-0.05cm}
\[  \liminf_{T\to\infty} \inf_b \Pr\left( X^b\in D_{A+\epsilon', T}\right) =1\]
by Lemma~\ref{lemma_D_A} for a proper choice of $\epsilon'>0$, and \vspace{-0.05cm}
\[ \lim_{H\to\frac12} \sup_b \Pr\left( \|X^{H,b}-X^b\|_{[0,T]} > \frac{\epsilon'}{2} \right) =0 \]
by Proposition~\ref{lemma_X^H-X}, we directly get from \eqref{eqp: G4} that
\[ \liminf_{T\to\infty}\liminf_{H\to\frac12} \inf_b \Pr\left( X^{H,b}, X^b\in D_{A,T}\right) =1.\]
For the other order of $\liminf$ we note that 
\[ \left\{ (X_t^{H,b})_{t\leq T_0},(X_t^{b})_{t\leq T_0} \in D_{A,T}\right\} \subset  \left\{ (X_t^{H,b})_{t\leq T},(X_t^{b})_{t\leq T} \in D_{A,T}\right\} \]
for $T_0\leq T$ and hence for any $T_0>0$
\[ \liminf_{T\to\infty} \inf_{b} \Pr \left( X^{H,b}, X^b\in D_{A,T}\right) \geq \inf_{b}\Pr \left( (X_t^{H,b})_{t\leq T_0}, (X_t^b)_{t\leq T_0}\in D_{A,T_0}\right).  \]
Let $\epsilon>0$ be arbitrary. Choose $T_0$ large enough and $\delta>0$ such that $\inf_{b}\Pr\left( (X_t^b)_{t\leq T_0} \in D_{A+\delta,T_0}\right) \geq 1-\epsilon$, which is possible by Lemma~\ref{lemma_D_A}. Then again by \eqref{eqp: G4} we find
\begin{align*}
&\liminf_{H\to\frac12} \liminf_{T\to\infty}\inf_{b\in \Sigma(C,A,\gamma,\sigma)\cap \mathcal{H}(1,L)} \Pr \left( X^{H,b}, X^b\in D_{A,T}\right)\\
&\hspace{1cm} \geq \liminf_{H\to\frac12} \inf_{b\in \Sigma(C,A,\gamma,\sigma)\cap \mathcal{H}(1,L)} \Pr \left( (X_t^{H,b})_{t\leq T_0}, (X_t^b)_{t\leq T_0}\in D_{A,T_0}\right) \\
&\hspace{1cm} \geq \inf_{b\in \Sigma(C,A,\gamma,\sigma)\cap \mathcal{H}(1,L)}\Pr\left( (X_t^b)_{t\leq T_0} \in D_{A+\delta,T_0}\right) \\
&\hspace{1.5cm} - \limsup_{H\to\frac12} \sup_{b\in \Sigma(C,A,\gamma,\sigma)\cap \mathcal{H}(1,L)} \Pr\left(\|X^{H,b} - X^b\|_{[0,T_0]} > \delta/2\right) \\
&\hspace{1cm} \geq 1-\epsilon, 
\end{align*}
where we applied Proposition~\ref{lemma_X^H-X} in the last step. The assertion follows by taking $\epsilon\searrow 0$.
\end{proof}

\begin{proof}[Proof of Remark \ref{remark_stability}]
As a first step, for any $\epsilon>0$,
\begin{align}\label{eqp: G5}
\begin{split}
\Pr\left( \tilde T_T^\eta(X^{H,b}) >\kappa\right) &\leq \Pr\left(  \tilde T_T^\eta(X^{H,b}) >\kappa, \left| \tilde T_T^\eta(X^{H,b}) - \tilde T_T^\eta(X^b)\right|\leq \epsilon \right)\\
&\hspace{1cm} +\Pr\left( \left| \tilde T_T^\eta(X^{H,b}) - \tilde T_T^\eta(X^b)\right|> \epsilon\right)\\
&\leq \Pr\left(  \tilde T_T^\eta(X^b) >\kappa-\epsilon\right)+\Pr\left( \left| \tilde T_T^\eta(X^{H,b}) - \tilde T_T^\eta(X^b)\right|> \epsilon\right).
\end{split}
\end{align}
To estimate the second summand, we denote $B_{H,b} := \{ X^{H,b}, X^b\in D_{A,T}\}$ and apply the law of total probability and get
\begin{align*}
&\Pr\left( \left| \tilde T_T^\eta(X^{H,b}) - \tilde T_T^\eta(X^b)\right|> \epsilon\right)\\
&\hspace{1cm} = \Pr\left( \left.\left| \tilde T_T^\eta(X^{H,b}) - \tilde T_T^\eta(X^b)\right|> \epsilon\ \right| B_{H,b}\right)\cdot\Pr\left( B_{H,b}\right)\\
&\hspace{2.5cm} + \Pr\left( \left.\left| \tilde T_T^\eta(X^{H,b}) - \tilde T_T^\eta(X^b)\right|> \epsilon\ \right| B_{H,b}^c\right)\cdot\Pr\left( B_{H,b}^c\right)\\
&\hspace{1cm} \leq \Pr\left( \left.\left| \tilde T_T^\eta(X^{H,b}) - \tilde T_T^\eta(X^b)\right|>\epsilon\ \right| B_{H,b}\right)+\Pr\left( B_{H,b}^c\right).
\end{align*}
By Theorem~\ref{prop_stability},
\[ \limsup_{T\to\infty}\limsup_{H\to\frac12} \sup_{b\in H_0(b_0,\eta)\cap\mathcal{H}(1,L)}\Pr\left( \left| \tilde T_T^\eta(X^{H,b}) - \tilde T_T^\eta(X^b)\right|> \epsilon\right) =0. \]
In conclusion, \eqref{eqp: G5} yields for any $\epsilon>0$ with $\kappa-\epsilon>\kappa_{\eta,\alpha}$
\begin{align*}
& \limsup_{T\to\infty}\limsup_{H\to\frac12} \sup_{b\in H_0(b_0,\eta)\cap\mathcal{H}(1,L)}\Pr\left( \tilde T_T^\eta(X^{H,b}) >\kappa\right)\\
&\hspace{2cm} \leq \limsup_{T\to\infty} \sup_{b\in H_0(b_0,\eta)\cap\mathcal{H}(1,L)}\Pr\left( \tilde T_T^\eta(X^{b}) >\kappa-\epsilon\right)\\
&\hspace{2cm} \leq \limsup_{T\to\infty} \sup_{b\in H_0(b_0,\eta)\cap\mathcal{H}(1,L)}\Pr\left( \tilde T_T^\eta(X^{b}) >\kappa_{\eta,\alpha}\right)\leq \alpha.
\end{align*}
The last step used that $T_T^\eta$ is uniformly asymptotically of level $\alpha$, see \eqref{eq: asym_level_alpha}, in combination with \eqref{eqp: C4} that implies that the restriction to $\mathcal{T}_T$ in the definition of $\tilde T_T^\eta$ compared to all of $\mathcal{T}$ is negligible in the limit.\pagebreak
\end{proof}

\addtocontents{toc}{\protect\enlargethispage{\baselineskip}}
\subsection{Proof of Theorem \ref{Continuity_lower_H}}\label{SubSec_G5}

First, we present the following result that is used later in our proof.

\begin{lemma}[\cite{Saussereau}, Lemma $8$] \label{lemma_moment_Hoelder_norm}
Let $T>0$ and $\frac12<\beta<H<1$. Then we have the following moment estimate for any $k\geq 1$:
\[ \E\left[ \left(\sup_{s\neq r\in [0,T]}\frac{|W_s^H- W_r^H|}{|s-r|^\beta}\right)^{2k}\right] \leq 32^k (2T)^{2k(H-\beta)} \frac{(2k)!}{k!}.\]
\end{lemma}

Next, we establish $L^2(\Pr)$-convergence of the exponent of the density given by the fractional Girsanov theorem \ref{Girsanov_fBM} as $H\to\frac12$.

\begin{lemma}\label{lemma_continuity_likelihood}
Let $b\in\Sigma(C,A,\gamma,\sigma)\cap\mathcal{H}(\beta,L)$, $L>0$, and $g:\R\rightarrow\R$ be a function with $\|g\|_\infty <\infty$ that is Lipschitz continuous with constant $\|g\|_L$. Define
\[ v_H^g(s) := K_H^{-1}\left( \int_0^\cdot g(X_u^{H,b})du\right)(s).\]
Then we have
\[ \int_0^T v_H^g(s) dW_s - \frac12\int_0^T v_H^g(s)^2 ds \stackrel{H\to\frac12}{\longrightarrow}_{L^2(\Pr)} \int_0^T g(X_s^b) dW_s - \frac12\int_0^T g(X_s^b)^2 ds.\]
\end{lemma}
\begin{proof}
As the representation of $v_H^g$ in terms of fractional integrals and derivatives is different for $H<\frac12$ and $H>\frac12$ we distinguish between those two cases. \\

\textit{Case 1: $H<\frac12$}. First of all, we show that $v_H^g(s)$ is bounded on $[0,T]$ uniformly in $H<\frac12$. This follows by
\begin{align}\label{eqp: G20}
\begin{split}
\left| v_H^g(s)\right| &= \left| s^{H-\frac12} I_{0+}^{\frac12 -H}\left( (\cdot)^{\frac12-H} g(X_\cdot^{H,b})\right)(s)\right|\\
&= s^{H-\frac12} \frac{1}{\Gamma(\frac12-H)} \left| \int_0^s \frac{u^{\frac12-H}g(X_u^{H,b})}{(s-u)^{\frac12 +H}} du \right|\\
&\leq s^{H-\frac12} \frac{\|g\|_\infty}{\Gamma(\frac12-H)} \int_0^s \frac{u^{\frac12-H}}{(s-u)^{\frac12+H}} du \\
&= s^{H-\frac12} \frac{\|g\|_\infty}{\Gamma(\frac12-H)}  \frac{\Gamma(\frac32-H)}{\Gamma(2-2H)} s^{1-2H}\\
&=\frac{\|g\|_\infty}{\Gamma(\frac12-H)}  \frac{\left(\frac12 -H\right)\Gamma(\frac12-H)}{\Gamma(2-2H)} s^{\frac12-H}\\
&\leq \|g\|_\infty \sqrt{T},
\end{split}
\end{align}
where we evaluated the integral in the forth step according to \eqref{eq: I_x^beta}. Furthermore, we used in the fifth stept that $x\Gamma(x) = \Gamma(x+1)$ for $x>0$ and in the last one that $s\in [0,T]$ and $(\frac12-H)(\Gamma(2-2H)^{-1}\leq 1$. This last property follows from $\Gamma(x)\geq\frac12$ for $x\in [1,2]$ which can be seen in the following way: For all those $x$ we have
\begin{align*}
\Gamma(x) = \int_0^\infty t^{x-1} e^{-t} dt &= \int_0^1 t^{x-1} e^{-t} dt + \int_1^\infty t^{x-1} e^{-t} dt\\
&\geq \int_0^1 t e^{-t} dt + \int_1^\infty e^{-t} dt \\
&= \left[ -te^{-t}\right]_0^1 + \left[ -e^{-t}\right]_0^1 + \left[ -e^{-t}\right]_1^\infty \\
&= 1-e^{-1}
\end{align*}
and now we simply use that $e>2$.\\
To proceed with the proof, we will show $L^2(\Pr)$-convergence of the stochastic integral and the Lebesgue integral separately. For the stochastic integrals we have by It\^{o}'s isometry
\begin{align*}
\E\left[ \left( \int_0^T v_H^g(s) dW_s - \int_0^T g(X_s^b) dW_s\right)^2\right] &= \E\left[ \left( \int_0^T v_H^g(s) - g(X_s^b) dW_s\right)^2\right]\\
&= \E\left[ \int_0^T \left( v_H^g(s) -g(X_s^b) \right)^2 ds\right].
\end{align*}
By the uniform in $H$ boundedness of $v_H^g$ and $g$ we get uniform integrability of the inner integral process (indexed in $H$) and are done if we can show that it converges to zero in probability for $H\nearrow\frac12$ (cf. \cite{Kallenberg}, Theorem $5.12$). This holds true, if $\| v_H^g -g(X_\cdot^b)\|_{L^2([0,T])}\rightarrow 0$ in probability.\\
Using again uniform integrability, we may show $L^2(\Pr)$-convergence of the Lebesgue integral if we establish that
\begin{align}\label{eqp: G15}
\|v_H^g\|_{L^2([0,T])}^2 = \int_0^T v_H^g(s)^2 ds\ \stackrel{H\nearrow\frac12}{\longrightarrow}\ \int_0^T g(X_s^b)^2 ds = \|g(X_\cdot^b)\|_{L^2([0,T])}^2
\end{align} 
in probability. By the reverse triangle inequality we have
\begin{align}\label{eqp: G16}
\left| \|v_H^g\|_{L^2([0,T])} - \| g(X_\cdot^b)\|_{L^2([0,T])} \right| \leq \|v_H^g - g(X_\cdot^b)\|_{L^2([0,T])}, 
\end{align} 
so convergence of $\|v_H^g - g(X_\cdot^b)\|_{L^2([0,T])}$ to zero in probability suffices to show $\|v_H^g\|_{L^2([0,T])} \rightarrow_\Pr \|g(X_\cdot^b)\|_{L^2([0,T])}$. By taking squares and applying the continuous mapping theorem we conclude that convergence in probability holds in \eqref{eqp: G15}.
Thus, we are finished for both the It\^{o} and the Lebesgue integral, if we show $\|v_H^g -g(X_\cdot^b)\|_{L^2([0,T])}\rightarrow_\Pr 0$. To this aim, we split in the following way:
\begin{align*}
v_H^g(s) - g(X_s^b) &= s^{H-\frac12} I_{0+}^{\frac12-H}\left( (\cdot)^{\frac12-H} g(X_\cdot^{H,b})\right) (s) - g(X_s^b) \\
&= \left( I_{0+}^{\frac12-H}\left( (\cdot)^{\frac12-H} g(X_\cdot^{H,b})\right)(s) - I_{0+}^{\frac12-H}\left( g(X_\cdot^b)\right)(s)\right)\\
&\hspace{1cm} + \left( I_{0+}^{\frac12-H}\left( g(X_\cdot^b)\right)(s) -g(X_s^b)\right) \\
&\hspace{1cm} + \left( s^{H-\frac12} -1\right) I_{0+}^{\frac12-H}\left( (\cdot)^{\frac12-H}g(X_\cdot^{H,b})\right)(s).
\end{align*}
By Minkowski's inequality for the $L^2([0,T])$-norm, it is enough to show  convergence in probability to zero for the $L^2([0,T])$-norm of each summand separately.
\begin{enumerate}
\item[(1)] For the first summand, we use the $L^2([0,T])$-estimate of equation~$(2.72)$ in~\cite{Samko} to get
\begin{align*}
&\left\| I_{0+}^{\frac12-H}\left( (\cdot)^{\frac12-H} g(X_\cdot^{H,b})\right) - I_{0+}^{\frac12-H}\left( g(X_\cdot^b)\right)\right\|_{L^2([0,T])}\\
&\hspace{1cm}\leq \frac{T^{\frac12 -H}}{(\frac12-H)\Gamma(\frac12-H)} \left\| (\cdot)^{\frac12-H} g(X_\cdot^{H,b}) - g(X_\cdot^b)\right\|_{L^2([0,T])}.
\end{align*}
As $x\Gamma(x)=\Gamma(x+1)$ for $x>0$, we have for $H<\frac12$,
\[  \frac{T^{\frac12 -H}}{(\frac12-H)\Gamma(\frac12-H)} = \frac{T^{\frac12-H}}{\Gamma(\frac32 -H)} \stackrel{H\nearrow\frac12}{\longrightarrow}\ 1.\]
Next, we evaluate the squared $L^2([0,T])$-norm:
\begin{align*}
&\int_0^T  \left( u^{\frac12-H} g(X_u^{H,b}) - g(X_u^b)\right)^2 du \\
&\hspace{0.1cm} = \int_0^T \left( u^{\frac12-H} (g(X_u^{H,b}) - g(X_u^b)) - g(X_u^b)\left( 1- u^{\frac12-H} \right)\right)^2 du \\
&\hspace{0.1cm}\leq 2\int_0^T u^{1-2H}(g(X_u^{H,b}) - g(X_u^b))^2 du + 2\int_0^T g(X_u^b)^2 \left( u^{\frac12-H} -1\right)^2 du.
\end{align*}
The second summand is bounded by \pagebreak
\begin{align*}
2\|g\|_\infty^2 \int_0^T \left( u^{\frac12-H} -1\right)^2 du &= 2\|g\|_\infty^2 \int_0^T \left( u^{1-2H} - 2u^{\frac12-H} +1\right) du\\
&\hspace{-1.5cm}= 2\|g\|_\infty^2 \left( \frac{1}{2-2H} T^{2-2H} - \frac{2}{\frac32-H} T^{\frac32-H} + T\right),
\end{align*} 
which converges to zero for $H\nearrow\frac12$. The first summand, on the other hand, can be bounded by
\begin{align*}
&2 \left\|g(X_\cdot^{H,b}) - g(X_\cdot^b)\right\|_{[0,T]} \int_0^T u^{1-2H} du \\
&\hspace{3cm}\leq 2\|g\|_L  \| X_\cdot^{H,b} - X_\cdot^b\|_{[0,T]}\cdot\frac{1}{2-2H}T^{2-2H}.
\end{align*} 
This term converges to zero in probability for $H\nearrow\frac12$ because the term $\|X_\cdot^{H,b}-X_\cdot^b\|_{[0,T]}$ does by Proposition~\ref{lemma_X^H-X} and the other factors are bounded.

\item[(2)] By Theorem $2.6$ in \cite{Samko}, we simply have for all $\omega\in\Omega$, 
\[\left\| I_{0+}^{\frac12-H}\left( g(X_\cdot^b)\right) -g(X_\cdot^b)\right\|_{L^2([0,T])}\ \stackrel{H\nearrow\frac12}{\longrightarrow}\ 0.\] 

\item[(3)] As in (2), we argue $\omega$-wise. We repeat the argument used for bounding $v_H^b$ in \eqref{eqp: G20} without the factor $s^{H-\frac12}$ and immediately see that
\[ \left| I_{0+}^{\frac12-H}\left( (\cdot)^{\frac12-H} g(X_\cdot^{H,b})\right)(s)\right| \leq \|g\|_\infty  T.\]
Hence,
\begin{align*}
&\left\| \left( (\cdot)^{H-\frac12} -1\right) I_{0+}^{\frac12-H}\left( (\cdot)^{\frac12-H} g(X_\cdot^{H,b})\right) \right\|_{L^2([0,T])}\\
&\hspace{1cm}\leq \|g\|_\infty T \left( \int_0^T  \left( s^{H-\frac12} -1\right)^2 ds\right)^\frac12\\
&\hspace{1cm} = \|g\|_\infty T\left( \frac{T^{2H}}{2H} - 2\frac{2T^{\frac12+H}}{2H+1} + T\right)^\frac12\ \stackrel{H\nearrow\frac12}{\longrightarrow}\ 0.
\end{align*} 
\end{enumerate}
This finishes the proof in the case $H<\frac12$.\\

\textit{Case 2: $H>\frac12$}. The idea is again to show $L^2(\Pr)$-convergence of the It\^{o} integral and the Lebesgue integral separately. For the stochastic integral we have by It\^{o}'s isometry
\begin{align*}
\E\left[ \left( \int_0^T v_H^g(s) dW_s - \int_0^T g(X_s^b) dW_s\right)^2\right]  = \E\left[ \int_0^T \left( v_H^g(s) -g(X_s^b) \right)^2 ds\right],
\end{align*}
but this time we use \eqref{eq: Weil_rep} to represent $v_H^g$ as
\begin{align*}
v_H^g(s) &= s^{H-\frac12} D_{0+}^{H-\frac12}\left( (\cdot)^{\frac12-H} g(X_\cdot^{H,b})\right)(s) \\
&= \frac{s^{H-\frac12}}{\Gamma\left(\frac32-H\right)}\left( \frac{s^{\frac12-H} g(X_s^{H,b})}{s^{H-\frac12}} \right. \\
&\hspace{2.5cm}\left. + \left( H-\frac12\right)\int_0^s \frac{s^{\frac12-H} g(X_s^{H,b}) - r^{\frac12-H} g(X_r^{H,b})}{(s-r)^{\frac12+H}} dr \right).
\end{align*}
Using repeatedly $(a+b)^2\leq 2a^2 +2b^2$, we see that
\begin{align}\label{eqp: G12}
\begin{split}
\E\left[ \int_0^T \left( v_H^g(s) - g(X_s^b)\right)^2 ds\right]  &\leq 2A_H+2B_H \\
&\leq 2A_H + 4B_H^1 + 4B_H^2,
\end{split}
\end{align} 
where
\begin{align*}
A_H &= \E\left[ \int_0^T \left( \frac{s^{\frac12-H} g(X_s^{H,b})}{\Gamma\left(\frac32-H\right)} - g(X_s^b)   \right)^2 ds \right], \\
B_H &= \E\left[ \int_0^T \left( \frac{\left(H-\frac12\right)s^{H-\frac12}}{\Gamma\left(\frac32-H\right)} \int_0^s \frac{s^{\frac12-H} g(X_s^{H,b}) - r^{\frac12-H} g(X_r^{H,b})}{(s-r)^{\frac12+H}} dr\right)^2 ds\right],
\end{align*}
and we subsequently splitted $B_H$ in a similar way as in the proof of Theorem~$3$ in \cite{Nualart} into
\begin{align*}
B_H^1 &= \E\left[\int_0^T \left( \frac{\left(H-\frac12\right)s^{H-\frac12}}{\Gamma\left(\frac32-H\right)} g(X_s^{H,b}) \int_0^s \frac{s^{\frac12-H} - r^{\frac12-H}}{(s-r)^{\frac12+H}} dr \right)^2  ds\right],\\
B_H^2 &= \E\left[\int_0^T  \left( \frac{\left(H-\frac12\right)s^{H-\frac12}}{\Gamma\left(\frac32-H\right)}  \int_0^s \frac{g(X_s^{H,b}) - g(X_r^{H,b})}{(s-r)^{\frac12+H}} r^{\frac12-H} dr \right)^2  ds\right].
\end{align*}
In the following we will show for each of $A_H, B_H^1$ and $B_H^{2}$ that it vanishes in the limit $H\searrow\frac12$. \\

\textit{$A_H$:} For this term, we start by
\begin{align*}
&A_H = \E\left[\int_0^T \left( \left(\frac{s^{\frac12-H}}{\Gamma\left(\frac32-H\right)}-1\right) g(X_s^{H,b}) + \left( g(X_s^{H,b}) - g(X_s^b)\right) \right)^2 ds\right]\\
&\hspace{0.2cm}\leq 2\E\left[ \int_0^T \left(\frac{s^{\frac12-H}}{\Gamma\left(\frac32-H\right)}-1\right)^2 g(X_s^{H,b})^2 ds + \int_0^T \left( g(X_s^{H,b}) - g(X_s^b) \right)^2 ds\right] \\
&\hspace{0.2cm}\leq 2 \|g\|_\infty^2 \int_0^T \left(\frac{s^{\frac12-H}}{\Gamma\left(\frac32-H\right)}-1\right)^2 ds + T\E\left[ \|g(X_\cdot^{H,b}) - g(X_\cdot^b)\|_{[0,T]}^2 \right].
\end{align*}
For the first summand, we evaluate the integral and get
\begin{align*}
&\int_0^T \left(\frac{s^{\frac12-H}}{\Gamma\left(\frac32-H\right)}-1\right)^2 ds \\
&\hspace{1cm}= \int_0^T \frac{s^{1-2H}}{\Gamma\left(\frac32-H\right)^2} - \frac{2s^{\frac12-H}}{\Gamma\left(\frac32-H\right)} +1 \ ds\\
& \hspace{1cm} =\frac{T^{2-2H}}{(2-2H)\Gamma\left(\frac32-H\right)^2} - \frac{2 T^{\frac32-H}}{\left(\frac32-H\right)\Gamma\left(\frac32-H\right)} +T\stackrel{H\searrow\frac12}{\longrightarrow} 0.
\end{align*}
For the second summand, we note that the family of random variables 
\[ \left\{\|g(X_\cdot^{H,b}) - g(X_\cdot^b)\|_{[0,T]}^2, \frac12\leq H<1 \right\} \]
is uniformly integrable as all of them are bounded by $4T\|g\|_\infty^2$. Thus, it suffices to establish convergence in probability to zero of this term instead (cf. \cite{Kallenberg}, Theorem $5.12$). But this follows easily by the Lipschitz condition on $g$ and Proposition~\ref{lemma_X^H-X} as
\[ \|g(X_\cdot^{H,b}) - g(X_\cdot^b)\|_{[0,T]}^2 \leq \|g\|_L^2 \| X_\cdot^{H,b} - X_\cdot^b\|_{[0,T]}^2. \] 
In conclusion, we have shown so far that $A_H\rightarrow 0$ for $H\searrow\frac12$. \\

\textit{$B_H^1$:} Bounding $B_H^1$ essentially relies on finding the value of the inner integral. For this, we have
\begin{align}\label{eqp: G22}
\int_0^s \frac{s^{\frac12-H} - r^{\frac12-H}}{(s-r)^{\frac12+H}} dr  = \left(\frac{\Gamma\left(\frac32-H\right)^2}{\Gamma(2-2H)} -1\right)\left(H-\frac12\right)^{-1} s^{1-2H}
\end{align} 
almost everywhere, which can be seen by evaluating the fractional derivative $D_{0+}^{H-1/2}\left((\cdot)^{1/2-H}\right)$ in two ways. First, by
direct evaluation via \eqref{eq: D_x^mu},
\[ D_{0+}^{H-\frac12}\left((\cdot)^{\frac12-H}\right)(s) = \frac{\Gamma\left( \frac32-H\right)}{\Gamma(2-2H)} s^{1-2H}.\]
On the other hand, $s^{\frac12-H}\in I_{0+}^{H-1/2}(L^1([0,T])$ as $\int_0^T u^{1-2H}du <\infty$ and
\[ I_{0+}^{H-\frac12}\left( \frac{\Gamma(\frac32-H)}{\Gamma(2-2H)} (\cdot)^{1-2H}\right)(s) = s^{\frac12-H} \]
by \eqref{eq: I_x^beta}.
Then, by the Weyl representation \eqref{eq: Weil_rep}, $D_{0+}^{H-1/2}\left((\cdot)^{1/2-H}\right)$ is given by
\[ \frac{1}{\Gamma\left(\frac32-H\right)} \left( s^{1-2H} + \left(H-\frac12\right) \int_0^s \frac{s^{\frac12-H} - r^{\frac12-H}}{(s-r)^{\frac12+H}} dr\right) \quad a.e.\]
Combining these two results gives \eqref{eqp: G22} after rearranging. With this identity holding almost everywhere,
\begin{align*}
B_H^1 &\leq \frac{\left(H-\frac12\right)^2}{\Gamma\left(\frac32-H\right)^2} \|g\|_\infty^2 \left(\frac{\Gamma(\frac32-H)^2}{\Gamma(2-2H)}-1\right)^2 \left(H-\frac12\right)^{-2} \int_0^T s^{2H-1} s^{2-4H} ds\\
& = \frac{1}{(2-2H)\Gamma\left(\frac32-H\right)^2}\left(\frac{\Gamma(\frac32-H)^2}{\Gamma(2-2H)}-1\right)^2 \|g\|_\infty^2 T^{2-2H} 
\end{align*}
which tends to zero for $H\searrow\frac12$.  \\

\textit{$B_H^{2}$:} Here, it is crucial to find a bound for the inner integral. For this, we first use the Lipschitz property of $g$ and get
\begin{align*}
\left|\int_0^s \frac{g(X_s^{H,b}) - g(X_r^{H,b})}{(s-r)^{\frac12+H}} r^{\frac12-H} dr\right| &\leq \int_0^s \frac{|g(X_s^{H,b}) - g(X_r^{H,b})|}{(s-r)^{\frac12+H}}  r^{\frac12-H} dr \\
&\leq \|g\|_L \int_0^s \frac{|X_s^{H,b} - X_r^{H,b}|}{(s-r)^{\frac12+H}} r^{\frac12-H} dr.
\end{align*}
By the definition of $X^{H,b}$ we furthermore have for $s\geq r$,
\[ X_s^{H,b} - X_r^{H,b} = \int_r^s b(X_u^{H,b}) du + \sigma (W_s^H - W_r^H). \]
By the at most linear growth condition on $b$,
\begin{align*}
\left| \int_r^s b(X_u^{H,b}) du\right| &\leq |s-r| \cdot \|b(X_\cdot^{H,b})\|_{[0,T]}\\
&\leq C|s-r| \left( 1+ \|X^{H,b}\|_{[0,T]}\right)\\
&\leq C|s-r| \left( 1+ CT + \|W^H\|_{[0,T]}\right)e^{CT},
\end{align*}
where the last step used Lemma~$1$ in \cite{Tudor}. To proceed, we choose a number $ 0<\epsilon_H<\min\{H-\frac12,\frac14\} $
and estimate (note that $W_0=0$),
\begin{align*}
\sup_{0\leq t\leq T} |W_t^H| &= \sup_{0<t\leq T}\frac{|W_t^H|}{t^{H-\epsilon_H}} t^{H-\epsilon_H}\\
&\leq T^{H-\epsilon_H} \sup_{0<t\leq T}\frac{|W_t^H|}{t^{H-\epsilon_H}} \\
&\leq \left(T\vee T^{\frac14}\right) \sup_{s\neq t\in [0,T]} \frac{|W_t^H - W_s^H|}{|t-s|^{H-\epsilon_H}}.
\end{align*}
Introducing the random variable
\begin{align}\label{eqp: G13}
G_H := \sup_{s\neq r\in [0,T]} \frac{|W_s^H-W_r^H|}{|s-r|^{H-\epsilon_H}}, 
\end{align}
we have for a constant $C'=C'(T)>0$ not depending on $H$,
\[ |X_s^{H,b} - X_r^{H,b}| \leq C' |s-r| \left( 1+ G_H\right) + \sigma |W_s^H - W_r^H|.\]
Using $(a+b)^2\leq 2a^2 + 2b^2$ once again, we find $B_H^2\leq 2B_H^{2,1} + 2B_H^{2,2}$ with
\begin{align*}
B_H^{2,1} &= \E\left[ \|g\|_L^2 \int_0^T  \left( \frac{\left(H-\frac12\right)s^{H-\frac12}}{\Gamma\left(\frac32-H\right)} C' (1+G_H) \int_0^s \frac{(s-r)r^{\frac12-H} }{(s-r)^{\frac12+H}}  dr \right)^2  ds \right] , \\
B_H^{2,2} &= \E\left[\|g\|_L^2  \int_0^T  \left( \frac{\left(H-\frac12\right)s^{H-\frac12}}{\Gamma\left(\frac32-H\right)}  \sigma \int_0^s \frac{|W_s^H - W_r^H|}{(s-r)^{\frac12+H}} r^{\frac12-H} dr \right)^2  ds \right].
\end{align*}
For $B_H^{2,1}$ we first evaluate the inner integral and get
\begin{align*}
\int_0^s \frac{(s-r)}{(s-r)^{\frac12+H}} r^{\frac12-H} dr &=  \int_0^s \frac{r^{\frac12-H}}{(s-r)^{-\frac12+H}} dr = \frac{\Gamma\left(\frac32-H\right)}{\Gamma\left(3-2H\right)} s^{2-2H}
\end{align*} 
using \eqref{eq: I_x^beta}. This gives the bound
\begin{align*}
B_H^{2,1} &\leq (C')^2\|g\|_L^2 \frac{\left(H-\frac12\right)^2 }{\Gamma(3-2H)^2} \E\left[ (1+G_H)^2\right]\int_0^T s^{3-2H} ds \\
&= (C')^2 \|g\|_L^2 \frac{\left(H-\frac12\right)^2}{\Gamma(3-2H)^2}  \frac{1}{4-2H} T^{4-2H} \E\left[ (1+G_H)^2\right]\\
&\leq (C')^2 \|g\|_L^2 \frac{\left(H-\frac12\right)^2}{\Gamma(3-2H)^2}  \frac{1}{4-2H} T^{4-2H} \left( 2+ 128 (2T)^{\epsilon_H} \right),
\end{align*}
where we used Lemma~\ref{lemma_moment_Hoelder_norm} in the last step. An application of this lemma is possible since $H>H-\epsilon_H>\frac12$. It follows that $B_H^{2,1}\rightarrow 0$ for $H\searrow\frac12$. \\
A similar procedure will be applied for $B_H^{2,2}$. For $0<\epsilon_H< \min\{H-\frac12,\frac14\}$, we estimate
\begin{align*}
\int_0^s \frac{\sigma |W_s^H - W_r^H|}{(s-r)^{\frac12+H}} r^{\frac12-H} dr &\leq \sigma G_H \int_0^s \frac{r^{\frac12-H}}{(s-r)^{\frac12 + \epsilon_H}} dr \\
&= \sigma G_H \frac{\Gamma\left(\frac32-H\right)}{\Gamma(2-H-\epsilon_H)} s^{1-H-\epsilon_H} 
\end{align*}
with \eqref{eq: I_x^beta} and the random variable $G_H$ given in \eqref{eqp: G13}.
This gives rise to the bound 
\begin{align*}
B_H^{2,2} \leq \|g\|_L^2 \frac{\sigma^2\left(H-\frac12\right)^2}{\Gamma(2-H-\epsilon_H)^2} \E\left[ G_H^2\right] \int_0^T s^{1-2\epsilon_H} ds.
\end{align*}
Again, we can apply Lemma \ref{lemma_moment_Hoelder_norm} to bound the expectation whereas it is simple to evaluate the integral. This together gives
\[ B_H^{2,2}\leq \|g\|_L^2 \frac{\sigma^2\left(H-\frac12\right)^2}{\Gamma(2-H-\epsilon_H)^2}\left( \frac{1}{2-2\epsilon_H} T^{2-2\epsilon_H} \right) \left( 64 (2T)^{\epsilon_H}\right) \]
and it now can be seen that $B_H^{2,2}\rightarrow 0$ for $H\searrow\frac12$. This establishes the convergence of the It\^{o} integrals of the statement as the right-hand side of \eqref{eqp: G12} converges to zero for $H\searrow \frac12$.\\

To show 
\[ \int_0^T v_H^g(s)^2 ds \longrightarrow_{L^2(\Pr)} \int_0^T g(X_s^b)^2 ds \]
as $H\searrow \frac12$, we will prove that the family
\begin{align}\label{eqp: G17}
\left( \left(\int_0^T v_H^g(s)^2 ds\right)^2\right)_{\frac12< H\leq \frac34} 
\end{align} 
is uniformly integrable. Then it suffices to establish (cf. \cite{Kallenberg}, Theorem $5.12$) 
\[ \int_0^T v_H^g(s)^2 ds \longrightarrow_\Pr \int_0^T g(X_s^b)^2 ds. \]
By the reverse triangle inequality as in \eqref{eqp: G16} and the reasoning thereafter, it is enough to show that $\|v_H^g - g(X_\cdot^b)\|_{L^2([0,T])}$ converges to zero in probability. For any $\epsilon>0$, 
\[ \Pr\left( \|v_H^g - g(X_\cdot^b)\|_{L^2([0,T])} >\epsilon\right) \leq \frac{1}{\epsilon^2}\E\left[ \int_0^T \left( v_H^g(s) - g(X_s^b)\right)^2 ds\right] \]
by an application of Markov's inequality and convergence of the right-hand side to zero for $H\searrow \frac12$ has been shown above, starting in \eqref{eqp: G12}. Consequently, it remains to prove uniform integrability of the family in \eqref{eqp: G17}. Without the summand $-g(X_s^b)$ inside the integral in the left-hand side of \eqref{eqp: G12}, 
\[ \E\left[ \int_0^T  v_H^g(s)^2 ds \right] \leq 2A_H' + 2B_H, \] 
with
\[ A_H' = \|g\|_\infty^2 \int_0^T \left(\frac{s^{\frac12-H}}{\Gamma\left(\frac32-H\right)}\right)^2 ds\]
and $B_H$ just beyond \eqref{eqp: G12}.
Proceeding as above with $G_H$ in \eqref{eqp: G13} and $0<\epsilon_H< \min\{H-\frac12,\frac14\}$ therein,
\begin{align*} \nonumber
\int_0^T v_H^g(s)^2 ds & \leq 2A_H' + 4B_H^1 + 8 B_H^{2,1} + 8B_H^{2,2} \\ \nonumber
&\leq 2\|g\|_\infty^2 \left( \frac{T^{2-2H}}{(2-2H)\Gamma\left(\frac32-H\right)^2} - \frac{2 T^{\frac32-H}}{\left(\frac32-H\right)\Gamma\left(\frac32-H\right)} +T\right)&\hspace{1cm}\\\nonumber
&\hspace{0.5cm} + 4\frac{1}{(2-2H) \Gamma\left(\frac32-H\right)^2} \left( \frac{\Gamma\left(\frac32-H\right)^2}{\Gamma(2-2H)} -1\right)^2 \|g\|_\infty^2 T^{2-2H}\\\nonumber
&\hspace{0.5cm} + 8(C')^2\|g\|_L^2  \frac{\left(H-\frac12\right)^2}{\Gamma(3-2H)^2} \frac{1}{4-2H} T^{4-2H} (1+G_H)^2\\\nonumber
&\hspace{0.5cm} +8\|g\|_L^2\sigma^2  \frac{\left(H-\frac12\right)^2}{\Gamma(2-H-\epsilon_H)} \frac{1}{2-2\epsilon_H} T^{2-2\epsilon_H} G_H^2 \\
&\leq C(g,T) \left( 1+ G_H^2\right)
\end{align*}
for some constant $C(T,g)>0$ depending only on $T$ and $g$, where we used that uniformly over $H\in (\frac12,\frac34]$, the terms involving the Gamma function are bounded away from zero and infinity and the $T$-dependent terms are bounded. Hence, by Lemma \ref{lemma_moment_Hoelder_norm}, for any integer $p>1$,
\begin{align}\label{eqp: G18}
\begin{split}
\E\left[ \left(\int_0^T v_H^g(s)^{2} ds\right)^{2p}\right]& \leq 2^{2p}C(g,T)^{2p} \E\left[1+ G_H^{2p}\right] \\
&\leq 2^{2p}C(g,T)^{2p} \left( 1+  (32)^{2p} (2(T\vee 1))^{2p} \frac{(4p)!}{(2p)!}\right),
\end{split}
\end{align}
where we used that $\epsilon_H$ in the definition of $G_H$ was chosen to be $\leq\frac14$. The bound is independent of $H$ and by Corollary $6.21$ in \cite{Klenke} the family in \eqref{eqp: G17} is uniformly integrable which completes the proof for $H>\frac12$.
\end{proof}

Remember that in accordance with the fractional version of Girsanov's theorem, Proposition~\ref{Girsanov_fBM}, we denote
\[ Z^{1/2}(u_s) = \exp\left( \int_0^T u_s dW_s - \frac12 \int_0^T u_s^2 ds\right) \]
for the corresponding likelihood in the case $H=\frac12$.

\begin{proposition}\label{continuity_likelihood}
Let $b\in\Sigma(C,A,\gamma,\sigma)\cap\mathcal{H}(\beta,L)$, $L>0$. For any bounded and Lipschitz continuous function $g:\R\rightarrow\R$ we have
\[ \E\left[\left| Z_T^H(g(X_\cdot^{H,b})) - Z_T^{1/2}(g( X_\cdot^b))\right|\right] \stackrel{H\to\frac12}{\longrightarrow} 0.\]
\end{proposition}
\begin{proof}
As $L^2(\Pr)$-convergence implies convergence in probability, we have by the continuous mapping theorem and Lemma \ref{lemma_continuity_likelihood} that
\begin{align*}
&\exp\left( \int_0^T v_H^g(s) dW_s -\frac12\int_0^T v_H^g(s)^2 ds\right)\\
&\hspace{4cm} \stackrel{H\to\frac12}{\longrightarrow}_\Pr \exp\left(\int_0^T g(X_s^b) dW_s -\frac12\int_0^T g(X_s^b)^2 ds\right)
\end{align*} 
for any bounded and Lipschitz continuous function $g$. We can conclude convergence in $L^1(\Pr)$ if the left-hand side is a uniformly integrable $H$-indexed family. We prove this by showing that its $L^2(\Pr)$-norm is uniformly bounded (see Corollary $6.21$ in \cite{Klenke}), i.e.
\begin{align}\label{eqp: G19}
\E\left[ \exp\left( 2\int_0^T v_H^g(s) dW_s -\int_0^T v_H^g(s)^2 ds\right)\right]\leq c
\end{align} 
for some $c$ independent of $H$. 
As $v_H^g(s)^2\geq 0$, it is enough to bound the expectation 
\[ \E\left[ \exp\left(2\int_0^T v_H^g(s) dW_s\right)\right] \]
and we use monotone convergence and the Burkholder-Davis-Gundy inequality (cf. \cite{Kallenberg}, Theorem~$20.12$) to bound the moments of the integral. The latter gives
\[ \E\left[ \left| \int_0^T v_H^g(s) dW_s\right|^p\right] \leq c_p \E\left[ \left(\int_0^T v_H^g(s)^2 ds\right)^\frac{p}{2}\right].\]
For the constant $c_p$ we have $c_p \leq C'p$ for some constant $C'>0$ (see \cite{Barlow/Yor}, Proposition $4.2$).
Now we again distinguish the cases $H<\frac12$ and $H>\frac12$. In the first one, we have
\[ \E\left[ \left(\int_0^T v_H^g(s)^2 ds\right)^\frac{p}{2}\right] \leq \|g\|_\infty^p  T^p, \]
which follows by the bound \eqref{eqp: G20} established in the proof of Lemma~\ref{lemma_continuity_likelihood} and hence by monotone convergence,
\begin{align*}
\E\left[ \exp\left(2\int_0^T v_H^g(s) dW_s\right)\right] &\leq  \sum_{k=0}^\infty \frac{1}{k!} \E\left[ \left| 2\int_0^T v_H^g(s)dW_s\right|^k\right]\\
&\leq 1+ \sum_{k=1}^\infty \frac{c_k 2^k\|g\|_\infty^k  T^k}{k!}\\
&\leq 1 + C'\sum_{k=1}^\infty \frac{(2\|g\|_\infty T)^k}{(k-1)!}\\
&\leq 1+ C'(2\|g\|_\infty T) e^{2\|g\|_\infty T} <\infty.
\end{align*}
In the case $H>\frac12$ we restrict to the interval $\frac12<H<\frac34$, which is possible as we consider convergence $H\searrow \frac12$. First, as above, 
\[ \E\left[ \exp\left(\int_0^T v_H^g(s) dW_s\right)\right]\leq 1+C'\sum_{k=1}^\infty  \frac{k 2^k}{k!} \E\left[ \left(\int_0^T v_H^g(s)^2 ds\right)^\frac{k}{2}\right]. \]
Next, using that for any random variable $X\in L^q(\Pr)$ and $1\leq p\leq q$,
\[ \E\left[ |X|^p\right] \leq \E\left[ |X|^q\right]^{p/q} \leq 1 + \E\left[ |X|^q\right],  \]
and \eqref{eqp: G18} subsequently,
\begin{align*}
&C'\sum_{k=1}^\infty  \frac{k 2^k}{k!} \E\left[ \left(\int_0^T v_H^g(s)^2 ds\right)^\frac{k}{2}\right] \\
&\hspace{0.5cm}\leq  4C'\sum_{k=1}^\infty \frac{4k 2^{4k}}{(4(k-1))!} \left( \E\left[ \left(\int_0^T v_H^g(s)^2 ds\right)^{2k}\right] +1\right) \\
&\hspace{0.5cm}\leq  4C'\sum_{k=1}^\infty \frac{4k 2^{4k}}{(4(k-1))!} \left( 2^{2k}C(T,g)^{2k}\left( 1+ (32)^{2k} (2(T\vee 1))^{2k}\frac{(4k)!}{(2k)!}\right) +1 \right) \\
&\hspace{0.5cm}\leq 4C'\sum_{k=1}^\infty \frac{(4k)^5}{(2k)!}\left( 2^5 C(g,T) 32 (T\vee 1) +2^2\right)^{2k}\\
&\hspace{0.5cm}\leq 4^6C'\sum_{k=1}^\infty \frac{k^5}{k!}\frac{1}{k!}\left(\left( 2^5 C(g,T) 32 (T\vee 1) +2^2\right)^2\right)^{k},
\end{align*}
where convergence of the series follows by a comparison with the exponential series as $k^5/k!$ is bounded.\\
In consequence, \eqref{eqp: G19} holds for both cases $H<\frac12$ and $H>\frac12$, which establishes the $L^1(\Pr)$-convergence of the Girsanov densities for a bounded and Lipschitz continuous function $g:\R\rightarrow\R$.
\end{proof}

\begin{proof}[Proof of Theorem \ref{Continuity_lower_H}]
As in the proof of Theorem~\ref{lower_bound}, we bound
\begin{align}\label{eqp: G25}
\begin{split}
&\sup_{\psi_T^H} \inf_{\substack{b\in H_1(b_0,\eta)\cap \{b-b_0\in\mathcal{H}(\beta,L)\}:\\ \Delta_J(b)\geq (1-\epsilon_T)c_*\delta_T}}\  \E \left[\psi_T^H(X^{H,b})\right] - \alpha \\
& \hspace{2cm}\leq  \E\left[\left(\frac{1}{N_T}\sum_{k=1}^{N_T} Z_T^H\left( (b_k-b_{0,\eta})(X^{H,b_{0,\eta}})\right) -1\right)\phi_T^H\right] \\
&\hspace{2cm} \leq \E\left[\left|\frac{1}{N_T}\sum_{k=1}^{N_T} Z_T^H\left( (b_k-b_{0,\eta})(X^{H,b_{0,\eta}})\right) -1\right|\right]
\end{split}
\end{align} 
and in particular choose the same hypothesis $b_k$ as in this proof, which are bounded by one and Lipschitz continuous. 
Now, by adding zero, this term is again bounded by the sum of 
\[\E\left[\left|\frac{1}{N_T}\sum_{k=1}^{N_T} Z_T^{1/2}\left( (b_k-b_{0,\eta})(X^{b_{0,\eta}})\right) -1\right|\right] \]
and 
\begin{align}\label{eqp: G23}
\begin{split}
& \E\left[ \left| \frac{1}{N_T} \sum_{k=1}^{N_T} \left( Z_T^H\left( (b_k-b_{0,\eta})(X^{H,b_{0,\eta}})\right)-Z_T^{1/2}\left( (b_k-b_{0,\eta})(X^{b_{0,\eta}})\right)\right)\right| \right]\\
&\hspace{1cm}\leq \frac{1}{N_T} \sum_{k=1}^{N_T} \E\left[ \left| Z_T^H\left( (b_k-b_{0,\eta})(X^{H,b_{0,\eta}})\right)-Z_T^{1/2}\left( (b_k-b_{0,\eta})(X^{b_{0,\eta}})\right)\right|\right].
\end{split}
\end{align} 
Let $\epsilon>0$. Then, by the proof of Theorem~\ref{lower_bound} in combination with Remark~\ref{remark_fixed_point},
\begin{align}\label{eqp: G24}
\limsup_{T\to\infty} \E\left[\left|\frac{1}{N_T}\sum_{k=1}^{N_T} Z_T^{1/2}\left( (b_k-b_{0,\eta})(X^{b_{0,\eta}})\right) -1\right|\right] <  \frac{\epsilon}{2}, 
\end{align} 
By Proposition~\ref{continuity_likelihood} there exists $\delta_k(T,\epsilon)>0$ such that for $|H-\frac12|<\delta_k(T,\epsilon)$,
\[ \E\left[\left| Z_T^H\left( (b_k-b_{0,\eta})(X^{H,b_{0,\eta}})\right)-Z_T^{1/2}\left( (b_k-b_{0,\eta})(X^{b_{0,\eta}})\right)\right|\right] <\frac{\epsilon}{2} \]
for every $k=1,\dots, N_T$.
Define $\delta(T) := \delta(T,\epsilon):=\min_{1\leq k\leq N}\delta_k(T,\epsilon)>0$. Then, 
\[ \sup_{|H-\frac12|<\delta(T)} \max_{1\leq k\leq N_T} \E\left[\left| Z_T^H\hspace{-0.05cm} \left( (b_k-b_{0,\eta})(X^{H,b_{0,\eta}})\right)\hspace{-0.05cm}-\hspace{-0.05cm} Z_T^{1/2}\hspace{-0.05cm}\left( (b_k-b_{0,\eta})(X^{b_{0,\eta}})\right)\right|\right] \]
is bounded from above by $\frac{\epsilon}{2}$. In particular, $\sup_{|H-\frac12|<\delta(T)} $ of the right-hand side of \eqref{eqp: G23} is bounded from above by $\frac{\epsilon}{2}$ and the claim of the theorem follows with \eqref{eqp: G25} and \eqref{eqp: G24}.
\end{proof}



\end{document}